\documentclass[10pt]{amsart}
\usepackage{amssymb, latexsym}
\usepackage{hyperref}

\usepackage{mathrsfs} 

\usepackage{backref}

\allowdisplaybreaks

\newcommand{\R}{{\mathbb{R}}}

\def\XXint#1#2#3{{\setbox0=\hbox{$#1{#2#3}{\int}$} \vcenter{\hbox{$#2#3$}}\kern-.5\wd0}}
 
\renewcommand{\phi}{\varphi}

\newcommand{\nm}[1]{\left\|#1\right\|} 
\newcommand{\abs}[1]{\left|#1\right|} 
\newcommand{\f}{\frac}

\renewcommand{\bar}[1]{\overline{#1}}

\newcommand{\lec}{\lesssim}

\newcommand{\Th}{\Theta}

\newcommand{\paren}[1]{\left(#1\right)}
\newcommand{\bracket}[1]{\left[#1\right]}
\newcommand{\braces}[1]{\left\{#1\right\}}

\newcommand{\HH}{\mathbb{H}}
\newcommand{\af}{\mathfrak{a}}
\newcommand{\aft}{\mathfrak a_t}
\newcommand{\taf}{\tilde{\mathfrak{a}}}
\newcommand{\taft}{\tilde{\mathfrak a}_t}
\newcommand{\tb}{\tilde b}
\newcommand{\Ac}{{\mathcal A}}
\newcommand{\tAc}{\tilde {\mathcal A}}
\newcommand{\tAone}{\tilde {A_1}}
\renewcommand{\th}{\tilde h}

\newcommand{\tD}{\tilde  D}

\newcommand{\Zf}{{\mathfrak{Z}}}
\newcommand{\zf}{{\mathfrak{z}}}

\newcommand{\tTheta}{\tilde{\Theta}}
\newcommand{\ttheta}{\tilde{\theta}}
\newcommand{\tEf}{\tilde{\mathfrak E}}

\renewcommand{\a}{\alpha} 
\renewcommand{\aa}{{\alpha '}} 

\newcommand{\bb}{{\beta '}} 

\renewcommand{\vec}[1]{{\bf #1}}

\let\Re=\undefined\DeclareMathOperator*{\Re}{Re}
\let\Im=\undefined\DeclareMathOperator*{\Im}{Im}

\theoremstyle{plain}
\newtheorem{theorem}{Theorem}
\newtheorem{proposition}[theorem]{Proposition}
\newtheorem{lemma}[theorem]{Lemma}

\theoremstyle{definition}
\newtheorem{definition}[theorem]{Definition}
\newtheorem{remark}[theorem]{Remark}

\newcounter{smalllist}

\numberwithin{equation}{section} \numberwithin{theorem}{section}
\begin{document}

\title[Wellposedness of the 2D water waves in a regime allowing for angled crests]{Wellposedness of the 2D full water wave equation in a regime that allows  for non-$C^1$ interfaces}
\author{Sijue Wu
}
\address{Department of Mathematics, University of Michigan, Ann Arbor, MI}

\thanks{Financial support in part by NSF grants DMS-1101434, DMS-1361791 and a Simons fellowship.}

\begin{abstract}
We consider the two dimensional gravity water wave equation in a regime where the free interface is allowed to be non-$C^1$.   
In this regime, only a degenerate Taylor inequality $-\frac{\partial P}{\partial\bold{n}}\ge 0$  holds, with degeneracy  at the singularities.  In \cite{kw} an energy functional $\mathcal E(t)$ was constructed 
and an a-prori estimate was proved. The energy functional $\mathcal E(t)$ is not only finite for interfaces and velocities in Sobolev spaces, but also finite for a class of non-$C^1$ interfaces with angled crests.  
In this paper  we prove the existence, uniqueness and stability of the solution of the 2d gravity water wave equation in the class where $\mathcal E(t)<\infty$, locally in time, 
 for any given data satisfying $\mathcal E(0)<\infty$.

\end{abstract}

\maketitle

\baselineskip15pt

\section{Introduction}

A class of water wave problems concerns the
motion of the 
interface separating an inviscid, incompressible, irrotational fluid,
under the influence of gravity, 
from a region of zero density (i.e. air) in 
$n$-dimensional space. It is assumed that the fluid region is below the
air region. Assume that
the density  of the fluid is $1$, the gravitational field is
$-{\bold k}$, where ${\bold k}$ is the unit vector pointing in the  upward
vertical direction, and at  
 time $t\ge 0$, the free interface is $\Sigma(t)$, and the fluid
occupies  region
$\Omega(t)$. When surface tension is
zero, the motion of the fluid is  described by 
\begin{equation}\label{euler}
\begin{cases}   \ \bold v_t + (\bold v\cdot \nabla) \bold v = -\bold k-\nabla P
\qquad  \text{on } \Omega(t),\ t\ge 0,
\\
\ \text{div}\,\bold v=0 , \qquad \text{curl}\,\bold v=0, \qquad  \text{on }
\Omega(t),\ t\ge 
0,
\\  
\ P=0, \qquad\qquad\qquad\qquad\qquad\text{on }
\Sigma(t) \\ 
\ (1, \bold v) \text{ 
is tangent to   
the free surface } (t, \Sigma(t)),
\end{cases}
\end{equation}
where $ \bold v$ is the fluid velocity, $P$ is the fluid
pressure. 
There is an important condition for these problems:
\begin{equation}\label{taylor}
-\frac{\partial P}{\partial\bold n}\ge 0
\end{equation}
pointwise on the interface, where $\bold n$ is the outward unit normal to the fluid interface 
$\Sigma(t)$ \cite{ta};
it is well known that when surface tension is neglected and the Taylor sign condition \eqref{taylor} fails, the water wave motion can be subject to the Taylor instability \cite{ ta, bi, bhl, ebi}.

The study on water waves dates back centuries. Early mathematical works include  Newton \cite{newton}, Stokes \cite{st}, Levi-Civita \cite{le}, and G.I. Taylor \cite{ta}.   Nalimov
\cite{na},  Yosihara \cite{yo} and Craig \cite{cr} proved local in time existence and uniqueness of solutions for the 2d water wave equation \eqref{euler} for small and smooth initial data. In  \cite{wu1, wu2}, we showed that for dimensions $n\ge 2$, the strong Taylor sign condition
 \begin{equation}\label{taylor-s}
-\frac{\partial P}{\partial\bold n}\ge c_0>0
\end{equation}
always holds for the
 infinite depth water wave problem \eqref{euler},  as long as the interface is in $C^{1+\epsilon}$, $\epsilon>0$; and the initial value problem of equation \eqref{euler} is locally well-posed in Sobolev spaces $H^s$, $s\ge 4$ for arbitrary given data.  Since then,
local wellposedness for water waves with additional effects such as the surface tension, bottom and non-zero vorticity,  under the assumption \eqref{taylor-s},\footnote{When there is surface tension, or bottom, or vorticity,  \eqref{taylor-s} does not always hold, it needs to be assumed.} were obtained, c.f. \cite{am, cl, cs, ig1, la, li, ot, sz, zz}.  Alazard, Burq \& Zuily \cite{abz, abz14} proved local wellposedness of \eqref{euler} in low regularity Sobolev spaces where the interfaces are only in $C^{3/2}$. Hunter, Ifrim \& Tararu \cite{hit} obtained a low regularity result for the 2d water waves that improves on \cite{abz}. 
The author \cite{wu3, wu4}, Germain, Masmoudi \& Shatah \cite{gms}, Ionescu \& Pusateri  \cite{ip} and Alazard \& Delort  \cite{ad} obtained almost global and global existence for two and three dimensional water wave equation \eqref{euler} for  small, smooth and localized data;  see \cite{hit, it, dipp, wang1, wang2, bmsw} for some additional developments. Furthermore in \cite{cf},
Castro, C\'ordoba, Fefferman, Gancedo and G\'omez-Serrano proved that for the 2d water wave equation \eqref{euler}, there exist initially non-self-intersecting interfaces that become self-intersecting at a later time; 
and as was shown in \cite{cs1}, the same result holds in 3d. 

All these work either prove or assume the strong Taylor sign condition \eqref{taylor-s}, and the lowest regularity 
considered are $C^{3/2}$ interfaces. 

A common phenomena  we observe in the ocean are waves with angled crests, with the interface  possibly non-$C^1$. 
A natural question is:
is the water wave equation \eqref{euler} well-posed in any class that includes non-$C^1$ interfaces?  

We focus on the two dimensional case in this paper.
 
 As was explained in \cite{kw}, the main difficulty in allowing for non-$C^1$ interfaces with angled crests is that in this case, both the quantity $-\frac{\partial P}{\partial \bf{n}}$ and the Dirichlet-to-Neumann operator $\nabla_{\bf n}$ degenerate, with degeneracy at the singularities on the interface;\footnote{We assume the acceleration is finite.}  and  only a weak Taylor inequality $-\f{\partial P}{\partial \vec{n}}\ge 0$ holds.  From earlier work \cite{wu1, wu2, am, la, abz, sz}, we know
the problem of solving the water wave equation \eqref{euler} can be reduced to solving a quasilinear equation of the interface $z=z(\alpha,t)$, of type
\begin{equation}\label{quasi1}
\partial_t^2\frak u+ a \nabla_{\bf n} \frak u=f(\frak u, \partial_t \frak u)
\end{equation}
where $a=-\frac{\partial P}{\partial \bf{n}}$. 
When the strong Taylor sign condition \eqref{taylor-s} holds and $\nabla_{\bf n}$ is non-degenerative,  equation \eqref{quasi1} is of the hyperbolic type with the right hand side consisting of lower order terms, and the Cauchy problem can be solved using classical tools.  In the case where the solution dependent quantity $a=-\frac{\partial P}{\partial \bf{n}}$ and  operator $\nabla_{\bf n}$ 
degenerate,  equation \eqref{quasi1} losses its hyperbolicity,  classical tools do not apply. 
New ideas are required to solve the problem.

In \cite{kw}, R. Kinsey and the author constructed an energy functional $\mathcal E(t)$ and proved an a-priori estimate,  
which states that for  solutions of the water wave equation \eqref{euler}, if $\mathcal E(0)<\infty$, then $\mathcal E(t)$ remains finite for a time period that depends only on $\mathcal E(0)$.
The energy functional $\mathcal E(t)$ 
is finite for interfaces and velocities in Sobolev classes,  and most importantly, it is also finite for a class of non-$C^1$ interfaces with angled crests.\footnote{In particular, the class where $\mathcal E(t)<\infty$ allows for angled crest type interfaces with interior angles at the crest $<\frac\pi 2$, which coincides with the range of the angles of the self-similar solutions in \cite{wu5}. Stokes extreme waves is not in the class where $\mathcal E(t)<\infty$. }
In this paper, we show that for any given data satisfying $\mathcal E(0)<\infty$, there is a $T>0$, depending only on $\mathcal E(0)$, such that 
the 2d water wave equation 
\eqref{euler} has a unique solution in the class where $\mathcal E(t)<\infty$ for time $0\le t\le T$, and the solution is stable.  We will work on the free surface equations that were  derived in \cite{wu1, wu2}. The novelty of this paper is that we study the degenerative case, and solve the equation in a broader class that includes non-$C^1$ interfaces.

\subsection{Outline of the paper} In \S\ref{notation1} we introduce some basic notations and conventions; further notations will be introduced  throughout  the paper.  In \S\ref{prelim}  we recall the  results in \cite{wu1, wu2}, and derive the free surface equation and its quasi-linearization,  
from system \eqref{euler}, in both the Lagrangian and Riemann mapping variables, for  interfaces and velocities in Sobolev spaces. 
We derived the quasilinear equation  in terms of the horizontal component in the Riemann mapping variable in \cite{wu1}, and in terms of   full components in the Lagrangian coordinates  in \cite{wu2}.
 Here we re-derive the equations for the sake of coherence. In \S\ref{general-soln} we will recover the water wave equation \eqref{euler} from the interface equation \eqref{interface-r}-\eqref{interface-holo}-\eqref{a1}-\eqref{b}, showing the equivalence of the two systems for smooth and non-self-intersecting interfaces. In \S\ref{a priori},
we present the energy functional $\mathcal E(t)$ constructed and the a-priori estimate proved in \cite{kw}.   
In \S\ref{prelim-result}, we  give a blow-up criteria in terms of the energy functional $\mathcal E(t)$ and 
a stability inequality for solutions of the interface equation \eqref{interface-r}-\eqref{interface-holo}-\eqref{a1}-\eqref{b} with a bound depending only on  $\mathcal E(t)$. In \S\ref{main} we present the main result, that is, 
the local in time wellposedness of the Cauchy problem for the water wave equation \eqref{euler}  in the class where $\mathcal E(t)<\infty$.  
In \S\ref{proof}, we give  the proof for the blow-up criteria,  Theorem~\ref{blow-up} and in \S\ref{proof3}, the stability inequality, Theorem~\ref{unique}. For the sake of completeness, we will also provide a proof for the a-priori estimate of \cite{kw} in the current setting in \S\ref{proof}.
In \S\ref{proof2}, we will prove the main result, Theorem~\ref{th:local}.

Some basic preparatory results are given in Appendix~\ref{ineq}; various identities that are useful for the paper are derived in Appendix~\ref{iden}. And in Appendix~\ref{quantities}, we  list the quantities that are controlled by $\mathcal E$. A majority of these are already shown in \cite{kw}.

{\bf Remark}: The blow-up criteria and the proof for the existence part of Theorem~\ref{th:local} are from  the unpublished manuscript of the author \cite{wu7}, with some small modifications. 

\subsection{Notation and convention}\label{notation1}
We consider solutions of the water wave equation \eqref{euler} in the setting where the fluid domain $\Omega(t)$ is simply connected, with the free interface $\Sigma(t):=\partial\Omega(t)$ being a Jordan curve,\footnote{That is, $\Sigma(t)$ is homeomorphic to the line $\mathbb R$.}
 $${\bold v}(z, t)\to 0,\qquad\text{as } |z|\to\infty$$
 and the interface $\Sigma(t)$ tending to horizontal lines at infinity.\footnote{The problem with velocity $\bold v(z,t)\to (c,0)$ as $|z|\to\infty$ can be reduced to the one with $\bold v\to 0$ at infinity by studying  the solutions in a moving frame. $\Sigma(t)$ may tend to two different lines at $+\infty$ and $-\infty$.} 
 
 We use the following notations and conventions:  $[A, B]:=AB-BA$ is the commutator of operators $A$ and $B$.  $H^s=H^s(\mathbb R)$ is the Sobolev space with norm $\|f\|_{H^s}:=(\int (1+|\xi|^2)^s|\hat f(\xi)|^2\,d\xi)^{1/2}$, $\dot H^{s}=\dot H^{s}(\mathbb R)$ is the Sobolev space with norm $\|f\|_{\dot H^{s}}:= c(\int |\xi|^{2s} |\hat f(\xi)|^2\,d\xi)^{1/2}$, $L^p=L^p(\mathbb R)$ is the $L^p$ space with $\|f\|_{L^p}:=(\int|f(x)|^p\,dx)^{1/p}$ for $1\le p<\infty$, and $f\in L^\infty$ if $\|f\|_{L^\infty}:=\text{ sup }|f(x)|<\infty$. When not specified, all the 
 norms $\|f\|_{H^s}$, $\|f\|_{\dot H^{s}}$, $\|f\|_{L^p}$, $1\le p\le\infty$ are in terms of the spatial variable only, and $\|f\|_{H^s(\mathbb R)}$, $\|f\|_{\dot H^{s}(\mathbb R)}$, $\|f\|_{L^p(\mathbb R)}$,  $1\le p\le\infty$ are in terms of the spatial variables. We say $f\in C^j([0, T], H^s)$ if the mapping $f=f(t):=f(\cdot, t): t\in [0, T]\to H^s$ is $j$-times continues differentiable, with $\sup_{[0, T], \ 0\le k\le j}\|\partial_t^k f(t)\|_{H^s}<\infty$; we say $f\in L^\infty([0, T], H^s)$ if $\sup_{[0, T]}\|f(t)\|_{H^s}<\infty$. 
 $C^j(X)$ is the space of $j$-times continuously differentiable functions on the set $X$; $C^j_0(\mathbb R)$ is the space of $j$-times continuously differentiable functions that decays at the infinity.

 Compositions are always in terms of the spatial variables and we write for $f=f(\cdot, t)$, $g=g(\cdot, t)$, $f(g(\cdot,t),t):=f\circ g(\cdot, t):=U_gf(\cdot,t)$. 
We identify $(x,y)$ with the complex number $x+iy$; $\Re z$, $\Im z$ are the real and imaginary parts of $z$; $\bar z=\Re z-i\Im z$ is the complex conjugate of $z$. $\overline \Omega$ is the closure of the domain $\Omega$, $\partial\Omega$ is the boundary of $\Omega$, ${\mathscr P}_-:=\{z\in \mathbb C: \Im z<0\}$ is the lower half plane. We write 
\begin{equation}\label{eq:comm}
[f,g; h]:=\frac1{\pi i}\int\frac{(f(x)-f(y))(g(x)-g(y))}{(x-y)^2}h(y)\,dy.
\end{equation}

We use $c$, $C$  to denote universal constants. $c(a_1,  \dots )$, $C(a_1, \dots)$, $M(a_1, \dots)$  are constants depending on $a_1, \dots $; constants appearing in different contexts need not be the same. We write $f\lec g$ if there is a universal constant $c$, such that $f\le cg$.  

\section{Preliminaries}\label{prelim}

Equation \eqref{euler} is a nonlinear equation defined on moving domains, it is difficult to study it directly. A classical approach is to reduce from \eqref{euler} to an equation on the interface, and study solutions of the interface equation. Then use the incompressibility and irrotationality of the velocity field to recover the velocity in the fluid domain by solving a boundary value problem for the Laplace equation. 

In what follows we derive the interface equations from \eqref{euler},  and vice versa; 
we  assume that the interface, velocity and acceleration are in Sobolev spaces.

\subsection{The equation for the free surface in Lagrangian variable}\label{surface-equation-l}

Let the free interface $\Sigma(t): z=z(\alpha, t)$, $\alpha\in\mathbb R$ be given by Lagrangian parameter $\alpha$, so $z_t(\alpha, t)={\bold v}(z(\alpha,t);t)$ is the velocity  of the fluid particles on the interface, $z_{tt}(\alpha,t)={\bold v_t + (\bold v\cdot \nabla) \bold v}(z(\alpha,t); t)$ is the acceleration. 
Notice that $P=0$ on $\Sigma(t)$ implies that $\nabla P$ is normal to $\Sigma(t)$, therefore $\nabla P=-i\frak a z_\alpha$, where 
\begin{equation}\label{frak-a}
\frak a =-\frac1{|z_\alpha|}\frac{\partial P}{\partial {\bold n}};
\end{equation}
 and the first and third equation of \eqref{euler} gives 
\begin{equation}\label{interface-l}
z_{tt}+i=i\frak a z_\alpha.
\end{equation}
The second equation of \eqref{euler}: $\text{div } \bold v=\text{curl } \bold v=0$ implies that $\bar {\bold v}$ is holomorphic in the fluid domain $\Omega(t)$, hence $\bar z_t$ is the boundary value of a holomorphic function in $\Omega(t)$. 

Let $\Omega\subset \mathbb C$ be a domain with boundary $\Sigma: z=z(\alpha)$, $\alpha\in  I$, oriented clockwise. Let $\mathfrak H$ be the Hilbert transform associated to $\Omega$:
\begin{equation}\label{hilbert-t}
\frak H f(\alpha)=\frac1{\pi i}\, \text{pv.}\int\frac{z_\beta(\beta)}{z(\alpha)-z(\beta)}f(\beta)\,d\beta
\end{equation}
We have the following characterization of the trace of a holomorphic function on $\Omega$.

\begin{proposition}\cite{jour}\label{prop:hilbe}
a.  Let $g \in L^p$ for some $1<p <\infty$. 
  Then $g$ is the boundary value of a holomorphic function $G$ on $\Omega$ with $G(z)\to 0$ at infinity if and only if 
  \begin{equation}
    \label{eq:1571}
    (I-\mathfrak H) g = 0.
  \end{equation}

b. Let $ f \in L^p$ for some $1<p<\infty$. Then $ \frac12(I+\mathfrak H)  f$ is the boundary value of a holomorphic function $\frak G$ on $\Omega$, with $\frak G(z)\to 0$ at infinity. 

c. $\mathfrak H1=0$.
\end{proposition}

By Proposition~\ref{prop:hilbe}  the second equation of \eqref{euler} is equivalent to 
$\bar z_t=\mathfrak H {\bar z_t}$. 
So the motion of the fluid interface $\Sigma(t): z=z(\alpha,t)$ is given by 
\begin{equation}\label{interface-e}
\begin{cases}
z_{tt}+i=i\frak a z_\alpha\\
\bar z_t=\frak H \bar z_t.
\end{cases}
\end{equation}
\eqref{interface-e} is a fully nonlinear equation.  In \cite{wu1},  Riemann mapping was introduced to analyze equation \eqref{interface-e} and to derive the quasilinear equation. 

\subsection{The free surface equation in Riemann mapping variable}\label{surface-equation-r}
Let $\Phi(\cdot, t): \Omega(t)\to  {\mathscr P}_-$ be the Riemann mapping taking $\Omega(t)$ to  the lower half plane ${\mathscr P}_-$, 
satisfying $\Phi(z(0,t),t)=0$ and $\lim_{z\to\infty}\Phi_z(z,t)=1$. Let 
\begin{equation}\label{h}
h(\alpha,t):=\Phi(z(\alpha,t),t),
\end{equation}
so $h(0,t)=0$ and $h:\mathbb R\to\mathbb R$ is a homeomorphism. Let $h^{-1}$ be defined by 
$$h(h^{-1}(\alpha',t),t)=\alpha',\quad \alpha'\in \mathbb R;$$
and 
\begin{equation}\label{1001}
Z(\alpha',t):=z\circ h^{-1}(\alpha',t),\quad Z_t(\alpha',t):=z_t\circ h^{-1}(\alpha',t),\quad Z_{tt}(\alpha',t):=z_{tt}\circ h^{-1}(\alpha',t)
\end{equation}
be the reparametrization of the position, velocity and acceleration of the interface in the Riemann mapping variable $\alpha'$. Let
\begin{equation}\label{1002}
Z_{,\alpha'}(\alpha', t):=\partial_{\alpha'}Z(\alpha', t),\quad Z_{t,\alpha'}(\alpha', t):=\partial_{\alpha'}Z_t(\alpha',t), \quad Z_{tt,\alpha'}(\alpha', t):=\partial_{\alpha'}Z_{tt}(\alpha',t), \text{ etc.}
\end{equation} 
Notice that ${\bar {\bold v}}\circ \Phi^{-1}: {\mathscr P}_-\to \mathbb C$ is holomorphic in the lower half plane ${\mathscr P}_-$ with  ${\bar {\bold v}}\circ \Phi^{-1}(\alpha', t)={\bar Z}_t(\alpha',t)$. Precomposing \eqref{interface-l}  with $h^{-1}$ and applying Proposition~\ref{prop:hilbe} to $
{\bar {\bold v}}\circ \Phi^{-1}$ in ${\mathscr P}_-$, we have the free surface equation in the Riemann mapping variable:
\begin{equation}\label{interface-r}
\begin{cases}
Z_{tt}+i=i\mathcal AZ_{,\alpha'}\\
\bar{Z}_t=\mathbb H \bar{Z}_t
\end{cases}
\end{equation}
where $\mathcal A\circ h=\frak a h_\alpha$ and $\mathbb H$ is the Hilbert transform associated with the lower half plane ${\mathscr P}_-$:
\begin{equation}\label{ht}
\mathbb H f(\alpha')=\frac1{\pi i}\text{pv.}\int\frac1{\alpha'-\beta'}\,f(\beta')\,d\beta'.
\end{equation}
Observe that $\Phi^{-1}(\alpha', t)=Z(\alpha', t)$ and $(\Phi^{-1})_{z'}(\alpha',t)=Z_{,\alpha'}(\alpha',t)$. So 
$Z_{,\alpha'}$, $\dfrac1{Z_{,\alpha'}}$ are boundary values of the holomorphic functions $(\Phi^{-1})_{z'}$ and $\dfrac1{(\Phi^{-1})_{z'}}$, tending to 1 at the spatial infinity. By Proposition~\ref{prop:hilbe}, \footnote{We work in the regime where $\frac1{Z_{,\alpha'}}-1\in L^2(\mathbb R)$. }
\begin{equation}\label{interface-holo}  
\frac1{Z_{,\alpha'}}-1=\mathbb H\paren{\frac1{Z_{,\alpha'}}-1}. \end{equation}
By the chain rule, we know  for any function $f$,  $U_h^{-1}\partial_t U_h f =(\partial_t+ b \partial_{\alpha'})f$, where 
 $$b:=h_t\circ h^{-1}.$$
 So $Z_{tt}=(\partial_t+b\partial_{\alpha'})Z_t$, and $Z_{t}=(\partial_t+b\partial_{\alpha'})Z$.  
 
 \subsubsection{Some additional notations}
We will often use the fact that $\mathbb H$ is purely imaginary,  and decompose a function into the sum of its holomorphic and antiholomorphic parts. We define 
the projections to the space of holomorphic functions in the lower, and respectively,  upper half planes  by
\begin{equation}\label{proj}
\mathbb P_H :=\frac12(I+\mathbb H),\qquad\text{and }\quad \mathbb P_A:=\frac12(I-\mathbb H).
\end{equation}

We also define \begin{equation}\label{da-daa}
D_\a =\dfrac {1}{z_\a}\partial_\a ,\quad \text{ and } \quad  D_\aa = \dfrac { 1}{Z_{,\aa}}\partial_\aa .
\end{equation}
We know by the chain rule that $\paren{D_\a f} \circ h^{-1}= D_\aa \paren{f \circ h^{-1}}$; and for any holomorphic
function $G$ on $\Omega(t)$ with boundary value $g(\alpha,t):=G(z(\alpha,t),t)$, $D_\a g=G_z\circ z$, and $D_\aa (g\circ h^{-1})=G_z\circ Z$. Hence $D_\a$, $D_\aa$ preserves the holomorphicity of $g$, $g\circ h^{-1}$.

 \subsubsection{The formulas for  $A_1$ and $b$.}

 Let $A_1:=\mathcal A |Z_{,\alpha'}|^2$. Notice that $\mathcal A\circ h=\frak a h_\alpha=-\frac{\partial P}{\partial\vec n}\frac{h_\alpha}{\abs{z_\alpha}}$, so 
  $A_1$ is related to the important quantity $-\frac{\partial P}{\partial\vec n}$ by
 $$-\frac{\partial P}{\partial\vec n}\circ Z=\frac{A_1}{\abs{Z_{,\aa}}}.$$
Using Riemann mapping,   we analyzed the quantities $A_1$ and $b$ in \cite{wu1}.   
Here we re-derive the formulas  for the sake of completeness; we will carefully note the a-priori assumptions made in the derivation. We mention that the same derivation can also be found in  \cite{wu6}. We also mention that in \cite{hit},  using the formulation of Ovsjannikov \cite{ovs},  the authors also re-derived the formulas \eqref{b} and \eqref{a1}. 
 
Assume that 
\begin{equation}\label{assume1}\lim_{z'\in {\mathscr P}_-, z'\to\infty} \Phi_t\circ \Phi^{-1}(z',t)=0,\qquad \lim_{z'\in {\mathscr P}_-,z'\to\infty} \braces{i(\Phi^{-1})_{z'} (z',t)-(\Phi^{-1})_{z'} \overline{ \vec{v}}_t\circ \Phi^{-1}(z',t)}=i. \ \footnote{ It was shown in \cite{wu1} that the water wave equation \eqref{euler} is well-posed in this regime.}\end{equation}
 \begin{proposition}[Lemma 3.1 and (4.7) of \cite{wu1}, or Proposition 2.2 and (2.18) of \cite{wu6}]\label{prop:a1}  
We have  \begin{equation}\label{b}
b:=h_t\circ h^{-1}=\Re (I-\mathbb H)\paren{\frac{Z_t}{Z_{,\alpha'}}};
\end{equation}
 \begin{equation}\label{a1}
A_1=1-\Im [Z_t,\mathbb H]{\bar Z}_{t,\alpha'}=1+\frac1{2\pi }\int \frac{|Z_t(\alpha', t)-Z_t(\beta', t)|^2}{(\alpha'-\beta')^2}\,d\beta'\ge 1;
\end{equation}
 \begin{equation}\label{taylor-formula}
-\frac{\partial P}{\partial\bold n}\Big |_{Z=Z(\cdot,t)}= \frac{A_1}{|Z_{,\alpha'}|}.
\end{equation}
In particular,  if the interface $\Sigma(t)\in C^{1+\epsilon}$ for some $\epsilon>0$, then the strong Taylor sign condition \eqref{taylor-s} holds.
\end{proposition}
\begin{proof}

Taking complex conjugate of the first equation in \eqref{interface-r}, then multiplying by $Z_{,\alpha'}$ yields
\begin{equation}\label{interface-a1}
Z_{,\alpha'}({\bar Z}_{tt}-i)=-i\mathcal A|Z_{,\alpha'}|^2:=-i A_1.
\end{equation}
The left hand side of \eqref{interface-a1} is almost holomorphic since $Z_{,\alpha'}$ is the boundary value of the holomorphic function $(\Phi^{-1})_{z'}$ and $\bar z_{tt}$ is the time derivative of the holomorphic function $\bar z_t$.  We explore the almost holomorphicity of $\bar z_{tt}$ by expanding.  Let $F=\bar {\bold v}$, we know $F$ is holomorphic in $\Omega(t)$ and $\bar z_t=F(z(\alpha, t),t)$, so
\begin{equation}\label{eq:1}
\bar z_{tt}=F_t(z(\alpha, t),t)+F_z(z(\alpha, t),t) z_t(\alpha, t),\qquad \bar z_{t\alpha}=F_z(z(\alpha, t),t) z_\alpha(\alpha, t)
\end{equation}
therefore \begin{equation}\label{eq:2}
\bar z_{tt}= F_t\circ z+ \frac{\bar z_{t\alpha}}{z_\alpha} z_t.
\end{equation}
 Precomposing with $h^{-1}$, subtracting $-i$, then multiplying by $Z_{,\alpha'}$, we have
$$Z_{,\alpha'}({\bar Z}_{tt}-i)= Z_{,\alpha'} F_t\circ Z+ Z_t {\bar Z}_{t,\alpha'}-i Z_{,\alpha'}=-iA_1 $$
Apply $(I-\mathbb H)$ to both sides of the equation. Notice that  $F_t\circ Z$ is the boundary value of the holomorphic function  $F_t\circ \Phi^{-1}$. By assumption \eqref{assume1} and Proposition~\ref{prop:hilbe},  $(I-\mathbb H)(Z_{,\alpha'} F_t\circ Z-iZ_{,\alpha'})=-i$;  therefore
$$-i(I-\mathbb H) A_1= (I-\mathbb H)(Z_t {\bar Z}_{t,\alpha'})-i$$
Taking imaginary parts on both sides and using the fact $(I-\mathbb H){\bar Z}_{t,\alpha'}=0$ \footnote{Because $(I-\mathbb H){\bar Z}_{t}=0$.}  to rewrite $(I-\mathbb H)(Z_t {\bar Z}_{t,\alpha'})$ as $[Z_t,\mathbb H]{\bar Z}_{t,\alpha'}$ yields
\begin{equation}\label{A_1}
A_1=1-\Im [Z_t,\mathbb H]{\bar Z}_{t,\alpha'}.
\end{equation}
The identity \begin{equation}
-\Im[Z_t,\mathbb H]{\bar Z}_{t,\alpha'}
=\frac1{2\pi }\int \frac{|Z_t(\alpha', t)-Z_t(\beta', t)|^2}{(\alpha'-\beta')^2}\,d\beta'
\end{equation}
is obtained by integration by parts. 

The quantity $b:=h_t\circ h^{-1}$ can be calculated similarly.
Recall $h(\alpha,t)=\Phi(z(\alpha,t),t)$, so 
$$h_t=\Phi_t\circ z+(\Phi_z\circ z) z_t,\qquad h_\alpha=(\Phi_z\circ z) z_\alpha$$
hence
$h_t= \Phi_t\circ z+ \frac{h_\alpha}{z_\alpha}  z_t$. Precomposing with $h^{-1}$ yields
\begin{equation}\label{b1}
h_t\circ h^{-1}=\Phi_t\circ Z+ \frac{Z_t}{Z_{,\alpha'}}.
\end{equation}
Now $\Phi_t\circ Z$ is the boundary value of the holomorphic function  $\Phi_t\circ \Phi^{-1}$.  By assumption \eqref{assume1} and Proposition~\ref{prop:hilbe}, $(I-\mathbb H)\Phi_t\circ Z=0$.  Apply $(I-\mathbb H)$ to both sides of \eqref{b1} then take the real parts, we get  
$$b=h_t\circ h^{-1}= \Re (I-\mathbb H)\paren{\frac{Z_t}{Z_{,\alpha'}}}.$$

A classical result in complex analysis states that if the interface is in $C^{1+\epsilon}$, $\epsilon>0$, tending to lines at infinity, then 
$c_0\le \abs{Z_{,\aa}}\le C_0$, for some constants $c_0, C_0>0$. So in this case, the strong Taylor sign condition \eqref{taylor-s} holds.

\end{proof}

\subsection{The quasilinear equation}\label{surface-quasi-r} In \cite{wu1, wu2} we showed that the quasi-linearization of the free surface equation \eqref{interface-e} can be accomplished by just taking one time derivative to equation \eqref{interface-l}. 

Taking derivative to $t$ to \eqref{interface-l} we get
\begin{equation}\label{quasi-l}
{\bar z}_{ttt}+i\frak a {\bar z}_{t\alpha}=-i\frak a_t {\bar z}_{\alpha}=\frac{\frak a_t}{\frak a} ({\bar z}_{tt}-i).
\end{equation}
Precomposing
with $h^{-1}$ on both sides of \eqref{quasi-l},  we have the equation in the Riemann mapping variable 
\begin{equation}\label{quasi-r1}
{\bar Z}_{ttt}+i\mathcal A {\bar Z}_{t,\alpha'}=\frac{\frak a_t}{\frak a}\circ h^{-1} ({\bar Z}_{tt}-i) 
\end{equation}
We compute $\dfrac{\frak a_t}{\frak a}$ by the identities $\frak a h_\alpha=\mathcal A\circ h$, and $\mathcal A\circ h=\dfrac{A_1}{\abs{Z_{,\aa}}^2}\circ h=A_1\circ h \dfrac{ h_\a^2}{\abs{z_{\a}}^2}$, so
\begin{equation}\label{eq:frak a}
\frak a =A_1\circ h \dfrac{ h_\a}{\abs{z_{\a}}^2};
\end{equation}
and we obtain, by taking derivative to $t$ to \eqref{eq:frak a}, 
$$
\dfrac{\frak a_t}{\frak a}= \frac{\partial_t \paren{ A_1\circ h}}{A_1\circ h}+\frac{h_{t\a}}{h_\a}-2\Re \frac{z_{t\a}}{z_\a}.
$$
Notice that $\frac{h_{t\a}}{h_\a}\circ h^{-1}=(h_t\circ h^{-1})_{\aa}:=b_\aa$. So
\begin{equation}\label{at}
\dfrac{\frak a_t}{\frak a}\circ h^{-1}= \frac{(\partial_t +b\partial_\aa) A_1}{A_1}+b_\aa -2\Re D_\aa Z_t;
\end{equation}
where we calculate from \eqref{b} that
\begin{equation}\label{ba}
\begin{aligned}
b_\aa&= \Re \paren{ (I-\mathbb H) \frac{Z_{t,\alpha'}}{Z_{,\aa}}+ (I-\mathbb H) \paren{Z_t\partial_\aa\frac 1{Z_{,\aa}}}}\\& =2\Re D_\aa Z_t+ \Re \paren{ (-I-\mathbb H) \frac{Z_{t,\alpha'}}{Z_{,\aa}}+ (I-\mathbb H) \paren{Z_t\partial_\aa \frac 1{Z_{,\aa}}}}\\&
=2\Re D_\aa Z_t+\Re \paren{\bracket{ \frac1{Z_{,\aa}}, \mathbb H}  Z_{t,\alpha'}+ \bracket{Z_t, \mathbb H}\partial_\aa \frac1{Z_{,\aa}}  }, 
\end{aligned}
\end{equation}
here in the last step we used the fact that  $(I+\mathbb H)Z_{t,\aa}=0$ and $(I-\mathbb H)\partial_\aa  \frac 1{Z_{,\aa}}=0$ \footnote{These follows from $(I-\mathbb H)\bar Z_t=0$ and  \eqref{interface-holo}.}  to rewrite the terms as commutators; and
we compute, by \eqref{A_1} and \eqref{eq:c14}, 
\begin{equation}\label{dta1}
(\partial_t +b\partial_\aa) A_1= -\Im \paren{\bracket{Z_{tt},\mathbb H}\bar Z_{t,\alpha'}+\bracket{Z_t,\mathbb H}\partial_\aa \bar Z_{tt}-[Z_t, b; \bar Z_{t,\aa}]}.
\end{equation}

We now sum up the above calculations and write the quasilinear system in the Riemann mapping variable. 
We have 
\begin{equation}\label{quasi-r}
\begin{cases}
(\partial_t+b\partial_\aa)^2{\bar Z}_{t}+i\dfrac{A_1}{\abs{Z_{,\aa}}^2}\partial_\aa {\bar Z}_{t}=\dfrac{\frak a_t}{\frak a}\circ h^{-1} ({\bar Z}_{tt}-i) 
\\ \bar Z_t=\mathbb H \bar Z_t
\end{cases}
\end{equation}
where
\begin{equation}\label{aux}
\begin{cases}
b:=h_t\circ h^{-1}=\Re (I-\mathbb H)\paren{\dfrac{Z_t}{Z_{,\alpha'}}}\\
A_1=1-\Im [Z_t,\mathbb H]{\bar Z}_{t,\alpha'}=1+\dfrac1{2\pi }\int \dfrac{|Z_t(\alpha', t)-Z_t(\beta', t)|^2}{(\alpha'-\beta')^2}\,d\beta'\\
\dfrac1{Z_{,\aa}}= i\dfrac{\bar Z_{tt}-i}{A_1}\\
\dfrac{\frak a_t}{\frak a}\circ h^{-1}= \dfrac{(\partial_t +b\partial_\aa) A_1}{A_1}+b_\aa -2\Re D_\aa Z_t
\end{cases}
\end{equation}
Here the third equation in \eqref{aux} is obtained by rearranging the terms of the equation \eqref{interface-a1}. 
Using it to replace $\frac1{Z_{,\aa}}$ by $i\frac{\bar Z_{tt}-i}{A_1}$,  we get a system for the complex conjugate velocity  and  acceleration $(\bar Z_t, \bar Z_{tt})$.  The initial data for the system \eqref{quasi-r}-\eqref{aux} is set up as follows.

\subsubsection{The initial data}\label{id-r}
Without loss of generality, we choose the parametrization of the initial interface $\Sigma(0): Z(\cdot,0):=Z(0)$  by the Riemann mapping variable, so $h(\alpha,0)=\alpha$ for $\alpha\in\mathbb R$; 
we take the initial  
velocity $Z_t(0)$,  
such that it satisfies $\bar Z_t(0)=\mathbb H \bar Z_t(0)$. 
And we take the initial acceleration $Z_{tt}(0)$ so that it solves the equation
\eqref{interface-a1} or the third equation in \eqref{aux}.

\subsection{Local wellposedness in Sobolev spaces}
 By \eqref{taylor-formula} and \eqref{a1},  if $Z_{,\aa}\in L^\infty$,  then the strong Taylor stability criterion \eqref{taylor-s} holds.  
In this case, the system \eqref{quasi-r}-\eqref{aux} is quasilinear of the hyperbolic type, with the left hand side of the first equation in \eqref{quasi-r} consisting of the higher order terms.\footnote{ $i\partial_\aa =|\partial_\aa|$ when acting on holomorphic functions. The Dirichlet-to-Neumann operator $\nabla_{\bf n}= \frac1{|Z_{,\aa}|}|\partial_\aa|$.} In \cite{wu1} we showed that the Cauchy problem of \eqref{quasi-r}-\eqref{aux},  and equivalently of \eqref{interface-r}-\eqref{interface-holo}-\eqref{b}-\eqref{a1},
 is uniquely solvable in Sobolev spaces $H^s$, $s\ge 4$.

Let the initial data be given as in \S\ref{id-r}.

\begin{theorem}[Local wellposedness in Sobolev spaces, cf. Theorem 5.11, \S6 of \cite{wu1}]\label{prop:local-s} Let $s\ge 4$. Assume that 
$Z_t(0)\in H^{s+1/2}(\mathbb R)$, $Z_{tt}(0)\in H^s(\mathbb R)$ and $Z_{,\aa}(0)\in L^\infty(\mathbb R)$. 
Then there is $T>0$,  such that on $[0, T]$, the initial value problem of \eqref{quasi-r}-\eqref{aux}, or equivalently of \eqref{interface-r}-\eqref{interface-holo}-\eqref{b}-\eqref{a1},  has a unique solution 
$Z=Z(\cdot, t)$,  satisfying
 $(Z_{t}, Z_{tt})\in C^
l([0, T], H^{s+1/2-l}(\mathbb R)\times H^{s-l}(\mathbb R))$,  and $Z_{,\alpha'}-1\in C^l([0, T], H^{s-l}(\mathbb R))$, for $l=0,1$.

Moreover if $T^*$ is the supremum over all such times $T$, then either $T^*=\infty$, or $T^*<\infty$, but
\begin{equation}\label{eq:1-1}
\sup_{[0, T^*)}(\|Z_{,\aa}(t)\|_{L^\infty}+\|Z_{tt}(t)\|_{H^3}+\| Z_t(t)\|_{H^{3+1/2}})=\infty. \end{equation}
\end{theorem}

\begin{remark} 1. Let $h=h(\a, t)$ be the solution of the ODE 
\begin{equation}\label{b1-1}
\begin{cases}
h_t=b(h,t),\\
h(\a, 0)=\a
\end{cases}
\end{equation}
 where $b$ is as given by \eqref{b}. Then $z=Z\circ h$ satisfies equation
 \eqref{interface-l}, cf. \S6 of \cite{wu1}.

2.    
 \eqref{quasi-r}-\eqref{aux} is a system for the complex conjugate velocity  and  acceleration $(\bar Z_t, \bar Z_{tt})$, the interface doesn't appear explicitly, 
so a solution can exist even if $Z=Z(\cdot, t)$ becomes self-intersecting. Similarly, equation \eqref{interface-r}-\eqref{interface-holo}-\eqref{b}-\eqref{a1} makes sense even if $Z=Z(\cdot, t)$ self-intersects. To obtain the solution of the water wave equation \eqref{euler} from the solution of the  quasilinear equation \eqref{quasi-r}-\eqref{aux} as given in Theorem~\ref{prop:local-s} above, in \S 6 of \cite{wu1}, 
an additional chord-arc condition  is assumed for the initial interface, and it was shown that the solution $Z=Z(\cdot, t)$ remains non-self-intersecting for a time period depending only on the initial chord-arc constant and the initial Sobolev norms. \footnote{\eqref{interface-r}-\eqref{interface-holo}-\eqref{b}-\eqref{a1} is equivalent to the water wave equation \eqref{euler} only when the interface is non-self-intersecting, see \S\ref{general-soln}.}   

3. 
Observe that we arrived at 
\eqref{quasi-r}-\eqref{aux} from \eqref{euler} using only the following properties of the domain: 1. there is a conformal mapping taking the fluid region $\Omega(t)$ to ${\mathscr P}_-$; 2. $P=0$ on $\Sigma(t)$. We note that  $z\to z^{1/2}$ is a conformal map that takes the region $\mathbb C\setminus \{z=x+i0, x>0\}$ to the upper half plane; so a domain with its boundary self-intersecting at the positive real axis can  be mapped conformally onto the lower half plane ${\mathscr P}_-$. Taking such a domain as the initial fluid domain, 
assuming  $P=0$ on $\Sigma(t)$ even when $\Sigma(t)$ self-intersects,\footnote{We note that when $\Sigma(t)$ self-intersects, the condition $P=0$ on $\Sigma(t)$ is unphysical.} one can still solve equation \eqref{quasi-r}-\eqref{aux} for a short time, by Theorem~\ref{prop:local-s}. Indeed this is one of the main ideas in the work of \cite{cf}. Using this idea and the time reversibility of the water wave equation, by choosing an appropriate initial velocity field that pulls the initial domain apart,  Castro, Cordoba et.\ al. \cite{cf} proved the existence of "splash" and "splat" singularities starting from a smooth non-self-intersecting fluid interface. 
 
\end{remark}

\subsection{Recovering the water wave equation \eqref{euler} from the interface equations}\label{general-soln} In this section, we derive the equivalent system in the lower half plane ${\mathscr P}_-$ for the interface equations 
\eqref{interface-r}-\eqref{interface-holo}-\eqref{a1}-\eqref{b}, and show how to recover from here the water wave equation \eqref{euler}. Although the derivation is quite straight-forward, to the best knowledge of the author, it has not been done before.

Let $Z=Z(\cdot, t)$ be a solution of \eqref{interface-r}-\eqref{interface-holo}-\eqref{a1}-\eqref{b}, 
satisfying the regularity properties  of Theorem~\ref{prop:local-s}; let $U(\cdot, t): \mathscr P_-\to \mathbb C$, $\Psi(\cdot, t): \mathscr P_-\to \mathbb C$ be the holomorphic functions, continuous on $\bar {\mathscr P}_-$, such that 
\begin{equation}\label{eq:270}
U(\alpha',t)=\bar Z_t(\alpha', t),\qquad \Psi(\alpha',t)=Z(\alpha',t),\qquad \Psi_{z'}(\alpha',t)=Z_{,\alpha'}(\alpha', t),
\end{equation}
and $\lim_{z'\to\infty} U(z',t)=0$, $\lim_{z'\to\infty}\Psi_{z'}(z',t)=1$.\footnote{We know $U(z', t)=K_{y'}\ast \bar Z_t$, $\Psi_{z'}=K_{y'}\ast Z_{,\aa}$ and by the Maximum principle, $\frac1{\Psi_{z'}}=K_{y'}\ast \frac1{Z_{,\aa}}$, here $K_{y'}$ is the Poisson kernel defined by \eqref{poisson}. By \eqref{interface-a1}, $\frac1{Z_{,\aa}}-1\in C([0, T], H^s(\mathbb R))$ for $s\ge 4$, so $\Psi_{z'}\ne 0$ on $\bar {\mathscr P}_-$.  } 
From $Z_t=(\partial_t+b\partial_\aa)Z=\Psi_t (\aa, t)+b \Psi_{z'}(\aa,t)$, we have
\begin{equation}\label{eq:271}
b=\frac{Z_t}{ \Psi_{z'} }-\frac{\Psi_t}{\Psi_{z'}}=\frac{\bar U}{\Psi_{z'}}-\frac{\Psi_t}{\Psi_{z'}},\qquad \text{on }\partial {\mathscr P}_- ,
\end{equation}
and substituting in we get 
$$\bar Z_{tt}=(\partial_t+b\partial_\aa) \bar Z_t=U_t+\paren{\frac{\bar U}{\Psi_{z'}}-\frac{\Psi_t}{\Psi_{z'}}   } U_{z'}\qquad \text{on }\partial {\mathscr P}_- ;$$
so $\bar Z_{tt}$ is the trace  of the function $U_t-\frac{\Psi_t}{\Psi_{z'}}U_{z'}+\frac{\bar U}{\Psi_{z'}}U_{z'} $ on $\partial {\mathscr P}_-$; and $Z_{,\alpha'}(\bar Z_{tt}-i)$ is then the trace of the function $\Psi_{z'} 
U_t- {\Psi_t}U_{z'}+{\bar U}U_{z'} -i\Psi_{z'}$  on $\partial {\mathscr P}_-$.  This gives, from \eqref{interface-a1} that 
\begin{equation}\label{eq:272}
\Psi_{z'} U_t- {\Psi_t}U_{z'}+{\bar U}U_{z'}-i\Psi_{z'}=-iA_1,\qquad \text{on }\partial {\mathscr P}_- .
\end{equation}
Observe that on the left hand side of \eqref{eq:272}, $\Psi_{z'} U_t- {\Psi_t}U_{z'}-i\Psi_{z'}$ is holomorphic on ${\mathscr P}_-$, 
while
${\bar U}U_{z'}=\partial_{z'}(\bar U U)$.  
So there is a real valued function $\frak P: {\mathscr P}_-\to \mathbb R$, such that
\begin{equation}\label{eq:273}
\Psi_{z'} U_t- {\Psi_t}U_{z'}+{\bar U}U_{z'}-i\Psi_{z'}=-2\partial_{z'}\frak P=-(\partial_{x'}-i\partial_{y'})\frak P,\qquad \text{on }{\mathscr P}_-;
\end{equation}
and by \eqref{eq:272}, because $iA_1$ is purely imaginary, 
\begin{equation}\label{eq:274}
\frak P=c,\qquad \text{on }\partial {\mathscr P}_-.
\end{equation}
where $c\in \mathbb R$ is a constant. 
Applying $\partial_{x'}+i\partial_{y'}:=2\overline{ \partial}_{z'}$ to \eqref{eq:273} yields
\begin{equation}\label{eq:275}
\Delta \frak P= -2|U_{z'}|^2\qquad\text{on }{\mathscr P}_-.
\end{equation}
It is easy to check that for $y'\le 0$ and $t\in [0, T]$, $\paren{U, U_t, U_{z'}, \Psi_{z'}-1, \frac1{\Psi_{z'}}-1, \Psi_t}(\cdot+iy', t)\in L^2(\mathbb R)\cap L^\infty(\mathbb R)$, and $(U, \Psi, \frac1{\Psi_{z'}}, \frak P)\in C^1( \overline{\mathscr P}_-\times [0, T])$.

It is clear that the above process is reversible. From a solution $(U, \Psi, \frak P)\in C^1( \overline{\mathscr P}_-\times [0, T])$ of the system \eqref{eq:273}-\eqref{eq:275}-\eqref{eq:274}-\eqref{eq:271}, with  $\paren{U, U_t, U_{z'}, \Psi_{z'}-1, \frac1{\Psi_{z'}}-1, \Psi_t}(\cdot+iy', t)\in L^2(\mathbb R)\cap L^\infty(\mathbb R)$ for $y'\le 0,\ t\in [0, T]$,  $U(\cdot, t)$, $\Psi(\cdot, t)$ holomorphic in ${\mathscr P}_-$, and  $b$ real valued, the boundary value $(Z(\aa,t), Z_t(\aa, t)):=(\Psi(\aa,t), \bar U(\aa, t))$ satisfies the interface equation \eqref{interface-r}-\eqref{interface-holo}-\eqref{b}-\eqref{a1}. Therefore the systems  \eqref{interface-r}-\eqref{interface-holo}-\eqref{b}-\eqref{a1} and  \eqref{eq:273}-\eqref{eq:275}-\eqref{eq:274}-\eqref{eq:271}, with $U(\cdot, t)$, $\Psi(\cdot, t)$ holomorphic in ${\mathscr P}_-$, and $\Psi_{z'}(\cdot, t)\ne 0$, $b$ real valued,  are equivalent in the smooth regime.

Assume $(U, \Psi)\in C( \overline{\mathscr P}_-\times [0, T])\cap C^1( {\mathscr P}_-\times (0, T))$ is a solution of the system \eqref{eq:273}-\eqref{eq:275}-\eqref{eq:274}, with  $U(\cdot, t)$, $\Psi(\cdot, t)$ holomorphic in ${\mathscr P}_-$, assume in addition that $\Sigma(t)=\{Z=Z(\alpha',t):=\Psi(\alpha',t)\ | \ \alpha'\in\mathbb R\}$  is a Jordan curve with $$\lim_{|\alpha'|\to\infty} Z_{,\alpha'}(\alpha',t)=1.$$ Let $\Omega(t)$ be the domain bounded by $Z=Z(\cdot,t)$ from the above, then $Z=Z(\alpha',t) $, $\alpha'\in\mathbb R$ winds the boundary of $\Omega(t)$ exactly once. By the argument principle, $\Psi: \bar {\mathscr P}_-\to \bar \Omega(t)$ is one-to-one and onto,  $\Psi^{-1}:\Omega(t)\to {\mathscr P}_-$
exists and is a holomorphic function; and by equation \eqref{eq:273} and  the chain rule, \begin{equation}\label{eq:276}
(U\circ \Psi^{-1})_t+\bar U\circ \Psi^{-1}(U\circ \Psi^{-1})_{z}+(\partial_x-i\partial_y)(\frak P\circ \Psi^{-1})=i,\qquad \text{on }\Omega(t).
\end{equation}
Let $\bar {\bold v}=U\circ \Psi^{-1}$, $P=\frak P\circ \Psi^{-1}$. Observe that ${\bf v}\bar{\bf v}_{z}= ({\bold v}\cdot \nabla) \bar {\bold v}$. So  $({\bf v}, P)$ satisfies the water wave equation \eqref{euler} in the  domain $\Omega(t)$.

\subsection{Non-$C^1$ interfaces} 
Assume that the interface $Z=Z(\cdot, t)$ has an angled crest at $\aa_0$  with interior angle  $\nu$, 
we know from the discussion in \S3.3.2 of \cite{kw} that  if the acceleration is finite, then it is necessary that $\nu\le \pi$; and if $\nu<\pi$ then 
$\frac 1{Z_{,\aa}}(\aa_0, t)=0$. We henceforth call those points at which  $\frac 1{Z_{,\aa}}=0$ the singularities.  

If the interface is allowed to be non-$C^1$ with interior angles at the crests $<\pi$, then the coefficient $\frac{A_1}{\abs{Z_{,\aa}}^2}$ of the second term on the left hand side of the first equation in \eqref{quasi-r} can be degenerative, and in this case 
  it is not clear if equation \eqref{quasi-r} is still hyperbolic. 
In order to handle this situation, we need to understand how the singularities  propagate. In what follows we derive the evolution equation for $\frac 1{Z_{,\aa}}$. We will also give the evolution equations for the basic quantities $\bar Z_t$ and  $\bar Z_{tt}$.\footnote{In Lagrangian coordinates, the first equation in \eqref{quasi-r} is of the form $(\partial_t^2+a\nabla_{\bf n} )\bar z_t=f$, where $a=-\frac{\partial P}{\partial {\bf n}}=\frac{A_1}{|Z_{,\aa}|}\circ h$, and the Dirichilet-Neumann operator $\nabla_{\bf n}\circ h^{-1}=\frac{i}{|Z_{,\aa}|}\partial_\aa$. So at the singularities both $a$ and $\nabla_{\bf n}$ are denegerative.}

\subsection{Some basic evolution equations} 
We begin with
$$\frac1{Z_{,\aa}}\circ h=\frac{h_\a}{z_\a},$$
 taking derivative to $t$  yields, 
$$\partial_t \paren{\frac1{Z_{,\aa}}\circ h}=\frac1{Z_{,\aa}}\circ h \paren{\frac{h_{t\a}}{h_\a}-\frac{z_{t\a}}{z_\a}};$$
precomposing with $h^{-1}$ gives
\begin{equation}\label{eq:dza}
(\partial_t+b\partial_{\alpha'})\paren{\frac1{Z_{,\aa}}}=\frac1{Z_{,\aa}} \paren{b_\aa-D_\aa Z_t}.
\end{equation}
The evolution equations for $\bar Z_t$ and $\bar Z_{tt}$ can be obtained from  \eqref{interface-a1} and \eqref{quasi-l}.

We have, by \eqref{interface-a1}, 
\begin{equation}\label{eq:dzt}
(\partial_t+b\partial_{\alpha'})\bar Z_t:= \bar Z_{tt} =-i \frac {A_1}{Z_{,\aa}}+i.
\end{equation}
Using \eqref{interface-l} to replace $i\frak a$ by $-\frac{\bar z_{tt}-i} {\bar z_\a}$ in equation \eqref{quasi-l} yields
$$\bar z_{ttt}=(\bar z_{tt}-i) \paren{ \bar { D_\a z_t}+ \frac{\frak a_t}{\frak a}};$$
precomposing with $h^{-1}$ gives 
\begin{equation}\label{eq:dztt}
(\partial_t+b\partial_{\alpha'})\bar Z_{tt}= (\bar Z_{tt}-i) \paren{ \bar { D_\aa Z_t}+ \frac{\frak a_t}{\frak a}\circ h^{-1}}.
\end{equation}

Equations \eqref{eq:dza}, \eqref{eq:dzt} and \eqref{eq:dztt} describe the time evolution of the basic quantities $\frac1{Z_{,\aa}}$, $\bar Z_t$ and $\bar Z_{tt}$. In fact,  equations \eqref{eq:dza}-\eqref{eq:dzt} together with \eqref{b}, \eqref{a1} and \eqref{ba} give a complete evolutionary system for the holomorphic quantities $\frac1{Z_{,\aa}}$ and $\bar Z_t$, which characterize the fluid domain $\Omega(t)$ and  the complex conjugate velocity $\bar {\vec{v}}$. We will explore this evolution system  in our future work. These equations give a first indication that it is natural to study the water wave problem in a setting where bounds are only imposed on $\frac1{Z_{,\aa}}$,  $\bar Z_t$ and their derivatives. 

\subsection{An important equation} Here we record an important equation, which is obtained by rearranging the terms of \eqref{interface-a1}. 
\begin{equation}\label{aa1}
\dfrac1{Z_{,\aa}}= i\dfrac{\bar Z_{tt}-i}{A_1}.
\end{equation}

\section{Well-posedness  in a broader class that includes non-$C^1$ interfaces.}\label{main-results}

We are now ready to study the Cauchy problem for the 
water wave equation \eqref{euler}
in a regime that allows for non-$C^1$ interfaces. 
We begin with an a-priori estimate.

\subsection{A-priori estimate for water waves with angled crests}\label{a priori}

Motivated by the question of the interaction of the free interface with a fixed vertical boundary, in \cite{kw}, Kinsey and the author studied the water wave equation \eqref{euler} in a regime that includes non-$C^1$ interfaces with angled crests in a periodic setting, 
constructed an energy functional and proved an a-priori estimate  
which does not require a positive lower bound for $\frac1{\abs{Z_{,\alpha'}}}$. 
A similar result holds for the whole line case. While a similar proof as that in \cite{kw} applies to the whole line, for the sake of completeness, 
we will provide  a slightly different argument  in \S\ref{proof0}. In the first proof  in \cite{kw}, we expanded and then re-organized the terms to ensure that there is no further cancelations and the estimates can be closed. Here instead we will rely on the estimates for the quantities $b_\aa$, $A_1$ and their derivatives.\footnote{These estimates become available in the work \cite{wu7}. The same results in the current paper hold in the periodic setting. 
}

Let 
\begin{equation}  \label{eq:ea}
 {\bf E}_a(t)=\int\frac1{A_1}|Z_{,\aa}(\partial_t+b\partial_\aa) D_\aa\bar Z_t |^2\,d\aa+ \nm{ D_\aa\bar Z_t(t)}_{\dot H^{1/2}}^2,
\end{equation}
and
\begin{equation} \label{eq:eb}
 {\bf E}_b(t) =  \int\frac1{A_1}\abs{Z_{,\aa}(\partial_t+b\partial_\aa)\paren{ \frac1{Z_{,\aa}}D_\aa^2\bar Z_t }}^2\,d\aa+  \nm{\frac1{Z_{,\aa}}D_\aa^2\bar Z_t (t)}_{\dot{H}^{1/2}}^2.
 \end{equation}
Let
\begin{equation}\label{energy}
\frak E(t)= {\bf E}_a(t)+{\bf E}_b(t)+ \|\bar{Z}_{t,\alpha'}(t)\|_{L^2}^2+\| D_{\alpha'}^2 \bar{Z}_t(t)\|_{L^2}^2+\nm{\partial_\aa\frac1{Z_{,\aa}}(t)}^2_{L^2} + \abs{\frac1{Z_{,\aa}}(0, t)}^2.
\end{equation}

\begin{theorem}[cf. Theorem 2 of \cite{kw} for the periodic version]\label{prop:a priori}
Let 
$Z=Z(\cdot,t)$, $t\in [0, T]$ be a solution of the system \eqref{interface-r}-\eqref{interface-holo}-\eqref{b}-\eqref{a1}, satisfying 
$(Z_{t}, Z_{tt})\in C^
l([0, T], H^{s+1/2-l}(\mathbb R)\times H^{s-l}(\mathbb R))$,  $l=0,1$ for some $s\ge 4$. There is a polynomial $C$ with universal  nonnegative coefficients, such that 
\begin{equation}\label{a priori-eq}
\frac{d}{dt}\frak E(t)\le C(\frak E(t)),\qquad \text{for } t\in [0, T].
\end{equation}

\end{theorem}

 For the sake of completeness we will give a proof of Theorem~\ref{prop:a priori} in \S\ref{proof}. 

\begin{remark}\label{remark3.2}

It appears that there is an  $\infty\cdot 0$ ambiguity in the definition of ${\bf E}_a$ and ${\bf E}_b$. This can be resolved by replacing the ambiguous quantities by the right hand sides of \eqref{2008-1} and \eqref{2010-2}. The same remark applies to Lemmas~\ref{basic-e}, ~\ref{basic-4-lemma},~\ref{dlemma1}. We opt for the current version for the clarity of the origins of the definitions and the more intuitive proofs.\footnote{The assumptions in Theorems~\ref{prop:a priori},~\ref{prop:a-priori},~\ref{unique} and Proposition~\ref{prop:energy-eq} is consistent with the completeness of the evolutionary equations \eqref{eq:dza}-\eqref{eq:dzt}.   We mention that to obtain the wellposed-ness result, Theorem~\ref{th:local}, we only apply Theorems~\ref{prop:a priori},~\ref{prop:a-priori},~\ref{unique} and Proposition~\ref{prop:energy-eq}  to solutions that satisfy in addition that $Z_{,\aa}\in L^\infty$.}

 By \eqref{eq:dza} and product rules, 
\begin{equation}\label{2008-1}
Z_{,\aa}(\partial_t+b\partial_\aa) D_\aa\bar Z_t=(b_\aa-D_\aa Z_t) \bar Z_{t,\aa}+(\partial_t+b\partial_\aa) \bar Z_{t,\aa}=
\bar Z_{tt,\aa}- (D_\aa Z_t) \bar Z_{t,\aa},
\end{equation}
and 
\begin{equation}\label{2010-2}
Z_{,\aa}(\partial_t+b\partial_\aa) \paren{\frac1{Z_{,\aa}}D^2_\aa\bar Z_t}=(b_\aa-D_\aa Z_t)D^2_\aa\bar Z_t+ (\partial_t+b\partial_\aa) D^2_\aa\bar Z_t.
\end{equation}
Let 
\begin{equation}\label{2010-3}
\begin{aligned}
\frak e(t)= \nm{\bar Z_{tt,\aa}(t)}_{L^2}^2&+\nm{D_\aa \bar Z_t(t)}_{\dot H^{1/2}}^2+ \nm{  D^2_\aa\bar Z_{tt}(t)}_{L^2}^2+\nm{\frac1{Z_{,\aa}}D^2_\aa\bar Z_t (t)}_{\dot H^{1/2}}^2\\&+
\|\bar{Z}_{t,\alpha'}(t)\|_{L^2}^2+\| D_{\alpha'}^2 \bar{Z}_t(t)\|_{L^2}^2+\nm{\partial_\aa\frac1{Z_{,\aa}}(t)}^2_{L^2} + \abs{\frac1{Z_{,\aa}}(0, t)}^2.
\end{aligned}
\end{equation}
It is easy to check that the argument in \S\ref{basic-quantities}  gives
\begin{equation}\label{equi-1}
\frak E(t)\lec c_1( \frak e(t) ),\qquad \text{and } \qquad \frak e(t)\lec c_2(\frak E(t)).
\end{equation}
for some universal polynomials $c_1=c_1(x)$ and $c_2=c_2(x)$. 

\end{remark}

In fact, as was shown in \S10 in \cite{kw}, we  have the following characterization, which is essentially a consequence of \eqref{equi-1} and equation \eqref{aa1},  of the energy functional $\frak E$  in terms of the holomorphic quantities $\frac1{Z_{,\aa}}$ and $\bar Z_t$. Since the proof in  \cite{kw} applies to the current setting,  
we omit the proof. 

Let \begin{equation}\label{energy1}
\begin{aligned}
\mathcal E(t)=\|\bar Z_{t,\alpha'}(t)\|_{L^2}^2&+ \|D_{\alpha'}^2\bar Z_t(t)\|_{L^2}^2+\nm{\partial_{\alpha'}\frac1{Z_{,\alpha'}}(t)}_{L^2}^2+\nm{D_{\alpha'}^2\frac1{Z_{,\alpha'}}(t)}_{L^2}^2\\&+\nm{\frac1{Z_{,\alpha'} }D_{\alpha'}^2\bar Z_t (t) }_{\dot H^{1/2}}^2+\| D_{\alpha'}\bar Z_t (t) \|_{\dot H^{1/2}}^2+\abs{\frac1{Z_{,\alpha'}}(0,t)}^2.
\end{aligned}
\end{equation}

\begin{proposition}[A characterization of $\frak E$ via $\mathcal E$, cf. \S10 of \cite{kw}] \label{prop:energy-eq} 
There are polynomials $C_1=C_1(x)$ and $C_2=C_2(x)$,
with nonnegative universal coefficients, such that for any solution $Z$ of \eqref{interface-r}-\eqref{interface-holo}-\eqref{b}-\eqref{a1}, satisfying the assumption of Theorem~\ref{prop:a priori},
\begin{equation}\label{energy-equiv}
\mathcal E(t)\le C_1(\frak E(t)),\qquad\text{and}\quad \frak E(t)\le C_2(\mathcal E(t)).
\end{equation}

\end{proposition}

A corollary of Theorem~\ref{prop:a priori} and Proposition~\ref{prop:energy-eq} is the following
\begin{theorem}[A-priori estimate \cite{kw}]\label{prop:a-priori} 
Let 
$Z=Z(\cdot,t)$, $t\in [0, T']$ be a solution of the system \eqref{interface-r}-\eqref{interface-holo}-\eqref{b}-\eqref{a1}, satisfying the assumption of Theorem~\ref{prop:a priori}.  
There are  constants $T=T(\mathcal E(0))>0$, $C=C(\mathcal E(0))>0$ that depend only on $\mathcal E(0)$, and with $-T(e)$, $C(e)$ increasing with respect to $e$,  such that 
\begin{equation}\label{a priori-e}
\sup_{[0, \min\{T, T'\}]}\mathcal E(t)\le C(\mathcal E(0))<\infty.
\end{equation}
\end{theorem}

\begin{remark}  
1. Let $t$ be fixed, $s\ge 2$,  and assume $Z_t(t)\in H^s(\R)$. By Proposition~\ref{B2} and Sobolev embeddings, $A_1(t)-1=-\Im[Z_t,\mathbb H]\bar Z_{t,\aa}\in H^s(\R)$; and by \eqref{aa1}, $Z_{tt}(t)\in H^s(\mathbb R)$ is equivalent to $\frac1{Z_{,\aa}}(t)-1\in H^s(\mathbb R)$.

2. Assume that $\paren{Z_t(t), \frac1{Z_{,\aa}}(t)-1}\in (H^{s+1/2}(\mathbb R), H^s(\mathbb R))$, $s\ge 2$, or equivalently $(Z_t(t), Z_{tt}(t))\in (H^{s+1/2}(\mathbb R), H^s(\mathbb R))$.   
It is easy to check  
that $\mathcal E(t)<\infty$. So in the class where $\mathcal E(t)<\infty$, it  allows  for interfaces and velocities in Sobolev classes; it is clear that in the class where $\mathcal E(t)<\infty$ it also allows for $\frac1{Z_{,\aa}}=0$, that is, singularities on the interface.

\end{remark}

\subsubsection{A description of the class $\mathcal E<\infty$ in $\mathscr P_-$}\label{e-1}
We give here an equivalent description of the class $\mathcal E<\infty$ in the lower half plane $\mathscr P_-$.

Let $1< p\le \infty$, and 
\begin{equation}\label{poisson}
K_y(x)=\frac{-y}{\pi(x^2+y^2)},\qquad y<0
\end{equation}
be the Poisson kernel. 
We know for any holomorphic function $G$ on $P_-$, 
$$\sup_{y<0}\|G(x+iy)\|_{L^p(\mathbb R,dx)}<\infty$$
if and only if there exists $g\in L^p(\mathbb R)$ such that $G(x+iy)=K_y\ast g(x)$. In this case,
$\sup_{y<0}\|G(x+iy)\|_{L^p(\mathbb R,dx)}=\|g\|_{L^p}$.  Moreover, if $g\in L^p(\mathbb R)$, $1<p<\infty$,  then $\lim_{y\to 0-} K_y\ast g(x)=g(x)$  in $L^p(\mathbb R)$ and if $g\in L^\infty\cap C(\mathbb R)$,  then $\lim_{y\to 0-} K_y\ast g(x)=g(x)$ for all $x\in\mathbb R$.

Let $Z=Z(\cdot, t)$ be a solution of \eqref{interface-r}-\eqref{interface-holo}-\eqref{a1}-\eqref{b}, satisfying the assumption of Theorem~\ref{prop:a priori}; 
let $\Psi$, $U$ be the holomorphic functions as given in \S\ref{general-soln}, so 
$$U(x'+iy',t)=K_{y'}\ast \bar Z_t(x', t),\qquad \frac1{\Psi_{z'}}(x'+iy',t)=K_{y'}\ast \frac1{Z_{,\aa}}(x',t),\qquad \text{for }y'< 0.$$
Let $z'=x'+iy'$.  We have 
\begin{equation}\label{domain-energy1}
\mathcal E(t)=\mathcal E_1(t)+\abs{\frac1{Z_{,\aa} }(0, t)  }^2,
\end{equation}
where
\begin{equation}\label{domain-energy}
\begin{aligned}
\mathcal E_1(t)&:=\sup_{y'<0}\nm{U_{z'}(\cdot+iy', t)}_{L^2(\mathbb R)}^2+\sup_{y'<0}\nm{\frac1{\Psi_{z'}}\partial_{z'}\paren{\frac1{\Psi_{z'} }U_{z'}}(\cdot+iy',t)}_{L^2(\mathbb R)}^2
\\&+\sup_{y'<0}\nm{\frac1{\{\Psi_{z'}\}^2}\partial_{z'}\paren{\frac1{\Psi_{z'} }U_{z'}}(\cdot+iy',t)}_{\dot H^{1/2}(\mathbb R)}^2+\sup_{y'<0}\nm{\frac1{\Psi_{z'} }U_{z'}(\cdot+iy',t)}_{\dot H^{1/2}(\mathbb R)}^2
\\&+\sup_{y'<0}\nm{ \frac1{\Psi_{z'} }\partial_{z'}\paren{\frac1{\Psi_{z'} }\partial_{z'}\paren{\frac1{\Psi_{z'} }}}(\cdot+iy',t)  }_{L^2(\mathbb R)}^2+\sup_{y'<0}\nm{ \partial_{z'}\paren{\frac1{\Psi_{z'} }}(\cdot+iy',t)  }_{L^2(\mathbb R)}^2.
\end{aligned}
\end{equation}

\subsection{A blow-up criteria and a stability inequality}\label{prelim-result}

The main objective of this paper is to show the unique solvability of the Cauchy problem for the water wave equation \eqref{euler} in the class where $\mathcal E<\infty$. We will build on the existing result, Theorem~\ref{prop:local-s}, by mollifying the initial data, constructing an approximating sequence and passing to the limit. However the existence time of the solution  as given in Theorem~\ref{prop:local-s} depends on the Sobolev norm of the initial data. In order to have an approximating sequence  defined on a time interval that has a uniform positive lower bound, we need a blow-up criteria; a uniqueness and stability theorem will allow us to prove the convergence of the sequence, and the uniqueness and stability of the solutions obtained by this process. 

Let the initial data be  as given in \S\ref{id-r}. 

\begin{theorem}[A blow-up criteria via $\mathcal E$]\label{blow-up}
Let $s\ge 4$. Assume $Z_{,\alpha'}(0)\in L^\infty (\mathbb R)$, $Z_t(0)\in H^{s+1/2}(\mathbb R)$ and $Z_{tt}(0)\in H^s(\mathbb R)$. Then there is $T>0$, such that on $[0, T]$, the initial value problem of \eqref{interface-r}-\eqref{interface-holo}-\eqref{b}-\eqref{a1} has a unique solution 
$Z=Z(\cdot, t)$,  satisfying
 $(Z_{t}, Z_{tt})\in C^
l([0, T], H^{s+1/2-l}(\mathbb R)\times H^{s-l}(\mathbb R))$ for $l=0,1$, and $Z_{,\alpha'}-1\in C([0, T], H^s(\mathbb R))$.

Moreover if $T^*$ is the supremum over all such times $T$, then either $T^*=\infty$, or $T^*<\infty$, but
\begin{equation}\label{eq:30}
\sup_{[0, T^*)}\mathcal E(t)=\infty 
\end{equation}

\end{theorem}

The proof for Theorem~\ref{blow-up} will be given in \S\ref{proof}. We now give the uniqueness and stability theorem. 

Let $Z=Z(\alpha',t)$, $\Zf=\Zf(\alpha',t)$ be solutions of the system \eqref{interface-r}-\eqref{interface-holo}-\eqref{b}-\eqref{a1}, with $z=z(\alpha,t)$, $\zf=\zf(\alpha,t)$ being their re-parametrizations in Lagrangian coordinates, and their initial data as given in \S\ref{id-r}; let $$Z_t,\ Z_{tt},\ Z_{,\aa},\ z_\a,\ h,\ A_1,\ \mathcal A, \ b,\ \frak a, \ D_\aa,\ D_\a, \  \frak H,\ \frak E(t), \ \mathcal E(t),\quad etc. $$ be the quantities associated with $Z$, $z$ as defined in \S\ref{prelim}, \S\ref{a priori}, and 
$$\Zf_t,\ \Zf_{tt},\ \Zf_{,\aa},\ \zf_\a,\ \th,\ \tAone,\ \tAc, \ \tb,\ \taf,\ \tD_\aa,\ \tD_\a, \  \tilde{\frak H}, \ \tilde{\frak E}(t), \ \tilde{\mathcal E}(t),\quad etc.$$
be the corresponding quantities for $\Zf$, $\zf$. Define
\begin{equation}\label{def-l}
l= \th\circ h^{-1}.
\end{equation}
so $l(\aa, 0)=\aa$, for $\aa\in \R$.
\begin{theorem}[Uniqueness and Stability in $\mathcal E<\infty$]\label{unique}
Assume that $Z$, $\frak Z$ are solutions of equation \eqref{interface-r}-\eqref{interface-holo}-\eqref{a1}-\eqref{b}, satisfying
$(Z_t, Z_{tt}), (\frak Z_t, \frak Z_{tt})\in C^l([0, T], H^{s+1/2-l}(\mathbb R)\times H^{s-l}(\mathbb R))$ for $l=0,1$, $s\ge 4$.
There is a constant $C$, depending only on $T$, $\sup_{[0, T]} \mathcal E(t)$ and $\sup_{[0, T]}\tilde{\mathcal E}(t)$, such that  
\begin{equation}\label{stability}
\begin{aligned}
&\sup_{[0, T]}\paren{\|\paren{\bar Z_t-\bar \Zf_t\circ l}(t)\|_{\dot{H}^{1/2}}+\|\paren{\bar Z_{tt}-\bar \Zf_{tt}\circ l}(t)\|_{\dot{H}^{1/2}}+\nm{\paren{\frac1{ Z_{,\aa}}-\frac 1{ \Zf_{,\aa}}\circ l}(t)}_{\dot{H}^{1/2}}}+\\&\sup_{[0, T]}\paren{\|\paren{l_\aa-1}(t)\|_{L^2}+\|D_\aa Z_t-(\tD_\aa \Zf_t)\circ l\|_{L^2}
 +\|(A_1-\tAone\circ l)(t)\|_{L^2}+\|(b_\aa-\tb_\aa\circ l)(t)\|_{L^2}}\\&\le C\paren{ \|\paren{\bar Z_t-\bar \Zf_t}(0)\|_{\dot{H}^{1/2}}+\|\paren{\bar Z_{tt}-\bar \Zf_{tt}}(0)\|_{\dot{H}^{1/2}}+\nm{\paren{\frac1{ Z_{,\aa}}-\frac 1{ \Zf_{,\aa}}}(0)}_{\dot{H}^{1/2}}}\\&+C\paren{\|\paren{D_\aa Z_t-(\tD_\aa \Zf_t)}(0)\|_{L^2}
 +\nm{\paren{\frac1{ Z_{,\aa}}-\frac 1{ \Zf_{,\aa}}}(0)}_{L^\infty}}
\end{aligned}
\end{equation}
\end{theorem}

By precomposing with $h$, we see that inequality \eqref{stability} effectively gives control of the differences, $z_t-\zf_t$, $z_{tt}-\zf_{tt}$ etc,  in Lagrangian coordinates. 

Notice that in the stability inequality \eqref{stability}, we control  the $\dot{H}^{1/2}$ norms of the differences of $Z_t$ and $\Zf_t\circ l$, $Z_{tt}$ and $\Zf_{tt}\circ l$, and  $\frac1{ Z_{,\aa}}$ and $\frac 1{ \Zf_{,\aa}}\circ l$, and the $L^2$ norms of the differences of $D_\aa Z_t$ and $(\tD_\aa \Zf_t)\circ l$, and $A_1$ and $\tAone\circ l$,  while  the energy functional $\frak E(t)$, or equivalently $\mathcal E(t)$, gives us control of the $L^2$ norms of $Z_{t,\aa}$, $Z_{tt,\aa}$ and $\partial_\aa\frac1{ Z_{,\aa}}$, and the $L^\infty$ and $\dot H^{1/2}$ norms \footnote{see \S\ref{basic-quantities} and \S\ref{hhalf-norm}.} of $D_\aa Z_t$ and $A_1$. Indeed, because the coefficient $\frac {A_1}{|Z_{,\aa}|^2}$ in equation \eqref{quasi-r} is solution dependent and possibly degenerative, for given solutions $Z=Z(\alpha', t)$, $\frak Z=\frak Z(\alpha', t)$ of  equation \eqref{quasi-r}-\eqref{aux},  the sets of zeros in $\frac1{Z_{,\alpha'}}(t)$ and $\frac1{\frak Z_{,\alpha'}}(t)$ are different and move with the solutions, hence one cannot simply subtract the two solutions and perform energy estimates, as is usually done in classical cases. 
Our approach is to first get a good understanding of the evolution of the degenerative factor $\frac1{Z_{,\aa}}$ via equation \eqref{eq:dza}, this allows us to construct a series of  model equations that capture the key degenerative features of the equation \eqref{quasi-r} to get some ideas of what would work. We then tailor the ideas to the specific structure of our equations. We give the proof for Theorem~\ref{unique} in \S\ref{proof3}.

\subsection{The wellposedness of the water wave equation \eqref{euler} in $\mathcal E<\infty$}\label{main}

Since it can be tricky to define solutions for the interface equation \eqref{interface-r} when the interface is allowed to have singularities, we will directly solve the water wave equation \eqref{euler} via the system \eqref{eq:273}-\eqref{eq:275}-\eqref{eq:274}-\eqref{eq:271}. 
As we know from the discussions in \S\ref{general-soln} and \S\ref{e-1},  equation \eqref{euler} is equivalent to \eqref{eq:273}-\eqref{eq:275}-\eqref{eq:274}, for $(U, \Psi)\in C(\overline{\mathscr P}_-\times [0, T])\cap C^1({\mathscr P}_-\times (0, T))$ with $U(\cdot, t)$, $\Psi(\cdot, t)$ holomorphic,  provided $\Psi(\cdot, t)$ is a Jordan curve; and the energy functionals $\mathcal E=\mathcal E_1+|\frac1{Z_{,\aa}}(0,t)|^2$. 
Observe that the energy functional $\mathcal E(t)$ does not give direct control of the lower order norms $\|Z_t(t)\|_{L^2}$, $\|Z_{tt}(t)\|_{L^2}$ and $\nm{\frac1{Z_{,\aa}}(t)-1}_{L^2(\mathbb R)}$;  
 in the class where we want to solve the water wave equation we require in addition that $Z_t(t)\in {L^2(\mathbb R)}$ and $\frac1{Z_{,\aa}}(t)-1\in {L^2(\mathbb R)}$.  This is consistent with the decay assumption made in \S\ref{notation1}.   

\subsubsection{The initial data}\label{id}
Let $\Omega(0)$ be the initial fluid domain, with the interface $\Sigma(0):=\partial\Omega(0)$ being a Jordan curve that tends to horizontal lines at the infinity, and let $\Psi(\cdot, 0):{\mathscr P}_-\to \Omega(0)$ 
be the Riemann Mapping such that $\lim_{z'\to\infty} \partial_{z'}\Psi(z', 0)=1$. We know $\Psi(\cdot, 0) :{\mathscr P}_-\to \Omega(0)$ is a homeomorphism. Let $Z(\alpha', 0):=\Psi(\alpha', 0)$ for $\aa\in\mathbb R$, so $Z=Z(\cdot, 0):\mathbb R\to\Sigma(0)$  is the parametrization of $\Sigma(0)$ in the Riemann Mapping variable. Let $\bold v(\cdot, 0):\Omega(0)\to \mathbb C$ be the initial velocity field, and $U(z', 0)=\bar{\bold v}(\Psi(z', 0),0)$. Assume $\bar{\bold v}(\cdot, 0)$ is holomorphic on $\Omega(0)$, so $U(\cdot, 0)$ is holomorphic on ${\mathscr P}_-$.  Assume that the energy functional $\mathcal E_{1}(0)$ for $(U(\cdot, 0),\Psi(\cdot, 0))$ as given in \eqref{domain-energy} satisfy $\mathcal E_1(0)<\infty$.
Assume in addition that \footnote{This is equivalent to  $\|U(\cdot+i0, 0)\|_{L^2(\mathbb R)}+ \nm{\frac1{Z_{,\aa}} (0)-1}_{L^2(\mathbb R)}    <\infty$, see \S\ref{e-1}.} 
\begin{equation}\label{iid}
c_0:=\sup_{y'<0}\|U(\cdot+iy', 0)\|_{L^2(\mathbb R)}+\sup_{y'<0}\nm{\frac1{\Psi_{z'}(\cdot+iy',0)}-1}_{L^2(\mathbb R)}<\infty.
\end{equation}

In light of  the discussion in \S\ref{general-soln} and the uniqueness and stability Theorem~\ref{unique}, we define solutions for the Cauchy problem of the system \eqref{eq:273}-\eqref{eq:275}-\eqref{eq:274} as follows.

\begin{definition}\label{de} 
Let the data be as given in \S\ref{id}, and 
$(U, \Psi, \frak P)\in C(\bar{\mathscr P}_-\times [0, T])$, with $(U, \Psi)\in C^1(\mathscr P_-\times (0, T))$, $\lim_{z'\to\infty} (U, \Psi_{z'}-1)(z',t)=(0,0)$ and
$U(\cdot, t)$, $\Psi(\cdot, t)$ holomorphic in the lower half plane ${\mathscr P}_-$ for $t\in [0, T]$. We say $(U, \Psi, \frak P)$  is a solution of the Cauchy problem of the system \eqref{eq:273}-\eqref{eq:275}-\eqref{eq:274}, if it satisfies the system \eqref{eq:273}-\eqref{eq:275}-\eqref{eq:274} on $\mathscr P_-\times [0, T]$, and  if there is a sequence  $Z_n=Z_n(\aa,t)$, $(\aa,t)\in \mathbb R\times[0, T]$, which are solutions of the system \eqref{interface-r}-\eqref{interface-holo}-\eqref{a1}-\eqref{b}, satisfying $(Z_{n,t}, \frac1{\partial_\aa Z_{n}}-1, \partial_\aa Z_{n} -1 )\in C^j([0, T],  H^{s+1/2-l}(\mathbb R)\times H^{s-l}(\mathbb R)\times H^{s-l}(\mathbb R)  )$ for some $s\ge 4$, $l=0,1$,  $\sup_{n, t\in [0, T]} \mathcal E_n(t)<\infty$ and $\sup_{n, t\in [0, T]}(\|Z_{n,t}(t)\|_{L^2}+ \nm{\frac1{\partial_\aa Z_{n}}(t)-1}_{L^2})<\infty$,  and the holomorphic extension $(U_n, \Psi_n)$ in ${\mathscr P}_-$ of $(\bar Z_{n,t}, Z_n)$, with $\lim_{z'\to\infty} (U_n, \partial_{z'}\Psi_{n}-1)(z',t)=(0,0)$, and the function $\frak P_n$ defined by \eqref{eq:273}-\eqref{eq:274}-\eqref{eq:275}, 
such that $\lim_{n\to \infty} U_n=U$, $\lim_{n\to \infty} \Psi_n=\Psi$, $\lim_{n\to \infty} \frak P_n=\frak P$ and $\lim_{n\to\infty}\frac1{\partial_{z'}\Psi_n}= \frac1{\partial_{z'}\Psi}$, uniformly on compact subsets of $\bar{\mathscr P}_-\times [0, T]$, and the data $(Z_n(\cdot, 0), Z_{n,t}(\cdot,0))$ converges in the topology of the right hand side of the inequality \eqref{stability} to the trace $(\Psi(\cdot+i0, 0), \bar U(\cdot+i0, 0))$. 
\end{definition}

Let $\mathcal E(0)=\mathcal E_1(0)+|\frac1{Z_{,\aa}}(0,0)|^2$.

\begin{theorem}[Local wellposedness in the $\mathcal E<\infty$ regime]\label{th:local}
1. There exists $T>0$, depending only on $\mathcal E(0)$, such that on $[0,T]$, the initial value problem of the  system \eqref{eq:273}-\eqref{eq:275}-\eqref{eq:274}  has a unique  solution $(U, \Psi, \frak P)$, with the properties that $U(\cdot, t),\Psi(\cdot, t)$ are holomorphic on ${\mathscr P}_-$ for each fixed $t\in [0, T]$, $U, \Psi, \frac1{\Psi_{z'}}, \frak P$ are continuous on $\bar {\mathscr P}_-\times [0, T]$,   $U, \Psi, \frak P$ are continuous differentiable on ${\mathscr P}_-\times [0, T]$,  
$\sup_{[0, T]}\mathcal E_1(t)<\infty$ and 
\begin{equation}\label{iidt}
\sup_{[0, T]}\sup_{y'<0}\paren{\|U(\cdot+iy', t)\|_{L^2(\mathbb R)}+\|\frac1{\Psi_{z'}(\cdot+iy',t)}-1\|_{L^2(\mathbb R)}}<\infty.
\end{equation}
The  solution $(U, \Psi, \frak P)$ gives rise to a solution $(\bar{\bold v}, P)=(U\circ \Psi^{-1}, \frak P\circ \Psi^{-1})$ of the water wave equation \eqref{euler} so long as $\Sigma(t)=\{Z=\Psi(\alpha',t)\ | \ \alpha'\in \mathbb R\}$ is a Jordan curve.

2. If in addition that the initial interface is chord-arc, that is, $Z_{,\alpha'}(\cdot,0)\in L^1_{loc}(\mathbb R)$ and there is $0<\delta<1$, such that
$$\delta \int_{\alpha'}^{\beta'} |Z_{,\alpha'}(\gamma,0)|\,d\gamma\le |Z(\alpha', 0)-Z(\beta', 0)|\le \int_{\alpha'}^{\beta'} |Z_{,\alpha'}(\gamma,0)|\,d\gamma,\quad \forall -\infty<\alpha'<  \beta'<\infty.$$
Then there is $T>0, T_1>0$, $T, T_1$ depend only on $\mathcal E(0)$, such that on $[0,  \min\{T, \frac{\delta}{T_1}\}]$, the initial value problem of the water wave equation \eqref{euler} has a unique solution, satisfying  $\mathcal E_1(t)<\infty$ and \eqref{iidt}, and the interface $Z=Z(\cdot, t)$ remains chord-arc.
\end{theorem}

We prove Theorem~\ref{th:local} in \S\ref{proof2}.

\section{The proof of  Theorem~\ref{prop:a priori} and  Theorem~\ref{blow-up}}\label{proof}

 We need the following basic  inequalities in the proof of Theorems~\ref{prop:a priori} and  \ref{blow-up}.
The basic energy inequality in Lemma~\ref{basic-e} has already appeared in \cite{wu3}. We give a proof nevertheless.

\begin{lemma}[Basic energy inequality I, cf. \cite{wu3}, lemma 4.1]\label{basic-e}
Assume $\Theta=\Theta(\alpha',t)$, $\alpha'\in \mathbb R$, $t\in [0, T)$ is smooth, decays fast at the spatial infinity, satisfying $(I-\mathbb H)\Theta=0$ and 
\begin{equation}\label{eq:40}
(\partial_t+b\partial_\aa)^2\Theta+i\mathcal A\partial_\aa \Theta=G_\Theta.
\end{equation}
Let
\begin{equation}\label{eq:41}
E_\Theta(t):=\int\frac1{\mathcal A}|(\partial_t+b\partial_\aa)\Theta|^2\,d\aa+ i\int(\partial_\aa\Theta) \bar\Theta\,d\aa.
\end{equation}
Then
\begin{equation}\label{eq:42}
\frac d{dt} E_\Theta(t)\le \nm{\frac{\frak a_t}{\frak a}\circ h^{-1}}_{L^\infty} E_\Theta(t)+2 E_\Theta(t)^{1/2}\paren{\int\frac{|G_\Theta|^2}{\mathcal A}\,d\aa}^{1/2}.
\end{equation}
\end{lemma}
\begin{remark}
By $\Theta=\mathbb H\Theta$ and \eqref{def-hhalf},
\begin{equation}\label{hhalf}
i\int(\partial_{\alpha'}\Theta) \bar\Theta\,d\alpha'=\int(i\partial_{\alpha'}\mathbb H \Theta) \bar\Theta\,d\alpha' =\|\Theta\|_{\dot{H}^{1/2}}^2\ge 0.
\end{equation}
\end{remark}
\begin{proof}
 By a change of the variables in  \eqref{eq:40},  we have
  $$(\partial_t^2+i\frak a \partial_\a)(\Theta\circ h)=G_\Th\circ h$$
where $\frak a h_\a=\mathcal A\circ h$; and  in  \eqref{eq:41}, 
$$E_\Theta(t)=\int\frac1{\frak a}|\partial_t(\Theta\circ h)|^2\,d\a+\int i\partial_\a (\Theta\circ h) \bar{\Theta\circ h}\,d\a.$$
So
\begin{equation}\label{eq:43}
\begin{aligned}
\frac d{dt} E_\Theta(t)&=\int 2\Re\braces{\frac1{\frak a}\partial_t^2(\Theta\circ h) \partial_t(\bar{\Theta\circ h})}-\frac{\frak a_t}{\frak a^2}|\partial_t(\Theta\circ h)|^2+2\Re \braces{i\int\partial_\alpha(\Theta\circ h)  \partial_t(\bar{\Theta\circ h}) \,d\alpha }
\\&= 2\Re \int \frac1{\frak a} G_\Th\circ h \partial_t(\bar{\Theta\circ h})\,d\alpha-\int \frac{\frak a_t}{\frak a^2}|\partial_t(\Theta\circ h)|^2\,d\a,
\end{aligned}
\end{equation}
where  we used integration by parts  in the first step. Changing back to the Riemann mapping variable, applying Cauchy-Schwarz inequality and \eqref{hhalf} yields  \eqref{eq:42}.

\end{proof}

We also need the following simple energy inequality.
\begin{lemma}[Basic energy inequality II]\label{basic-e2} Assume  $\Theta=\Theta(\aa,t)$ is smooth and decays fast at the spatial infinity. And assume
\begin{equation}\label{evolution-equation}
(\partial_t+b\partial_\aa)\Theta=g_\Theta.
\end{equation}
Then
\begin{equation}\label{basic-2}
\frac{d}{dt}\nm{\Theta(t)}_{L^2}^2\le 2\nm{g_\Th(t)}_{L^2}\nm{\Theta(t)}_{L^2}+\|b_\aa(t)\|_{L^\infty}\nm{\Theta(t)}_{L^2}^2
\end{equation}
\end{lemma}
\begin{proof}
We have, upon changing variables,
$$\int |\Theta(\aa,t)|^2\,d\aa=\int |\Th( h(\a, t),t)|^2h_\a \,d\a,$$
so
\begin{equation}
\begin{aligned}
\frac{d}{dt}\int |\Theta(\aa,t)|^2\,d\aa &=\int 2\Re \partial_t(\Th\circ h)\bar{\Th\circ h}\ h_\a+ |\Th\circ h|^2h_{t\a} \,d\a\\&= \int 2\Re \paren{(\partial_t+b\partial_\aa)\Theta}\bar\Th (\aa,t) + b_\aa  |\Theta(\aa,t)|^2\,d\aa;
\end{aligned}
\end{equation}
here in the second step we changed back to the Riemann mapping variable, and used the fact that $\frac{h_{t\a}}{h_\a}=b_\aa\circ h$. Inequality \eqref{basic-2} follows from Cauchy-Schwarz inequality. 
\end{proof}
Let \begin{equation}\label{P}
\mathcal P=(\partial_t+b\partial_\aa)^2+i\mathcal A\partial_\aa.
\end{equation} We need two more basic inequalities.
\begin{lemma}[Basic inequality III]\label{basic-3-lemma} Assume that $\Th=\Th(\aa,t)$ is smooth and decays fast at the spatial infinity, and assume $\Th=\mathbb H\Th$. Then
\begin{equation}\label{basic-3}
\begin{aligned}
&\nm{(I-\mathbb H)\paren{
\mathcal P\Th}(t)
}_{L^2}\le \nm{\partial_\aa (\partial_t+b\partial_\aa)b}_{L^\infty}\nm{\Th(t)}_{L^2}\\&\qquad\qquad\qquad+\nm{b_\aa}_{L^\infty}\nm{(\partial_t+b\partial_\aa)\Th(t)}_{L^2}+\nm{ b_\aa}_{L^\infty}^2\nm{\Th(t)}_{L^2}+\nm{\mathcal A_\aa}_{L^\infty}\nm{\Th(t)}_{L^2}.
\end{aligned}
\end{equation}
\end{lemma}
\begin{proof}
Because $\Th=\mathbb H\Th$, we have
$$(I-\mathbb H)(\mathcal P\Th)=\bracket{\mathcal P, \mathbb H}\Th;$$
and by \eqref{eq:c25},
$$\bracket{\mathcal P, \mathbb H}\Th=\bracket{(\partial_t+b\partial_\aa)b,\mathbb H}\partial_\aa \Th+2\bracket{b,\mathbb H}\partial_\aa (\partial_t+b\partial_\aa)\Th-[b,b; \partial_\aa \Th]+\bracket{i\mathcal A,\mathbb H}\partial_\aa \Th.
$$
Inequality \eqref{basic-3} follows from \eqref{3.20}.
\end{proof}
\begin{lemma}[Basic inequality IV]\label{basic-4-lemma}
Assume $f$ is  smooth and decays fast at the spatial infinity. Then
\begin{equation}\label{basic-4}
\begin{aligned}
\nm{Z_{,\aa}\bracket{\mathcal P, \frac1{Z_{,\aa}}}f}_{L^2}&\lec \nm{ (\partial_t+b\partial_\aa)(b_\aa-D_\aa Z_t)}_{L^\infty}\nm{f}_{L^2}\\&+  \nm{(b_\aa-D_\aa Z_t)}^2_{L^\infty}\nm{f}_{L^2}+  \nm{ (b_\aa-D_\aa Z_t)}_{L^\infty}\nm{(\partial_t+b\partial_\aa)f}_{L^2}\\&+  \nm{A_1}_{L^\infty}\nm{\frac1{Z_{,\aa}}\partial_\aa\frac1{Z_{,\aa}}}_{L^\infty}\nm{f}_{L^2}.
\end{aligned}
\end{equation}
\end{lemma}
\begin{proof}
Lemma~\ref{basic-4-lemma} is straightforward from the commutator relation \eqref{eq:c16}, identities
\eqref{eq:c26}, \eqref{eq:c27} and the definition $A_1:=\mathcal A\abs{Z_{,\aa}}^2$.
\end{proof}

Let $Z=Z(\cdot, t)$ be a solution of the system \eqref{interface-r}-\eqref{interface-holo}-\eqref{b}-\eqref{a1}, satisfying the assumption of Theorem~\ref{prop:a priori}.  By \eqref{quasi-r1} and \eqref{eq:c10},  we have
\begin{equation}\label{base-eq}
\mathcal P \bar Z_{t,\aa}=-(\partial_t+b\partial_\aa)(b_\aa \partial_{\aa}\bar Z_{t})-b_\aa\partial_\aa \bar Z_{tt}-i\mathcal A_\aa \partial_\aa \bar Z_t+\partial_\aa\paren{\frac{\frak a_t}{\frak a}\circ h^{-1} (\bar Z_{tt}-i)}
\end{equation}
Equation \eqref{base-eq} is our base equation in the proof of Theorems~\ref{prop:a priori} and \ref{blow-up}.

\subsection{The proof of  Theorem~\ref{prop:a priori}.}\label{proof0}
We begin with computing a few evolutionary equations. We have 
\begin{equation}\label{eq-dt}
\mathcal P D_\aa\bar Z_t=\bracket{\mathcal P, \frac1{Z_{,\aa}}}\bar Z_{t,\aa}+\frac1{Z_{,\aa}}\mathcal P \bar Z_{t,\aa};
\end{equation}
\begin{equation}\label{eq-ddt}
\mathcal P \paren{\frac1{Z_{,\aa}}D^2_\aa\bar Z_t}=\bracket{\mathcal P, \frac1{Z_{,\aa}}}D_\aa^2\bar Z_{t}+\frac1{Z_{,\aa}}\bracket{\mathcal P, D_\aa^2}\bar Z_{t}+\frac1{Z_{,\aa}}D_\aa^2\mathcal P\bar Z_t.\end{equation}
And, by the commutator identity \eqref{eq:c7} and the fact that $(\partial_t+b\partial_\aa)\bar Z_{t}=\bar Z_{tt}$, 
\begin{equation}\label{eq-zta}
(\partial_t+b\partial_\aa)\bar Z_{t,\aa}=\bar Z_{tt,\aa}-b_\aa \bar Z_{t,\aa};
\end{equation}
and by \eqref{eq:dza} and \eqref{eq:c7}
\begin{equation}\label{eq-ddza}
\begin{aligned}
(\partial_t+b\partial_\aa)\partial_{\alpha'} \frac{1}{ Z_{,\alpha'}}&=\partial_{\alpha'}(\partial_t+b\partial_\aa) \frac{1}{ Z_{,\alpha'}}+[(\partial_t+b\partial_\aa), \partial_{\alpha'}] \frac{1}{ Z_{,\alpha'}}\\&
=\paren{\partial_{\alpha'}\frac{1}{ Z_{,\alpha'}}}\paren{b_\aa-D_\aa Z_t}+D_\aa \paren{b_\aa-D_\aa Z_t} -b_{\alpha'}\partial_{\alpha'}\frac{1}{ Z_{,\alpha'}}\\&
=-D_\aa Z_t\paren{\partial_{\alpha'}\frac{1}{ Z_{,\alpha'}}}+D_\aa \paren{b_\aa-D_\aa Z_t}.
\end{aligned}
\end{equation}

We know from the definition of ${\bf E}_a(t)$,  ${\bf E}_b(t)$, and $A_1:=\mathcal A |Z_{,\aa}|^2$,
 $${\bf E}_a(t):=E_{D_\aa\bar Z_t}(t),\qquad\text{and }\quad {\bf E}_b(t):=E_{ \frac1{Z_{,\aa}}D^2_\aa\bar Z_t}(t),$$
where $E_\Th(t)$ is the basic energy as defined in \eqref{eq:41}. Notice that the quantities $D_\aa\bar Z_t$ and $
\frac1{Z_{,\aa}}D^2_\aa\bar Z_t$ are holomorphic. So the energy functional 
\begin{equation}\label{energy-functional}
\frak E(t)=E_{D_\aa\bar Z_t}(t)+ E_{ \frac1{Z_{,\aa}}D^2_\aa\bar Z_t}(t) +\|\bar Z_{t,\aa}(t)\|_{L^2}^2+\|D_\aa^2\bar Z_t(t)\|_{L^2}^2+\nm{\partial_\aa \frac{1}{ Z_{,\alpha'}}(t)}_{L^2}^2+\abs{\frac{1}{ Z_{,\alpha'}}(0,t)}^2.
\end{equation}
Our goal is to show that there is a universal polynomial $C=C(x)$, such that
\begin{equation}\label{energy-ineq}
\frac d{dt}\frak E(t)\le C(\frak E(t)).
\end{equation}

We begin with a list of quantities controlled by $\frak E(t)$. \subsubsection{Quantities controlled by $\frak E(t)$.}\label{basic-quantities} It is clear that $\frak E(t)$ controls the following quantities:
\begin{equation}\label{list1}
\nm{ D_\aa\bar Z_t}_{\dot H^{1/2}}, \quad\nm{\frac{1}{ Z_{,\alpha'}} D_\aa^2\bar Z_t}_{\dot H^{1/2}}, \quad \|\bar Z_{t,\aa}\|_{L^2},\quad \|D_\aa^2\bar Z_t\|_{L^2},\quad \nm{\partial_\aa \frac{1}{ Z_{,\alpha'}}}_{L^2},\quad \abs{\frac{1}{ Z_{,\alpha'}}(0,t)}.
\end{equation}
By \eqref{eq:b13} and \eqref{a1},
\begin{equation}\label{2000}
1\le A_1,\qquad{and}\quad \nm{ A_1}_{L^\infty}\lec 1+\|\bar Z_{t,\aa}\|_{L^2}^2\le 1+\frak E.
\end{equation}
We also have, by \eqref{ba} and \eqref{eq:b13}, that
\begin{equation}\label{2001}
\|b_\aa-2\Re D_\aa Z_t\|_{L^\infty}\lec \nm{\partial_\aa \frac{1}{ Z_{,\alpha'}}}_{L^2}\|\bar Z_{t,\aa}\|_{L^2}\le \frak E.
\end{equation}
We now estimate $\nm{D_\aa Z_t}_{L^\infty}$. We have, by the fundamental Theorem of calculus, 
\begin{equation}\label{2002}\paren{D_{\gamma'} \bar Z_t}^2-\int_0^1 \paren{D_\bb \bar Z_t}^2\,d\bb=2\int_0^1\int_\bb^{\gamma'}
D_\aa \bar Z_t\partial_\aa D_\aa \bar Z_t\,d\aa d\bb=2\int_0^1\int_\bb^{\gamma'}
\partial_\aa \bar Z_t D_\aa^2 \bar Z_t\,d\aa d\bb,
\end{equation}
where in the last equality, we moved $\frac1{Z_{,\aa}}$ from the first to the second factor. So for any $\gamma'\in \mathbb R$,
\begin{equation}\label{2003}
\abs{\paren{D_{\gamma'} \bar Z_t(\gamma',t)}^2-\int_0^1 \paren{D_\bb \bar Z_t(\bb,t)}^2\,d\bb}\le 2\|\bar Z_{t,\aa}\|_{L^2} \|D_\aa^2\bar Z_t\|_{L^2}\le 2\frak E.
\end{equation}
Now by the fundamental Theorem of calculus and Cauchy-Schwarz inequality we have,  for $\bb\in [0, 1]$, 
\begin{equation}\label{2004}\abs{\frac{1}{ Z_{,\bb}}(\bb, t)-\frac{1}{ Z_{,\alpha'}}(0,t)}\le \int_0^1\abs{\partial_\aa \frac{1}{ Z_{,\alpha'}}}\,d\aa\le \nm{\partial_\aa \frac{1}{ Z_{,\alpha'}}}_{L^2};
\end{equation}
so
$$\abs{\int_0^1 \paren{D_\bb \bar Z_t}^2\,d\bb}\le \nm{\frac{1}{ Z_{,\alpha'}}}_{L^\infty[0,1]}^2\|\bar Z_{t,\aa}\|_{L^2}^2\le \paren{\abs{\frac{1}{ Z_{,\alpha'}}(0,t)}+ \nm{\partial_\aa \frac{1}{ Z_{,\alpha'}}}_{L^2}}^2\|\bar Z_{t,\aa}\|_{L^2}^2\lec \frak E^2.
$$
Combining the above argument, we get
\begin{equation}\label{2005}
\nm{D_\aa Z_t}_{L^\infty}=\nm{D_\aa \bar Z_t}_{L^\infty}\lec C(\frak E).
\end{equation}
This together with \eqref{2001} gives us
\begin{equation}\label{2006}
\nm{b_\aa}_{L^\infty}\lec C(\frak E).
\end{equation}

We now explore the remaining terms in ${\bf E}_a(t)$ and ${\bf E}_b(t)$. 
We know 
\begin{equation}\label{2007}{\bf E}_a(t)=\int\frac1{A_1}|Z_{,\aa}(\partial_t+b\partial_\aa) D_\aa\bar Z_t |^2\,d\aa+ \nm{ D_\aa\bar Z_t}_{\dot H^{1/2}}^2.\end{equation}
Now by \eqref{eq:dza}, product rules and \eqref{eq-zta}, \footnote{One can also compute by changing to the Lagrangian coordinate and using the commutator relation \eqref{eq:c1}.}
\begin{equation}\label{2008}
Z_{,\aa}(\partial_t+b\partial_\aa) D_\aa\bar Z_t=(b_\aa-D_\aa Z_t) \bar Z_{t,\aa}+(\partial_t+b\partial_\aa) \bar Z_{t,\aa}=
\bar Z_{tt,\aa}- (D_\aa Z_t) \bar Z_{t,\aa};
\end{equation}
so
\begin{equation}\label{2009}
\begin{aligned}
\|Z_{tt,\aa}\|_{L^2}\le &\|D_\aa Z_t\|_{L^\infty}\|Z_{t,\aa}\|_{L^2}+\|Z_{,\aa}(\partial_t+b\partial_\aa) D_\aa\bar Z_t\|_{L^2}\\&\le \|D_\aa Z_t\|_{L^\infty}\|Z_{t,\aa}\|_{L^2}  + \paren{\|A_1\|_{L^\infty}{\bf E}_a}^{1/2}\lec C(\frak E).
\end{aligned}
\end{equation}
Similarly,
\begin{equation}\label{2010}
{\bf E}_b(t)=\int\frac1{A_1}\abs{Z_{,\aa}(\partial_t+b\partial_\aa) \paren{\frac1{Z_{,\aa}}D^2_\aa\bar Z_t} }^2\,d\aa+ \nm{ \frac1{Z_{,\aa}} D^2_\aa\bar Z_t}_{\dot H^{1/2}}^2,
\end{equation}
and by product rule and \eqref{eq:dza},
\begin{equation}\label{2010-1}
Z_{,\aa}(\partial_t+b\partial_\aa) \paren{\frac1{Z_{,\aa}}D^2_\aa\bar Z_t}=(b_\aa-D_\aa Z_t)D^2_\aa\bar Z_t+ (\partial_t+b\partial_\aa) D^2_\aa\bar Z_t;
\end{equation}
so
\begin{equation}\label{2011}
\nm{(\partial_t+b\partial_\aa) D^2_\aa\bar Z_t}_{L^2}\le \nm{Z_{,\aa}(\partial_t+b\partial_\aa) \paren{\frac1{Z_{,\aa}}D^2_\aa\bar Z_t}}_{L^2}+\|b_\aa-D_\aa Z_t\|_{L^\infty}\|D^2_\aa\bar Z_t\|_{L^2}\lec C(\frak E).
\end{equation}
Now from
\begin{align}
D_\aa^2 Z_t=\partial_\aa \frac1{Z_{,\aa}}D_\aa Z_{t}+\frac1{Z_{,\aa}^2}\partial_\aa^2 Z_t,\label{2012-1}\\
D_\aa^2 \bar Z_t=\partial_\aa \frac1{Z_{,\aa}}D_\aa \bar Z_{t}+\frac1{Z_{,\aa}^2}\partial_\aa^2 \bar Z_t,\label{2012-2}
\end{align}
we have
\begin{equation}\label{2012}
\|D_\aa^2 Z_t\|_{L^2}\le 2 \nm{\partial_\aa\frac1{Z_{,\aa}}}_{L^2}\|D_\aa Z_t\|_{L^\infty}+\|D_\aa^2 \bar Z_t\|_{L^2}\lec
 C(\frak E).
\end{equation}
Commuting $\partial_t+b\partial_\aa$ with $D^2_\aa$ by \eqref{eq:c2-1}, we get
\begin{equation}\label{2013}
D^2_\aa\bar Z_{tt}= (\partial_t +b\partial_\aa) D^2_\aa \bar Z_t+2(D_\aa Z_t) D_\aa^2\bar Z_t +(D_\aa^2 Z_t) D_\aa\bar Z_t;
\end{equation}
by \eqref{list1}, \eqref{2005}, \eqref{2012} and \eqref{2011}, we have
\begin{equation}\label{2014}
\|D^2_\aa\bar Z_{tt}\|_{L^2}\le C(\frak E).
\end{equation}
From \eqref{2014} and \eqref{2009}, we can work through the same argument as from \eqref{2002} to \eqref{2005} and get
\begin{equation}\label{2015}
\|D_\aa Z_{tt}\|_{L^\infty}=\|D_\aa \bar Z_{tt}\|_{L^\infty}\lec C(\frak E);
\end{equation}
and then by 
a similar calculation as in \eqref{2012-1}-\eqref{2012-2} and \eqref{2014}, \eqref{2015}, 
\begin{equation}\label{2016}
\|D^2_\aa Z_{tt}\|_{L^2}\le C(\frak E).
\end{equation}
Additionally,  by \eqref{eq-zta},
\begin{equation}\label{2017}
\|(\partial_t+b\partial_\aa)\bar Z_{t,\aa}\|_{L^2}\le \|Z_{tt,\aa}\|_{L^2}+\|b_\aa\|_{L^\infty} \|Z_{t,\aa}\|_{L^2}\lec C(\frak E).
\end{equation}
Sum up the estimates from \eqref{list1} through \eqref{2017}, we have that the following quantities are controlled by $\frak E$:
\begin{equation}\label{2020}
\begin{aligned}
&\nm{ D_\aa\bar Z_t}_{\dot H^{1/2}}, \quad\nm{\frac{1}{ Z_{,\alpha'}} D_\aa^2\bar Z_t}_{\dot H^{1/2}}, \quad \|\bar Z_{t,\aa}\|_{L^2},\quad \|D_\aa^2\bar Z_t\|_{L^2},\quad \nm{\partial_\aa \frac{1}{ Z_{,\alpha'}}}_{L^2},\quad \abs{\frac{1}{ Z_{,\alpha'}}(0,t)},\\&
\|A_1\|_{L^\infty}, \quad \|b_\aa\|_{L^\infty}, \quad \|D_\aa Z_t\|_{L^\infty},\quad \|D_\aa Z_{tt}\|_{L^\infty}, 
\quad \|(\partial_t+b\partial_\aa)\bar Z_{t,\aa}\|_{L^2} \\&
\|Z_{tt,\aa}\|_{L^2}, \quad \|D_\aa^2 \bar Z_{tt}\|_{L^2},\quad \|D_\aa^2  Z_{tt}\|_{L^2},\quad \|(\partial_t+b\partial_\aa)D_\aa^2\bar Z_t\|_{L^2},\quad \|D_\aa^2 Z_t\|_{L^2}.
\end{aligned}
\end{equation}
  
  We will use Lemmas~\ref{basic-e}-\ref{basic-4-lemma} to do estimates. Hence we need to control 
  the quantities that appear on the right hand sides of the inequalities in these Lemmas.

\subsubsection{Controlling $\nm{\frac{\frak a_t}{\frak a}\circ h^{-1}}_{L^\infty}$ and 
$\nm{(\partial_t+b\partial_\aa)A_1}_{L^\infty}$}\label{ata-da1}
 By \eqref{at}, 
$$\dfrac{\frak a_t}{\frak a}\circ h^{-1}= \frac{(\partial_t +b\partial_\aa) A_1}{A_1}+b_\aa -2\Re D_\aa Z_t.
$$
We have controlled $\|b_\aa\|_{L^\infty}$ and $\|D_\aa Z_t\|_{L^\infty}$ in \S\ref{basic-quantities}. We are left with the quantity $\nm{(\partial_t+b\partial_\aa)A_1}_{L^\infty}$. By \eqref{dta1}, 
$$
(\partial_t +b\partial_\aa) A_1= -\Im \paren{\bracket{Z_{tt},\mathbb H}\bar Z_{t,\alpha'}+\bracket{Z_t,\mathbb H}\partial_\aa \bar Z_{tt}-[Z_t, b; \bar Z_{t,\aa}]}.
$$
Applying \eqref{eq:b13} to the first two terms and \eqref{eq:b15} to the last we get
\begin{equation}\label{2021}
\nm{(\partial_t+b\partial_\aa)A_1}_{L^\infty}\lec \|Z_{tt,\aa}\|_{L^2}\|Z_{t,\aa}\|_{L^2}+\|b_\aa\|_{L^\infty}\|Z_{t,\aa}\|^2_{L^2}\lec C(\frak E);
\end{equation}
consequently
\begin{equation}\label{2022}
\nm{\frac{\frak a_t}{\frak a}\circ h^{-1}}_{L^\infty}\le \nm{(\partial_t+b\partial_\aa)A_1}_{L^\infty}+ 
\|b_\aa\|_{L^\infty}+2\|D_\aa Z_t\|_{L^\infty}\lec
C(\frak E).
\end{equation}

\subsubsection{Controlling $\nm{\mathcal A_\aa}_{L^\infty}$ and $\nm{\frac1{Z_{,\aa}}\partial_\aa \frac1{Z_{,\aa}}}_{L^\infty}$}\label{aa-zdz}
By \eqref{interface-r},  we have
\begin{equation}\label{2028}
i\mathcal A=\frac{Z_{tt}+i}{Z_{,\aa}}.
\end{equation}
Differentiating with respect to $\aa$ yields
\begin{equation}\label{2029}
i\mathcal A_\aa=(Z_{tt}+i)\partial_\aa\frac{1}{Z_{,\aa}}+D_\aa Z_{tt}.
\end{equation}
Apply $I-\mathbb H$ to both sides of the equation and use the fact that $\partial_\aa\frac{1}{Z_{,\aa}}=\mathbb H\paren{\partial_\aa\frac{1}{Z_{,\aa}}}$ to rewrite the first term on the right hand side as a commutator,  we get
\begin{equation}\label{2030}
i(I-\mathbb H)\mathcal A_\aa=\bracket{Z_{tt}, \mathbb H}\partial_\aa\frac{1}{Z_{,\aa}}+(I-\mathbb H)D_\aa Z_{tt}.
\end{equation}
Notice that $\mathcal A_\aa$ is purely real, so $\Im \paren{i(I-\mathbb H)\mathcal A_\aa}=\mathcal A_\aa$, and $|\mathcal A_\aa|\le |\paren{i(I-\mathbb H)\mathcal A_\aa}|$. Therefore,  
\begin{equation}\label{2031}
|\mathcal A_\aa|\le \abs{\bracket{Z_{tt}, \mathbb H}\partial_\aa\frac{1}{Z_{,\aa}}}+ 2|D_\aa Z_{tt}|+|(I+\mathbb H)D_\aa Z_{tt}|.
\end{equation}
We estimate the first term by \eqref{eq:b13},
$$\nm{\bracket{Z_{tt}, \mathbb H}\partial_\aa\frac{1}{Z_{,\aa}}}_{L^\infty}\lec \|Z_{tt,\aa}\|_{L^2}\nm{\partial_\aa\frac{1}{Z_{,\aa}}}_{L^2},$$
and the second term has been controlled in \S\ref{basic-quantities}. We are left with the third term, $(I+\mathbb H)D_\aa Z_{tt}$. We rewrite it by commuting out $\frac{1}{Z_{,\aa}}$:
\begin{equation}\label{2032}
(I+\mathbb H)D_\aa Z_{tt}=D_\aa (I+\mathbb H) Z_{tt} -\bracket{ \frac{1}{Z_{,\aa}},\mathbb H} Z_{tt,\aa},
\end{equation}
where we can estimate the second term by \eqref{eq:b13}. For the first term, we know $(I+\mathbb H)Z_t=0$  because $(I-\mathbb H)\bar Z_t=0$ and $\mathbb H$ is purely imaginary; 
 and $Z_{tt}=(\partial_t+b\partial_\aa)Z_t$. So
\begin{equation}\label{2033}
(I+\mathbb H) Z_{tt}=-[\partial_t+b\partial_\aa, \mathbb H]Z_t=-[b, \mathbb H]Z_{t,\aa}.
\end{equation}
We further rewrite it by \eqref{b}: 
\begin{equation}\label{bb}
b=\mathbb P_A\paren{\frac{Z_t}{Z_{,\aa}}}+\mathbb P_H\paren{\frac{\bar Z_t}{\bar Z_{,\aa}}}=\frac{\bar Z_t}{\bar Z_{,\aa}}+\mathbb P_A\paren{\frac{Z_t}{Z_{,\aa}}-\frac{\bar Z_t}{\bar Z_{,\aa}}},
\end{equation}
Prop~\ref{prop:comm-hilbe}, the fact that $(I+\mathbb H)Z_{t,\aa}=0$ and $(I+\mathbb H)\bar{D_\aa \bar Z_{t}}=0$. We have
\begin{equation}\label{2034}
(I+\mathbb H) Z_{tt}=-\bracket{\frac{\bar Z_t}{\bar Z_{,\aa}}  , \mathbb H}Z_{t,\aa}=-\bracket{\bar Z_t , \mathbb H}\bar{D_\aa \bar Z_{t}}.
\end{equation}
 
 We have reduced the task of estimating $D_\aa (I+\mathbb H) Z_{tt}$  to estimating $D_\aa \bracket{\bar Z_t , \mathbb H}\bar{D_\aa \bar Z_{t}}$. 
We  compute, for general functions $f$ and $g$, 
 \begin{equation}\label{2026}
\partial_\aa [f,\mathbb H]g= f_\aa \mathbb H g-\frac1{\pi i}\int\frac{(f(\aa)-f(\bb))}{(\aa-\bb)^2}g(\bb)\,d\bb
\end{equation}
therefore
\begin{equation}\label{2027}
\begin{aligned}
D_\aa & [f,\mathbb H]g= \frac1{Z_{,\aa}} f_\aa \mathbb H g\\&-\frac1{\pi i}\int\frac{\paren{f(\aa)-f(\bb)}\paren{\frac1{Z_{,\aa}}-\frac1{Z_{,\bb}}}}{(\aa-\bb)^2}g(\bb)\,d\bb-\frac1{\pi i}\int\frac{(f(\aa)-f(\bb))}{(\aa-\bb)^2} \frac1{Z_{,\bb}}g(\bb)\,d\bb.
\end{aligned}
\end{equation}
Now using \eqref{2034},  \eqref{2027}, and the fact that  $(I+\mathbb H)\bar{D_\aa \bar Z_{t}}=0$,  we have
\begin{equation}\label{2035}
\begin{aligned}
D_\aa & (I+\mathbb H) Z_{tt}=\abs{D_\aa \bar Z_t}^2
\\&+\frac1{\pi i}\int\frac{\paren{\bar Z_t(\aa)-\bar Z_t(\bb)}\paren{\frac1{Z_{,\aa}}-\frac1{Z_{,\bb}}}}{(\aa-\bb)^2}\bar{D_\bb \bar Z_{t}}\,d\bb+\frac1{\pi i}\int\frac{(\bar Z_t(\aa)-\bar Z_t(\bb))}{(\aa-\bb)^2}\frac1{Z_{,\bb}}\bar{D_\bb \bar Z_{t}}\,d\bb,
\end{aligned}
\end{equation}
where we rewrite the third term further 
\begin{equation}\label{2036}
\begin{aligned}
\frac1{\pi i}\int\frac{(\bar Z_t(\aa)-\bar Z_t(\bb))}{(\aa-\bb)^2}\frac1{Z_{,\bb}}\bar{D_\bb \bar Z_{t}}\,d\bb=&
\frac1{\pi i}\int\frac{(\bar Z_t(\aa)-\bar Z_t(\bb))}{(\aa-\bb)^2}\paren{\frac1{Z_{,\bb}}\bar{D_\bb \bar Z_{t}}-\frac1{Z_{,\aa}}\bar{D_\aa \bar Z_{t}}}\,d\bb\\&+\frac1{Z_{,\aa}}\bar{D_\aa \bar Z_{t}}\,\bar Z_{t,\aa};
\end{aligned}
\end{equation}
here we simplified the second term on the right hand side by the fact that $\bar Z_{t}=\mathbb H\bar Z_t$.

We can now estimate $\nm{D_\aa  (I+\mathbb H) Z_{tt}}_{L^\infty}$. We apply \eqref{eq:b16} to the second term on the right side of  \eqref{2035}; for the third term we use \eqref{2036}, and apply \eqref{eq:b16} to the first term on the right hand side of \eqref{2036}, and notice that
\begin{equation}\label{2037}
\partial_\aa \paren{\frac1{Z_{,\aa}}\bar{D_\aa \bar Z_{t}}}=\partial_\aa \paren{\frac1{Z_{,\aa}}}\bar{D_\aa \bar Z_{t}}+D_\aa \bar{D_\aa \bar Z_{t}};
\end{equation}
we have
\begin{equation}\label{2038}
\nm{D_\aa  (I+\mathbb H) Z_{tt}}_{L^\infty}\lec \nm{D_\aa \bar Z_t}^2_{L^\infty}+\nm{\partial_\aa\frac1{Z_{,\aa}}}_{L^2}\nm{Z_{t,\aa}}_{L^2}\nm{D_\aa \bar Z_t}_{L^\infty}+\nm{Z_{t,\aa}}_{L^2}\nm{D_\aa^2 \bar Z_t}_{L^2}.
\end{equation}

Sum up the calculations from \eqref{2031} through \eqref{2038}, and use the estimates in \S\ref{basic-quantities}, we conclude
\begin{equation}\label{2039}
\nm{\mathcal A_\aa}_{L^\infty}\lec C(\frak E).
\end{equation}
Observe that  the same argument also gives, by taking the real parts in \eqref{2030}, 
\begin{equation}\label{2039-1}
\nm{\mathbb H\mathcal A_\aa}_{L^\infty}\lec C(\frak E).
\end{equation}

Now from \eqref{2029} and \eqref{aa1},
\begin{equation}
\frac{iA_1}{\bar Z_{,\aa}}\partial_\aa\frac{1}{Z_{,\aa}} =i\mathcal A_\aa- D_\aa Z_{tt};
\end{equation}
Because $A_1\ge 1$, we have
\begin{equation}\label{2040}
\nm{\frac{1}{ Z_{,\aa}}\partial_\aa\frac{1}{Z_{,\aa}} }_{L^\infty}\le \nm{\mathcal A_\aa}_{L^\infty}+\| D_\aa Z_{tt}\|_{L^\infty}\lec C(\frak E).
\end{equation}

  \subsubsection{Controlling   $\nm{\partial_\aa (\partial_t+b\partial_\aa)\frac1{Z_{,\aa}}}_{L^2}$ and $\nm{ (\partial_t+b\partial_\aa)\partial_\aa\frac1{Z_{,\aa}}}_{L^2}$ }\label{dadtza}
  We begin with \eqref{eq-ddza}, and rewrite the second term on the right hand side to get
  \begin{equation}\label{ddza}
  \begin{aligned}
(\partial_t+b\partial_\aa)\partial_{\alpha'} \frac{1}{ Z_{,\alpha'}}
&=-D_\aa Z_t\paren{\partial_{\alpha'}\frac{1}{ Z_{,\alpha'}}}+D_\aa \paren{b_\aa-D_\aa Z_t}\\&
=-D_\aa Z_t\paren{\partial_{\alpha'}\frac{1}{ Z_{,\alpha'}}}+D_\aa \paren{b_\aa-2\Re D_\aa Z_t}+D_\aa \bar {D_\aa Z_t}.
\end{aligned}
\end{equation}
  We control the first and third terms by
  \begin{equation}\label{2024}
\nm{  D_\aa Z_t\paren{\partial_{\alpha'}\frac{1}{ Z_{,\alpha'}}}}_{L^2}\le \nm{  D_\aa Z_t}_{L^\infty}\nm{\partial_{\alpha'}\frac{1}{ Z_{,\alpha'}}}_{L^2}\lec C(\frak E)
  \end{equation}
  and
  \begin{equation}\label{2025}
  \nm{D_\aa \bar {D_\aa Z_t}}_{L^2}= \nm{ {D^2_\aa Z_t}}_{L^2}\lec  C(\frak E).
  \end{equation}
  We are left with the term $D_\aa \paren{b_\aa-2\Re D_\aa Z_t}$. We begin with \eqref{ba}:\begin{equation}\label{ba-1}
b_\aa-2\Re D_\aa Z_t=\Re \paren{\bracket{ \frac1{Z_{,\aa}}, \mathbb H}  Z_{t,\alpha'}+ \bracket{Z_t, \mathbb H}\partial_\aa \frac1{Z_{,\aa}}  }.
\end{equation}
Notice that the right hand side consists of  $\bracket{ \frac1{Z_{,\aa}}, \mathbb H}  Z_{t,\alpha'}$, 
$\bracket{Z_t, \mathbb H}\partial_\aa \frac1{Z_{,\aa}}$ and their complex conjugates. We use \eqref{2027} to compute
\begin{equation}\label{2041}
\begin{aligned}
D_\aa &\bracket{ \frac1{Z_{,\aa}}, \mathbb H}  Z_{t,\alpha'}= - \partial_\aa \frac1{Z_{,\aa}}  D_\aa  Z_{t}\\&-\frac1{\pi i}\int\frac{\paren{\frac1{Z_{,\aa}}-\frac1{Z_{,\bb}}}^2}{(\aa-\bb)^2} Z_{t,\bb}  \,d\bb-\frac1{\pi i}\int\frac{\paren{\frac1{Z_{,\aa}}-\frac1{Z_{,\bb}}} }{(\aa-\bb)^2} D_\bb Z_t   \,d\bb.
\end{aligned}
\end{equation}
Applying \eqref{eq:b12} to the second term and \eqref{3.17} to the third term yields
\begin{equation}\label{2042}
\nm{D_\aa \bracket{ \frac1{Z_{,\aa}}, \mathbb H}  Z_{t,\alpha'}}_{L^2}\lec \nm{\partial_\aa \frac1{Z_{,\aa}} }_{L^2}\nm{ D_\aa  Z_{t}}_{L^\infty}+ \nm{\partial_\aa \frac1{Z_{,\aa}} }_{L^2}^2\nm{ Z_{t,\aa}}_{L^2}.
\end{equation}
Similarly
\begin{equation}\label{2043}
\begin{aligned}
D_\aa &\bracket{Z_t, \mathbb H}\partial_\aa  \frac1{Z_{,\aa}} =   Z_{t,\aa} \frac1{Z_{,\aa}} \partial_\aa \frac1{Z_{,\aa}}  \\&-\frac1{\pi i}\int\frac{\paren{Z_t(\aa)-Z_t(\bb)}\paren{\frac1{Z_{,\aa}}-\frac1{Z_{,\bb}}}}{(\aa-\bb)^2} \partial_\bb \frac1{Z_{,\bb}}   \,d\bb-\frac1{\pi i}\int\frac{ \paren{Z_t(\aa)-Z_t(\bb)} }{(\aa-\bb)^2} \frac1{Z_{,\bb}} \partial_\bb \frac1{Z_{,\bb}}    \,d\bb,
\end{aligned}
\end{equation}
and applying \eqref{eq:b12} to the second term and \eqref{3.17} to the third term yields
\begin{equation}\label{2044}
\nm{  D_\aa \bracket{Z_t, \mathbb H}\partial_\aa  \frac1{Z_{,\aa}}  }_{L^2}\lec  \nm{\partial_\aa \frac1{Z_{,\aa}} }_{L^2}^2\nm{ Z_{t,\aa}}_{L^2}+ \nm{Z_{t,\aa}}_{L^2}\nm{  \frac1{Z_{,\aa}}\partial_\aa  \frac1{Z_{,\aa}}}_{L^\infty}   .
\end{equation}
The estimate of the complex conjugate terms is similar,  we omit. This concludes,  with an application of the results in \S\ref{basic-quantities} and \S\ref{aa-zdz}, that
\begin{equation}\label{2045}
\nm{D_\aa\paren{b_\aa-2\Re D_\aa Z_t}}_{L^2}\lec C(\frak E),
\end{equation}
therefore
\begin{equation}\label{2046}
\nm{(\partial_t+b\partial_\aa)\partial_{\alpha'} \frac{1}{ Z_{,\alpha'}}}_{L^2}\lec C(\frak E).
\end{equation}

 Now by
 \begin{equation}\label{2047}
\partial_{\alpha'} (\partial_t+b\partial_\aa) \frac{1}{ Z_{,\alpha'}}=(\partial_t+b\partial_\aa)\partial_{\alpha'} \frac{1}{ Z_{,\alpha'}}+b_\aa \partial_\aa \frac{1}{ Z_{,\alpha'}},
 \end{equation}
we also have
 \begin{equation}\label{2048}
\nm{\partial_{\alpha'} (\partial_t+b\partial_\aa)\frac{1}{ Z_{,\alpha'}}}_{L^2}\lec C(\frak E).
\end{equation}

\subsubsection{Controlling $\nm{\partial_\aa(\partial_t+b\partial_\aa)b}_{L^\infty}$, $\nm{(\partial_t+b\partial_\aa)b_\aa}_{L^\infty}$ and $\nm{(\partial_t+b\partial_\aa)D_\aa Z_t}_{L^\infty}$}\label{dtdab}

We apply \eqref{eq:c14} to \eqref{ba-1} and get
\begin{equation}\label{dba-1}
\begin{aligned}
(\partial_t+b\partial_\aa)\paren{b_\aa-2\Re D_\aa Z_t}&=\Re \paren{\bracket{ (\partial_t+b\partial_\aa)\frac1{Z_{,\aa}}, \mathbb H}  Z_{t,\alpha'}+ \bracket{ \frac1{Z_{,\aa}}, \mathbb H}  Z_{tt,\alpha'}-\bracket{  \frac1{Z_{,\aa}}, b;  Z_{t,\alpha'}   } }
\\&+\Re\paren{  \bracket{Z_{tt}, \mathbb H}\partial_\aa \frac1{Z_{,\aa}}    + \bracket{Z_{t}, \mathbb H}\partial_\aa (\partial_t+b\partial_\aa)\frac1{Z_{,\aa}} -\bracket{ Z_{t}, b; \partial_\aa \frac1{Z_{,\aa}}   } };
\end{aligned}
\end{equation}
using \eqref{eq:b13}, \eqref{eq:b15} and results from previous subsections we obtain
\begin{equation}\label{2023}
\begin{aligned}
\nm{(\partial_t+b\partial_\aa)\paren{b_\aa-2\Re D_\aa Z_t}}_{L^\infty}&\lec \nm{\partial_\aa (\partial_t+b\partial_\aa)\frac1{Z_{,\aa}}}_{L^2}\nm
{ Z_{t,\alpha'}}_{L^2}\\&+ \nm{\partial_\aa \frac1{Z_{,\aa}}}_{L^2}\nm{  Z_{tt,\alpha'}}_{L^2}+ \nm{ \partial_\aa \frac1{Z_{,\aa}}}_{L^2}\|b_\aa\|_{L^\infty}\|  Z_{t,\alpha'} \|_{L^2}
\\&\lec C(\frak E).
\end{aligned}
\end{equation}
We now compute $(\partial_t+b\partial_\aa)D_\aa Z_t$.   By \eqref{eq:c1-1},
\begin{equation}\label{2049}
(\partial_t+b\partial_\aa) D_\aa Z_t=D_\aa Z_{tt}-\paren{D_\aa Z_t}^2.
\end{equation}
So by the estimates in \S\ref{basic-quantities}, we have
\begin{equation}\label{2050}
\nm{(\partial_t+b\partial_\aa) D_\aa Z_t}_{L^\infty}\le \nm{D_\aa Z_{tt}}_{L^\infty}+\nm{D_\aa Z_t}^2_{L^\infty}\lec C(\frak E).
\end{equation}
This combine with \eqref{2023} yields
\begin{equation}\label{2051}
\nm{(\partial_t+b\partial_\aa)b_\aa}_{L^\infty}\lec C(\frak E).
\end{equation}

From  $\partial_\aa(\partial_t+b\partial_\aa)b=(\partial_t+b\partial_\aa)b_\aa+(b_\aa)^2$,
\begin{equation}\label{2052}
\nm{\partial_\aa(\partial_t+b\partial_\aa)b}_{L^\infty}\le \nm{(\partial_t+b\partial_\aa)b_\aa}_{L^\infty}+ \nm{b_\aa}_{L^\infty}^2\lec C(\frak E).
\end{equation}

We are now ready to estimate $\frac{d}{dt}\frak E$. 

\subsubsection{Controlling $\frac d{dt} \nm{\bar Z_{t,\aa}}_{L^2}^2$, $\frac d{dt} \nm{D_\aa^2 \bar Z_{t}}_{L^2}^2$, $\frac d{dt} \nm{   \partial_\aa \frac1{Z_{,\aa}}   }_{L^2}^2$ and $\frac d{dt} \abs{ \frac1{Z_{,\aa}} (0,t)  }^2$}\label{ddtlower} We use Lemma~\ref{basic-e2} to control $\frac d{dt} \nm{\bar Z_{t,\aa}}_{L^2}^2$, $\frac d{dt} \nm{D_\aa^2 \bar Z_{t}}_{L^2}^2$ and $\frac d{dt} \nm{   \partial_\aa \frac1{Z_{,\aa}}   }_{L^2}^2$. Notice that when we substitute  
$$\Th=Z_{t,\aa},\qquad \Th=  D_\aa^2 \bar Z_{t},\qquad \text{and}\quad  \Th= \partial_\aa \frac1{Z_{,\aa}} $$
in \eqref{basic-2}, all the terms on the right hand sides are already controlled in subsections \S\ref{basic-quantities} and \S\ref{dadtza}. So we have
\begin{equation}\label{2053}
\frac d{dt} \nm{\bar Z_{t,\aa}}_{L^2}^2+\frac d{dt} \nm{D_\aa^2 \bar Z_{t}}_{L^2}^2+\frac d{dt} \nm{   \partial_\aa \frac1{Z_{,\aa}}   }_{L^2}^2\lec C(\frak E).
\end{equation}

To estimate $\frac d{dt} \abs{ \frac1{Z_{,\aa}} (0,t)  }^2$, we start with \eqref{eq:dza} and compute
\begin{equation}\label{2054}
(\partial_t+b\partial_{\alpha'})\abs{\frac1{Z_{,\aa}}}^2=2\Re \paren{\frac1{\bar Z_{,\aa}}(\partial_t+b\partial_{\alpha'})\frac1{Z_{,\aa}}}=  \abs{\frac1{Z_{,\aa}} }^2\paren{2b_\aa-2\Re D_\aa Z_t}.
\end{equation}
Recall we chose the Riemann mapping so that $h(0,t)=0$ for all $t$. So $h_t\circ h^{-1}(0,t)=b(0,t)=0$ and
\begin{equation}\label{2055}
\frac{d}{dt}\abs{\frac1{Z_{,\aa}}(0,t)}^2=  \abs{\frac1{Z_{,\aa}}(0,t)}^2\paren{2b_\aa(0,t)-2\Re D_\aa Z_t(0,t)}\lec C(\frak E).
\end{equation}

We use Lemma~\ref{basic-e} to estimate the two main terms $\frac{ d}{dt} {\bf E}_a(t)$ and $\frac{ d}{dt}{\bf E}_b(t)$. 

\subsubsection{Controlling $\frac{ d}{dt} {\bf E}_a(t)$}\label{ddtea} We begin with $\frac{ d}{dt} {\bf E}_a(t)$.  Apply   Lemma~\ref{basic-e} to  $\Th=D_\aa \bar Z_t$ we get
\begin{equation}\label{2056}
\frac d{dt} {\bf E}_a(t)\le \nm{\frac{\frak a_t}{\frak a}\circ h^{-1}}_{L^\infty} {\bf E}_a(t)+2 {\bf E}_a(t)^{1/2}\paren{\int\frac{|\mathcal PD_\aa \bar Z_t  |^2}{\mathcal A}\,d\aa}^{1/2}.
\end{equation}
By \eqref{2022}, we know the first term is controlled by $C(\frak E)$.  We need to estimate the factor $\paren{\int\frac{|\mathcal PD_\aa \bar Z_t  |^2}{\mathcal A}\,d\aa}^{1/2}$ in the second term.  By \eqref{eq-dt}:
\begin{equation}\label{2057}
\mathcal P D_\aa\bar Z_t=\bracket{\mathcal P, \frac1{Z_{,\aa}}}\bar Z_{t,\aa}+\frac1{Z_{,\aa}}\mathcal P \bar Z_{t,\aa},
\end{equation}
and we have 
\begin{equation}\label{2058}
\int\frac{| \bracket{\mathcal P, \frac1{Z_{,\aa}}}\bar Z_{t,\aa}   |^2}{\mathcal A}\,d\aa= 
\int\frac{|Z_{,\aa} \bracket{\mathcal P, \frac1{Z_{,\aa}}}\bar Z_{t,\aa}   |^2}{ A_1}\,d\aa
\le 
\nm{Z_{,\aa}\bracket{\mathcal P, \frac1{Z_{,\aa}}}\bar Z_{t,\aa}}_{L^2}\lec C(\frak E),
\end{equation}
here in the last step we used \eqref{basic-4}, notice that all the terms on the right hand side of \eqref{basic-4}   with $f=\bar Z_{t,\aa}$ are controlled in subsections \S\ref{basic-quantities}--\S\ref{dtdab}. We are left with the term
$\int\frac{|  \frac1{Z_{,\aa}}\mathcal P \bar Z_{t,\aa}   |^2}{\mathcal A}\,d\aa$. Because $A_1\ge 1$,
$$\int\frac{|  \frac1{Z_{,\aa}}\mathcal P \bar Z_{t,\aa}   |^2}{\mathcal A}\,d\aa\le \int |\mathcal P \bar Z_{t,\aa}   |^2\,d\aa.$$

By the base equation \eqref{base-eq},
\begin{equation}\label{2059}
\mathcal P \bar Z_{t,\aa}=-(\partial_t+b\partial_\aa)(b_\aa \partial_{\aa}\bar Z_{t})-b_\aa\partial_\aa \bar Z_{tt}-i\mathcal A_\aa \partial_\aa \bar Z_t+\partial_\aa\paren{\frac{\frak a_t}{\frak a}\circ h^{-1} (\bar Z_{tt}-i)};
\end{equation}
we expand the last term by product rules,
\begin{equation}\label{2061}
\partial_\aa\paren{\frac{\frak a_t}{\frak a}\circ h^{-1} (\bar Z_{tt}-i)}=\frac{\frak a_t}{\frak a}\circ h^{-1} \bar Z_{tt,\aa}+\partial_\aa\paren{\frac{\frak a_t}{\frak a}\circ h^{-1}} (\bar Z_{tt}-i).
\end{equation}
It is clear that the first three terms in \eqref{2059} are controlled by $\frak E$, by the results of \S\ref{basic-quantities} - \S\ref{dtdab}:
\begin{equation}\label{2060}
\|-(\partial_t+b\partial_\aa)(b_\aa \partial_{\aa}\bar Z_{t})-b_\aa\partial_\aa \bar Z_{tt}-i\mathcal A_\aa \partial_\aa \bar Z_t\|_{L^2}\lec C(\frak E),
\end{equation}
and the first term in \eqref{2061} satisfies
\begin{equation}\label{2062}
\nm{\frac{\frak a_t}{\frak a}\circ h^{-1} \bar Z_{tt,\aa}}_{L^2}\le \nm{\frac{\frak a_t}{\frak a}\circ h^{-1}}_{L^\infty}\| \bar Z_{tt,\aa}\|_{L^2}\lec C(\frak E).
\end{equation}
We are left with one last term $\partial_\aa\paren{\frac{\frak a_t}{\frak a}\circ h^{-1}} (\bar Z_{tt}-i)$ in \eqref{2059}. 
We write
\begin{equation}\label{2063}
\mathcal P\bar Z_{t,\aa}=\partial_\aa\paren{\frac{\frak a_t}{\frak a}\circ h^{-1}} (\bar Z_{tt}-i)+\mathcal R
\end{equation}
where $\mathcal R= -(\partial_t+b\partial_\aa)(b_\aa \partial_{\aa}\bar Z_{t})-b_\aa\partial_\aa \bar Z_{tt}-i\mathcal A_\aa \partial_\aa \bar Z_t+\frac{\frak a_t}{\frak a}\circ h^{-1} \bar Z_{tt,\aa}$. 
We want to take advantage of the fact that $\partial_\aa\paren{\frac{\frak a_t}{\frak a}\circ h^{-1}} $ is purely real; notice that we have control of $\|(I-\mathbb H)\mathcal P \bar Z_{t,\aa}\|_{L^2}$ and  $\|\mathcal R\|_{L^2}$, 
 by Lemma~\ref{basic-3-lemma} and \S\ref{basic-quantities} - \S\ref{dtdab},  and by \eqref{2060} and \eqref{2062}.

Apply $(I-\mathbb H)$ to both sides of equation \eqref{2063}, we get
\begin{equation}\label{2065}
\begin{aligned}
(I-\mathbb H)\mathcal P\bar Z_{t,\aa}&=(I-\mathbb H)\paren{\partial_\aa\paren{\frac{\frak a_t}{\frak a}\circ h^{-1}} (\bar Z_{tt}-i)}+(I-\mathbb H)\mathcal R\\&
= (\bar Z_{tt}-i)(I-\mathbb H)\partial_\aa\paren{\frac{\frak a_t}{\frak a}\circ h^{-1}} +\bracket{\bar Z_{tt}, \mathbb H} \partial_\aa\paren{\frac{\frak a_t}{\frak a}\circ h^{-1}}  +(I-\mathbb H)\mathcal R
\end{aligned}
\end{equation}
where we commuted  $\bar Z_{tt}-i$ out in the second step. Now because $\partial_\aa\paren{\frac{\frak a_t}{\frak a}\circ h^{-1}} $  is purely real, 
\begin{equation}\label{2066}
\abs{   (\bar Z_{tt}-i) \partial_\aa\paren{\frac{\frak a_t}{\frak a}\circ h^{-1}}  } \le \abs{ (\bar Z_{tt}-i)(I-\mathbb H)\partial_\aa\paren{\frac{\frak a_t}{\frak a}\circ h^{-1}}  },
\end{equation}
so by \eqref{2065},
\begin{equation}\label{2067}
\abs{   (\bar Z_{tt}-i) \partial_\aa\paren{\frac{\frak a_t}{\frak a}\circ h^{-1}}  } \le  \abs{(I-\mathbb H)\mathcal P\bar Z_{t,\aa}}+ \abs{\bracket{\bar Z_{tt}, \mathbb H} \partial_\aa\paren{\frac{\frak a_t}{\frak a}\circ h^{-1}} }+ \abs{(I-\mathbb H)\mathcal R}.
\end{equation}
We estimate the $L^2$ norm of the first term by Lemma~\ref{basic-3-lemma}, the second term by \eqref{3.21}, and the third term by \eqref{2060} and \eqref{2062}. We obtain
\begin{equation}\label{2068}
\nm{   (\bar Z_{tt}-i) \partial_\aa\paren{\frac{\frak a_t}{\frak a}\circ h^{-1}}  }_{L^2}\lec C(\frak E)+ \nm{Z_{tt,\aa}}_{L^2}\nm{\frac{\frak a_t}{\frak a}\circ h^{-1}}_{L^\infty}+C(\frak E)\lec C(\frak E).
\end{equation}
This concludes
\begin{equation}\label{2069}
\frac d{dt}{\bf E}_a(t)\lec C(\frak E(t)).
\end{equation}

We record here the following estimate that will be used later. By \eqref{2068}, $\bar Z_{tt}-i=-\frac{iA_1}{Z_{,\aa}}$ and $A_1\ge 1$, we have
\begin{equation}\label{2068-1}
\nm{  D_\aa\paren{\frac{\frak a_t}{\frak a}\circ h^{-1}}  }_{L^2}\lec C(\frak E).
\end{equation}

\subsubsection{Controlling $\frac d{dt}{\bf E}_b(t)$}\label{ddteb} Taking $\Th= \frac1{Z_{,\aa}}D_\aa^2\bar Z_t$ in Lemma~\ref{basic-e}, we have, 
\begin{equation}\label{2070}
\frac d{dt} {\bf E}_b(t)\le \nm{\frac{\frak a_t}{\frak a}\circ h^{-1}}_{L^\infty} {\bf E}_b(t)+2 {\bf E}_b(t)^{1/2}\paren{\int\frac{|\mathcal P\paren{\frac1{Z_{,\aa}}D^2_\aa \bar Z_t}  |^2}{\mathcal A}\,d\aa}^{1/2}.
\end{equation}
By \eqref{2022}, the first term is controlled by $\mathfrak E$. We consider the second term. We know
\begin{equation}\label{2071}
\mathcal P \paren{\frac1{Z_{,\aa}}D^2_\aa\bar Z_t}=\bracket{\mathcal P, \frac1{Z_{,\aa}}}D_\aa^2\bar Z_{t}+\frac1{Z_{,\aa}}\bracket{\mathcal P, D_\aa^2}\bar Z_{t}+\frac1{Z_{,\aa}}D_\aa^2\mathcal P\bar Z_t,\end{equation}
and because $A_1\ge 1$, 
\begin{equation}\label{2072}
\int\frac{|\mathcal P\frac1{Z_{,\aa}}D^2_\aa \bar Z_t  |^2}{\mathcal A}\,d\aa\lec \int \abs{Z_{,\aa}\bracket{\mathcal P, \frac1{Z_{,\aa}}}D_\aa^2\bar Z_{t}}^2\,d\aa+\int\abs{ \bracket{\mathcal P, D_\aa^2}\bar Z_{t}}^2\,d\aa+\int \abs{D_\aa^2\mathcal P\bar Z_t}^2\,d\aa.
\end{equation}
Now by Lemma~\ref{basic-4-lemma} and the results of \S\ref{basic-quantities} - \S\ref{dtdab},
the first term on the right hand side of \eqref{2072} is controlled by $\frak E$. For the second term, we compute, using \eqref{eq:c4-1}, 
 \begin{equation} \label{2074}
   \begin{aligned}
 \bracket{\mathcal P,D_\aa^2}\bar Z_t & =-4(D_\aa Z_{tt}) D_\aa^2\bar Z_t + 6(D_\aa Z_t)^2 D_\aa^2\bar Z_t  - (2D_\aa^2 Z_{tt}) D_\aa\bar Z_t\\&
+ 6(D_\aa Z_t) (D_\aa^2 Z_t) D_\aa\bar Z_t - 2(D_\aa^2 Z_t) D_\aa \bar Z_{tt} - 4(D_\aa Z_t) D_\aa^2 \bar Z_{tt}.
   \end{aligned}
   \end{equation}
By results in \S\ref{basic-quantities}, we have
\begin{equation}\label{2075}
\| \bracket{\mathcal P,D_\aa^2}\bar Z_t\|_{L^2}\lec C(\frak E).
\end{equation}
We are left with the  term $\int \abs{D_\aa^2\mathcal P\bar Z_t}^2\,d\aa$, where
$$\mathcal P\bar Z_t:=\bar Z_{ttt}+i\mathcal A \bar Z_{t,\aa}=\frac{\frak a_t}{\frak a}\circ h^{-1} (\bar Z_{tt}-i).$$

We expand $D_\aa^2 \mathcal P\bar Z_t$  by product rules,
\begin{equation}\label{2076}
D_\aa^2 \mathcal P\bar Z_t=D_\aa^2\paren{\frac{\frak a_t}{\frak a}\circ h^{-1}}(\bar Z_{tt}-i)+2D_\aa\paren{\frac{\frak a_t}{\frak a}\circ h^{-1}}D_\aa \bar Z_{tt}+\paren{\frac{\frak a_t}{\frak a}\circ h^{-1}} D_\aa^2\bar Z_{tt}.
\end{equation}
We know how to handle the second and third terms, thanks to the work in the previous subsections.
We want to use the same idea as in the previous subsection to control the first term,   however $D_\aa$ is not purely real, so we go through the following slightly evoluted process. 

First,  we have
\begin{equation}\label{2077}
\begin{aligned}
&\nm{2D_\aa\paren{\frac{\frak a_t}{\frak a}\circ h^{-1}}D_\aa \bar Z_{tt}+\paren{\frac{\frak a_t}{\frak a}\circ h^{-1}} D_\aa^2\bar Z_{tt}}_{L^2}\\& \qquad\lec \|D_\aa \bar Z_{tt}\|_{L^\infty}\nm{D_\aa\paren{   \frac{\frak a_t}{\frak a}\circ h^{-1}}}_{L^2}+\nm{   \frac{\frak a_t}{\frak a}\circ h^{-1}}_{L^\infty}\nm{D_\aa^2 \bar Z_{tt}}_{L^2}\lec C(\frak E);
\end{aligned}
\end{equation}
and by Lemma~\ref{basic-3-lemma}  and \eqref{2075}, 
\begin{equation}\label{2084}
\nm{(I-\mathbb H)D_\aa^2 \mathcal P\bar Z_t}_{L^2}\le \nm{(I-\mathbb H)\mathcal PD_\aa^2 \bar Z_t}_{L^2}+\nm{(I-\mathbb H)\bracket{D_\aa^2, \mathcal P}\bar Z_t}_{L^2}\lec C(\frak E).
\end{equation}
So
\begin{equation}\label{2078}
\nm{(I-\mathbb H)\paren{D_\aa^2\paren{\frac{\frak a_t}{\frak a}\circ h^{-1}}(\bar Z_{tt}-i)}}_{L^2}\lec C(\frak E).
\end{equation}
This gives, from 
\begin{equation}\label{2079}
\begin{aligned}
(I-\mathbb H)\paren{D_\aa^2\paren{\frac{\frak a_t}{\frak a}\circ h^{-1}}(\bar Z_{tt}-i)}&= \frac{(\bar Z_{tt}-i)}{Z_{,\aa}}(I-\mathbb H)\paren{\partial_\aa D_\aa \paren{\frac{\frak a_t}{\frak a}\circ h^{-1}}}\\&+\bracket{\frac{(\bar Z_{tt}-i)}{Z_{,\aa}}, \mathbb H}\paren{\partial_\aa D_\aa \paren{\frac{\frak a_t}{\frak a}\circ h^{-1}}},
\end{aligned}
\end{equation}
and \eqref{3.20} that 
\begin{equation}\label{2080}
\nm{\frac{(\bar Z_{tt}-i)}{Z_{,\aa}}(I-\mathbb H)\paren{\partial_\aa D_\aa \paren{\frac{\frak a_t}{\frak a}\circ h^{-1}}}}_{L^2}\lec C(\frak E).
\end{equation}
Now we move the factor $\frac{(\bar Z_{tt}-i)}{|Z_{,\aa}|}$ back into $(I-\mathbb H)$ to get
\begin{equation}\label{2081}
\begin{aligned}
\frac{(\bar Z_{tt}-i)}{|Z_{,\aa}|}(I-\mathbb H)\paren{\partial_\aa D_\aa \paren{\frac{\frak a_t}{\frak a}\circ h^{-1}}}&=- \bracket{\frac{(\bar Z_{tt}-i)}{|Z_{,\aa}|}, \mathbb H}\paren{\partial_\aa D_\aa \paren{\frac{\frak a_t}{\frak a}\circ h^{-1}}}\\&+(I-\mathbb H) \paren{\frac{(\bar Z_{tt}-i)}{|Z_{,\aa}|}\partial_\aa D_\aa \paren{\frac{\frak a_t}{\frak a}\circ h^{-1}}};
\end{aligned}
\end{equation}
and observe that
\begin{equation}\label{2082}
\begin{aligned}
&\frac{(\bar Z_{tt}-i)}{|Z_{,\aa}|}\partial_\aa D_\aa \paren{\frac{\frak a_t}{\frak a}\circ h^{-1}}=\frac{(\bar Z_{tt}-i)}{Z_{,\aa}}\partial_\aa \paren{\frac1{|Z_{,\aa}|}\partial_\aa \paren{\frac{\frak a_t}{\frak a}\circ h^{-1}}}\\&+\frac{(\bar Z_{tt}-i)}{|Z_{,\aa}|}\partial_\aa\paren{\frac1{Z_{,\aa}} }\partial_\aa \paren{\frac{\frak a_t}{\frak a}\circ h^{-1}}-\frac{(\bar Z_{tt}-i)}{Z_{,\aa}}\partial_\aa\paren{\frac1{|Z_{,\aa}|} }\partial_\aa \paren{\frac{\frak a_t}{\frak a}\circ h^{-1}}.
\end{aligned}
\end{equation}
We know the $L^2$ norms of the last two terms on the right hand side of \eqref{2082} are controlled by $C(\frak E)$; and by \eqref{3.20}, the $L^2$ norm of the commutator  in \eqref{2081} is also controlled by $C(\frak E)$, therefore by \eqref{2081}, \eqref{2080}, \eqref{2082},
\begin{equation}\label{2083}
\nm{(I-\mathbb H) \paren{\frac{(\bar Z_{tt}-i)}{Z_{,\aa}}\partial_\aa \paren{\frac1{|Z_{,\aa}|}\partial_\aa \paren{\frac{\frak a_t}{\frak a}\circ h^{-1}}} } }_{L^2}\lec C(\frak E). 
\end{equation}
Now we commute out the factor $\frac{(\bar Z_{tt}-i)}{Z_{,\aa}}$ from $(I-\mathbb H)$ to get
\begin{equation}\label{2085}
\begin{aligned}
(I-\mathbb H) \paren{\frac{(\bar Z_{tt}-i)}{Z_{,\aa}}\partial_\aa \paren{\frac1{|Z_{,\aa}|}\partial_\aa \paren{\frac{\frak a_t}{\frak a}\circ h^{-1}}} }&= \frac{(\bar Z_{tt}-i)}{Z_{,\aa}} (I-\mathbb H) \partial_\aa \paren{\frac1{|Z_{,\aa}|}\partial_\aa \paren{\frac{\frak a_t}{\frak a}\circ h^{-1}}} \\&+\bracket{\frac{(\bar Z_{tt}-i)}{Z_{,\aa}}, \mathbb H}\partial_\aa \paren{\frac1{|Z_{,\aa}|} \partial_\aa \paren{\frac{\frak a_t}{\frak a}\circ h^{-1}}}
\end{aligned}
\end{equation}
Observe that the quantity the operator $(I-\mathbb H)$ acts on in the first term on the right hand side of \eqref{2085} is purely real. Applying \eqref{3.20} again to the commutator in \eqref{2085} and using \eqref{2083} and the fact that $|f|\le |(I-\mathbb H)f|$ for $f$ real, we obtain
\begin{equation}\label{2086}
\begin{aligned}
&\nm{ \frac{(\bar Z_{tt}-i)}{Z_{,\aa}} \partial_\aa \paren{\frac1{|Z_{,\aa}|}\partial_\aa \paren{\frac{\frak a_t}{\frak a}\circ h^{-1}}} }_{L^2}\\&\qquad\le \nm{\frac{(\bar Z_{tt}-i)}{Z_{,\aa}} (I-\mathbb H) \partial_\aa \paren{\frac1{|Z_{,\aa}|}\partial_\aa \paren{\frac{\frak a_t}{\frak a}\circ h^{-1}}}  }_{L^2}\lec C(\frak E).
\end{aligned}
\end{equation}
Applying \eqref{2086} to \eqref{2082} yields, 
$$\nm{D_\aa^2\paren{\frac{\frak a_t}{\frak a}\circ h^{-1}}(\bar Z_{tt}-i)}_{L^2}\lec C(\frak E);$$
and by \eqref{2077}, \eqref{2076}, 
\begin{equation}\label{2088}
\nm{D_\aa^2\mathcal P\bar Z_t}_{L^2}\lec C(\frak E).
\end{equation}
This finishes the proof of 
\begin{equation}\label{2089}
\frac d{dt}{\bf E}_b(t)\lec C(\frak E(t)).
\end{equation}
Sum up the results in subsections \S\ref{ddtlower} - \S\ref{ddteb}, we obtain
\begin{equation}\label{2090}
\frac d{dt}\frak E(t)\lec C(\frak E(t)).
\end{equation}

\subsection{The proof of Theorem ~\ref{blow-up}}\label{proof1} Assume that the initial data satisfies the assumption of Theorem~\ref{blow-up}, we know by \eqref{interface-a1}, Proposition~\ref{B2} and Sobolev embedding,  that $\frac1{Z_{,\aa}}(0)-1, Z_{,\aa}(0)-1\in H^s(\mathbb R)$, with 
\begin{equation}\label{2000-1}
\begin{aligned}
\nm{\frac1{Z_{,\aa}}(0)}_{L^\infty}&\le \|Z_{tt}(0)\|_{L^\infty}+1\lec \|Z_{tt}(0)\|_{H^1}+1<\infty;\\
\nm{\frac1{Z_{,\aa}}(0)-1}_{H^s}&\lec C\paren{ \|Z_t(0)\|_{H^s}, \|Z_{tt}(0)\|_{H^s}};\\
\nm{Z_{,\aa}(0)-1}_{H^s}&\lec C\paren{ \|Z_t(0)\|_{H^s}, \|Z_{tt}(0)\|_{H^s}}.
\end{aligned}
\end{equation}
From  Theorem~\ref{prop:local-s} and Proposition~\ref{prop:energy-eq}, we know to prove the blow-up criteria, Theorem~\ref{blow-up}, it suffices to show that for any solution of \eqref{interface-r}-\eqref{interface-holo}-\eqref{b}-\eqref{a1},  satisfying the regularity properties in Theorem~\ref{blow-up}, and for any $T_0>0$, 
$$\sup_{[0, T_0)} \frak E(t)<\infty \quad \text{implies} \quad \sup_{[0,T_0)}( \|Z_{,\aa}(t)\|_{L^\infty}+\|Z_t(t)\|_{H^{3+1/2}}+\|Z_{tt}(t)\|_{H^3})<\infty.$$

We begin with the lower order norms. We first show that, as a consequence of equation \eqref{eq:dza},  if $\|Z_{,\aa}(0)\|_{L^\infty} <\infty$, then \begin{equation}\label{2000-2}\sup_{[0, T_0)}\|Z_{,\aa}(t)\|_{L^\infty}<\infty\qquad \text{ as  long as }\quad \sup_{[0, T_0)}\frak E(t)<\infty.\end{equation}

Solving equation \eqref{eq:dza} we get, because $\partial_t+b\partial_\aa=U_h^{-1}\partial_t U_h$, 
\begin{equation}\label{2100}
\frac1{Z_{,\aa}}(h(\a,t),t)=\frac1{Z_{,\aa}}(\a,0)e^{\int_0^t (b_\aa\circ h(\a,\tau)-D_\a z_t(\a,\tau))\,d\tau};
\end{equation}
so by  \eqref{2020} of \S\ref{basic-quantities},
\begin{equation}\label{2101}
\sup_{[0, T]}\nm{Z_{,\aa}(t)}_{L^\infty} \le \nm{Z_{,\aa}(0)}_{L^\infty}e^{\int_0^T \|b_\aa(\tau)-D_\aa Z_t(\tau)\|_{L^\infty}\,d\tau}\lec \nm{Z_{,\aa}(0)}_{L^\infty}e^{T\sup_{[0, T]}C(\frak E(t))},
\end{equation}
 hence \eqref{2000-2} holds. Notice that from \eqref{2100}, we also have 
\begin{equation}\label{2101-1}
\sup_{[0, T]}\nm{\frac1 {Z_{,\aa}}(t)}_{L^\infty} \lec \nm{\frac1{Z_{,\aa}}(0)}_{L^\infty}e^{T\sup_{[0, T]}C(\frak E(t))}.
\end{equation}

Now by Lemma~\ref{basic-e2},
\begin{align}\label{2102}
\frac d{dt} \|Z_t(t)\|^2_{L^2}&\lec \|Z_{tt}(t)\|_{L^2}\|Z_t(t)\|_{L^2} +\|b_\aa(t)\|_{L^\infty}\|Z_t(t)\|^2_{L^2},\\
\label{2103}\frac d{dt} \|Z_{tt}(t)\|^2_{L^2}&\lec \|Z_{ttt}(t)\|_{L^2}\|Z_{tt}(t)\|_{L^2} +\|b_\aa(t)\|_{L^\infty}\|Z_{tt}(t)\|^2_{L^2};
\end{align}
and from equations \eqref{eq:dztt} and \eqref{aa1},
\begin{equation}\label{2104}
\bar Z_{ttt}= (\bar Z_{tt}-i)\paren{\frac{\frak a_t}{\frak a}\circ h^{-1}+ \bar {D_\aa Z_t}}=-\frac{iA_1}{Z_{,\aa}}\paren{\frac{\frak a_t}{\frak a}\circ h^{-1}+ \bar {D_\aa Z_t}},
\end{equation}
so 
\begin{equation}\label{2105}
\|\bar Z_{ttt}(t)\|_{L^2}\lec \|A_1(t)\|_{L^\infty}\nm{\frac{1}{Z_{,\aa}}(t)}_{L^\infty} \paren{ \nm{\frac{\frak a_t}{\frak a}\circ h^{-1}(t)}_{L^2}+ \nm{ D_\aa Z_t(t)}_{L^2}}.
\end{equation}
We want to show that $\nm{\frac{\frak a_t}{\frak a}\circ h^{-1}(t)}_{L^2}$ and $\nm{ D_\aa Z_t(t)}_{L^2}$ can be controlled by $\frak E$ and the initial data;  by \eqref{at}, it suffices to control $\|b_\aa(t)\|_{L^2}$, $\nm{ D_\aa Z_t(t)}_{L^2}$ and $\|(\partial_t+b\partial_\aa)A_1(t)\|_{L^2}$.

 Applying H\"older's inequality and \eqref{3.21} to \eqref{ba} yields
\begin{equation}\label{2106}
\|b_\aa(t)\|_{L^2}+\|D_\aa Z_t(t)\|_{L^2}\lec  \nm{\frac1{Z_{,\aa}}(t)}_{L^\infty}\|Z_{t,\aa}(t)\|_{L^2},
\end{equation}
and applying H\"older's inequality, \eqref{3.21}, \eqref{eq:b12} to \eqref{dta1} gives
\begin{equation}\label{2106-1}
\|(\partial_t+b\partial_\aa)A_1(t)\|_{L^2}\lec  \nm{Z_{tt}(t)}_{L^\infty}\|Z_{t,\aa}(t)\|_{L^2}+\|b_\aa(t)\|_{L^2}\|Z_{t,\aa}(t)\|_{L^2}^2;
\end{equation}
so by \eqref{at}, using the fact \eqref{aa1}, we have
\begin{equation}\label{2107}
 \nm{\frac{\frak a_t}{\frak a}\circ h^{-1}}_{L^2}+\nm{ D_\aa Z_t}_{L^2} \lec \paren{\nm{A_1}_{L^\infty}\nm{\frac1{Z_{,\aa}}}_{L^\infty}+1}\|Z_{t,\aa}\|_{L^2}+\|b_\aa\|_{L^2}\|Z_{t,\aa}\|_{L^2}^2.
\end{equation}
This gives, by further applying the estimates \eqref{2020} in \S\ref{basic-quantities} and \eqref{2101-1},  that for $t\in [0, T]$,
\begin{equation}\label{2108}
\|Z_{ttt}(t)\|_{L^2}\lec C(T, \sup_{[0, T]}\frak E(t)) \paren{1+\nm{\frac1{Z_{,\aa}}(t)}_{L^\infty}^2}\lec C(T, \sup_{[0, T]}\frak E(t)) \paren{1+\nm{\frac1{Z_{,\aa}}(0)}_{L^\infty}^2}.
\end{equation}

We now apply Gronwall's inequality to \eqref{2103}. This yields
\begin{equation}\label{2109}
\sup_{[0, T]} \|Z_{tt}(t)\|_{L^2}\lec C\paren{T, \sup_{[0, T]}\frak E(t), \|Z_{tt}(0)\|_{L^2}, \nm{\frac1{Z_{,\aa}}(0)}_{L^\infty}}.
\end{equation}
We then apply Gronwall's inequality to \eqref{2102}, using \eqref{2109}. We obtain
\begin{equation}\label{2110}
\sup_{[0, T]} \|Z_{t}(t)\|_{L^2}\lec C\paren{T, \sup_{[0, T]}\frak E(t), \|Z_{t}(0)\|_{L^2}, \|Z_{tt}(0)\|_{L^2}, \nm{\frac1{Z_{,\aa}}(0)}_{L^\infty}}.
\end{equation}
Therefore the lower order norm $\sup_{[0, T]}(\|Z_t(t)\|_{L^2}+\|Z_{tt}(t)\|_{L^2})$ is controlled by $\sup_{[0, T]}\frak E(t)$,  the $L^2$ norm of  $(Z_t(0), Z_{tt}(0))$ and the $L^\infty$ norm of $\frac1{Z_{,\aa}(0)}$.\footnote{$\nm{\frac1{Z_{,\aa}(0)}}_{L^\infty}$  is controlled by the $H^1$ norm of $Z_{tt}(0)$, see \eqref{2000-1}.}

We are left with proving 
\begin{equation}\label{2111}
\sup_{[0, T_0)}\frak E(t)<\infty \qquad\text{implies }\quad \sup_{[0, T_0)}(\|\partial_\aa^3Z_{t}(t)\|_{\dot H^{1/2}}+\|\partial_\aa^3 Z_{tt}(t)\|_{L^2})<\infty.
\end{equation}
We do so via two stronger results, Propositions~\ref{step1} and \ref{step2}.

Let
\begin{equation}\label{2114}
\begin{aligned}
 E_{k}(t):&=E_{D_\aa \partial_\aa^{k-1}\bar Z_{t}}(t)+\|\partial_\aa^k \bar Z_t(t)\|_{L^2}^2
\\&:=\int \frac1{A_1}\abs{Z_{,\alpha'}(\partial_t+b\partial_\aa) \paren{\frac1{Z_{,\alpha'}}\partial_{\alpha'}^k\bar Z_t}}^2\,d\alpha'+\nm{\frac1{Z_{,\alpha'}}\partial_{\alpha'}^k\bar Z_t(t)}_{\dot H^{1/2}}^2+\|\partial_\aa^k \bar Z_t(t)\|_{L^2}^2,
\end{aligned}
\end{equation}
where $k=2, 3$. We have
\begin{proposition}\label{step1}
There exists a polynomial $p_1=p_1(x)$ with universal coefficients such that 
\begin{equation}\label{2115}
\frac d{dt} E_2(t)\le p_1\paren{\frak E(t)} E_2(t).
\end{equation}
\end{proposition}

\begin{proposition}\label{step2}
There exists a polynomial $p_2=p_2(x,y, z)$ with universal coefficients such that 
\begin{equation}\label{2116}
\frac d{dt} E_3(t)\le p_2\paren{\frak E(t), E_2(t), \nm{\frac1{Z_{,\aa}}(t)}_{L^\infty}} (E_3(t)+1).
\end{equation}
\end{proposition}

By Gronwall's inequality, we have from \eqref{2115} and \eqref{2116}  that
\begin{equation}\label{step1-2}
\begin{aligned}
E_2(t)&\le E_2(0)e^{\int_0^t p_1(\frak E(s))\,ds};\qquad\text{and }\\
E_3(t)&\le \paren{E_3(0)+\int_0^t p_2\paren{\frak E(s), E_2(s), \nm{\frac1{Z_{,\aa}}(s)}_{L^\infty}}\,ds}e^{\int_0^t p_2\paren{\frak E(s), E_2(s), \nm{\frac1{Z_{,\aa}}(s)}_{L^\infty}}\,ds},
\end{aligned}
\end{equation}
so $\sup_{[0, T]} E_2(t)$ is controlled by $E_2(0)$ and $\sup_{[0, T]}\frak E(t)$; and $\sup_{[0, T]} E_3(t)$
is controlled by $E_3(0)$, $\sup_{[0, T]}\frak E(t)$,  $\sup_{[0, T]} E_2(t)$ and $\sup_{[0, T]} \nm{\frac1{Z_{,\aa}}(t)}_{L^\infty}$. And by \eqref{2101-1}, $\sup_{[0, T]} E_3(t)$ is  in turn controlled by $E_3(0)$, $\sup_{[0, T]}\frak E(t)$, $E_2(0)$ and $ \nm{\frac1{Z_{,\aa}}(0)}_{L^\infty}$.
 We will prove Propositions~\ref{step1} and ~\ref{step2} in the next two subsections.
In \S\ref{complete1}  we will exam the relation between the energy functionals $E_2$, $E_3$ and the Sobolev norms $\| Z_t(t)\|_{H^{3+1/2}}$, $\| Z_{tt}(t)\|_{H^3}$ and 
complete the proof of Theorem~\ref{blow-up}. 

\subsection{ The proof of Proposition~\ref{step1}}\label{proof-prop1}
 We begin with a list of quantities controlled by $E_2(t)$.

\subsubsection{Quantities controlled by $E_2(t)$.}\label{quantities-e2}
It is clear by the definition  that the following are controlled by $E_2(t)$. 
\begin{equation}\label{2117}
\|\partial_{\alpha'}^2\bar Z_t\|_{L^2}^2\le E_2,\quad    \nm{\frac1{Z_{,\alpha'}}\partial_{\alpha'}^2\bar Z_t}_{\dot H^{1/2}}^2 \le  E_2,\quad \nm{Z_{,\alpha'}(\partial_t+b\partial_\aa)\paren{ \frac1{Z_{,\alpha'}}\partial_{\alpha'}^2\bar Z_t}}_{L^2}^2\le C(\frak E) E_2,
\end{equation}
because $1\le A_1\le C(\frak E)$ by \eqref{2000}. We compute,  by product rules and \eqref{eq:dza}, that
\begin{equation}\label{2117-1}
 Z_{,\alpha'}(\partial_t+b\partial_\aa) \paren{\frac1{Z_{,\alpha'}}\partial_{\alpha'}^2\bar Z_t}= (\partial_t+b\partial_\aa)\partial_{\alpha'}^2\bar Z_t+ (b_\aa-D_\aa Z_t)\partial_{\alpha'}^2\bar Z_t,
\end{equation}
therefore, by estimates \eqref{2020} in \S\ref{basic-quantities},
\begin{equation}\label{2118}
\abs{\nm{(\partial_t+b\partial_\aa)\partial_{\alpha'}^2\bar Z_t}_{L^2}-\nm{Z_{,\alpha'}(\partial_t+b\partial_\aa) \frac1{Z_{,\alpha'}}\partial_{\alpha'}^2\bar Z_t}_{L^2}}
\le C(\frak E)\|\partial_{\alpha'}^2\bar Z_t\|_{L^2},
\end{equation}
so
\begin{equation}\label{2119}
\nm{(\partial_t+b\partial_\aa)\partial_{\alpha'}^2\bar Z_t}_{L^2}^2\le C(\frak E)E_2.
\end{equation}
Now by \eqref{eq:c7}, 
\begin{equation}
\partial_{\alpha'}(\partial_t+b\partial_\aa)\bar Z_{t,\alpha'}
= (\partial_t+b\partial_\aa)\partial_{\alpha'}^2\bar Z_{t}+b_{\alpha'}\partial_{\alpha'}^2\bar Z_{t},
\end{equation}
so by \eqref{2020}, 
\begin{equation}\label{2120}
\|\partial_{\alpha'}(\partial_t+b\partial_\aa)\bar Z_{t,\alpha'}\|_{L^2}^2\le C(\frak E)E_2.
\end{equation}
Using Sobolev inequality \eqref{eq:sobolev} and \eqref{2020}, we obtain
\begin{align}\label{2121}
\|Z_{t,\alpha'}\|_{L^\infty}^2&\le 2\|Z_{t,\alpha'}\|_{L^2}\|\partial_{\alpha'}^2Z_{t}\|_{L^2}\le C(\frak E)E_2^{1/2};\qquad\qquad\text{and}\\
\label{2122}
\|(\partial_t+b\partial_\aa)Z_{t,\alpha'}\|_{L^\infty}^2&\le 2\|(\partial_t+b\partial_\aa)Z_{t,\alpha'}\|_{L^2}\|\partial_{\alpha'}(\partial_t+b\partial_\aa)Z_{t,\aa}\|_{L^2}\le C(\frak E)E_2^{1/2}.
\end{align}

We need the estimates for some additional quantities,  which we give in the following subsections. 

\subsubsection{Controlling the  quantity $\nm{\partial_{\alpha'}(b_\aa-2\Re D_\aa Z_t)}_{L^2}$.} \label{ddb}
We begin with equation \eqref{ba}, and differentiate with respect to $\aa$. We get
\begin{equation}\label{2122-1}
\begin{aligned}
\partial_\aa(b_\aa-2\Re D_\aa Z_t)&=\Re \paren{\bracket{ \partial_\aa \frac1{Z_{,\aa}}, \mathbb H}  Z_{t,\alpha'}+ \bracket{Z_{t,\aa}, \mathbb H}\partial_\aa \frac1{Z_{,\aa}}  }\\&+\Re \paren{\bracket{  \frac1{Z_{,\aa}}, \mathbb H}  \partial_\aa^2 Z_{t}+ \bracket{Z_t, \mathbb H}\partial_\aa^2 \frac1{Z_{,\aa}}  };
\end{aligned}
\end{equation}
using $\mathbb H Z_{t,\aa}=-Z_{t,\aa}$ to rewrite the first term,
\begin{equation}\label{2129-1}
\bracket{\partial_\aa \frac1{Z_{,\aa}}, \mathbb H}  Z_{t,\alpha'}=-(I+\mathbb H)\paren{\partial_\aa \frac1{Z_{,\aa}}  Z_{t,\alpha'}}
\end{equation}
and then applying  \eqref{3.21} and \eqref{3.20} to the last two terms. We get, by \eqref{2121} and \eqref{2020},
\begin{equation}\label{2123}
\nm{\partial_\aa(b_\aa-2\Re D_\aa Z_t)}_{L^2}\lesssim \|Z_{t,\alpha'}\|_{L^\infty}\nm{\partial_{\alpha'}\frac1{Z_{,\alpha'}}}_{L^2}\le C(\frak E) E_2^{1/4}.
\end{equation}

\subsubsection{Controlling $\|\partial_{\alpha'}^2\bar Z_{tt}\|_{L^2}$}\label{ddzt}
We start with  $(\partial_t+b\partial_\aa)\partial_{\alpha'}^2\bar Z_t$,  and   commute $\partial_t+b\partial_\aa$ with $\partial_{\alpha'}^2$; by \eqref{eq:c11}, we have
\begin{equation}\label{2124}
\begin{aligned}
\partial_{\alpha'}^2\bar Z_{tt}-(\partial_t+b\partial_\aa)\partial_{\alpha'}^2\bar Z_t
&= [\partial_{\alpha'}^2, (\partial_t+b\partial_\aa)]\bar Z_t
\\&=2b_{\alpha'} \partial_{\alpha'}^2\bar Z_t +\paren{\partial_{\alpha'}b_{\alpha'}} \bar Z_{t,\alpha'},
\end{aligned}
\end{equation}
We further expand the second term 
\begin{equation}\label{2125}
\paren{\partial_{\alpha'}b_{\alpha'}} \bar Z_{t,\alpha'}=\paren{\partial_{\alpha'}(b_{\alpha'}-2\Re D_\aa Z_t)} \bar Z_{t,\alpha'}+2\Re\paren{\partial_\aa\frac1{Z_{,\aa}}Z_{t,\aa}}\bar Z_{t,\aa}+2\Re \paren{\frac1{Z_{,\aa}}\partial_\aa^2 Z_{t}}\bar Z_{t,\aa};
\end{equation}
we get, by \eqref{2124} and \eqref{2125} that
\begin{equation}\label{2126}
\begin{aligned}
\|\partial_{\alpha'}^2\bar Z_{tt}\|_{L^2}&\lec \|(\partial_t+b\partial_\aa)\partial_{\alpha'}^2\bar Z_t\|_{L^2}+
\|b_\aa\|_{L^\infty}\|\partial_{\alpha'}^2\bar Z_t\|_{L^2}+\nm{\partial_\aa(b_\aa-2\Re D_\aa Z_t)}_{L^2}\|Z_{t,\alpha'}\|_{L^\infty}\\&+\nm{ \partial_\aa\frac1{Z_{,\aa}}}_{L^2}\|Z_{t,\alpha'}\|_{L^\infty}^2+\|D_\aa Z_t\|_{L^\infty}\|\partial_{\alpha'}^2\bar Z_t\|_{L^2}
\end{aligned}
\end{equation}
Therefore by the estimates in \S\ref{quantities-e2}, \S\ref{basic-quantities} and \eqref{2123}, 
\begin{equation}\label{2127}
\|\partial_{\alpha'}^2\bar Z_{tt}\|^2_{L^2}\lesssim C(\frak E)E_2.
\end{equation}
As a consequence of  the Sobolev inequality \eqref{eq:sobolev}, and estimates \eqref{2020} in \S\ref{basic-quantities},
\begin{equation}\label{2128}
\|\partial_{\alpha'}\bar Z_{tt}\|^2_{L^\infty}\le 2\|\partial_{\alpha'}\bar Z_{tt}\|_{L^2}\|\partial_{\alpha'}^2\bar Z_{tt}\|_{L^2}    \lesssim C(\frak E)E_2^{1/2}.
\end{equation}
We also have, by the $L^2$ boundedness of $\mathbb H$, 
\begin{equation}\label{2128-1}
\|\mathbb H\bar Z_{tt,\aa}\|^2_{L^\infty}\le 2\|\partial_{\alpha'}\mathbb H\bar Z_{tt}\|_{L^2}\|\partial_{\alpha'}^2\mathbb H\bar Z_{tt}\|_{L^2}\lec \|\partial_{\alpha'}\bar Z_{tt}\|_{L^2}\|\partial_{\alpha'}^2\bar Z_{tt}\|_{L^2}    \lesssim C(\frak E)E_2^{1/2}.
\end{equation}

\subsubsection{Controlling the  quantity $\nm{(\partial_t+b\partial_\aa)\partial_{\alpha'}(b_\aa-2\Re D_\aa Z_t)}_{L^2}$.}\label{dtdabl2}
We begin with \eqref{2122-1}, replacing the first term  by \eqref{2129-1}, 
then use \eqref{eq:c21}  on the first term and use \eqref{eq:c14} to compute the remaining three terms,
\begin{equation}\label{2129}
\begin{aligned}
(\partial_t&+b\partial_\aa)\partial_\aa(b_\aa-2\Re D_\aa Z_t)\\&=-\Re \paren{[b,\mathbb H]\partial_\aa \paren{ \partial_\aa \frac1{Z_{,\aa}}  Z_{t,\alpha'}} +(I+\mathbb H) (\partial_t+b\partial_\aa)\paren{ \partial_\aa \frac1{Z_{,\aa}}  Z_{t,\alpha'}}}\\&
+\Re\paren{ \bracket{(\partial_t+b\partial_\aa) Z_{t,\aa}, \mathbb H}\partial_\aa \frac1{Z_{,\aa}} +\bracket{ Z_{t,\aa}, \mathbb H}\partial_\aa (\partial_t+b\partial_\aa)\frac1{Z_{,\aa}}-\bracket{ Z_{t,\aa},  b; \partial_\aa \frac1{Z_{,\aa}} }}\\&+\Re \paren{\bracket{  \frac1{Z_{,\aa}}, \mathbb H}  \partial_\aa (\partial_t+b\partial_\aa)  Z_{t,\aa} +\bracket{  (\partial_t+b\partial_\aa) \frac1{Z_{,\aa}}, \mathbb H}  \partial_\aa^2 Z_{t}-\bracket{  \frac1{Z_{,\aa}}, b; \partial_\aa^2 Z_{t} } } \\&+\Re \paren{\bracket{Z_{tt}, \mathbb H}\partial_\aa^2 \frac1{Z_{,\aa}} + \bracket{Z_t, \mathbb H}\partial_\aa(\partial_t+b\partial_\aa)\partial_\aa \frac1{Z_{,\aa}} - \bracket{Z_t, b; \partial_\aa^2 \frac1{Z_{,\aa}} } }.
\end{aligned}
\end{equation}
We have, by \eqref{3.20}, \eqref{3.16}, \eqref{3.21}, \eqref{3.17}, and estimates in \S\ref{basic-quantities}, \S\ref{dadtza}, \S\ref{quantities-e2}-\S\ref{ddzt},
\begin{equation}\label{2130}
\begin{aligned}
&\nm{(\partial_t+b\partial_\aa)\partial_{\alpha'}(b_\aa-2\Re D_\aa Z_t)}_{L^2}\lec \|b_\aa\|_{L^\infty}\nm{ \partial_\aa \frac1{Z_{,\aa}}}_{L^2}\| Z_{t,\alpha'}\|_{L^\infty}\\&+\nm{ (\partial_t+b\partial_\aa)\partial_\aa \frac1{Z_{,\aa}}}_{L^2}\| Z_{t,\alpha'}\|_{L^\infty}+\nm{ \partial_\aa \frac1{Z_{,\aa}}}_{L^2}\| (\partial_t+b\partial_\aa)Z_{t,\alpha'}\|_{L^\infty}\\&
+\nm{ \partial_\aa(\partial_t+b\partial_\aa) \frac1{Z_{,\aa}}}_{L^2}\| Z_{t,\alpha'}\|_{L^\infty}+\nm{ \partial_\aa \frac1{Z_{,\aa}}}_{L^2}\| Z_{tt,\alpha'}\|_{L^\infty}\\&
\lec C(\frak E) E_2^{1/4}.
\end{aligned}
\end{equation}

\subsubsection{Controlling the quantity $\partial_\aa \mathcal A_\aa$}
We begin with equation \eqref{2029}, 
 \begin{equation}\label{2131}
i \mathcal A_\aa= \frac1{Z_{,\aa}}\partial_{\alpha'} Z_{tt}+(Z_{tt}+i)\partial_{\alpha'}\frac{1}{Z_{,\alpha'}}
 \end{equation}
and differentiate with respect to $\aa$. We get
  \begin{equation}\label{2132}
  i\partial_{\alpha'}\mathcal A_\aa=\frac{\partial_{\alpha'}^2Z_{tt}}{Z_{,\alpha'}}+2\partial_{\alpha'} Z_{tt}\partial_{\alpha'}\frac{1}{Z_{,\alpha'}}+(Z_{tt}+i)\partial_{\alpha'}^2\frac{1}{Z_{,\alpha'}}.
  \end{equation}
  Applying $(I-\mathbb H)$ yields
    \begin{equation}\label{2133}
i(I-\mathbb H)  \partial_{\alpha'}\mathcal A_\aa=(I-\mathbb H)(\frac{\partial_{\alpha'}^2Z_{tt}}{Z_{,\alpha'}})+2(I-\mathbb H)(\partial_{\alpha'} Z_{tt}\partial_{\alpha'}\frac{1}{Z_{,\alpha'}})+(I-\mathbb H)((Z_{tt}+i)\partial_{\alpha'}^2\frac{1}{Z_{,\alpha'}}).
  \end{equation}
  We rewrite the first term on the right by commuting out $\frac1{Z_{,\alpha'}}$,
 and use $(I-\mathbb H) \partial_{\alpha'}^2\frac{1}{Z_{,\alpha'}}=0$  to rewrite the third term on the right of \eqref{2133} as a commutator.
  We have
  \begin{equation}\label{2136}
  i(I-\mathbb H)  \partial_{\alpha'}\mathcal A_\aa-\frac1{Z_{,\aa}}(I-\mathbb H)\partial_{\alpha'}^2Z_{tt}=
  [\frac{1}{Z_{,\alpha'}}, \mathbb H]{\partial_{\alpha'}^2Z_{tt}}
  +2(I-\mathbb H)(\partial_{\alpha'} Z_{tt}\partial_{\alpha'}\frac{1}{Z_{,\alpha'}})+ [Z_{tt}, \mathbb H]\partial_{\alpha'}^2\frac{1}{Z_{,\alpha'}}.
   \end{equation}
Taking imaginary parts, then applying \eqref{3.20}, \eqref{3.21} and H\"older's inequality gives
  \begin{equation}\label{2137}
  \nm{\partial_{\alpha'}\mathcal A_\aa-\Im\braces{\frac{1}{Z_{,\alpha'}}(I-\mathbb H)({\partial_{\alpha'}^2Z_{tt}})}}_{L^2}\lesssim \nm{\partial_{\alpha'}\frac{1}{Z_{,\alpha'}}}_{L^2}\|Z_{tt,\alpha'}\|_{L^\infty}\lec C(\frak E)E_2^{1/4}.
  \end{equation}

\subsubsection{Controlling $\nm{\partial_\aa\paren{\frac{\frak a_t}{\frak a}\circ h^{-1}}}_{L^2}$}\label{data}
We begin with \eqref{at}. We have controlled $\nm{\partial_\aa(b_\aa-2\Re D_\aa Z_t)}_{L^2}$ in \eqref{2123}, we are left with  $\nm{\partial_\aa \paren{\frac{(\partial_t+b\partial_\aa)A_1}{A_1}}}_{L^2}$.

We proceed with computing $\partial_\aa A_1$, using \eqref{a1}. We have
\begin{equation}\label{2138}
\begin{aligned}
\partial_\aa A_1&=-\Im\paren{ [Z_{t,\aa},\mathbb H]\bar Z_{t,\aa}+[Z_t, \mathbb H]\partial_\aa^2\bar Z_t}\\&=-\Im (-\mathbb H |\bar Z_{t,\aa}|^2+[Z_t, \mathbb H]\partial_\aa^2\bar Z_t);
\end{aligned}
\end{equation}
here we used the fact $\mathbb H \bar Z_{t,\aa}=\bar Z_{t,\aa}$ to expand the first term, then removed the term $\Im |Z_{t,\aa}|^2=0$. Applying \eqref{3.20}, \eqref{2121} and \eqref{2020} gives
\begin{equation}\label{2139}
\|\partial_\aa A_1\|_{L^2}\lec \|Z_{t,\aa}\|_{L^\infty}\|Z_{t,\aa}\|_{L^2}\lec C(\frak E) E_2^{1/4}.
\end{equation}  
Now taking derivative $\partial_t+b\partial_\aa$ to \eqref{2138}, using \eqref{eq:c21} and \eqref{eq:c14}, yields
\begin{equation}\label{2140}
\begin{aligned}
(\partial_t+b\partial_\aa)\partial_\aa A_1&=\Im \paren{[b, \mathbb H]\partial_\aa |\bar Z_{t,\aa}|^2+2 \mathbb H \Re\braces{Z_{t,\aa} (\partial_t+b\partial_\aa)\bar Z_{t,\aa} }}\\&
-\Im ([Z_{tt}, \mathbb H]\partial_\aa^2\bar Z_t)+[Z_t,\mathbb H]\partial_\aa (\partial_t+b\partial_\aa  )\bar Z_{t,\aa}-[Z_t, b; \partial_\aa^2\bar Z_t ]);
\end{aligned}
\end{equation}
By \eqref{3.20}, then use \eqref{2020}, \eqref{2121}, \eqref{2128} we get
\begin{equation}\label{2141}
\begin{aligned}
\nm{(\partial_t+b\partial_\aa)\partial_\aa A_1}_{L^2}&\lec \|b_\aa\|_{L^\infty}\|Z_{t,\aa}\|_{L^2}\|Z_{t,\aa}\|_{L^\infty}+ \|Z_{tt,\aa}\|_{L^\infty}\|Z_{t,\aa}\|_{L^2}\\&+\|Z_{t,\aa}\|_{L^\infty}\|(\partial_t+b\partial_\aa)\bar Z_{t,\aa}\|_{L^2}\lec C(\frak E) E_2^{1/4}.
\end{aligned}
\end{equation}
Commuting $\partial_\aa$ with $\partial_t+b\partial_\aa$ gives
\begin{equation}\label{2142}
\nm{\partial_\aa (\partial_t+b\partial_\aa)A_1}_{L^2}\lec \nm{(\partial_t+b\partial_\aa)\partial_\aa A_1}_{L^2}+\nm{b_\aa \partial_\aa A_1}_{L^2}\lec C(\frak E) E_2^{1/4}.
\end{equation}
Combine \eqref{2142} with \eqref{2139} and \eqref{2123}, using \eqref{2000}, \eqref{2021}, we obtain
\begin{equation}\label{2143}
\nm{\partial_\aa\paren{\frac{\frak a_t}{\frak a}\circ h^{-1}}}_{L^2} \lec C(\frak E) E_2^{1/4}.
\end{equation}

Sum up the estimates obtained in \S\ref{quantities-e2} - \S\ref{data}, we have that the following quantities are controlled by $C(\frak E)E_2^{1/2}$:
\begin{equation}\label{2144}
\begin{aligned}
&\|\partial_{\alpha'}^2\bar Z_t\|_{L^2}, \quad    \nm{\frac1{Z_{,\alpha'}}\partial_{\alpha'}^2\bar Z_t}_{\dot H^{1/2}},\quad \nm{Z_{,\alpha'}(\partial_t+b\partial_\aa)\paren{ \frac1{Z_{,\alpha'}}\partial_{\alpha'}^2\bar Z_t}}_{L^2}, \\& \quad \nm{(\partial_t+b\partial_\aa)\partial_{\alpha'}^2\bar Z_t}_{L^2}, \quad 
\nm{\partial_{\alpha'}(\partial_t+b\partial_\aa)\bar Z_{t,\aa}}_{L^2},\quad \|\partial_{\alpha'}^2\bar Z_{tt}\|_{L^2}\\&
\|Z_{t,\alpha'}\|_{L^\infty}^2,\quad
\|(\partial_t+b\partial_\aa)Z_{t,\alpha'}\|_{L^\infty}^2,\quad \|\partial_{\alpha'}\bar Z_{tt}\|^2_{L^\infty},
\quad \|\mathbb H\bar Z_{tt,\aa}\|^2_{L^\infty},\\&
\nm{\partial_\aa(b_\aa-2\Re D_\aa Z_t)}_{L^2}^2, \quad \nm{(\partial_t+b\partial_\aa)\partial_\aa(b_\aa-2\Re D_\aa Z_t)}_{L^2}^2, \\&\quad  \nm{\partial_{\alpha'}\mathcal A_\aa-\Im\braces{\frac{1}{Z_{,\alpha'}}(I-\mathbb H)({\partial_{\alpha'}^2Z_{tt}})}}_{L^2}^2,\quad \nm{\partial_\aa\paren{\frac{\frak a_t}{\frak a}\circ h^{-1}}}_{L^2}^2,\quad \|\partial_\aa A_1\|_{L^2}^2
\end{aligned}
\end{equation}

\subsubsection{Controlling $\frac d{dt}E_2(t)$} We are now ready to estimate $\frac d{dt}E_2(t)$. 
We know 
$$E_2(t)=E_{D_\aa \partial_\aa \bar Z_{t}}(t)+\|\partial_\aa^2 \bar Z_t(t)\|_{L^2}^2,$$
 where $E_{D_\aa \partial_\aa \bar Z_{t}}(t)$ is as defined in \eqref{eq:41}. We  use Lemma~\ref{basic-e} on $E_{D_\aa \partial_\aa \bar Z_{t}}(t)$  and Lemma~\ref{basic-e2} on $\|\partial_\aa^2 \bar Z_t(t)\|_{L^2}^2$. 

We start with $\|\partial_\aa^2 \bar Z_t(t)\|_{L^2}^2$. We know by Lemma~\ref{basic-e2} that
\begin{equation}\label{2144-1}
\frac d{dt} \|\partial_\aa^2 \bar Z_t(t)\|_{L^2}^2\lec \|(\partial_t+b\partial_\aa)\partial_\aa^2 \bar Z_t(t)\|_{L^2}\|\partial_\aa^2 \bar Z_t(t)\|_{L^2}+\|b_\aa\|_{L^\infty}\|\partial_\aa^2 \bar Z_t(t)\|_{L^2}^2
\end{equation}
We have controlled $\|(\partial_t+b\partial_\aa)\partial_\aa^2 \bar Z_t(t)\|_{L^2}$ in \eqref{2119}, and $\|b_\aa\|_{L^\infty}$ in \S\ref{basic-quantities},  therefore
\begin{equation}\label{2145}
\frac d{dt} \|\partial_\aa^2 \bar Z_t(t)\|_{L^2}^2\lec C(\frak E(t))E_2(t).
\end{equation}

We now estimate $\frac d{dt} E_{D_\aa \partial_\aa \bar Z_{t}}(t)$. Take  $\Th=D_\aa  \bar Z_{t,\aa}$ in  Lemma~\ref{basic-e}, we have
 \begin{equation}\label{2149}
\frac d{dt} E_{D_\aa \partial_\aa \bar Z_{t}}(t)  \le \nm{\frac{\frak a_t}{\frak a}\circ h^{-1}}_{L^\infty} E_{D_\aa \partial_\aa \bar Z_{t}}(t)+2 E_{D_\aa \partial_\aa \bar Z_{t}}(t)^{1/2}\paren{\int\frac{|\mathcal P D_\aa \bar Z_{t,\aa}|^2}{\mathcal A}\,d\aa}^{1/2}.
\end{equation}
We have controlled $\nm{\frac{\frak a_t}{\frak a}\circ h^{-1}}_{L^\infty}$ in \eqref{2022}. We need to control
$\int\frac{|\mathcal P D_\aa \bar Z_{t,\aa}|^2}{\mathcal A}\,d\aa$. 

We know  $\mathcal A=\frac{A_1}{|Z_{,\aa}|^2}$ and $A_1\ge 1$, so
\begin{equation}\label{2150}
\int\frac{|\mathcal P D_\aa \bar Z_{t,\aa}|^2}{\mathcal A}\,d\aa\le \int  |Z_{,\aa}\mathcal P D_\aa \bar Z_{t,\aa}|^2 \,d\aa.
\end{equation}
We compute
\begin{equation}\label{2112}
\mathcal P D_\aa \bar Z_{t,\aa}= \bracket{\mathcal P, D_\aa} \bar Z_{t,\aa}+ \frac1{Z_{,\aa}}\partial_\aa\mathcal P  \bar Z_{t,\aa};
\end{equation}
further expanding $\bracket{\mathcal P, D_\aa} \bar Z_{t,\aa}$ by \eqref{eq:c5-1} yields
\begin{equation}\label{2146}
 \bracket{\mathcal P, D_\aa}  \bar Z_{t,\aa}=  (-2D_\aa Z_{tt}) D_\aa  \bar Z_{t,\aa} -2(D_\aa Z_t)(\partial_t +b\partial_\aa)D_\aa  \bar Z_{t,\aa};
\end{equation}
and by \eqref{2020} and \eqref{2117},
\begin{equation}\label{2146-1}
\begin{aligned}
 \|Z_{,\aa}\bracket{\mathcal P, D_\aa}  \bar Z_{t,\aa}\|_{L^2}&\lec  \|D_\aa Z_{tt}\|_{L^\infty}\|\partial_\aa^2  \bar Z_{t}\|_{L^2}+\|D_\aa Z_t\|_{L^\infty}\nm{Z_{,\aa}(\partial_t +b\partial_\aa)D_\aa  \bar Z_{t,\aa}}_{L^2}\\&\lec C(\frak E) E_2^{1/2}.
 \end{aligned}
\end{equation}
We are left with controlling $\|\partial_\aa \mathcal P \bar Z_{t,\aa}\|_{L^2}$. 

Taking derivative to $\aa$ to \eqref{base-eq} yields
\begin{equation}\label{2147}
\begin{aligned}
\partial_\aa \mathcal P \bar Z_{t,\aa}&=-(\partial_t+b\partial_\aa)(\paren{\partial_\aa b_\aa} \partial_{\aa}\bar Z_{t})-\paren{\partial_\aa b_\aa}\partial_\aa \bar Z_{tt}-i\paren{\partial_\aa \mathcal A_\aa} \partial_\aa \bar Z_t\\&
-(\partial_t+b\partial_\aa)( b_\aa \partial_{\aa}^2\bar Z_{t})- b_\aa\partial_\aa^2 \bar Z_{tt}-i \mathcal A_\aa \partial_\aa^2 \bar Z_t-b_\aa^2 \partial_\aa^2 \bar Z_{t}- b_\aa (\partial_\aa  b_\aa) \partial_{\aa}\bar Z_{t}\\&
+\frac{\frak a_t}{\frak a}\circ h^{-1} \partial_\aa^2\bar Z_{tt}+2\paren{ \partial_\aa\frac{\frak a_t}{\frak a}\circ h^{-1} }\partial_\aa \bar Z_{tt}+\paren{\partial_\aa^2 \frac{\frak a_t}{\frak a}\circ h^{-1}} (\bar Z_{tt}-i);
\end{aligned}
\end{equation}
we further expand the terms in the first line and the last term in the second line according to the available estimates in \S\ref{ddb} - \S\ref{data}, 
\begin{equation}\label{2148} 
\begin{aligned}
(\partial_t+b\partial_\aa)&(\paren{\partial_\aa b_\aa} \partial_{\aa}\bar Z_{t})=(\partial_t+b\partial_\aa)\braces{\partial_\aa \paren{ b_\aa-2\Re D_\aa Z_t} \partial_{\aa}\bar Z_{t}}\\&+2 \braces{\Re \partial_\aa \paren{  D_\aa Z_t}} (\partial_t+b\partial_\aa)\partial_{\aa}\bar Z_{t}+2 \braces{\Re (\partial_t+b\partial_\aa)\partial_\aa \paren{  D_\aa Z_t}} \partial_{\aa}\bar Z_{t}
\end{aligned}
\end{equation}
we expand the factors in the second line further by product rules,
\begin{equation}\label{2151}
\Re \partial_\aa \paren{  D_\aa Z_t}=\Re \partial_\aa \frac1{Z_{,\aa}}  \partial_\aa Z_t +\Re\frac {\partial_\aa^2 Z_t}{Z_{,\aa}},
\end{equation}
\begin{equation}\label{2152}
\begin{aligned}
\Re (\partial_t+b\partial_\aa)\partial_\aa \paren{  D_\aa Z_t}&= \Re (\partial_t+b\partial_\aa)\partial_\aa \frac1{Z_{,\aa}}  \partial_\aa Z_t +\Re \partial_\aa \frac1{Z_{,\aa}}   (\partial_t+b\partial_\aa)\partial_\aa Z_t \\&+\Re\frac {(\partial_t+b\partial_\aa)\partial_\aa^2 Z_t}{Z_{,\aa}}+\Re\frac {\partial_\aa^2 Z_t}{Z_{,\aa}}(b_\aa-D_\aa Z_t),
\end{aligned}
\end{equation}
here we used \eqref{eq:dza} in the last term; and by \eqref{eq:c7},
\begin{equation}\label{2153}
(\partial_t+b\partial_\aa)\partial_{\aa}\bar Z_{t}=\bar Z_{tt,\aa}-b_\aa \bar Z_{t,\aa}.
\end{equation}
We are now ready to conclude, by \eqref{2123}, \eqref{2130}, \eqref{2121}, \eqref{2122}, \eqref{2020}, \eqref{2046}, \S\ref{quantities-e2} and the expansions \eqref{2148} - \eqref{2153} that
\begin{equation}\label{2154}
\|(\partial_t+b\partial_\aa)(\paren{\partial_\aa b_\aa} \partial_{\aa}\bar Z_{t})\|_{L^2}\lec C(\frak E) E_2^{1/2}.
\end{equation}
Similarly we can conclude, after expanding if necessary, with a similar estimate for all the terms on the right hand side of \eqref{2147} except for $\paren{\partial_\aa^2 \frac{\frak a_t}{\frak a}\circ h^{-1}} (\bar Z_{tt}-i)$. Let
 \begin{equation}\label{2155}
\partial_\aa \mathcal P \bar Z_{t,\aa}=\mathcal R_1+\paren{\partial_\aa^2 \frac{\frak a_t}{\frak a}\circ h^{-1}} (\bar Z_{tt}-i);
\end{equation}
where $\mathcal R_1$ is the sum of the remaining terms on the right hand side of \eqref{2147}. We have, by the argument above, that 
\begin{equation}\label{2156}
\|\mathcal R_1\|_{L^2}\lec C(\frak E) E_2^{1/2}.
\end{equation}
We control the term $\paren{\partial_\aa^2 \frac{\frak a_t}{\frak a}\circ h^{-1}} (\bar Z_{tt}-i)$ with  a similar  idea as  that in \S\ref{ddtea}, by taking advantage of the fact that $\partial_\aa^2 \frac{\frak a_t}{\frak a}\circ h^{-1}$ is purely real. 

Applying $(I-\mathbb H)$ to both sides of \eqref{2155}, and commuting out $\bar Z_{tt}-i$ yields
 \begin{equation}\label{2157}
(I-\mathbb H) \partial_\aa \mathcal P \bar Z_{t,\aa}=(I-\mathbb H)\mathcal R_1+\bracket{\bar Z_{tt}, \mathbb H}\partial_\aa^2 \frac{\frak a_t}{\frak a}\circ h^{-1} + (\bar Z_{tt}-i)(I-\mathbb H)\partial_\aa^2 \frac{\frak a_t}{\frak a}\circ h^{-1} ;
\end{equation}
Because $\mathbb H$ is purely imaginary, we have $\abs{\partial_\aa^2 \frac{\frak a_t}{\frak a}\circ h^{-1}}\le \abs{(I-\mathbb H)\partial_\aa^2 \frac{\frak a_t}{\frak a}\circ h^{-1}}$, and 
\begin{equation}\label{2158}
\abs{  (\bar Z_{tt}-i)\partial_\aa^2 \frac{\frak a_t}{\frak a}\circ h^{-1}}\le \abs{(I-\mathbb H) \partial_\aa \mathcal P \bar Z_{t,\aa}}+\abs{(I-\mathbb H)\mathcal R_1}+\abs{\bracket{\bar Z_{tt}, \mathbb H}\partial_\aa^2 \frac{\frak a_t}{\frak a}\circ h^{-1}}.
\end{equation}
Now by \eqref{eq:c10}, 
\begin{equation}\label{2159}
\bracket{\mathcal P, \partial_\aa} \bar Z_{t,\aa}
=-(\partial_t+b\partial_\aa)(b_\aa\partial_\aa \bar Z_{t,\aa})-b_\aa\partial_\aa (\partial_t+b\partial_\aa)\bar Z_{t,\aa}-i\mathcal A_\aa \partial_\aa \bar Z_{t,\aa};
\end{equation}
so
\begin{equation}\label{2160}
\|\bracket{\mathcal P, \partial_\aa} \bar Z_{t,\aa}\|_{L^2}\lec C(\frak E) E_2^{1/2}
\end{equation}
by \eqref{2020}, \S\ref{aa-zdz} and \S\ref{dtdab}. By Lemma~\ref{basic-3-lemma}, and \eqref{2160},
\begin{equation}\label{2161}
\|(I-\mathbb H) \partial_\aa \mathcal P \bar Z_{t,\aa}\|_{L^2}\le \|(I-\mathbb H)  \mathcal P \partial_\aa \bar Z_{t,\aa}\|_{L^2}+ \|(I-\mathbb H)\bracket{\mathcal P, \partial_\aa} \bar Z_{t,\aa}\|_{L^2}\lec C(\frak E) E_2^{1/2}.
\end{equation}
Now we apply \eqref{3.20} to the commutator on the right hand side of \eqref{2158}. By \eqref{2156} and \eqref{2161}, we have
\begin{equation}\label{2162}
\nm{  (\bar Z_{tt}-i)\partial_\aa^2 \frac{\frak a_t}{\frak a}\circ h^{-1}}_{L^2}\lec C(\frak E) E_2^{1/2}+\nm{\bar Z_{tt,\aa}}_{L^\infty}\nm{\partial_\aa \frac{\frak a_t}{\frak a}\circ h^{-1}}_{L^2}\lec C(\frak E) E_2^{1/2} .
\end{equation}
This together with \eqref{2156} and \eqref{2155} gives
\begin{equation}\label{2163}
\|\partial_\aa \mathcal P \bar Z_{t,\aa}\|_{L^2}\lec C(\frak E) E_2^{1/2}.
\end{equation}
We can now conclude, by \eqref{2150}, \eqref{2112}, \eqref{2146-1} and \eqref{2163}  that
\begin{equation}\label{2164}
\int\frac{|\mathcal P D_\aa \bar Z_{t,\aa}|^2}{\mathcal A}\,d\aa\lec  C(\frak E) E_2^{1/2};
\end{equation}
and consequently,
\begin{equation}\label{2165}
\frac d{dt} E_{D_\aa \partial_\aa \bar Z_t}(t) \lec  C(\frak E) E_2.
\end{equation}
Combining \eqref{2145} and \eqref{2165} yields
\begin{equation}\label{2166}
\frac d{dt} E_2(t) \lec  C(\frak E(t)) E_2(t).
\end{equation}
This concludes the proof for Proposition~\ref{step1}.

\subsection{The proof of Proposition~\ref{step2}}\label{proof-prop2} 
We begin with discussing quantities controlled by $E_3$. Since the idea is similar to that in previous sections, when the estimates are straightforward, we don't always give the full details. 

\subsubsection{Quantities controlled by $E_3$ and a polynomial of $\frak E$ and $E_2$}\label{quantities-e3}
By the definition of $E_3$, and the fact that $1\le A_1 \le C(\frak E)$, cf. \eqref{2000},
\begin{equation}\label{2200}
\|\partial_{\alpha'}^3\bar Z_t\|_{L^2}^2\le E_3,\quad   \nm{Z_{,\alpha'}(\partial_t+b\partial_\aa)\paren{ \frac1{Z_{,\alpha'}}\partial_{\alpha'}^3\bar Z_t}}_{L^2}^2\le C(\frak E) E_3,\quad \nm{\frac1{Z_{,\alpha'}}\partial_{\alpha'}^3\bar Z_t}_{\dot H^{1/2}}^2\le  E_3 .
\end{equation}
By \eqref{eq:dza} and product rules, 
\begin{equation}\label{2201}
 Z_{,\alpha'}(\partial_t+b\partial_\aa)\paren{ \frac1{Z_{,\alpha'}}\partial_{\alpha'}^3\bar Z_t}=(\partial_t+b\partial_\aa)\partial_{\alpha'}^3\bar Z_t+(b_\aa-D_\aa Z_t)
 \partial_{\alpha'}^3\bar Z_t
\end{equation}
so by \eqref{2020},
\begin{equation}\label{2202}
\abs{\nm{(\partial_t+b\partial_\aa)\partial_{\alpha'}^3\bar Z_t}_{L^2}-\nm{Z_{,\alpha'}(\partial_t+b\partial_\aa) \paren{\frac1{Z_{,\alpha'}}\partial_{\alpha'}^3\bar Z_t}}_{L^2}}
\le C(\frak E)\|\partial_{\alpha'}^3\bar Z_t\|_{L^2},
\end{equation}
therefore
\begin{equation}\label{2203}
\nm{(\partial_t+b\partial_\aa)\partial_{\alpha'}^3\bar Z_t}_{L^2}^2\le C(\frak E)E_3.
\end{equation}
We commute out $\partial_\aa$, by \eqref{eq:c7}, 
\begin{equation}\label{2204}
\partial_{\alpha'}(\partial_t+b\partial_\aa)\partial_{\alpha'}^2\bar Z_t=(\partial_t+b\partial_\aa)\partial_{\alpha'}^3\bar Z_t+b_{\alpha'}\partial_{\alpha'}^3\bar Z_t,
\end{equation}
so
\begin{equation}\label{2205}
\|\partial_{\alpha'}(\partial_t+b\partial_\aa)\partial_{\alpha'}^2\bar Z_t\|_{L^2}^2\le C(\frak E)E_3.
\end{equation}
As a consequence of the Sobolev inequality \eqref{eq:sobolev}, and \eqref{2144},
\begin{align}\label{2206}
\|\partial_{\alpha'}^2\bar Z_t\|_{L^\infty}^2&\le 2\|\partial_{\alpha'}^2\bar Z_t\|_{L^2}\|\partial_{\alpha'}^3\bar Z_t\|_{L^2}  \lec   C(\frak E, E_2)E_3^{1/2},\\
\label{2207}
 \|(\partial_t+b\partial_\aa)\partial_{\alpha'}^2\bar Z_t\|_{L^\infty}^2&\le 2\|(\partial_t+b\partial_\aa)\partial_{\alpha'}^2\bar Z_t\|_{L^2}\|\partial_\aa (\partial_t+b\partial_\aa)\partial_{\alpha'}^2\bar Z_t\|_{L^2} \lec C(\frak E, E_2)E_3^{1/2}.
\end{align}
Now we commute out $\partial_\aa^2$ by \eqref{eq:c11}, and get
\begin{equation}\label{2208}
\partial_{\alpha'}^2(\partial_t+b\partial_\aa)\partial_{\alpha'}\bar Z_t=(\partial_t+b\partial_\aa)\partial_{\alpha'}^3\bar Z_t+\paren{\partial_{\alpha'}b_{\alpha'}}\partial_{\alpha'}^2\bar Z_t+2b_{\alpha'}\partial_{\alpha'}^3\bar Z_t,
\end{equation}
We expand the second term further according to the available estimate \eqref{2123}, as we did in \eqref{2125}; we get
\begin{equation}\label{2209}
\|\partial_{\alpha'}^2(\partial_t+b\partial_\aa)\partial_{\alpha'}\bar Z_t\|_{L^2}^2\le C\paren{\frak E, E_2, \nm{\frac1{Z_{,\aa}}}_{L^\infty}}(E_3+1);
\end{equation}
and consequently by Sobolev inequality \eqref{eq:sobolev} and \eqref{2144}, 
\begin{equation}\label{2210}
\|\partial_{\alpha'}(\partial_t+b\partial_\aa)\partial_{\alpha'}\bar Z_t\|_{L^\infty}^2\le C\paren{\frak E, E_2, \nm{\frac1{Z_{,\aa}}}_{L^\infty}}(E_3^{1/2}+1).
\end{equation}

We need to control some additional quantities.

\subsubsection{Controlling  $\|\partial_{\alpha'}A_1\|_{L^\infty}$, $\|\partial_{\alpha'}^2A_1\|_{L^2}$ and $\nm{\partial_{\alpha'}^2\frac{1}{Z_{,\alpha'}}}_{L^2}$} 
We begin with \eqref{2138}: 
\begin{equation}\label{2211}
\partial_\aa A_1=-\Im\paren{ [Z_{t,\aa},\mathbb H]\bar Z_{t,\aa}+[Z_t, \mathbb H]\partial_\aa^2\bar Z_t}.
\end{equation}
By \eqref{eq:b13}, \eqref{2020}, \eqref{2144},
  \begin{equation}\label{2212}
  \nm{\partial_{\alpha'}A_1}_{L^\infty}\lesssim \|Z_{t,\alpha'}\|_{L^2}\|\partial_{\alpha'}^2Z_{t}\|_{L^2}\lesssim C(\frak E)E_2^{1/2}. 
    \end{equation}
Differentiating \eqref{2211} with respect to $\alpha'$ then apply \eqref{3.20}, \eqref{3.22} and use $\mathbb H \bar Z_{t,\aa}=\bar Z_{t,\aa}$  gives
\begin{equation}\label{2213}
\|\partial_{\alpha'}^2A_1\|_{L^2}\lesssim \|Z_{t,\alpha'}\|_{L^\infty}\|\partial_{\alpha'}^2Z_{t}\|_{L^2}\le C(\frak E, E_2),
\end{equation}
where in the last step  we used \eqref{2144}. To estimate $\partial_{\alpha'}^2\frac{1}{Z_{,\alpha'}}$ we begin with 
 \eqref{interface-a1}:
$$-i\frac 1{Z_{,\alpha'}}=\frac{\bar Z_{tt}-i}{A_1}.$$
Taking two derivatives with respect to $\alpha'$ gives
\begin{equation}\label{2214}
-i\partial_{\alpha'}^2\frac{1}{Z_{,\alpha'}}=\frac{\partial_{\alpha'}^2\bar Z_{tt}}{A_1}-2\bar Z_{tt,\alpha'}\frac{\partial_{\alpha'}A_1}{A_1^2}+(\bar Z_{tt}-i)\paren{-\frac{\partial_{\alpha'}^2A_1}{A_1^2}+2\frac{(\partial_{\alpha'}A_1)^2}{A_1^3}};
\end{equation}
therefore, because $A_1\ge 1$, and \eqref{aa1}, \eqref{2020}, \eqref{2144}, \eqref{2212}, \eqref{2213},
\begin{equation}\label{2215}
\begin{aligned}
\nm{\partial_{\alpha'}^2\frac{1}{Z_{,\alpha'}}}_{L^2}\lesssim &\|\partial_{\alpha'}^2\bar Z_{tt}\|_{L^2}+\|\partial_{\alpha'}\bar Z_{tt}\|_{L^2}\|\partial_{\alpha'}A_1\|_{L^\infty}\\&+\nm{\frac{1}{Z_{,\alpha'}}}_{L^\infty}(\|\partial_{\alpha'}^2A_1\|_{L^2}+\|\partial_{\alpha'}A_1\|_{L^2}\|\partial_{\alpha'}A_1\|_{L^\infty})\le C\paren{\frak E, E_2, \nm{\frac{1}{Z_{,\alpha'}}}_{L^\infty}},
\end{aligned}
\end{equation}
and consequently by Sobolev inequality \eqref{eq:sobolev}, and \eqref{2020},
\begin{equation}\label{2216}
\nm{\partial_{\alpha'}\frac{1}{Z_{,\alpha'}}}_{L^\infty} \le C\paren{\frak E, E_2, \nm{\frac{1}{Z_{,\alpha'}}}_{L^\infty}}.
\end{equation}

\subsubsection{Controlling $\|\partial_{\alpha'}^2 b_{\alpha'}\|_{L^2}$ and $\|\partial_{\alpha'}^3Z_{tt}\|_{L^2}$} 
We are now ready to give the estimates for $ \|\partial_{\alpha'}^2b_{\alpha'}\|_{L^2}  $ and $\| \partial_{\alpha'}^3\bar Z_{tt}  \|_{L^2}$.   We begin with \eqref{2122-1}, differentiating with respect to $\aa$, then use \eqref{3.20}, \eqref{3.21}, the fact that $\mathbb H Z_{t,\aa}=-Z_{t,\aa}$, $\mathbb H\frac1{Z_{,\aa}}=\frac1{Z_{,\aa}}$,  and H\"older's inequality; we get
\begin{equation}\label{2217}
\begin{aligned}
\|\partial_{\alpha'}^2(b_{\alpha'}-2\Re D_\aa Z_t)\|_{L^2}&\lesssim \|\partial_{\alpha'}^2Z_{t}\|_{L^2}\nm{\partial_{\alpha'}\frac{1}{Z_{,\alpha'}}}_{L^\infty}+\|Z_{t,\alpha'}\|_{L^\infty}\nm{\partial_{\alpha'}^2\frac{1}{Z_{,\alpha'}}}_{L^2}\\& \le  C\paren{\frak E, E_2, \nm{\frac{1}{Z_{,\alpha'}}}_{L^\infty}}  .
\end{aligned}
\end{equation}
It is easy to show, by product rules and H\"older's inequality that
\begin{equation}\label{2217-1}
\begin{aligned}
\|\partial_\aa^2 D_\aa Z_t\|_{L^2}&\lec  \|\partial_{\alpha'}^2Z_{t}\|_{L^2}\nm{\partial_{\alpha'}\frac{1}{Z_{,\alpha'}}}_{L^\infty}+\|Z_{t,\alpha'}\|_{L^\infty}\nm{\partial_{\alpha'}^2\frac{1}{Z_{,\alpha'}}}_{L^2}+\nm{\frac1{Z_{,\aa}}\partial_\aa^3 Z_t}_{L^2}\\&\lec C\paren{\frak E, E_2, \nm{\frac{1}{Z_{,\alpha'}}}_{L^\infty}}+\nm{\frac1{Z_{,\aa}}}_{L^\infty}E_3^{1/2},
\end{aligned}
\end{equation}
so
\begin{equation}\label{2218}
\|\partial_{\alpha'}^2 b_{\alpha'}\|_{L^2}\lec C\paren{\frak E, E_2, \nm{\frac{1}{Z_{,\alpha'}}}_{L^\infty}}+\nm{\frac1{Z_{,\aa}}}_{L^\infty}E_3^{1/2}. 
\end{equation}
Now starting from \eqref{eq-zta} and taking two derivatives  to $\aa$ gives
\begin{equation}\label{2219}
\begin{aligned}
\partial_{\alpha'}^3\bar Z_{tt}&=\partial_{\alpha'}^2(\partial_t+b\partial_\aa)\partial_{\alpha'}\bar Z_t+\partial_\aa^2 (b_\aa\bar Z_{t,\aa})
\\&=\partial_{\alpha'}^2(\partial_t+b\partial_\aa)\partial_{\alpha'}\bar Z_t+(\partial_\aa^2 b_\aa)\bar Z_{t,\aa}+2(\partial_\aa b_\aa)\partial_\aa \bar Z_{t,\aa}+ b_\aa \partial_\aa^3\bar Z_{t},
\end{aligned}
\end{equation}
so
\begin{equation}\label{2220}
\| \partial_{\alpha'}^3\bar Z_{tt}  \|_{L^2}^2\le C\paren{\frak E, E_2, \nm{\frac{1}{Z_{,\alpha'}}}_{L^\infty}}(E_3+1),
\end{equation}
and as a consequence of \eqref{eq:sobolev},
\begin{equation}\label{2221}
\| \partial_{\alpha'}^2\bar Z_{tt}  \|_{L^\infty}^2\le C\paren{\frak E, E_2, \nm{\frac{1}{Z_{,\alpha'}}}_{L^\infty}}(E_3^{1/2}+1) .
\end{equation}

\subsubsection {Controlling $\partial_{\alpha'}^2\mathcal A_\aa$.} We differentiate \eqref{2136} with respect to $\alpha'$ then take the imaginary parts and use H\"older's inequality, \eqref{3.20}, \eqref{3.21}. We have,
\begin{equation}\label{2222}
\begin{aligned}
\|\partial_{\alpha'}^2\mathcal A_\aa\|_{L^2}&\le \nm{\frac1{Z_{,\alpha'}}}_{L^\infty}\|\partial_{\alpha'}^3\bar Z_{tt}\|_{L^2}+ \nm{\partial_{\alpha'}\frac1{Z_{,\alpha'}}}_{L^\infty}\|\partial_{\alpha'}^2\bar Z_{tt}\|_{L^2}\\&+\nm{\partial_{\alpha'}^2\frac1{Z_{,\alpha'}}}_{L^2}\|\partial_{\alpha'}\bar Z_{tt}\|_{L^\infty}\le C\paren{\frak E, E_2, \nm{\frac1{Z_{,\alpha'}}}_{L^\infty}} ( E_3^{1/2}+1).
\end{aligned}
\end{equation}

\subsubsection{Controlling $\nm{\partial_\aa^2\frac{\frak a_t}{\frak a}\circ h^{-1}}_{L^2}$}
We begin with \eqref{at}, and take two derivatives to $\aa$. 
\begin{equation}
\begin{aligned}
\partial_\aa^2\frac{\frak a_t}{\frak a}\circ h^{-1}&=\frac{\partial_\aa^2 (\partial_t+b\partial_\aa)A_1}{A_1}-2\frac{\partial_\aa (\partial_t+b\partial_\aa)A_1\partial_\aa A_1}{A_1^2}\\&+(\partial_t+b\partial_\aa)A_1 \paren{\frac{-\partial_\aa^2 A_1}{A_1^2} +2\frac{(\partial_\aa A_1)^2}{A_1^3}}+ \partial_\aa^2 (b_\aa-2\Re D_\aa Z_t)
\end{aligned}
\end{equation}
We have controlled $\nm{\partial_\aa^2 (b_\aa-2\Re D_\aa Z_t)}_{L^2}$, $\nm{\partial_\aa^2 A_1}_{L^2}$, $\nm{\partial_\aa A_1}_{L^\infty}$, $\nm{\partial_\aa A_1}_{L^2}$ and  $\nm{\partial_\aa (\partial_t+b\partial_\aa) A_1}_{L^2}$ etc. in \eqref{2217}, \eqref{2212}, \eqref{2213}, \eqref{2139}, \eqref{2142} and \eqref{2000}, \eqref{2021}. We are left with $\partial_\aa^2 (\partial_t+b\partial_\aa)A_1$. 

We begin with \eqref{a1}, taking two derivatives to $\aa$, then one derivative to $\partial_t+b\partial_\aa$. We have
\begin{equation}\label{2222-1}
(\partial_t+b\partial_\aa)\partial_\aa^2 A_1=-\sum_{k=0}^2 C_2^k\Im(\partial_t+b\partial_\aa) \paren{\bracket{ \partial_\aa^k Z_t, \mathbb H}\partial_\aa^{2-k} \bar Z_{t,\alpha'}}
\end{equation}
where $C_2^0=1, C_2^1=2, C_2^2=1$. Using \eqref{eq:c14} to expand the right hand side, then use 
\eqref{3.20}, \eqref{3.21}  and \eqref{3.22} to do the estimates. We have
\begin{equation}
\|(\partial_t+b\partial_\aa)\partial_\aa^2 A_1\|_{L^2}\lec C\paren{\frak E, E_2, \nm{\frac 1{Z_{,\aa}}}_{L^\infty}}(E_3^{1/4}+1).
\end{equation}
Now we use \eqref{eq:c11} to compute
\begin{equation}\label{2223-1}
\partial_\aa^2 (\partial_t+b\partial_\aa)A_1= (\partial_t+b\partial_\aa)\partial_\aa^2 A_1+\partial_\aa b_\aa \partial_\aa A_1+2 b_\aa \partial_\aa^2 A_1.
\end{equation}
Therefore
\begin{equation}
\|\partial_\aa^2(\partial_t+b\partial_\aa) A_1\|_{L^2}\lec C\paren{\frak E, E_2, \nm{\frac 1{Z_{,\aa}}}_{L^\infty}}(E_3^{1/4}+1),
\end{equation}
consequently 
\begin{equation}\label{2224-1}
\nm{\partial_{\alpha'}^2\frac{\frak a_t}{\frak a}\circ h^{-1} }_{L^2}\lec C\paren{\frak E, E_2, \nm{\frac 1{Z_{,\aa}}}_{L^\infty}}(E_3^{1/2}+1).
\end{equation}

\subsubsection{ Controlling $\nm{(\partial_t+b\partial_\aa)\frac{1}{Z_{,\alpha'}}}_{L^\infty}$  and $\nm{(\partial_t+b\partial_\aa)\partial_{\alpha'}^2\frac{1}{Z_{,\alpha'}}}_{L^2}$ }
We begin with \eqref{eq:dza}, 
\begin{equation}\label{2225}
(\partial_t+b\partial_\aa)\frac{1}{Z_{,\alpha'}}=\frac{1}{Z_{,\alpha'}}(b_\aa-D_\aa Z_t),
\end{equation}
differentiating twice with respect to $\aa$; we get
\begin{equation}\label{2226}
\begin{aligned}
&\partial_{\alpha'}^2 (\partial_t+b\partial_\aa)\frac{1}{Z_{,\alpha'}}=\paren{\partial_{\alpha'}^2\frac{1}{Z_{,\alpha'}}}( b_{\alpha'}-D_{\alpha'}Z_t)\\&+2\paren{\partial_{\alpha'}\frac{1}{Z_{,\alpha'}}}( \partial_{\alpha'}b_{\alpha'}-\partial_{\alpha'}D_{\alpha'}Z_t)+\frac{1}{Z_{,\alpha'}}( \partial_{\alpha'}^2b_{\alpha'}-\partial_{\alpha'}^2D_{\alpha'}Z_t).
\end{aligned}
\end{equation}
We further expand
$\partial_{\alpha'}D_{\alpha'}Z_t$ 
and
$\partial_{\alpha'}^2D_{\alpha'}Z_t$ 
by product rules then use H\"older's inequality, \eqref{2020}, \eqref{2144} and \eqref{2215}, \eqref{2216}, \eqref{2217-1}, \eqref{2218}. We have
\begin{equation}\label{2227}
\begin{aligned}
\nm{\partial_{\alpha'}^2(\partial_t+b\partial_\aa)\frac{1}{Z_{,\alpha'}}}_{L^2}&\lec C(\frak E)\nm{\partial_{\alpha'}^2\frac{1}{Z_{,\alpha'}}}_{L^2}+\nm{\partial_{\alpha'}\frac{1}{Z_{,\alpha'}}}_{L^\infty}\| \partial_{\alpha'}b_{\alpha'}-\partial_{\alpha'}D_{\alpha'}Z_t\|_{L^2}\\&+\nm{\frac{1}{Z_{,\alpha'}}}_{L^\infty}\| \partial_{\alpha'}^2b_{\alpha'}-\partial_{\alpha'}^2D_{\alpha'}Z_t\|_{L^2}\lec C\paren{\frak E, E_2, \nm{\frac{1}{Z_{,\alpha'}}}_{L^\infty}}(E_3^{1/2}+1).
\end{aligned}
\end{equation}

Now by \eqref{eq:c11},
\begin{equation}\label{2228}
(\partial_t+b\partial_\aa)\partial_{\alpha'}^2\frac{1}{Z_{,\alpha'}}=\partial_{\alpha'}^2 (\partial_t+b\partial_\aa)\frac{1}{Z_{,\alpha'}}-(\partial_{\alpha'}b_{\alpha'})\partial_{\alpha'}\frac{1}{Z_{,\alpha'}}-2b_{\alpha'}\partial_{\alpha'}^2\frac{1}{Z_{,\alpha'}}
\end{equation}
so
\begin{equation}\label{2229}
 \nm{(\partial_t+b\partial_\aa)\partial_{\alpha'}^2\frac{1}{Z_{,\alpha'}}}_{L^2}\lec
 C\paren{\frak E, E_2, \nm{\frac{1}{Z_{,\alpha'}}}_{L^\infty}}(E_3^{1/2}+1).
 \end{equation}

 We apply \eqref{2020} to \eqref{2225}, and obtain
\begin{equation}\label{2230}
\nm{(\partial_t+b\partial_\aa)\frac{1}{Z_{,\alpha'}}}_{L^\infty} \le C(\frak E)\nm{\frac{1}{Z_{,\alpha'}}}_{L^\infty}.
\end{equation}

\subsubsection{ Controlling $\|\partial_{\alpha'}^2(\partial_t+b\partial_\aa) b_{\alpha'}\|_{L^2}$
}\label{da2dtba} By   \eqref{eq:c11}, 
\begin{equation}\label{2223}
\partial_{\alpha'}^2(\partial_t+b\partial_\aa) b_{\alpha'}=(\partial_t+b\partial_\aa)\partial_{\alpha'}^2 b_{\alpha'}+(\partial_{\alpha'} b_{\alpha'})^2+2 b_{\alpha'}\partial_{\alpha'}^2 b_{\alpha'}
\end{equation}
where by H\"older's and Sobolev inequalities \eqref{eq:sobolev},
\begin{equation}\label{2224}
\begin{aligned}
&\|(\partial_{\alpha'} b_{\alpha'})^2\|_{L^2}+\| b_{\alpha'}\partial_{\alpha'}^2 b_{\alpha'}\|_{L^2} \lesssim \|\partial_{\alpha'} b_{\alpha'}\|_{L^2}\|\partial_{\alpha'} b_{\alpha'}\|_{L^\infty}+\|\partial_{\alpha'}^2 b_{\alpha'}\|_{L^2}\| b_{\alpha'}\|_{L^\infty}\\&\lesssim  \|\partial_{\alpha'} b_{\alpha'}\|_{L^2}^{3/2}\|\partial_{\alpha'}^2b_{\alpha'}\|_{L^2}^{1/2}+\|\partial_{\alpha'}^2 b_{\alpha'}\|_{L^2}\| b_{\alpha'}\|_{L^\infty}\\&\le C\paren{\frak E, E_2, \nm{\frac 1{Z_{,\aa}}}_{L^\infty}} (E_3^{1/2}+1).
\end{aligned}
\end{equation}
Now we consider $(\partial_t+b\partial_\aa)\partial_{\alpha'}^2 b_{\alpha'}$. We begin with \eqref{ba-1}, 
differentiating twice with respect to $\aa$, then to $\partial_t+b\partial_\aa$,
\begin{equation}\label{2232}
\begin{aligned}
(\partial_t+b\partial_\aa)\partial_\aa^2(b_\aa-2\Re D_\aa Z_t)&=\sum_{k=0}^2 C_2^k\Re(\partial_t+b\partial_\aa) \paren{\bracket{ \partial_\aa^k \frac1{Z_{,\aa}}, \mathbb H}\partial_\aa^{2-k}  Z_{t,\alpha'}}\\&+ \sum_{k=0}^2 C_2^k\Re(\partial_t+b\partial_\aa) \paren{\bracket{\partial_\aa^k Z_t, \mathbb H}\partial_\aa^{3-k} \frac1{Z_{,\aa}}  }.
\end{aligned}
\end{equation}
where $C_2^0=1, C_2^1=2, C_2^2=1$. We expand $(\partial_t+b\partial_\aa)\partial_{\alpha'}^2 D_\aa Z_t$ by product rules, and use \eqref{eq:c14} to further expand the right hand side of \eqref{2232}; 
we then use \eqref{3.20}, \eqref{3.21}, \eqref{3.22}, \eqref{eq:b12} and H\"older's inequality to do the estimates. We have
\begin{equation}\label{2233}
\begin{aligned}
&\|(\partial_t+b\partial_\aa)\partial_{\alpha'}^2b_{\alpha'}\|_{L^2}\lesssim \|(\partial_t+b\partial_\aa)\partial_{\alpha'}^3Z_{t}\|_{L^2}\nm{\frac1{Z_{,\alpha'}}}_{L^\infty}\\&+\|\partial_{\alpha'}^3Z_{t}\|_{L^2}\nm{(\partial_t+b\partial_\aa)\frac1{Z_{,\alpha'}}}_{L^\infty}+ \|(\partial_t+b\partial_\aa)\partial_{\alpha'}^2Z_{t}\|_{L^\infty}\nm{\partial_{\alpha'}\frac{1}{Z_{,\alpha'}}}_{L^2}\\&+\|\partial_{\alpha'}^2Z_{t}\|_{L^\infty}\paren{\nm{(\partial_t+b\partial_\aa)\partial_{\alpha'}\frac{1}{Z_{,\alpha'}}}_{L^2}+\nm{\partial_{\alpha'}(\partial_t+b\partial_\aa)\frac{1}{Z_{,\alpha'}}}_{L^2}}\\&
+\|\partial_{\alpha'}^2Z_{t}\|_{L^\infty}\|b_{\alpha'}\|_{L^\infty}\nm{\partial_{\alpha'}\frac{1}{Z_{,\alpha'}}}_{L^2}+\|(\partial_t+b\partial_\aa)\partial_{\alpha'}Z_{t}\|_{L^\infty}\nm{\partial_{\alpha'}^2\frac{1}{Z_{,\alpha'}}}_{L^2}\\&+\|\partial_{\alpha'}Z_{t}\|_{L^\infty}\|b_{\alpha'}\|_{L^\infty}\nm{\partial_{\alpha'}^2\frac{1}{Z_{,\alpha'}}}_{L^2}+\|Z_{tt,\alpha'}\|_{L^\infty}\nm{\partial_{\alpha'}^2\frac{1}{Z_{,\alpha'}}}_{L^2}\\&+\|Z_{t,\alpha'}\|_{L^\infty}\nm{(\partial_t+b\partial_\aa)\partial_{\alpha'}^2\frac{1}{Z_{,\alpha'}}}_{L^2}.
\end{aligned}
\end{equation}
This, together with \eqref{2224} gives, 
\begin{equation}\label{2234}
\|(\partial_t+b\partial_\aa)\partial_{\alpha'}^2b_{\alpha'}\|_{L^2}+\|\partial_{\alpha'}^2(\partial_t+b\partial_\aa)b_{\alpha'}\|_{L^2}\lec C\paren{\frak E, E_2, \nm{\frac1{Z_{,\aa}}}_{L^\infty}}(E_3^{1/2}+1).
\end{equation}

\subsubsection{Controlling $\frac d{dt} E_3(t)$}\label{ddte3} We know $E_3(t)$ consists of $E_{D_\aa \partial_\aa^2 \bar Z_t}$ and $\|\partial_\aa^3 \bar Z_t\|_{L^2}^2$. We apply Lemma~\ref{basic-e} to $E_{D_\aa \partial_\aa^2 \bar Z_t}$ and Lemma~\ref{basic-e2} to $\|\partial_\aa^3 \bar Z_t\|_{L^2}^2$. We begin with $\|\partial_\aa^3 \bar Z_t\|_{L^2}^2$. We have, by Lemma~\ref{basic-e2},
\begin{equation}\label{2235}
\frac d{dt}\|\partial_\aa^3 \bar Z_t\|_{L^2}^2\lec \|(\partial_t+b\partial_\aa)\partial_\aa^3 \bar Z_t\|_{L^2}\|\partial_\aa^3 \bar Z_t\|_{L^2}+\|b_\aa\|_{L^\infty}\|\partial_\aa^3 \bar Z_t\|_{L^2}^2
\end{equation}
We have controlled all the factors, in \eqref{2020} and \eqref{2203}. We have
\begin{equation}\label{2236}
\frac d{dt}\|\partial_\aa^3 \bar Z_t\|_{L^2}^2\lec C(\frak E) E_3(t).
\end{equation}

We now consider $E_{D_\aa \partial_\aa^2 \bar Z_t}$. Applying Lemma~\ref{basic-e} to $\Th=D_\aa \partial_\aa^2 \bar Z_t$ yields
\begin{equation}\label{2237}
\frac d{dt} E_{D_\aa \partial_\aa^2 \bar Z_t}(t)\le \nm{\frac{\frak a_t}{\frak a}\circ h^{-1}}_{L^\infty} E_{D_\aa \partial_\aa^2 \bar Z_t}(t)+ 2E_{D_\aa \partial_\aa^2 \bar Z_t}(t)^{1/2}\paren{\int \frac{|\mathcal PD_\aa \partial_\aa^2 \bar Z_t|^2}{\mathcal A}\,d\aa}^{1/2}
\end{equation}
We have controlled the factor $\nm{\frac{\frak a_t}{\frak a}\circ h^{-1}}_{L^\infty} $ in \S\ref{ata-da1}, we are left with the second term. We know, by $\mathcal A |Z_{,\aa}|^2=A_1\ge 1$, that 
\begin{equation}\label{2238}
\int \frac{|\mathcal PD_\aa \partial_\aa^2 \bar Z_t|^2}{\mathcal A}\,d\aa\le \int |Z_{,\aa}\mathcal PD_\aa \partial_\aa^2 \bar Z_t|^2\,d\aa.
\end{equation}
We  compute 
\begin{equation}
\label{2239}
\mathcal P D_\aa \partial_\aa^2 \bar Z_{t}= \bracket{\mathcal P, D_\aa}  \partial_\aa^2 \bar Z_{t}+ D_\aa\bracket{\mathcal P,\partial_\aa}  \bar Z_{t,\aa}+ D_\aa\partial_\aa [\mathcal P , \partial_\aa] \bar Z_{t}+D_\aa \partial_\aa^2\mathcal P \bar Z_t;
\end{equation}
and expand further by \eqref{eq:c5-1},
\begin{equation}\label{2240}
 \bracket{\mathcal P, D_\aa} \partial_\aa^2 \bar Z_{t}=  (-2D_\aa Z_{tt}) D_\aa \partial_\aa^2 \bar Z_{t}-2(D_\aa Z_t)(\partial_t+b\partial_\aa) D_\aa \partial_\aa^2 \bar Z_{t};
\end{equation}
 by
\eqref{eq:c10} and product rules,
\begin{equation}\label{2241}
\begin{aligned}
\partial_\aa[\mathcal P, \partial_\aa]\bar Z_{t,\aa}& =-(\partial_t+b\partial_\aa)(\partial_\aa b_\aa\partial_\aa^2 \bar Z_{t})-\partial_\aa b_\aa\partial_\aa (\partial_t+b\partial_\aa)\bar Z_{t,\aa} -i\partial_\aa\mathcal A_\aa \partial_\aa^2 \bar Z_{t}\\&
-(\partial_t+b\partial_\aa)( b_\aa\partial_\aa^3 \bar Z_{t})- b_\aa\partial_\aa^2 (\partial_t+b\partial_\aa)\bar Z_{t,\aa} -i\mathcal A_\aa \partial_\aa^3 \bar Z_{t}\\&- b_\aa^2\partial_\aa^3 \bar Z_{t}-b_\aa\partial_\aa b_\aa \partial_\aa^2 \bar Z_{t};
\end{aligned}
\end{equation}
and by \eqref{eq:c10},
\begin{equation}\label{2242}
\partial_\aa^2[\mathcal P,\partial_\aa]\bar Z_t=-\partial_\aa^2(\partial_t+b\partial_\aa)(b_\aa \partial_{\aa}\bar Z_{t})-\partial_\aa^2(b_\aa\partial_\aa \bar Z_{tt})-i\partial_\aa^2(\mathcal A_\aa \partial_\aa \bar Z_t),
\end{equation}
and then expand \eqref{2242} by product rules. 
We have controlled all the factors on the right hand sides of \eqref{2240}, \eqref{2241} and \eqref{2242} in \eqref{2020}, \eqref{2144} and \S\ref{quantities-e3} - \S\ref{da2dtba}. We have, by H\"older's inequality,
\begin{equation}\label{2243}
\int |Z_{,\aa} \bracket{\mathcal P, D_\aa} \partial_\aa^2 \bar Z_{t}|^2+|\partial_\aa[\mathcal P, \partial_\aa]\bar Z_{t,\aa}|^2+|\partial_\aa^2[\mathcal P,\partial_\aa]\bar Z_t|^2\,d\aa\lec C\paren{\frak E, E_2, \nm{\frac1{Z_{,\aa}}}_{L^\infty}}(E_3+1).
\end{equation}
We are left with the last term $\partial_\aa^3 \mathcal P\bar Z_{t}$. We expand by product rules, starting from \eqref{quasi-r1}; we have
\begin{equation}\label{2244}
\begin{aligned}
\partial_\aa^3\mathcal P \bar Z_{t}&=
\frac{\frak a_t}{\frak a}\circ h^{-1} \partial_\aa^3\bar Z_{tt}+3\partial_\aa\paren{\frac{\frak a_t}{\frak a}\circ h^{-1} }\partial_\aa^2\bar Z_{tt}+3\partial_\aa^2\paren{\frac{\frak a_t}{\frak a}\circ h^{-1} }\bar Z_{tt,\aa}\\&+\partial_\aa^3\paren{\frac{\frak a_t}{\frak a}\circ h^{-1} }(\bar Z_{tt}-i).
\end{aligned}
\end{equation}
Let 
\begin{equation}\label{2245}
\mathcal R_2= \frac{\frak a_t}{\frak a}\circ h^{-1} \partial_\aa^3\bar Z_{tt}+3\partial_\aa\paren{\frac{\frak a_t}{\frak a}\circ h^{-1} }\partial_\aa^2\bar Z_{tt}+3\partial_\aa^2\paren{\frac{\frak a_t}{\frak a}\circ h^{-1} }\bar Z_{tt,\aa}.
\end{equation}
We have controlled all the factors of the terms in $\mathcal R_2$, with
\begin{equation}\label{2246}
\begin{aligned}
&\nm{\mathcal R_2}_{L^2} \le \nm{\frac{\frak a_t}{\frak a}\circ h^{-1}}_{L^\infty}\nm{ \partial_\aa^3\bar Z_{tt}}_{L^2}+3\nm{\partial_\aa\paren{\frac{\frak a_t}{\frak a}\circ h^{-1} }}_{L^2}\nm{\partial_\aa^2\bar Z_{tt}}_{L^\infty}\\&+3\nm{\partial_\aa^2\paren{\frac{\frak a_t}{\frak a}\circ h^{-1} }}_{L^2}\nm{\bar Z_{tt,\aa}}_{L^2}
\lec C\paren{\frak E, E_2, \nm{\frac1{Z_{,\aa}}}_{L^\infty}}(E_3^{1/2}+1).
\end{aligned}
\end{equation}
We are left with controlling $\nm{ \partial_\aa^3\paren{\frac{\frak a_t}{\frak a}\circ h^{-1} }(\bar Z_{tt}-i)}_{L^2}$. We  use a similar idea as that in \S\ref{ddtea}, that is, to take advantage of the fact that $
 \partial_\aa^3\paren{\frac{\frak a_t}{\frak a}\circ h^{-1} }$ is purely real. 
 
 Applying $(I-\mathbb H)$ to both sides of \eqref{2244}, with the first three terms replaced by $\mathcal R_2$, and commuting out $\bar Z_{tt}-i$ yields,
 \begin{equation}\label{2247}
 (I-\mathbb H)\partial_\aa^3\mathcal P \bar Z_{t}=(I-\mathbb H)\mathcal R_2+[\bar Z_{tt},\mathbb H]\partial_\aa^3\paren{\frac{\frak a_t}{\frak a}\circ h^{-1} }+ (\bar Z_{tt}-i)(I-\mathbb H)\partial_\aa^3\paren{\frac{\frak a_t}{\frak a}\circ h^{-1} }.
 \end{equation}
Now
\begin{equation}\label{2248}
\partial_\aa^3\mathcal P \bar Z_{t}=\partial_\aa^2 [\partial_\aa, \mathcal P]\bar Z_t+ \partial_\aa [\partial_\aa, \mathcal P] \bar Z_{t,\aa}+[\partial_\aa, \mathcal P]\partial_\aa^2\bar Z_t+\mathcal P\partial_\aa^3 \bar Z_t.
\end{equation}
and by \eqref{eq:c10}, 
\begin{equation}\label{2249}
[\mathcal P, \partial_\aa]\partial_\aa^2\bar Z_t=-(\partial_t+b\partial_\aa)(b_\aa\partial_\aa \partial_\aa^2\bar Z_t)-b_\aa\partial_\aa (\partial_t+b\partial_\aa)\partial_\aa^2\bar Z_t-i\mathcal A_\aa \partial_\aa \partial_\aa^2\bar Z_t
\end{equation}
so by \S\ref{basic-quantities}-\S\ref{dtdab}, \S\ref{quantities-e3} and H\"older's inequality, 
\begin{equation}\label{2250}
\|[\mathcal P, \partial_\aa]\partial_\aa^2\bar Z_t\|_{L^2}\lec C(\frak E) E_3^{1/2}.
\end{equation}
Applying Lemma~\ref{basic-3-lemma} to the last term in \eqref{2248}. We have, by \eqref{2243}, \eqref{2250} and Lemma~\ref{basic-3-lemma}, that
\begin{equation}\label{2251}
\| (I-\mathbb H)\partial_\aa^3\mathcal P \bar Z_{t}\|_{L^2}\lec C\paren{\frak E, E_2, \nm{\frac1{Z_{,\aa}}}_{L^\infty}}(E_3^{1/2}+1).
\end{equation}
This gives, by \eqref{2247}, that
\begin{equation}\label{2252}
\begin{aligned}
\nm{ \partial_\aa^3\paren{\frac{\frak a_t}{\frak a}\circ h^{-1} }(\bar Z_{tt}-i)}_{L^2}&\le \|(I-\mathbb H)\partial_\aa^3\mathcal P \bar Z_{t}\|_{L^2}+\|\mathcal R_2\|_{L^2}  +\nm{[\bar Z_{tt},\mathbb H]\partial_\aa^3\paren{\frac{\frak a_t}{\frak a}\circ h^{-1} }}_{L^2}\\&
\lec C\paren{\frak E, E_2, \nm{\frac1{Z_{,\aa}}}_{L^\infty}}(E_3^{1/2}+1);
\end{aligned}
\end{equation}
here for the commutator, we used \eqref{3.20}, \eqref{2144}, and \eqref{2224-1}. Combining with \eqref{2244},
 \eqref{2245},  \eqref{2246} yields
 \begin{equation}
 \nm{ \partial_\aa^3\mathcal P \bar Z_{t} }_{L^2} \lec C\paren{\frak E, E_2, \nm{\frac1{Z_{,\aa}}}_{L^\infty}}(E_3^{1/2}+1). 
 \end{equation}
Further combing with \eqref{2238}, \eqref{2239}, \eqref{2243} gives
\begin{equation}
\int \frac{|\mathcal PD_\aa \partial_\aa^2 \bar Z_t|^2}{\mathcal A}\,d\aa\lec C\paren{\frak E, E_2, \nm{\frac1{Z_{,\aa}}}_{L^\infty}}(E_3+1). 
\end{equation}
By \eqref{2237}, this shows that Proposition~\ref{step2} holds.

\subsection{Completing the proof for Theorem~\ref{blow-up}}\label{complete1}
 Now we continue the discussion in \S\ref{proof1}, 
assuming that the initial data satisfies the assumption of Theorem~\ref{blow-up}, and the solution $Z$ satisfies the regularity property in Theorem~\ref{blow-up}. By \eqref{step1-2} and the ensuing discussion, to complete the proof of Theorem~\ref{blow-up}, it suffices to show that for the given data,  $E_2(0)<\infty$ and $E_3(0)<\infty$;  and $\sup_{[0, T]}(E_2(t)+E_3(t)+\frak E(t))$ control the higher order Sobolev norm $\sup_{[0, T]}(\|\partial_\aa^3Z_t(t)\|_{\dot{H}^{1/2}}+\|\partial_\aa^3 Z_{tt}(t)\|_{L^2})$. 

By \eqref{2117-1} and \eqref{2124},
\begin{equation}
 Z_{,\alpha'}(\partial_t+b\partial_\aa) \paren{\frac1{Z_{,\alpha'}}\partial_{\alpha'}^2\bar Z_t}= \partial_{\alpha'}^2\bar Z_{tt}- (b_\aa+D_\aa Z_t)\partial_{\alpha'}^2\bar Z_t- (\partial_\aa b_\aa)\bar Z_{t,\aa};
\end{equation}
expanding further according to the available  estimates in \eqref{2001}, \eqref{2123}, and using H\"older's inequality, we have
\begin{equation}
\begin{aligned}
\nm{ Z_{,\alpha'}(\partial_t+b\partial_\aa) D_\aa\partial_{\alpha'}\bar Z_t}_{L^2}&\lec 
\|\partial_{\alpha'}^2\bar Z_{tt}\|_{L^2}+ \|Z_{t,\aa}\|_{L^2}\nm{\partial_\aa\frac1{Z_{,\aa}}}_{L^2}\|\partial_\aa^2 Z_{t}\|_{L^2}+\|D_\aa Z_t\partial_\aa^2 \bar Z_t\|_{L^2}\\&+ \|Z_{t,\aa}\|_{L^\infty}^2\nm{\partial_\aa\frac1{Z_{,\aa}}}_{L^2}+\|(\partial_\aa D_\aa Z_t)\partial_\aa \bar Z_t\|_{L^2};
\end{aligned}
\end{equation}
it is clear that we also have
\begin{equation}
\nm{\frac1{Z_{,\aa}}\partial_\aa^2\bar Z_t}_{\dot H^{1/2}}\le C\paren{\nm{\frac 1{Z_{,\aa}}-1}_{H^1}, \|Z_t\|_{H^{2+1/2}}};
\end{equation}
so for the given initial data, we have $E_2(0)<\infty$. A similar argument shows that we also have $E_3(0)<\infty$. This implies, by \eqref{step1-2} and the ensuing discussion, that $\sup_{[0, T_0)}(E_2(t)+E_3(t))<\infty$ provided $\sup_{[0, T_0)}\frak E(t)<\infty$. 
On the other hand, we have shown in \eqref{2220} that $\|\partial_\aa^3 \bar Z_{tt}(t)\|_{L^2}$ is controlled by $E_2(t)$, $E_3(t)$ and $\nm{\frac1{Z_{,\aa}}(t)}_{L^\infty}$; and  by \eqref{Hhalf},
$$\|\partial_{\alpha'}^3 \bar Z_{t}\|_{\dot H^{1/2}}\lesssim \|Z_{,\alpha'}\|_{L^\infty}\paren{\nm{\frac1{Z_{,\alpha'}}\partial_{\alpha'}^3 \bar Z_{t}}_{\dot H^{1/2}}+\nm{\partial_{\alpha'}\frac1{Z_{,\alpha'}}}_{L^2}\|\partial_{\alpha'}^3 \bar Z_{t}\|_{L^2}},
$$
so $\|\partial_{\alpha'}^3 \bar Z_{t}\|_{\dot H^{1/2}}$ is controlled by $E_3(t)$, $\frak E(t)$ and $\|Z_{,\alpha'}(t)\|_{L^\infty}$.  With a further application of \eqref{2101} and \eqref{2101-1}, we have
\begin{equation}
\sup_{[0, T_0)} \|\partial_\aa^3 \bar Z_{tt}(t)\|_{L^2}+ \|\partial_{\alpha'}^3 \bar Z_{t}\|_{\dot H^{1/2}}<\infty,\quad\text{provided}\quad \sup_{[0, T_0)}\frak E(t)<\infty.
\end{equation}
This, together with \eqref{2101}, \eqref{2109}, \eqref{2110} and Theorem~\ref{prop:local-s} shows 
that Theorem~\ref{blow-up} holds.

\section{The proof of Theorem~\ref{unique}}\label{proof3}
\subsection{Some basic preparations}\label{prepare}
 We begin with some basic preparatory analysis that will be used in the proof of Theorem~\ref{unique}. 
 
 In the first lemma we construct an energy functional for the difference of the solutions of an equation of the type \eqref{quasi-r}. 
We will apply Lemma~\ref{dlemma1} to $\Theta=\bar Z_t,\ \frac1{Z_{,\aa}}-1$ and $\bar Z_{tt}$.

\begin{lemma}\label{dlemma1}
Assume  $\Theta$, $\tTheta$ are smooth and decay at the spatial infinity, and satisfy
\begin{equation}\label{eqdf}
\begin{aligned}
&(\partial_t+b\partial_\aa)^2\Theta+i\Ac\partial_\aa\Theta=G,\\
&(\partial_t+\tb\partial_\aa)^2\tTheta+i\tAc\partial_\aa\tTheta=\tilde G.\\
\end{aligned}
\end{equation}
Let 
\begin{equation}
\mathfrak F(t)=\int \frac{\kappa}{A_1}\abs{ Z_{,\aa}\paren{(\partial_t+b\partial_\aa)\Theta+\mathfrak c}- \Zf_{,\aa}\circ l\paren{(\partial_t+\tb\partial_\aa)\tTheta\circ l+\mathfrak c}}^2+i\partial_\aa(\Theta-\tTheta\circ l)\bar{(\Theta-\tTheta\circ l)}\,d\aa,
\end{equation}
where $\kappa=\sqrt{\frac{A_1}{\tAone\circ l} l_\aa}$, $\frak c$ is a constant, and 
\begin{equation}
{\bf F}(t)=\int \abs{ Z_{,\aa}\paren{(\partial_t+b\partial_\aa)\Theta+\frak c}- \Zf_{,\aa}\circ l\paren{(\partial_t+\tb\partial_\aa)\tTheta\circ l+\frak c}}^2\ \ d\aa.
\end{equation}
Then
\begin{equation}\label{dlemma1-inq}
\begin{aligned}
&\frak F'(t)\lec {\bf F}(t)^{\frac12} \nm{\frac{\kappa}{A_1}}_{L^\infty}\nm{ Z_{,\aa}G- \Zf_{,\aa}\circ l\tilde G\circ l}_{L^2}\\&\quad+{\bf F}(t)^{\frac12} \nm{\frac{\kappa}{A_1}}_{L^\infty}\nm{\frac{(\partial_t+b\partial_\aa)\kappa }{\kappa}}_{L^2}\paren{\nm{Z_{,\aa}((\partial_t+b\partial_\aa)\Theta+\mathfrak c) }_{L^\infty}+\nm{\Zf_{,\aa}((\partial_t+\tb\partial_\aa)\tTheta+\mathfrak c)}_{L^\infty}}\\&
+{\bf F}(t)^{\frac12} \nm{\frac{\kappa}{A_1}}_{L^\infty}\paren{\nm{\frac{\frak a_t}{\frak a}\circ h^{-1}-\frac{\taft}{\taf}\circ h^{-1}}_{L^2}+\nm{D_\aa Z_t-\tD_\aa\Zf_t\circ l }_{L^2}}\nm{\Zf_{,\aa}((\partial_t+\tb\partial_\aa)\tTheta+\mathfrak c)}_{L^\infty}\\&+{\bf F}(t) \nm{\frac{\kappa}{A_1}}_{L^\infty}\paren{
\nm{\frac{\frak a_t}{\frak a}}_{L^\infty}+\nm{D_\aa Z_t}_{L^\infty}}
+\nm{1-\kappa}_{L^2}{\bf F}(t)^{\frac12}\paren{\nm{D_\aa\Theta}_{L^\infty}+\nm{\frac{l_\aa}{\kappa}}_{L^\infty}\nm{\tD_\aa\tTheta}_{L^\infty}}\\&+2\Re i\int \bar{\paren{\frac1{Z_{,\aa}}-U_l\frac1{\Zf_{,\aa}} }} \paren{\bar{U_l(\Zf_{,\aa}((\partial_t+\tb\partial_\aa)\tTheta+\frak c))}\Theta_\aa-\bar{Z_{,\aa}((\partial_t+b\partial_\aa)\Theta+\frak c)}(\tTheta\circ l)_\aa}\,d\aa.
\end{aligned}
\end{equation}
\end{lemma}
\begin{remark}
By definition, $\frac{\kappa}{A_1}= \sqrt{\frac{l_\aa}{A_1\tAone\circ l}}$. And in what follows, $\sqrt{\tilde a h_\aa}\, \kappa\circ h=\frac {\sqrt {A_1}}{{\Zf_{,\aa}}\circ l}\circ h$, $\sqrt{a h_\aa}=\frac {\sqrt {A_1}}{Z_{,\aa}}\circ h$.
\end{remark}

\begin{proof}    
Let $\theta=\Theta\circ h$, and $\tilde\theta=\tTheta\circ \th$. We know $\theta$, $\ttheta$ satisfy
\begin{equation}\label{eqdfl}
\begin{aligned}
&\partial_t^2\theta+i\af \partial_\a\theta=G\circ h,\\
&\partial_t^2\ttheta+i\taf\partial_\a\ttheta=\tilde G\circ \th.\\
\end{aligned}
\end{equation}
Changing coordinate by $h$, we get
\begin{equation}
\mathfrak F(t)=\int\abs{ \sqrt{\frac{k}{a}} \paren{\theta_t+\mathfrak c}- \frac1{\sqrt{{\tilde a} k}}\paren{\ttheta_t+\mathfrak c}}^2+i\partial_\a(\theta-\ttheta)\bar{(\theta-\ttheta)}\,d\a,
\end{equation}
where $k=\kappa\circ h$, $\sqrt{a}:= \frac{\sqrt{A_1\circ h h_\a}}{z_{\a}}$ and 
 $\sqrt{{\tilde a} }:= \frac{\sqrt{\tAone\circ \th \th_\a}}{\zf_{\a}}$. Notice that here, $\sqrt{a}$ and 
 $\sqrt{{\tilde a} }$ are complex valued, and $|\sqrt{a}|^2=\frak a$, $|\sqrt{\tilde a}|^2=\taf$.  Differentiating to $t$, integrating by parts, then applying equations \eqref{eqdfl}, we get
 \begin{equation}\label{2300}
 \begin{aligned}
 &\mathfrak F'(t)=2\Re\int  \bar{\braces{
 \sqrt{\frac{k}{a}} \paren{\theta_t+\mathfrak c}- \frac1{\sqrt{{\tilde a} k}}\paren{\ttheta_t+\mathfrak c}}}\left\{\sqrt{\frac{k}{a}} G\circ h- \frac1{\sqrt{{\tilde a} k}}\tilde G\circ \th\right\}\\&\qquad+\bar{\braces{
 \sqrt{\frac{k}{a}} \paren{\theta_t+\mathfrak c}- \frac1{\sqrt{{\tilde a} k}}\paren{\ttheta_t+\mathfrak c}}}\left\{\frac12{\frac{k_t}k \paren{\sqrt{\frac{k}{a}} \paren{\theta_t+\mathfrak c}+ \frac1{\sqrt{{\tilde a} k}}\paren{\ttheta_t+\mathfrak c} }  }\right\}\\&-\bar{\braces{
 \sqrt{\frac{k}{a}} \paren{\theta_t+\mathfrak c}- \frac{\ttheta_t+\mathfrak c}{\sqrt{{\tilde a} k}}}}\braces{\paren{\frac12\frac{\aft}{\af}-i\Im D_\a z_t}\sqrt{\frac{k}{a}}\paren{\theta_t+\mathfrak c}-\paren{\frac12\frac{\taft}{\taf}-i\Im D_\a \zf_t}\frac{\ttheta_t+\mathfrak c}{\sqrt{{\tilde a} k}}   }\,d\a \\&+
 2\Re i\int \bar{(\theta_t-\ttheta_t)}(\theta_\a-\ttheta_a)-\bar{\braces{
 \sqrt{\frac{k}{a}} \paren{\theta_t+\mathfrak c}- \frac1{\sqrt{{\tilde a} k}}\paren{\ttheta_t+\mathfrak c}}}\braces{\sqrt k\bar{\sqrt{  a}}\theta_\a-\frac{\bar{\sqrt{{\tilde a}}}}{\sqrt{k}}\ttheta_\a }   \,d\a\\&=I+II
  \end{aligned}
  \end{equation}
  where $I$ consists of the terms in the first three lines and $II$ is the last line
  $$II:=2\Re i\int \bar{(\theta_t-\ttheta_t)}(\theta_\a-\ttheta_a)-\bar{\braces{
 \sqrt{\frac{k}{a}} \paren{\theta_t+\mathfrak c}- \frac1{\sqrt{{\tilde a} k}}\paren{\ttheta_t+\mathfrak c}}}\braces{\sqrt k\bar{\sqrt{  a}}\theta_\a-\frac{\bar{\sqrt{{\tilde a}}}}{\sqrt{k}}\ttheta_\a }   \,d\a.$$
 Further regrouping terms in $II$ we get
 \begin{equation}\label{2301}
 \begin{aligned}
 II=&2\Re i\int \frac{(1-k)}{\sqrt{k}}\bar{ \braces{
 \sqrt{\frac{k}{a}} \paren{\theta_t+\mathfrak c}- \frac1{\sqrt{{\tilde a} k}}\paren{\ttheta_t+\mathfrak c}}  }\braces{\bar{\sqrt{a}} \theta_\a+\bar{\sqrt{{\tilde a}}}\ttheta_\a} \,d\a\\&+
 2\Re i\int \bar{(\sqrt{a}-\sqrt{{\tilde a}} k )}\paren{\bar{\frac1{\sqrt{{\tilde a}} k}(\ttheta_t+\frak c)}\theta_\a-\bar{\frac1{\sqrt{a}}(\theta_t+\frak c)}\ttheta_\a}\,d\a.
 \end{aligned}
 \end{equation}
Changing variables by $h^{-1}$ in the integrals in \eqref{2300} and \eqref{2301}, and then applying  Cauchy-Schwarz and H\"older's inequalities, we obtain \eqref{dlemma1-inq}.
\end{proof}

We have the following basic identities and inequalities.
\begin{proposition} Let $\mathcal Q_l= U_{l}\mathbb H U_{l^{-1}}-\mathbb H$, where $l:\mathbb R\to \mathbb R$ is a diffeomorphism,\footnote{We say $l:\mathbb R\to\mathbb R$ is a diffeomorphism, if 
$l:\mathbb R\to\mathbb R$ is one-to-one and onto, and $l, \ l^{-1}\in C^1(\mathbb R)$, with $\|l_\aa\|_{L^\infty}+\|(l^{-1})_\aa\|_{L^\infty}<\infty$.}
with $l_\aa-1\in L^2$.  For any $f\in H^1(\mathbb R)$, we have
\begin{align}
\nm{\mathcal Q_l f}_{\dot H^{\frac12}}&\le C(\nm{(l^{-1})_\aa}_{L^\infty}, \nm{l_\aa}_{L^\infty})\|l_\aa-1\|_{L^2}\|\partial_\aa f\|_{L^2};\label{q1}\\
 \nm{\mathcal Q_l f}_{L^\infty}&\le C(\nm{(l^{-1})_\aa}_{L^\infty}, \nm{l_\aa}_{L^\infty})\|l_\aa-1\|_{L^2}\|\partial_\aa f\|_{L^2};\label{q2}\\
 \nm{\mathcal Q_l f}_{L^2}&\le C(\nm{(l^{-1})_\aa}_{L^\infty}, \nm{l_\aa}_{L^\infty})\|l_\aa-1\|_{L^2}\| f\|_{L^\infty};\label{q3}\\
 \nm{\mathcal Q_l f}_{L^2}&\le C(\nm{(l^{-1})_\aa}_{L^\infty}, \nm{l_\aa}_{L^\infty})\|l_\aa-1\|_{L^2}\| f\|_{\dot H^{1/2}  }.\label{q4}
\end{align}

\end{proposition}
\begin{proof}
We know 
\begin{equation}\label{2305}
U_{l}\mathbb H U_{l^{-1}}f(\aa)=\frac1{\pi i}\int\frac{f(\bb) l_\bb(\bb)}{l(\aa)-l(\bb)}\,d\bb
\end{equation}
so
\begin{equation}\label{2306}
\begin{aligned}
\mathcal Q_l f&=\frac1{\pi i}\int\paren{\frac{ l_\bb(\bb)-1}{l(\aa)-l(\bb)}+\frac{ \aa-l(\aa)-\bb+l(\bb)}{(l(\aa)-l(\bb))(\aa-\bb)}  } f(\bb)\,d\bb\\&
=\frac1{\pi i}\int\paren{\frac{ l_\bb(\bb)-1}{l(\aa)-l(\bb)}+  \frac{ \aa-l(\aa)-\bb+l(\bb)}{(l(\aa)-l(\bb))(\aa-\bb)} } (f(\bb)-f(\aa))\,d\bb,
\end{aligned}
\end{equation}
here in the second step we inserted $-f(\aa)$ because $\mathbb H1=0$. Apply Cauchy-Schwarz inequality and Hardy's inequality  \eqref{eq:77} on the second equality in \eqref{2306} we obtain \eqref{q2} and \eqref{q4}. Using \eqref{3.16} and \eqref{3.17} on the first equality in \eqref{2306} we get \eqref{q3}. We are left with \eqref{q1}.

Differentiate with respect to $\aa$ and integrate by parts gives
\begin{equation}\label{2307}
\partial_\aa \mathcal Q_lf(\aa)=\frac1{\pi i}\int\paren{\frac{l_\aa(\aa)}{l(\aa)-l(\bb)}- \frac1{\aa-\bb}} f_\bb(\bb) \,d\bb
\end{equation}
Let $p\in C^\infty_0(\mathbb R)$. We have, by using the fact $\mathbb H1=0$ to insert $-p(\bb)$, that
\begin{equation}\label{2308}
\begin{aligned}
\int p(\aa)\partial_\aa \mathcal Q_lf(\aa)\,d\aa&= \frac1{\pi i}\iint\paren{\frac{l_\aa(\aa)}{l(\aa)-l(\bb)}- \frac1{\aa-\bb}} f_\bb(\bb) (p(\aa)-p(\bb)) \,d\aa d\bb\\&=
 \frac1{\pi i}\iint\frac{p(\aa)-p(\bb) }{l(\aa)-l(\bb)} (l_\aa(\aa)-1)f_\bb(\bb) \,d\aa d\bb\\&+
  \frac1{\pi i}\iint\frac{ \aa-l(\aa)-\bb+l(\bb)}{(l(\aa)-l(\bb))(\aa-\bb)} f_\bb(\bb) (p(\aa)-p(\bb)) \,d\aa d\bb.
\end{aligned}
\end{equation}
Applying Cauchy-Schwarz inequality and Hardy's inequality \eqref{eq:77} to \eqref{2308}. We get, for some constant $c$ depending only on $\|l_\aa\|_{L^\infty}$ and $\|(l^{-1})_\aa\|_{L^\infty}$, 
\begin{equation}\label{2309}
\abs{\int p(\aa)\partial_\aa \mathcal Q_lf(\aa)\,d\aa}\le c\|p\|_{\dot H^{1/2}}\|l_\aa-1\|_{L^2}\|\partial_\aa f\|_{L^2}.
\end{equation}
This proves inequality \eqref{q1}.

\end{proof}

\begin{lemma}\label{hhalf2}
 Assume that $f, \ g,\ f_1,\ g_1 \in H^1(\mathbb R)$ are the boundary values of some holomorphic functions on $\mathscr P_-$. Then
\begin{equation}\label{halfholo}
\int \partial_\aa \mathbb P_A (\bar f g)(\aa)f_1(\aa) \bar g_1(\aa)\,d\aa=-\frac1{2\pi i}\iint \frac{(\bar f(\aa)-\bar f(\bb))( f_1(\aa)- f_1(\bb))}{(\aa-\bb)^2}g(\bb)\bar g_1(\aa)\,d\aa d\bb.
\end{equation}

\end{lemma}

\begin{proof}
Let $f, \ g,\ f_1,\ g_1\in H^1(\mathbb R)$, and are the boundary values of some holomorphic functions in $\mathscr P_-$. We have
\begin{equation}\label{2310}
2\mathbb P_A (\bar f g)=(I-\mathbb H)(\bar f g)= [\bar f,\mathbb H]g
\end{equation}
and
\begin{equation}\label{2311}
2\partial_\aa\mathbb P_A (\bar f g)= \partial_\aa \bar f\, \mathbb H g-\frac1{\pi i}\int\frac{\bar f(\aa)-\bar f(\bb)}{(\aa-\bb)^2} g(\bb)\,d\bb.
\end{equation}
Because $\bar g_1 \partial_\aa\mathbb P_A (\bar f f_1 g)\in L^1(\mathbb R)$ is the boundary value of an anti-holomorphic function in $\mathscr P_-$, by Cauchy integral theorem, 
\begin{equation}\label{2312}
0=2 \int \bar g_1 \partial_\aa\mathbb P_A (\bar f f_1 g)\,d\aa
=\int \partial_\aa \bar f\, f_1 g \bar g_1\,d\aa-\frac1{\pi i}\iint\frac{\bar f(\aa)-\bar f(\bb)}{(\aa-\bb)^2} f_1(\bb)g(\bb)\bar g_1(\aa)\,d\aa d\bb,
\end{equation}
here we applied formula \eqref{2311} to the pair of holomorphic functions $f$ and $f_1 g$, and used the fact that $\mathbb H(f_1 g)=f_1 g$.
Now we use \eqref{2311} to compute, because $\mathbb H g=g$,
\begin{equation}\label{2313}
2\int \partial_\aa \mathbb P_A (\bar f g) \, f_1 \bar g_1\,d\aa
=\int \partial_\aa \bar f\, g f_1  \bar g_1\,d\aa- \frac1{\pi i}\iint\frac{\bar f(\aa)-\bar f(\bb)}{(\aa-\bb)^2} g(\bb)f_1(\aa)\bar g_1(\aa)\,d\aa d\bb.
\end{equation}
Substituting \eqref{2312} in \eqref{2313},  we get \eqref{halfholo}.
\end{proof}

\begin{remark}\label{hhalf3}
By  Cauchy integral theorem, we know for $f, \ g,\ f_1,\ g_1 \in H^1(\mathbb R)$,
$$\int \partial_\aa \mathbb P_A (\bar f g)(\aa)f_1(\aa) \bar g_1(\aa)\,d\aa=\int \partial_\aa \mathbb P_A (\bar f g)\mathbb P_H(f_1 \bar g_1)\,d\aa=\int \partial_\aa \mathbb P_A (\bar f g)\bar {\mathbb P_A(\bar f_1 g_1)}\,d\aa .$$

\end{remark}
As a  corollary of Lemma~\ref{hhalf2} and Remark~\ref{hhalf3} we have
\begin{proposition}\label{hhalf4}
Assume  that $f, \ g \in H^1(\mathbb R)$. We have
\begin{align}
\nm{\bracket{f, \mathbb H} g}_{\dot H^{1/2}}&\lec \|f\|_{\dot H^{1/2}}(\|g\|_{L^\infty} +\|\mathbb H g\|_{L^\infty});\label{hhalf41}\\
\nm{ \bracket{f, \mathbb H} g }_{\dot H^{1/2}}&\lec \|\partial_\aa f\|_{L^2}\|g\|_{L^2};\label{hhalf42}\\
\nm{\bracket{f, \mathbb H} \partial_\aa g}_{\dot H^{1/2}}&\lec \|g\|_{\dot H^{1/2}}(\|\partial_\aa f\|_{L^\infty}+\|\partial_\aa \mathbb H f\|_{L^\infty}).\label{hhalf43}
\end{align}
\end{proposition}

\begin{proof}
By Proposition~\ref{prop:comm-hilbe} and the decompositions $f=\mathbb P_A f+\mathbb P_H f$, $g=\mathbb P_A g+\mathbb P_H g$, 
\begin{equation}\label{2318}
\bracket{f,\mathbb H}g=\bracket{\mathbb P_A f,\mathbb H}\mathbb P_H g+\bracket{\mathbb P_H f,\mathbb H}\mathbb P_A g.
\end{equation}
So without loss of generality, we assume $f$ is anti-holomorphic and $g$ is holomorphic, i.e. $f=-\mathbb H f$, $g=\mathbb Hg$.
\eqref{hhalf41} is straightforward from \eqref{halfholo}, Remark~\ref{hhalf3} and the definition \eqref{def-hhalf}; and 
\eqref{hhalf42} can be easily obtained by applying Cauchy-Schwarz inequality and Hardy's inequality \eqref{eq:77} to \eqref{halfholo}. We are left with \eqref{hhalf43}.

By integration by parts, we know
\begin{equation}\label{2314}
 [ f,\mathbb H] \partial_\aa g+[g,\mathbb H]\partial_\aa f=\frac1{\pi i}\int \frac{( f(\aa)-f(\bb))(g(\aa)-g(\bb))}{(\aa-\bb)^2}\,d\bb:={\bf r};
\end{equation}
and by \eqref{hhalf41},
$$\nm{[g,\mathbb H]\partial_\aa f}_{\dot H^{1/2}}\lec \|g\|_{\dot H^{1/2}}\|\partial_\aa f\|_{L^\infty}.$$ For the  term $\bf r$ in the right hand side of \eqref{2314}, we have
\begin{equation}\label{2315}
\partial_\aa {\bf r}=\frac{-2}{\pi i}\int \frac{( f(\aa)- f(\bb))(g(\aa)-g(\bb))}{(\aa-\bb)^3}\,d\bb+ f_\aa \mathbb H g_\aa+g_\aa \mathbb H  f_\aa;
\end{equation}
and using $f=-\mathbb Hf$, $g=\mathbb H g$, we find
$$ f_\aa \mathbb H g_\aa+g_\aa \mathbb H  f_\aa= f_\aa   g_\aa-g_\aa   f_\aa=0.$$
Let $p\in C_0^\infty(\mathbb R)$. We have, using the symmetry of the integrand, 
\begin{equation}\label{2316}
\int p\partial_\aa {\bf r}\,d\aa=\frac{-1}{\pi i}\iint \frac{( f(\aa)- f(\bb))(g(\aa)-g(\bb))(p(\aa)-p(\bb))}{(\aa-\bb)^3}\,d\aa d\bb;
\end{equation}
applying Cauchy-Schwarz inequality and the definition \eqref{def-hhalf}, we get
\begin{equation}\label{2317}
\abs{\int p\partial_\aa {\bf r}\,d\aa}\lec \|\partial_\aa f\|_{L^\infty} \|g\|_{\dot H^{1/2}}\|p\|_{\dot H^{1/2}},
\end{equation}
so $ \|{\bf r}\|_{\dot H^{1/2}}\lec \|\partial_\aa f\|_{L^\infty} \|g\|_{\dot H^{1/2}}$. This finishes the proof for \eqref{hhalf43}.

\end{proof}

\begin{proposition}\label{dl21}
Assume  $f,\ g, \ f_1, \ g_1\in H^1(\mathbb R)$, and $l:\mathbb R\to \mathbb R$ is a diffeomorphism, with $l_\aa-1\in L^2$. Then
\begin{equation}\label{dl21-inq}
\begin{aligned}
&\nm{\bracket{f,\mathbb H}\partial_\aa g-U_l\bracket{f_1,\mathbb H}\partial_\aa g_1}_{L^2}\lec \nm{f-f_1\circ l}_{\dot H^{1/2}}\|\partial_\aa g\|_{L^2}\\&\qquad\qquad+ \|\partial_\aa f_1\|_{L^2} \|l_\aa\|_{L^\infty}^{\frac12} \nm{g-g_1\circ l}_{\dot H^{1/2}}+\|\partial_\aa f_1\|_{L^2}  \|\partial_\aa g_1\|_{L^2}\nm{l_\aa-1}_{L^2}.
\end{aligned}
\end{equation} 
\end{proposition}

\begin{proof}
We know
\begin{equation}\label{2320}
\begin{aligned}
\bracket{f,\mathbb H}\partial_\aa g&-U_l\bracket{f_1,\mathbb H}\partial_\aa g_1=\bracket{f,\mathbb H}\partial_\aa g-\bracket{f_1\circ l,U_l\mathbb HU_{l^{-1}}(l_\aa)^{-1}}\partial_\aa (g_1\circ l)\\&=
\bracket{f,\mathbb H}\partial_\aa g-\bracket{f_1\circ l,\mathbb H}\partial_\aa (g_1\circ l)+\bracket{f_1\circ l,\mathbb H-U_l\mathbb HU_{l^{-1}}(l_\aa)^{-1}}\partial_\aa (g_1\circ l);
\end{aligned}
\end{equation}
applying Proposition~\ref{prop:half-dir} to the term
\begin{equation}\label{2321}
\bracket{f,\mathbb H}\partial_\aa g-\bracket{f_1\circ l,\mathbb H}\partial_\aa (g_1\circ l)=\bracket{f-f_1\circ l,\mathbb H}\partial_\aa g+\bracket{f_1\circ l,\mathbb H}\partial_\aa (g-g_1\circ l)
\end{equation}
gives
\begin{equation}\label{2322}
\nm{\bracket{f,\mathbb H}\partial_\aa g-\bracket{f_1\circ l,\mathbb H}\partial_\aa (g_1\circ l)}_{L^2}\lec \nm{f-f_1\circ l}_{\dot H^{1/2}}\|\partial_\aa g\|_{L^2}+ \|\partial_\aa (f_1\circ l)\|_{L^2}  \nm{g-g_1\circ l}_{\dot H^{1/2}}.
\end{equation}
Now by \eqref{2305},
\begin{equation}\label{2323}
\begin{aligned}
&\bracket{f_1\circ l,\mathbb H-U_l\mathbb HU_{l^{-1}}(l_\aa)^{-1}}\partial_\aa (g_1\circ l)\\&\qquad\qquad=\frac1{\pi i} \int\frac{(f_1\circ l(\aa)-f_1\circ l(\bb))(l(\aa)-\aa-l(\bb)+\bb)}{(l(\aa)-l(\bb))(\aa-\bb)}\, \partial_\bb (g_1\circ l)(\bb) \,d\bb;
\end{aligned}
\end{equation}
applying Cauchy-Schwarz inequality and Hardy's inequality \eqref{eq:77} we get
\begin{equation}\label{2324}
\nm{\bracket{f_1\circ l,\mathbb H-U_l\mathbb HU_{l^{-1}}(l_\aa)^{-1}}\partial_\aa (g_1\circ l)}_{L^2}\lec 
\|\partial_\aa f_1\|_{L^2}\|l_\aa-1\|_{L^2}\|\partial_\aa g_1\|_{L^2}.\end{equation}
This finishes the proof for \eqref{dl21-inq}.
\end{proof}

\begin{proposition}\label{dl22}
Assume that $f,\ g,\ f_1, \ g_1$ are smooth and decay  at infinity, and $l:\mathbb R\to \mathbb R$ is a diffeomorphism, with $l_\aa-1\in L^2$. Then there is a constant $c(\|l_\aa\|_{L^\infty}, \|(l^{-1})_\aa\|_{L^\infty})$, depending on $\|l_\aa\|_{L^\infty}, \|(l^{-1})_\aa\|_{L^\infty}$, such that
\begin{equation}\label{dl221}
\begin{aligned}
&\nm{\bracket{f,\mathbb H}\partial_\aa g-U_l\bracket{f_1,\mathbb H}\partial_\aa g_1}_{L^2}\le c(\|l_\aa\|_{L^\infty}, \|(l^{-1})_\aa\|_{L^\infty})\nm{\partial_\aa f-\partial_\aa(f_1\circ l)}_{L^2}\|g\|_{L^\infty}\\&\qquad\qquad+ c(\|l_\aa\|_{L^\infty}, \|(l^{-1})_\aa\|_{L^\infty})(\|\partial_\aa f_1\|_{L^\infty}  \nm{g-g_1\circ l}_{L^2}+\|\partial_\aa f_1\|_{L^\infty} \| g_1\|_{L^\infty}\nm{l_\aa-1}_{L^2}).
\end{aligned}
\end{equation}
\begin{equation}\label{dl222}
\begin{aligned}
&\nm{\bracket{f,\mathbb H}\partial_\aa g-U_l\bracket{f_1,\mathbb H}\partial_\aa g_1}_{L^2}\lec \nm{\partial_\aa f-\partial_\aa(f_1\circ l)}_{L^2}\|g\|_{\dot H^{1/2}}\\&\qquad\qquad+ c(\|l_\aa\|_{L^\infty}, \|(l^{-1})_\aa\|_{L^\infty})(\|\partial_\aa f_1\|_{L^\infty}  \nm{g-g_1\circ l}_{L^2}+\|\partial_\aa f_1\|_{L^\infty} \| g_1\|_{\dot H^{1/2}}\nm{l_\aa-1}_{L^2}).
\end{aligned}
\end{equation}
\end{proposition}
\begin{proof}
We use the same computation as in the proof for Proposition~\ref{dl21}, and apply Proposition~\ref{B2} to the terms in \eqref{2321} and \eqref{2323} to get \eqref{dl221}. To obtain \eqref{dl222} we apply \eqref{eq:b11} and \eqref{3.20} to \eqref{2321}; and for the term in \eqref{2323},  we first integrate by parts, then apply Cauchy-Schwarz inequality and Hardy's inequality \eqref{eq:77}. 

\end{proof}

\begin{proposition}\label{dhhalf1}
Assume that $f,\ g,\ f_1, \ g_1$ are smooth and decay  at infinity, and $l:\mathbb R\to \mathbb R$ is a diffeomorphism, with $l_\aa-1\in L^2$. Then there is a constant $c:=c(\|l_\aa\|_{L^\infty}, \|(l^{-1})_\aa\|_{L^\infty})$, depending on $\|l_\aa\|_{L^\infty}, \|(l^{-1})_\aa\|_{L^\infty}$, such that
\begin{equation}\label{dhhalf1-inq}
\begin{aligned}
&\nm{\bracket{f,\mathbb H}\partial_\aa g-U_l\bracket{f_1,\mathbb H}\partial_\aa g_1}_{\dot H^{1/2}}\lec \nm{\partial_\aa f-\partial_\aa(f_1\circ l)}_{L^2}\|\partial_\aa g_1\|_{L^2}\|l_\aa\|_{L^\infty}^{\frac12}\\&\qquad+ (\|\partial_\aa f\|_{L^\infty}+\|\partial_\aa \mathbb H f\|_{L^\infty})  \nm{g-g_1\circ l}_{\dot H^{1/2}}+c\|\partial_\aa f_1\|_{L^\infty} \| \partial_\aa g_1\|_{L^2}\nm{l_\aa-1}_{L^2}.
\end{aligned}
\end{equation}

\end{proposition}

\begin{proof}
We begin with \eqref{2320} and write the first two terms on the right hand side as 
\begin{equation}\label{2325}
\bracket{f,\mathbb H}\partial_\aa g-\bracket{f_1\circ l,\mathbb H}\partial_\aa (g_1\circ l)=\bracket{f-f_1\circ l,\mathbb H}\partial_\aa (g_1\circ l)+\bracket{f,\mathbb H}\partial_\aa (g-g_1\circ l);
\end{equation}
applying \eqref{hhalf42} and \eqref{hhalf43} to \eqref{2325} we get
\begin{equation}\label{2326}
\begin{aligned}
\nm{\bracket{f,\mathbb H}\partial_\aa g-\bracket{f_1\circ l,\mathbb H}\partial_\aa (g_1\circ l)}_{\dot H^{1/2}}&\lec \nm{\partial_\aa(f-f_1\circ l)}_{L^2}\|\partial_\aa (g_1\circ l)\|_{L^2}\\&+ (\|\partial_\aa f\|_{L^\infty} + \|\partial_\aa \mathbb Hf\|_{L^\infty})\nm{g-g_1\circ l}_{\dot H^{1/2}}.
\end{aligned}
\end{equation}
Consider the last term on the right hand side of \eqref{2320}. For any $p\in C_0^\infty(\mathbb R)$, 
\begin{equation}\label{2327}
\begin{aligned}
\int\partial_\aa p \bracket{f_1\circ l,\mathbb H-U_l\mathbb HU_{l^{-1}}(l_\aa)^{-1}}&\partial_\aa (g_1\circ l)\,d\aa\\&=\int\partial_\aa (g_1\circ l)\bracket{f_1\circ l,\mathbb H-U_l\mathbb HU_{l^{-1}}(l_\aa)^{-1}}\partial_\aa p \,d\aa;
\end{aligned}
\end{equation}
the same argument as in the proof of \eqref{dl222}, that is, integrating by parts, then applying Cauchy-Schwarz inequality and Hardy's inequality \eqref{eq:77}  gives
$$\|\bracket{f_1\circ l,\mathbb H-U_l\mathbb HU_{l^{-1}}(l_\aa)^{-1}}\partial_\aa p\|_{L^2}\le c\,\|\partial_\aa f_1\|_{L^\infty} \|l_\aa-1\|_{L^2}\|p\|_{\dot H^{1/2}}, $$
where $c:= c(\|l_\aa\|_{L^\infty}, \|(l^{-1})_\aa\|_{L^\infty})$ is a constant depending on $\|l_\aa\|_{L^\infty}$ and $ \|(l^{-1})_\aa\|_{L^\infty}$; so 
$$\abs{\int\partial_\aa p \bracket{f_1\circ l,\mathbb H-U_l\mathbb HU_{l^{-1}}(l_\aa)^{-1}}\partial_\aa (g_1\circ l)\,d\aa}\le c\|\partial_\aa (g_1\circ l)\|_{L^2}\|\partial_\aa f_1\|_{L^\infty} \|l_\aa-1\|_{L^2}\|p\|_{\dot H^{1/2}}.$$
This finishes the proof for \eqref{dhhalf1-inq}.

\end{proof}

\begin{proposition}\label{dhhalf2}
Assume that $f,\ g,\ f_1, \ g_1$ are smooth and decay at infinity, and $l:\mathbb R\to \mathbb R$ is a diffeomorphism, with $l_\aa-1\in L^2$. Then there is a constant $c:=c(\|l_\aa\|_{L^\infty}, \|(l^{-1})_\aa\|_{L^\infty})$, depending on $\|l_\aa\|_{L^\infty}, \|(l^{-1})_\aa\|_{L^\infty}$, such that
\begin{equation}\label{dhhalf2-inq}
\begin{aligned}
&\nm{\bracket{f,\mathbb H} g-U_l\bracket{f_1,\mathbb H} g_1}_{\dot H^{1/2}}\lec \nm{ f-f_1\circ l}_{\dot H^{1/2}}(\| g\|_{L^\infty}+\|\mathbb H g\|_{L^\infty}) \\&\qquad+ \|\partial_\aa f_1\|_{L^2}\|l_\aa\|_{L^\infty}^{\frac12} \nm{g-g_1\circ l}_{L^2}+c\|\partial_\aa f_1\|_{L^2} \| g_1\|_{L^\infty}\nm{l_\aa-1}_{L^2}.
\end{aligned}
\end{equation}

\end{proposition}
\begin{proof}
Similar to the proof of Proposition~\ref{dl21}, we have  
\begin{equation}\label{2328}
\bracket{f,\mathbb H} g-U_l\bracket{f_1,\mathbb H} g_1=
\bracket{f,\mathbb H} g-\bracket{f_1\circ l,\mathbb H}(g_1\circ l) +\bracket{f_1\circ l,\mathbb H-U_l\mathbb HU_{l^{-1}}} (g_1\circ l);
\end{equation}
writing
\begin{equation}\label{2329}
\bracket{f,\mathbb H} g-\bracket{f_1\circ l,\mathbb H} (g_1\circ l)=\bracket{f-f_1\circ l,\mathbb H} g+\bracket{f_1\circ l,\mathbb H}(g-g_1\circ l)
\end{equation}
and applying \eqref{hhalf41} and \eqref{hhalf42} gives,
\begin{equation}\label{2330}
\begin{aligned}
\nm{\bracket{f,\mathbb H} g-\bracket{f_1\circ l,\mathbb H} (g_1\circ l)}_{\dot H^{1/2}}&\lec 
\nm{ f-f_1\circ l}_{\dot H^{1/2}}(\| g\|_{L^\infty}+\|\mathbb H g\|_{L^\infty})\\&+ \|\partial_\aa (f_1\circ l)\|_{L^2} \nm{g-g_1\circ l}_{L^2}.
\end{aligned}
\end{equation}
Consider the second term on the right hand side of \eqref{2328}.  We write
\begin{equation}\label{2331}
\bracket{f_1\circ l,\mathbb H-U_l\mathbb HU_{l^{-1}}} (g_1\circ l)=\bracket{f_1\circ l,\mathbb H-U_l\mathbb HU_{l^{-1}} \frac1{l_\aa}} (g_1\circ l)+\bracket{f_1\circ l, U_l\mathbb HU_{l^{-1}}}(\frac1{l_\aa}-1) (g_1\circ l).
\end{equation}
Now
\begin{equation}\label{2332}
\bracket{f_1\circ l, U_l\mathbb HU_{l^{-1}}}((l_\aa)^{-1}-1) (g_1\circ l)=U_l [f_1,\mathbb H]\paren{((l^{-1})_\aa-1)g_1}.
\end{equation}
Changing variables, and then using \eqref{hhalf42} yields
\begin{equation}\label{2333}
\nm{ U_l [f_1,\mathbb H]\paren{((l^{-1})_\aa-1)g_1}  }_{\dot H^{1/2}}\le 
c\|\partial_\aa f_1\|_{L^2} \| g_1\|_{L^\infty}\nm{l_\aa-1}_{L^2}
\end{equation}
for some constant $c$ depending on $\|l_\aa\|_{L^\infty}, \|(l^{-1})_\aa\|_{L^\infty}$. 

 For the first term on the right hand side of \eqref{2331} we use the duality argument in \eqref{2327}. Let $p\in C^\infty_0(\mathbb R)$,
\begin{equation}\label{2334}
\int \partial_\aa p \bracket{f_1\circ l,\mathbb H-U_l\mathbb HU_{l^{-1}} (l_\aa)^{-1}} (g_1\circ l)\,d\aa=\int  g_1\circ l\bracket{f_1\circ l,\mathbb H-U_l\mathbb HU_{l^{-1}} (l_\aa)^{-1}} \partial_\aa p \,d\aa,
\end{equation}
and
\begin{equation}\label{2335}
\begin{aligned}
&\bracket{f_1\circ l,\mathbb H-U_l\mathbb HU_{l^{-1}}(l_\aa)^{-1}}\partial_\aa p\\&\qquad\qquad=\frac1{\pi i} \int\frac{(f_1\circ l(\aa)-f_1\circ l(\bb))(l(\aa)-\aa-l(\bb)+\bb)}{(l(\aa)-l(\bb))(\aa-\bb)}\, \partial_\bb p(\bb) \,d\bb.
\end{aligned}
\end{equation}
Integrating by parts, then apply Cauchy-Schwarz inequality and Hardy's inequalities \eqref{eq:77} and \eqref{eq:771} gives
\begin{equation}\label{2336}
\nm{\bracket{f_1\circ l,\mathbb H-U_l\mathbb HU_{l^{-1}}(l_\aa)^{-1}}\partial_\aa p}_{L^1}\le c \|\partial_\aa f_1\|_{L^2} \nm{l_\aa-1}_{L^2}\|p\|_{\dot H^{1/2}},
\end{equation}
for some constant $c$ depending on $\|l_\aa\|_{L^\infty}, \|(l^{-1})_\aa\|_{L^\infty}$, so
\begin{equation}\label{2337}
\abs{\int \partial_\aa p \bracket{f_1\circ l,\mathbb H-U_l\mathbb HU_{l^{-1}} (l_\aa)^{-1}} (g_1\circ l)\,d\aa}\le c \|g_1\|_{L^\infty}\|\partial_\aa f_1\|_{L^2} \nm{l_\aa-1}_{L^2}\|p\|_{\dot H^{1/2}}.
\end{equation}
This finishes the proof for \eqref{dhhalf2-inq}.
\end{proof}

We define $$[f, m; \partial_\aa g]_n:=\frac1{\pi i}\int\frac{(f(\aa)-f(\bb))(m(\aa)-m(\bb))^n}{(\aa-\bb)^{n+1}}\partial_\bb g(\bb)\,d\bb.$$ So $[f,m;\partial_\aa g]=[f, m; \partial_\aa g]_1$, and $[f,\mathbb H]\partial_\aa g=[f, m; \partial_\aa g]_0$.
\begin{proposition}\label{d32}
Assume that $f,\ m,\ g,\ f_1,\ m_1,\ g_1$ are smooth and $f,\ g,\ f_1,\ g_1$ decay at infinity, and $l:\mathbb R\to\mathbb R$ is a diffeomorphism, with $l_\aa-1\in L^2$. Then there is a constant $c$, depending on $\|l_\aa\|_{L^\infty}, \|(l^{-1})_\aa\|_{L^\infty}$, such that
\begin{equation}\label{d32inq}
\begin{aligned}
&\nm{[f, m; \partial_\aa g]_n-U_l[f_1, m_1; \partial_\aa g_1]_n}_{L^2}\le c
\nm{f-f_1\circ l}_{\dot H^{1/2}}\nm{\partial_\aa m}^n_{L^\infty}\|\partial_\aa g\|_{L^2}\\&+c \|\partial_\aa f_1\|_{L^2} (\nm{\partial_\aa m}_{L^\infty}+\nm{\partial_\aa m_1}_{L^\infty})^{n-1}\nm{\partial_\aa(m-m_1\circ l)}_{L^2}\nm{\partial_\aa g}_{L^2}\\&+c\|\partial_\aa f_1\|_{L^2} \nm{\partial_\aa m_1}^n_{L^\infty} \nm{g-g_1\circ l}_{\dot H^{1/2}}+c\|\partial_\aa f_1\|_{L^2}\nm{\partial_\aa m_1}^n_{L^\infty}   \|\partial_\aa g_1\|_{L^2}\nm{l_\aa-1}_{L^2}.
\end{aligned}
\end{equation}
\end{proposition}
Proposition~\ref{d32} can be proved  similarly  as for Proposition~\ref{dl21}, we omit the details. 

\begin{proposition}\label{d33}
Assume that $f,\ m,\ g,\ f_1,\ m_1,\ g_1$ are smooth and decay at infinity, and $l:\mathbb R\to\mathbb R$ is a diffeomorphism, with $l_\aa-1\in L^2$. Then there is a constant $c$, depending on $\|l_\aa\|_{L^\infty}, \|(l^{-1})_\aa\|_{L^\infty}$, such that
\begin{equation}\label{d33inq}
\begin{aligned}
&\nm{[f, g;  m]-U_l[f_1, g_1; m_1]}_{L^2}\le c
\nm{f-f_1\circ l}_{\dot H^{1/2}}\|\partial_\aa g\|_{L^2}\nm{ m}_{L^\infty}\\&+c\|\partial_\aa f_1\|_{L^2}  \nm{g-g_1\circ l}_{\dot H^{1/2}}\nm{ m}_{L^\infty}+
c\|\partial_\aa f_1\|_{L^2}\nm{\partial_\aa g_1}_{L^2}\nm{m-m_1\circ l}_{L^2}\\&
+c\|\partial_\aa f_1\|_{L^2}\nm{ m_1}_{L^\infty}   \|\partial_\aa g_1\|_{L^2}\nm{l_\aa-1}_{L^2}.
\end{aligned}
\end{equation}

\end{proposition}
Proposition~\ref{d33} straightforwardly follows from Cauchy-Schwarz inequality, Hardy's inequality and the definition of $\dot H^{1/2}$ norm.

\begin{proposition}\label{half-product}
Assume $f\in \dot H^{1/2}(\mathbb R)\cap L^\infty(\mathbb R)$, $g\in \dot H^{1/2}(\mathbb R)$, and $g$ can be decomposed by
\begin{equation}\label{decomp}
g=g_1+ pq
\end{equation}
with $g_1\in L^\infty(\mathbb R)$, $q\in L^2(\mathbb R)$, and $\partial_\aa p\in L^2(\mathbb R)$, satisfying $\partial_\aa(pf)\in L^2(\mathbb R)$. Then $fg\in \dot H^{1/2}(\mathbb R)$, and
\begin{equation}\label{product-inq}
\nm{fg}_{\dot H^{1/2}}\lec \nm{f}_{L^\infty}\nm{g}_{\dot H^{1/2}}+ \nm{g_1}_{L^\infty}\nm{f}_{\dot H^{1/2}}+\nm{q}_{L^2}\nm{ \partial_\aa (pf)   }_{L^2}+\nm{q}_{L^2}\nm{ \partial_\aa p   }_{L^2}\nm{f}_{L^\infty}.
\end{equation}
\end{proposition}
\begin{proof}
The proof is straightforward by definition. We have
\begin{equation}
\begin{aligned}
&\nm{fg}_{\dot H^{1/2}}^2\lec \iint\frac{|f(\bb)|^2|g(\aa)-g(\bb)|^2}{(\aa-\bb)^2}\,d\aa d\bb+ \iint\frac{|g(\aa)|^2|f(\aa)-f(\bb)|^2}{(\aa-\bb)^2}\,d\aa d\bb\\&\lec
\nm{f}^2_{L^\infty}\nm{g}^2_{\dot H^{1/2}}+ \nm{g_1}^2_{L^\infty}\nm{f}^2_{\dot H^{1/2}}+\iint\frac{|q(\aa)|^2|p(\aa)f(\aa)-p(\bb)f(\bb)|^2}{(\aa-\bb)^2}\,d\aa d\bb\\&+\iint\frac{|q(\aa)|^2|p(\aa)-p(\bb)|^2|f(\bb)|^2}{(\aa-\bb)^2}\,d\aa d\bb\\&\lec \nm{f}^2_{L^\infty}\nm{g}^2_{\dot H^{1/2}}+ \nm{g_1}^2_{L^\infty}\nm{f}^2_{\dot H^{1/2}}+\nm{q}^2_{L^2}\nm{ \partial_\aa (pf)   }^2_{L^2}+\nm{q}^2_{L^2}\nm{ \partial_\aa p   }^2_{L^2}\nm{f}^2_{L^\infty},
\end{aligned}
\end{equation}
where in the last step we used Fubini's Theorem and Hardy's inequality \eqref{eq:77}.

\end{proof}
\subsection{The proof of Theorem~\ref{unique}}\label{proof3.5} In addition to what have already been given, we use the following convention in this section, \S\ref{proof3.5}: we write $A\lec B$ if there is a constant $c$, depending only on $\sup_{[0, T]}\frak E(t)$ and $\sup_{[0, T]}\tEf (t)$, such that $A\le cB$. We assume the reader is familiar with the quantities that are controlled by the functional $\frak E(t)$, see \S\ref{proof} and Appendix~\ref{quantities}. We don't always give precise references on these estimates. 

Let $Z=Z(\aa,t)$, $\Zf=\Zf(\aa,t)$, $t\in [0, T]$  be solutions of  the system \eqref{interface-r}-\eqref{interface-holo}-\eqref{b}-\eqref{a1}, satisfying the assumptions of Theorem~\ref{unique}. 
Recall we defined in \eqref{def-l}
\begin{equation}\label{2338}
l=l(\aa,t)=\th\circ h^{-1}(\aa, t)=\th(h^{-1}(\aa,t),t).
\end{equation}
We will apply Lemma~\ref{dlemma1} to $\Theta=\bar Z_t,\ \frac1{Z_{,\aa}}-1,\bar Z_{tt}$ and Lemma~\ref{basic-e2} to $l_\aa-1$ to construct an energy functional $\mathcal F(t)$, and show that the time derivative $\mathcal F'(t)$ can be controlled by $\mathcal F(t)$ and the initial data. 
We begin with computing the evolutionary equations for these quantities. We have
\begin{equation}\label{2339}
\partial_t ( l_\aa\circ h)=\partial_t\paren{\frac{\th_\a}{h_\a}}=\frac{\th_\a}{h_\a}\paren{\frac{\th_{t\a}}{\th_\a}-\frac{h_{t\a}}{h_\a}}=(l_\aa\circ h)(\tb_\aa\circ \th-b_\aa\circ h);
\end{equation}
precomposing with $h^{-1}$ yields 
\begin{equation}\label{2340}
(\partial_t +b\partial_\a) l_\aa=l_\aa(\tb_\aa\circ l-b_\aa).
\end{equation}
The equation for $\bar Z_t$ is given by \eqref{quasi-r}-\eqref{aux}.
To find the equation for $\bar Z_{tt}$ we take a  derivative to $t$ to \eqref{quasi-l}:
\begin{equation}\label{2341}
\begin{aligned}
(\partial_t^2+i\frak a\partial_\a)\bar z_{tt}&=-i\frak a_t \bar z_{t\a}+\partial_t\paren{\frac{\af_t}{\af}}(\bar z_{tt}-i)+\frac{\af_t}{\af}\bar z_{ttt}\\&=\partial_t\paren{\frac{\af_t}{\af}}(\bar z_{tt}-i)+\frac{\af_t}{\af}(\bar z_{ttt}-i\frak a \bar z_{t\a})\\&=(\bar z_{tt}-i)\paren{\partial_t\paren{\frac{\af_t}{\af}}+\paren{\frac{\af_t}{\af}}^2+2\paren{\frac{\af_t}{\af}}\bar{D_\a z_t}},
\end{aligned}
\end{equation}
here we used equation \eqref{quasi-l} and substituted by \eqref{interface-l}: $-i\frak a \bar z_\a=\bar z_{tt}-i$ in the last step. Precomposing with $h^{-1}$, and then substituting $\bar Z_{tt}-i$ by \eqref{aa1}, yields, for $\mathcal P=(\partial_t+b\partial_\aa)^2+i\mathcal A\partial_\aa$, 
\begin{equation}\label{eqztt}
\mathcal P\bar Z_{tt}= -i\,\frac{A_1}{Z_{,\aa}} \paren{(\partial_t+b\partial_\aa)\paren{\frac{\af_t}{\af}\circ h^{-1}}+\paren{\frac{\af_t}{\af}\circ h^{-1}}^2+2\paren{\frac{\af_t}{\af}\circ h^{-1}}\bar{D_\aa Z_t}}:=G_3.
\end{equation}
To find the equation for $\frac1{Z_{,\aa}}$ we begin with \eqref{eq:dza}. Precomposing with $h$, then differentiate with respect to $t$ gives
\begin{equation}\label{2342}
\partial_t^2\paren{\frac {h_\a}{z_\a}}=\frac {h_\a}{z_\a}\paren{(b_\aa\circ h-D_\a z_t)^2+\partial_t(b_\aa\circ h-2\Re D_\a z_t)+\partial_t \bar{D_\a z_t}}
\end{equation}
here  we subtracted $\partial_t\bar{D_\a z_t}$  and then added $\partial_t\bar{D_\a z_t}$  in the second factor on the right hand side to take advantage of the formula \eqref{ba}. We compute
\begin{equation}\label{2343}
\partial_t\bar {D_\a z_t}=\bar{D_\a z_{tt}}-(\bar{D_\a z_t})^2,
\end{equation}
replacing $\bar Z_{tt}$ by \eqref{aa1} yields
\begin{equation}\label{2344}
\bar{D_\aa Z_{tt}}=\frac1{\bar Z_{,\aa}}\partial_\aa\paren{-i\frac{A_1}{Z_{,\aa}}}=-i\frac{A_1}{\bar Z_{,\aa}}\partial_\aa\paren{\frac{1}{Z_{,\aa}}}-i\frac1{|Z_{,\aa}|^2}\partial_\aa A_1;
\end{equation}
precomposing equation \eqref{2342} with $h^{-1}$ and substitute in by \eqref{2343}-\eqref{2344}, we get
\begin{equation}\label{eqza}
\mathcal P\frac1{Z_{,\aa}}=\frac {1}{Z_\aa}\paren{(b_\aa-D_\aa Z_t)^2+(\partial_t+b\partial_\aa)(b_\aa-2\Re D_\aa Z_t)-\paren{\bar {D_\aa Z_t}}^2-i\frac{\partial_\aa A_1}{|Z_{,\aa}|^2}}:=G_2.
\end{equation}

We record here the equation for $\bar Z_t$, which is the first equation in \eqref{quasi-r}, in which we substituted in by \eqref{aa1},
\begin{equation}\label{eqzt}
\mathcal P\bar Z_t=-i\,\frac{\frak a_t}{\frak a}\circ h^{-1} \frac{A_1}{Z_{,\aa}}:=G_1.
\end{equation}
\subsubsection{The energy functional $\mathcal F(t)$} The energy functional $\mathcal F(t)$ for the differences of the solutions will consist of $\|l_\aa(t)-1\|_{L^2(\mathbb R)}^2$ and the functionals $\mathfrak F(t)$ when applied to $\Theta=\bar Z_t(t)$, $\frac1{Z_{,\aa}}-1$
and $\bar Z_{tt}$, taking $\frak c=-i, \ 0, \ 0$ respectively. Let
\begin{equation}\label{f0}
\frak F_0(t)=\|l_\aa(t)-1\|_{L^2(\mathbb R)}^2;
\end{equation}
\begin{equation}\label{f1}
\frak F_1(t):=\int \frac{\kappa}{A_1}\abs{ Z_{,\aa}\paren{\bar Z_{tt}-i}- \Zf_{,\aa}\circ l\paren{\bar \Zf_{tt}\circ l-i}}^2+i\partial_\aa(\bar Z_t-\bar {\Zf}_t\circ l)\bar{(\bar Z_t-\bar {\Zf}_t\circ l)}\,d\aa;
\end{equation}
\begin{equation}\label{f2}
\begin{aligned}
\frak F_2(t):&=\int \frac{\kappa}{A_1}\abs{ Z_{,\aa}(\partial_t+b\partial_\aa)\paren{\frac1{Z_{,\aa}}}- \Zf_{,\aa}\circ l(\partial_t+\tb\partial_\aa)\paren{\frac1{\Zf_{,\aa}}}\circ l}^2\,d\aa\\&\quad\qquad+i\int \partial_\aa\paren{\frac1{Z_{,\aa}}- \frac1{\Zf_{,\aa}}\circ l}\bar{\paren{\frac1{Z_{,\aa}}- \frac1{\Zf_{,\aa}}\circ l }}\,d\aa;
\end{aligned}
\end{equation}
and
\begin{equation}\label{f3}
\frak F_3(t):=\int \frac{\kappa}{A_1}\abs{ Z_{,\aa}\bar Z_{ttt}- \Zf_{,\aa}\circ l\bar \Zf_{ttt}\circ l}^2+i\partial_\aa(\bar Z_{tt}-\bar {\Zf}_{tt}\circ l)\bar{(\bar Z_{tt}-\bar {\Zf}_{tt}\circ l)}\,d\aa.
\end{equation}
Substituting the evolutionary equations \eqref{eq:dzt}, \eqref{eq:dza} and \eqref{eq:dztt} in the functionals $\frak F_i$ we get
\begin{equation}\label{f11}
\frak F_1(t)=\int \frac{\kappa}{A_1}\abs{ A_1- \tAone\circ l}^2+i\partial_\aa(\bar Z_t-\bar {\Zf}_t\circ l)\bar{(\bar Z_t-\bar {\Zf}_t\circ l)}\,d\aa;
\end{equation}
\begin{equation}\label{f22}
\frak F_2(t)=\int \frac{\kappa}{A_1}\abs{ (b_\aa-D_\aa Z_t)-(\tb_\aa-\tD_\aa \Zf_t)\circ l   }^2+ i\partial_\aa(\frac1{Z_{,\aa}}- \frac1{\Zf_{,\aa}}\circ l)\bar{(\frac1{Z_{,\aa}}- \frac1{\Zf_{,\aa}}\circ l )}\,d\aa;
\end{equation}
and
\begin{equation}\label{f33}
\begin{aligned}
\frak F_3(t)&=\int \frac{\kappa}{A_1}\abs{ A_1\paren{\frac{\frak a_t}{\frak a}\circ h^{-1}+\bar{D_\aa Z_t}}- \paren{\tAone\paren{\frac{\taf_t}{\taf}\circ \th^{-1}+\bar{\tD_\aa \Zf_t}}}\circ l }^2\,d\aa\\&\qquad\qquad+i\int\partial_\aa(\bar Z_{tt}-\bar {\Zf}_{tt}\circ l)\bar{(\bar Z_{tt}-\bar {\Zf}_{tt}\circ l)}\,d\aa.
\end{aligned}
\end{equation}

\begin{remark}\label{remark512}
Assume that the assumption of Theorem~\ref{unique} holds.  Because $h_t=b(h,t)$, and $h(\a, 0)=\a$, where $b$ is given by \eqref{b}, we have $h_{t\a}=h_\a b_\aa\circ h$, and 
\begin{equation}\label{2345}
h_\a(\cdot, t)=e^{\int_0^t b_\aa\circ h(\cdot, \tau)\, d\tau}.
\end{equation}
So there are constants $c_1>0$, $c_2>0$, depending only on $\sup_{[0, T]}\frak E(t)$, such that 
\begin{equation}\label{2346}
c_1\le h_\a(\a, t)\le c_2,\qquad\qquad \text{for all } \a\in \mathbb R, t\in [0, T]. 
\end{equation}
 Consequently, because $l_\aa=\frac{\th_a}{h_\a}\circ h^{-1}$,  there is a constant $0<c<\infty$, depending only on $\sup_{[0, T]}\frak E(t)$ and $\sup_{[0, T]}\tEf(t)$, such that 
\begin{equation}\label{2346-1}
c^{-1}\le l_\aa(\aa, t) \le c,  \qquad\qquad \text{for all } \a\in \mathbb R, t\in [0, T].
\end{equation}
It is easy to check that for each $t\in [0, T]$, $b_\aa(t)\in L^2(\mathbb R)$, so $h_\a(t)-1\in L^2(\mathbb R)$, and hence $l_\aa(t)-1\in L^2(\mathbb R)$.
It is clear that under the assumption of Theorem~\ref{unique}, the functionals $\frak F_i(t)$, $i=1,2,3$ are well-defined. 
 \end{remark}
 
Notice that the functionals $\frak F_i(t)$,  $i=1,2,3$ are not necessarily positive definite, see Lemma~\ref{hhalf1}. We prove the following
\begin{lemma}\label{dominate1}
There is a constant $M_0$, depending only on $\sup_{[0, T]}\frak E(t)$ and $\sup_{[0, T]}\tEf(t)$, such that for all $M\ge M_0$, and $t\in [0, T]$, 
\begin{align}
\|l_\aa(t)-1\|_{L^2}+\|(\bar Z_t-\bar \Zf_t\circ l)(t)\|_{\dot H^{1/2}}^2+\nm{\paren{\frac1{Z_{,\aa}}- \frac1{\Zf_{,\aa}}\circ l}(t)}_{\dot H^{1/2}}^2\le M\frak F_0(t)+\frak F_1(t)+\frak F_2(t);\label{dominate11}\\
\nm{(A_1-\tAone\circ l)(t)}_{L^2}^2+\nm{(D_\aa Z_t-\tD_\aa \Zf_t\circ l) (t)}_{L^2}^2+ \nm{b_\aa-\tb_\aa\circ l (t)}_{L^2}^2\lec M\frak F_0(t)+\frak F_1(t)+\frak F_2(t).\label{dominate12}
\end{align}
 
 \end{lemma}

 \begin{proof}
 By Lemma~\ref{hhalf1},  
 \begin{align}
 \int i\partial_\aa(\bar Z_t-\bar {\Zf}_t\circ l)\bar{(\bar Z_t-\bar {\Zf}_t\circ l)}\,d\aa=\|\mathbb P_H (\bar Z_t-\bar {\Zf}_t\circ l)\|_{\dot H^{1/2}}^2-\|\mathbb P_A (\bar Z_t-\bar {\Zf}_t\circ l)\|_{\dot H^{1/2}}^2;\\
\text{and }\qquad \|\bar Z_t-\bar {\Zf}_t\circ l\|_{\dot H^{1/2}}^2=\int i\partial_\aa(\bar Z_t-\bar {\Zf}_t\circ l)\bar{(\bar Z_t-\bar {\Zf}_t\circ l)}\,d\aa+2\|\mathbb P_A (\bar Z_t-\bar {\Zf}_t\circ l)\|_{\dot H^{1/2}}^2.
 \end{align}
 Because $\bar Z_t=\mathbb H\bar Z_t$ and $\bar \Zf_t=\mathbb H\bar \Zf_t$, 
 $$2\mathbb P_A (\bar Z_t-\bar {\Zf}_t\circ l)=-2\mathbb P_A (\bar {\Zf}_t\circ l)=-\mathcal Q_l (\bar {\Zf}_t\circ l)$$
 and by \eqref{q1},
 $$\|\mathcal Q_l (\bar {\Zf}_t\circ l)\|_{\dot H^{1/2}}\le C(\|l_\aa\|_{L^\infty}, \|(l^{-1})_\aa\|_{L^\infty})\|\partial_\aa \bar {\Zf}_t\|_{L^2}\|l_\aa-1\|_{L^2}\lec \|l_\aa-1\|_{L^2}. $$
 So there is a constant $M_0$, depending only on $\sup_{[0, T]}\frak E(t)$ and $\sup_{[0, T]}\tEf(t)$, such that for all $t\in [0, T]$ and $M\ge M_0$, 
 \begin{equation}\label{2350}
 \|(\bar Z_t-\bar {\Zf}_t\circ l)(t)\|_{\dot H^{1/2}}^2\le \int i\partial_\aa(\bar Z_t-\bar {\Zf}_t\circ l)\bar{(\bar Z_t-\bar {\Zf}_t\circ l)}\,d\aa+M\|l_\aa-1\|_{L^2}^2\le \frak F_1(t)+M\frak F_0(t)
 \end{equation}
 A similar argument holds for $\nm{\paren{\frac1{Z_{,\aa}}- \frac1{\Zf_{,\aa}}\circ l}(t)}_{\dot H^{1/2}}^2$. This proves  \eqref{dominate11}.
 
 Now by $\frac{\kappa}{A_1}= \sqrt{\frac{l_\aa}{A_1\tAone\circ l}}$, Remark~\ref{remark512} and the estimate \eqref{2000}, there is a constant $0<c<\infty$, depending on $\sup_{[0, T]}\frak E(t)$ and $\sup_{[0, T]}\tEf(t)$, such that 
 \begin{equation}\label{2350-1}
 \frac1c\le \frac{\kappa}{A_1} \le c,
 \end{equation}
  so 
 \begin{equation}\label{2351} 
  \nm{(A_1-\tAone\circ l)(t)}_{L^2}^2+\nm{(b_\aa-D_\aa Z_t-(\tb_\aa-\tD_\aa \Zf_t\circ l)) (t)}_{L^2}^2\lec  M\frak F_0(t)+\frak F_1(t)+\frak F_2(t),
 \end{equation}
for large enough $M$, depending only on  $\sup_{[0, T]}\frak E(t)$ and $\sup_{[0, T]}\tEf(t)$.  Using \eqref{ba} we have, from Proposition~\ref{dl21}, 
 \begin{equation}\label{2352}
 \nm{b_\aa-2\Re D_\aa Z_t-(\tb_\aa-2\Re \tD_\aa \Zf_t\circ l)}_{L^2}\lec \|\bar Z_t-\bar \Zf_t\circ l\|_{\dot H^{1/2}}+\nm{\frac1{Z_{,\aa}}- \frac1{\Zf_{,\aa}}\circ l}_{\dot H^{1/2}}+ \|l_\aa-1\|_{L^2};
 \end{equation}
 combining with \eqref{2351} yields
$$ \nm{(D_\aa Z_t-\tD_\aa \Zf_t\circ l) (t)}_{L^2}^2+ \nm{(b_\aa-\tb_\aa\circ l )(t)}_{L^2}^2\lec M\frak F_0(t)+\frak F_1(t)+\frak F_2(t),$$
 for  large enough $M$, depending only on  $\sup_{[0, T]}\frak E(t)$ and $\sup_{[0, T]}\tEf(t)$. This proves \eqref{dominate12}.

 \end{proof}
 
  \begin{lemma}\label{dominate2}
Let  $M_0$ be the constant in Lemma~\ref{dominate1}.  Then for all $M\ge M_0$, and $t\in [0, T]$, 
\begin{equation}\label{dominate21}
\|\mathbb P_A(\bar Z_{tt}-\bar \Zf_{tt}\circ l)(t)\|_{\dot H^{1/2}}^2\lec M\frak F_0(t)+\frak F_1(t)+\frak F_2(t)
\end{equation}
 
 \end{lemma}
\begin{proof}
We have
\begin{equation}\label{2353}
2\mathbb P_A(\bar Z_{tt}-\bar \Zf_{tt}\circ l)=2\mathbb P_A(\bar Z_{tt})-2U_l \mathbb P_A (\bar \Zf_{tt})-\mathcal Q_l (\bar \Zf_{tt}\circ l);
\end{equation}
and by \eqref{q1}, 
$$\nm{\mathcal Q_l (\bar \Zf_{tt}\circ l)}_{\dot H^{1/2}}\lec \|l_\aa-1\|_{L^2}.$$
Consider the first two terms on the right hand side of \eqref{2353}. We use \eqref{eq:c21} and the fact that $\bar Z_t=\mathbb H\bar Z_t$ to rewrite
\begin{equation}\label{2354}
2\mathbb P_A(\bar Z_{tt})=[\partial_t+b\partial_\aa, \mathbb H]\bar Z_{t}=[b, \mathbb H]\partial_\aa \bar Z_t.
\end{equation}
We would like to use \eqref{dhhalf1-inq} to estimate $\nm{2\mathbb P_A(\bar Z_{tt})-2U_l \mathbb P_A (\bar \Zf_{tt})}_{\dot H^{1/2}}$, 
observe that we have controlled all the quantities  
on the right hand side of \eqref{dhhalf1-inq},  
except for $\|\mathbb H b_\aa\|_{L^\infty}$. 

By \eqref{ba},
\begin{equation}\label{2355}
\begin{aligned}
b_\aa-2\Re D_\aa Z_t=\Re \paren{\bracket{ \frac1{Z_{,\aa}}, \mathbb H}  Z_{t,\alpha'}+ \bracket{Z_t, \mathbb H}\partial_\aa \frac1{Z_{,\aa}}  }, 
\end{aligned}
\end{equation}
and by the fact $Z_{t,\aa}=-\mathbb H Z_{t,\aa}$, 
\begin{equation}\label{2356}
\bracket{ \frac1{Z_{,\aa}}, \mathbb H}  Z_{t,\alpha'}=-(I+\mathbb H)D_\aa Z_t,
\end{equation}
so
\begin{equation}\label{2357}
\mathbb H \bracket{ \frac1{Z_{,\aa}}, \mathbb H}  Z_{t,\alpha'}=\bracket{ \frac1{Z_{,\aa}}, \mathbb H}  Z_{t,\alpha'}.
\end{equation}
This gives,  by \eqref{eq:b13},
\begin{equation}\label{2358}
\nm{\mathbb H \bracket{ \frac1{Z_{,\aa}}, \mathbb H}  Z_{t,\alpha'}}_{L^\infty}\lec \nm{\partial_\aa \frac1{Z_{,\aa}}}_{L^2}\nm{ Z_{t,\alpha'}}_{L^2}\lec C(\frak E(t));
\end{equation}
 similarly
 $\nm{\mathbb H \bracket{Z_t, \mathbb H}\partial_\aa \frac1{Z_{,\aa}} }_{L^\infty}\lec C(\frak E(t))$, therefore $\nm{\mathbb H(b_\aa-2\Re D_\aa Z_t)}_{L^\infty}\lec C(\frak E(t))$. Observe that the argument from \eqref{2356}-\eqref{2358} also shows that $$\nm{(I+\mathbb H)D_\aa Z_t}_{L^\infty}\lec C(\frak E(t)),$$ therefore 
 \begin{equation}\label{2358-1}
 \nm{\mathbb H D_\aa Z_t}_{L^\infty}\le \nm{(I+\mathbb H)D_\aa Z_t}_{L^\infty}+\nm{D_\aa Z_t}_{L^\infty}
 \lec C(\frak E(t)),
 \end{equation}
 and hence $\nm{\mathbb Hb_\aa}_{L^\infty}\lec C(\frak E(t))$. Notice that
 $$\nm{ \partial_\aa(b-\tb\circ l)}_{L^2}\lec \nm{ b_\aa-\tb_\aa\circ l}_{L^2}+\|l_\aa-1\|_{L^2}.$$
 Applying \eqref{dhhalf1-inq} to \eqref{2354} we get
 $$\nm{2\mathbb P_A(\bar Z_{tt})-2U_l \mathbb P_A (\bar \Zf_{tt})}_{\dot H^{1/2}}\lec \nm{ b_\aa-\tb_\aa\circ l}_{L^2}+\|l_\aa-1\|_{L^2}+\nm{\bar Z_t-\bar \Zf_t\circ l}_{\dot H^{1/2}}.$$
 A further application of Lemma~\ref{dominate1} yields Lemma~\ref{dominate2}.

\end{proof}

 As a consequence of Lemmas~\ref{dominate1}, and \ref{dominate2} we have
\begin{proposition}\label{denergy}
There is a constant $M_0>0$, depending only on $\sup_{[0, T]}\frak E(t)$ and $\sup_{[0, T]}\tEf(t)$, such that for all $M\ge M_0$, and $t\in [0, T]$, 
\begin{equation}\label{2359}
\begin{aligned}
&\|(\bar Z_t-\bar \Zf_t\circ l)(t)\|_{\dot H^{1/2}}^2+\nm{\paren{\frac1{Z_{,\aa}}- \frac1{\Zf_{,\aa}}\circ l}(t)}_{\dot H^{1/2}}^2+\|(\bar Z_{tt}-\bar \Zf_{tt}\circ l)(t)\|_{\dot H^{1/2}}^2\\ &+
\nm{(A_1-\tAone\circ l)(t)}_{L^2}^2+\nm{(D_\aa Z_t-\tD_\aa \Zf_t\circ l) (t)}_{L^2}^2+ \nm{(b_\aa-\tb_\aa\circ l) (t)}_{L^2}^2+\nm{l_\aa(t)-1}_{L^2}^2\\&
+\nm{\paren{\frac{\frak a_t}{\frak a}\circ h^{-1}-\frac{\taf_t}{\taf}\circ h^{-1}}(t)}_{L^2}^2\lec M\frak F_0(t)+\frak F_1(t)+\frak F_2(t)+M^{-1} \frak F_3(t).
\end{aligned}
 \end{equation}
 \end{proposition}

 \begin{remark}\label{remark517}
 It is clear that the reverse of  inequality \eqref{2359} also holds. Observe that by \eqref{a1} and Proposition~\ref{dl21}, we have for any $t\in [0, T]$, 
 \begin{equation}
 \nm{(A_1-\tAone\circ l)(t)}_{L^2}\lec \|(\bar Z_t-\bar \Zf_t\circ l)(t)\|_{\dot H^{1/2}}+ \|l_\aa(t)-1\|_{L^2}.
 \end{equation}
 and by \eqref{at}-\eqref{ba}-\eqref{dta1} and Propositions~\ref{dl21}, ~\ref{d32}, 
 \begin{equation}
 \begin{aligned}
& \nm{  \paren{\frac{\frak a_t}{\frak a}\circ h^{-1}-\frac{\taf_t}{\taf}\circ h^{-1}}(t) }_{L^2}=\nm{  \paren{\frac{\frak a_t}{\frak a}\circ h^{-1}-\frac{\taf_t}{\taf}\circ \th^{-1}\circ l}(t) }_{L^2} \\& \lec \|(\bar Z_t-\bar \Zf_t\circ l)(t)\|_{\dot H^{1/2}}+ \nm{\paren{\frac1{Z_{,\aa}}- \frac1{\Zf_{,\aa}}\circ l}(t)}_{\dot H^{1/2}}+\|(\bar Z_{tt}-\bar \Zf_{tt}\circ l)(t)\|_{\dot H^{1/2}}\\&\qquad+\nm{(b_\aa-\tb_\aa\circ l) (t)}_{L^2}+\|l_\aa(t)-1\|_{L^2}.
 \end{aligned}
 \end{equation}
 This, together with \eqref{2352} shows that for all $t\in [0, T]$,
 \begin{equation}\label{2360}
 \begin{aligned}
 M\frak F_0(t)+\frak F_1(t)+&\frak F_2(t)+M^{-1} \frak F_3(t)\lec \|(\bar Z_t-\bar \Zf_t\circ l)(t)\|_{\dot H^{1/2}}^2+\nm{\paren{\frac1{Z_{,\aa}}- \frac1{\Zf_{,\aa}}\circ l}(t)}_{\dot H^{1/2}}^2\\&+\|(\bar Z_{tt}-\bar \Zf_{tt}\circ l)(t)\|_{\dot H^{1/2}}^2+\nm{(D_\aa Z_t-\tD_\aa \Zf_t\circ l) (t)}_{L^2}^2
+\nm{l_\aa(t)-1}_{L^2}^2.
\end{aligned}
 \end{equation}
 
 \end{remark}

 Now fix a constant $M$,  with $M\ge M_0>0$, so that \eqref{2359} holds. We define
 \begin{equation}\label{dfunctional}
 \mathcal F(t):= M\frak F_0(t)+\frak F_1(t)+\frak F_2(t)+M^{-1}\frak F_3(t).
 \end{equation}
 We have
 \begin{proposition}\label{denergy-est}
 Assume that $Z=Z(\aa, t)$, $\Zf=\Zf(\aa,t)$ are solutions of the system \eqref{interface-r}-\eqref{interface-holo}-\eqref{a1}-\eqref{b}, satisfying the assumption of Theorem~\ref{unique}. Then there is a constant $C$,  depending only on $T$, $\sup_{[0, T]}\frak E(t)$ and $\sup_{[0, T]}\tEf(t)$, such that for $t\in [0, T]$,
 \begin{equation}\label{denergy-inq}
 \frac d{dt}\mathcal F(t)\le C\paren{ \mathcal F(t)+ \int_0^t \mathcal F(\tau)\,d\tau + \nm{\paren{\frac1{Z_{,\aa}}-\frac1{\Zf_{,\aa}}}(0)}_{L^\infty}\mathcal F(t)^{\frac12}}.
 \end{equation}

 \end{proposition}
 
Assuming Proposition~\ref{denergy-est} holds, we have,  by \eqref{denergy-inq},
 \begin{equation}\label{2460}
 \frac d{dt}\paren{\mathcal F(t)+\int_0^t \mathcal F(\tau)\,d\tau}\le C\paren{ \mathcal F(t)+ \int_0^t \mathcal F(\tau)\,d\tau + \nm{\paren{\frac1{Z_{,\aa}}-\frac1{\Zf_{,\aa}}}(0)}_{L^\infty}^2};
   \end{equation}
and by Gronwall's inequality,
\begin{equation}\label{2461}
\mathcal F(t)+\int_0^t \mathcal F(\tau)\,d\tau\le C\paren{\mathcal F(0)+ \nm{\paren{\frac1{Z_{,\aa}}-\frac1{\Zf_{,\aa}}}(0)}_{L^\infty}^2},\qquad \text{for }t\in [0, T],
\end{equation}
for some constant $C$ depending on $T$, $\sup_{[0, T]}\frak E(t)$ and $\sup_{[0, T]}\tEf(t)$. 
This together with \eqref{2359} and \eqref{2360} gives \eqref{stability}.

We now give the proof for Proposition~\ref{denergy-est}.
 
 \begin{proof}
 To prove Proposition~\ref{denergy-est} we apply Lemma~\ref{basic-e2} to $\Theta=l_\aa-1$, and Lemma~\ref{dlemma1} to $\Theta=\bar Z_t,\ \frac1{Z_{,\aa}}-1,\ \bar Z_{tt}$. We have, by Lemma~\ref{basic-e2} and \eqref{2340},
 \begin{equation}\label{2361}
 \frak F'_0(t)\le 2\nm{l_\aa (b_\aa-\tb_\aa\circ l)}_{L^2}\frak F_0(t)^{1/2}+\|b_\aa\|_{L^\infty}\frak F_0(t)\lec \mathcal F(t),
 \end{equation} 
 here we used \eqref{2346-1}, \eqref{2359}, and \S\ref{basic-quantities}, \eqref{2020}.
 
 Now we apply Lemma~\ref{dlemma1} to $\Theta=\bar Z_t,\ \frac1{Z_{,\aa}}-1,\ \bar Z_{tt}$ to get the estimates for $\frak F_1'(t)$, $\frak F_2'(t)$ and $\frak F_3'(t)$. Checking through the right hand sides of the inequalities \eqref{dlemma1-inq} for $\Theta=\bar Z_t,\ \frac1{Z_{,\aa}}-1,\ \bar Z_{tt}$, we find that we have controlled almost all of the quantities, respectively by $\mathcal F(t)$ or $\frak E(t)$, $\tEf(t)$, except for the following:
 
 \begin{itemize}
 \item
1.  $\nm{\frac{(\partial_t+b\partial_\aa)\kappa }{\kappa}}_{L^2}$;

 \item 
2.  $\nm{1-\kappa}_{L^2}$;

\item
 3. $2\Re i\int \bar{\paren{\frac1{Z_{,\aa}}-\frac1{\Zf_{,\aa}}\circ l }} \paren{\bar{\Zf_{,\aa}\circ l((\partial_t+\tb\partial_\aa)\tTheta\circ l+\frak c)}\Theta_\aa-\bar{Z_{,\aa}((\partial_t+b\partial_\aa)\Theta+\frak c)}(\tTheta\circ l)_\aa}\,d\aa$, \newline
 for $\Theta=\bar Z_t,\ \frac1{Z_{,\aa}}-1,\ \bar Z_{tt}$; with $\frak c=-i$ for $\Theta=\bar Z_t$, and $\frak c=0$ for $\Theta=\frac1{Z_{,\aa}}-1$ and  $\bar Z_{tt}$;

\item
4. $\nm{ Z_{,\aa}G_i- \Zf_{,\aa}\circ l\tilde G_i\circ l}_{L^2}$, for $i=1,2,3$.

 \end{itemize}
 
We begin with items 1. and 2.  By definition $\kappa=\sqrt{\frac{A_1}{\tAone\circ l} l_\aa}$, so
 \begin{equation}\label{2362}
2 \frac{(\partial_t+b\partial_\aa)\kappa }{\kappa}= \frac{(\partial_t+b\partial_\aa)A_1}{A_1}-\frac{(\partial_t+\tb \partial_\aa)\tAone }{\tAone}\circ l+(\tb_\aa\circ l-b_\aa);
  \end{equation}
 and by \eqref{at},  
  \begin{equation}\label{2363}
\frac{(\partial_t +b\partial_\aa) A_1}{A_1}=\dfrac{\frak a_t}{\frak a}\circ h^{-1}-(b_\aa -2\Re D_\aa Z_t);
\end{equation}
  therefore
  \begin{equation}\label{2364}
  \nm{\frac{(\partial_t+b\partial_\aa)\kappa }{\kappa}}_{L^2}\lec  \nm{\frac{\frak a_t}{\frak a}\circ h^{-1}-\frac{\taf_t}{\taf}\circ h^{-1}}_{L^2}+ \nm{b_\aa-\tb_\aa\circ l}_{L^2}+\nm{D_\aa Z_t-\tD_\aa\Zf_t \circ l}_{L^2}\lec \mathcal F(t)^{\frac12}.
  \end{equation}
 And it is clear that by the definition of $\kappa$,
 \begin{equation}\label{2365}
 \nm{1-\kappa}_{L^2}\lec \nm{A_1-\tAone\circ l}_{L^2}+\nm{l_\aa-1}_{L^2}\lec \mathcal F(t)^{\frac12}.
 \end{equation}
  
  What remains to be controlled are the quantities in items 3. and 4. We first consider item 4.  We have, by \eqref{eqzt}, ${Z_{,\aa}}G_1=-i\,\frac{\frak a_t}{\frak a}\circ h^{-1} A_1$, so
  \begin{equation}\label{2366}
  \nm{ Z_{,\aa}G_1- \Zf_{,\aa}\circ l\tilde G_1\circ l}_{L^2}\lec \nm{\frac{\frak a_t}{\frak a}\circ h^{-1}-\frac{\taf_t}{\taf}\circ h^{-1}}_{L^2}+\nm{A_1-\tAone\circ l}_{L^2}\lec \mathcal F(t)^{\frac12};
  \end{equation}
  and by \eqref{eqza}, ${Z_\aa}G_2=(b_\aa-D_\aa Z_t)^2+(\partial_t+b\partial_\aa)(b_\aa-2\Re D_\aa Z_t)-\paren{\bar {D_\aa Z_t}}^2-i\frac{\partial_\aa A_1}{|Z_{,\aa}|^2}$, so 
   \begin{equation}\label{2367}
   \begin{aligned}
 & \nm{ Z_{,\aa}G_2- \Zf_{,\aa}\circ l\tilde G_2\circ l}_{L^2}\lec \nm{b_\aa-\tb_\aa\circ l}_{L^2}+\nm{D_\aa Z_t-\tD_\aa\Zf_t \circ l}_{L^2}+\\&\nm{(\partial_t+b\partial_\aa)(b_\aa-2\Re D_\aa Z_t)-(\partial_t+\tb\partial_\aa)(\tb_\aa- 2\Re \tD_\aa\Zf_t )\circ l}_{L^2}+\nm{\frac{\partial_\aa A_1}{|Z_{,\aa}|^2}-\frac{\partial_\aa \tAone}{|\Zf_{,\aa}|^2}\circ l}_{L^2}; 
  \end{aligned}
  \end{equation}
observe that  we have controlled all but the last two quantities on the right hand side of \eqref{2367} by $\mathcal F(t)^{1/2}$. By \eqref{eqztt}, 
  $Z_{,\aa}G_3= -i\,A_1 \paren{(\partial_t+b\partial_\aa)\paren{\frac{\af_t}{\af}\circ h^{-1}}+\paren{\frac{\af_t}{\af}\circ h^{-1}}^2+2\paren{\frac{\af_t}{\af}\circ h^{-1}}\bar{D_\aa Z_t}}$, so 
   \begin{equation}\label{2368}
   \begin{aligned}
  &\nm{ Z_{,\aa}G_3- \Zf_{,\aa}\circ l\tilde G_3\circ l}_{L^2}\lec \nm{\frac{\frak a_t}{\frak a}\circ h^{-1}-\frac{\taf_t}{\taf}\circ h^{-1}}_{L^2}+\nm{D_\aa Z_t-\tD_\aa\Zf_t \circ l}_{L^2}\\&+\nm{A_1-\tAone\circ l}_{L^2}\paren{\nm{(\partial_t+b\partial_\aa)\paren{\frac{\frak a_t}{\frak a}\circ h^{-1} } }_{L^\infty}+1}
  \\&\qquad+
  \nm{(\partial_t+b\partial_\aa)\paren{\frac{\frak a_t}{\frak a}\circ h^{-1}}-(\partial_t+\tb\partial_\aa)\paren{\frac{\taf_t}{\taf} \circ \th^{-1} }\circ l}_{L^2}.
  \end{aligned}
  \end{equation}
  We have controlled all but the factor in the third quantity and the very last quantity on the right hand side of \eqref{2368}. 
  
 In the remaining part of the proof for Proposition~\ref{denergy-est}, we will 
  show the following inequalities 
  \begin{itemize}
  
    \item
 $\nm{\frac{\partial_\aa A_1}{|Z_{,\aa}|^2}-\frac{\partial_\aa \tAone}{|\Zf_{,\aa}|^2}\circ l}_{L^2}\lec \mathcal F(t)^{\frac12}$;
 
  \item
  $\nm{(\partial_t+b\partial_\aa)(b_\aa-2\Re D_\aa Z_t)-(\partial_t+\tb\partial_\aa)(\tb_\aa- 2\Re \tD_\aa\Zf_t )\circ l}_{L^2}\lec \mathcal F(t)^{\frac12}$;

\item
$\nm{(\partial_t+b\partial_\aa)\paren{\frac{\frak a_t}{\frak a}\circ h^{-1} } }_{L^\infty}\le C(\frak E(t));$

  \item
  $ \nm{(\partial_t+b\partial_\aa)\paren{\frac{\frak a_t}{\frak a}\circ h^{-1}}-(\partial_t+\tb\partial_\aa)\paren{\frac{\taf_t}{\taf} \circ \th^{-1} }\circ l}_{L^2}\lec \mathcal F(t)^{\frac12}$;
  
  \item
  and control the quantities in item 3. 
  \end{itemize}
  Our main strategy is the same as always, that is, to rewrite the quantities in forms to which the results in \S\ref{prepare} can be applied.

 \subsubsection{Some additional quantities controlled by $\frak E(t)$ and by $\mathcal F(t)$}\label{additional}
 
 We begin with deriving some additional estimates that will be used in the proof. First we record the conclusions from the computations of \eqref{2355}-\eqref{2358-1}, 
 \begin{equation}\label{2380}
\nm{\mathbb H D_\aa Z_t}_{L^\infty}\le C(\frak E(t)),\qquad \nm{\mathbb H b_\aa }_{L^\infty}\le C(\frak E(t)).
\end{equation}

 Because $\partial_\aa \frac1{Z_{,\aa}} =\mathbb H\paren{\partial_\aa \frac1{Z_{,\aa} }}$,
 \begin{equation}\label{2381}
 2\mathbb P_A\paren{Z_{t} \partial_\aa \frac1{Z_{,\aa} }}=\bracket{Z_{t},\mathbb H}\partial_\aa \frac1{Z_{,\aa} };
 \end{equation}
and we have, by \eqref{eq:b13} and \eqref{dl21-inq}, 
\begin{align}
\nm{\mathbb P_A\paren{Z_{t} \partial_\aa \frac1{Z_{,\aa} }}}_{L^\infty}
\lec \nm{Z_{t,\aa}}_{L^2}\nm{ \partial_\aa \frac1{Z_{,\aa} }}_{L^2}\le C(\frak E(t));\label{2382}\\
\nm{\mathbb P_A\paren{Z_{t} \partial_\aa \frac1{Z_{,\aa} }}-U_l \mathbb P_A\paren{\Zf_{t} \partial_\aa \frac1{\Zf_{,\aa} }}}_{L^2}\lec \mathcal F(t)^{1/2}.\label{2383}
\end{align}
Similarly we have
\begin{equation}\label{2384}
\nm{\mathbb P_A\paren{Z_{t} \partial_\aa \bar Z_t}-U_l \mathbb P_A\paren{\Zf_{t} \partial_\aa \bar{\Zf}_t}}_{L^2}\lec \mathcal F(t)^{1/2}.
\end{equation}

 By \eqref{a1}, $iA_1=i-\mathbb P_A(Z_t\bar Z_{t,\aa})+\mathbb P_H(\bar Z_t Z_{t,\aa})$,
 and by \eqref{aa1},
 \begin{equation}\label{2385}
 \bar Z_{tt}-i=-\frac{i}{Z_{,\aa}}+\frac{\mathbb P_A(Z_t\bar Z_{t,\aa})}{Z_{,\aa}}-\frac{\mathbb P_H(\bar Z_t Z_{t,\aa})}{Z_{,\aa}};
 \end{equation}
applying $\mathbb P_H$ to both sides of \eqref{2385} and rewriting the second term on the right hand side as a commutator gives
 \begin{equation}\label{2386}
 \frac{\mathbb P_H(\bar Z_t Z_{t,\aa})}{Z_{,\aa}}=-i\paren{\frac{1}{Z_{,\aa}}-1}-\frac12\bracket{\frac{1}{Z_{,\aa}},\mathbb H}\mathbb P_A(Z_t\bar Z_{t,\aa})-\mathbb P_H(\bar Z_{tt}).
 \end{equation}
 Now we apply \eqref{dhhalf2-inq}, \eqref{q1}, \eqref{2359} and \eqref{2384} to get
 \begin{align}
\nm{\bracket{\frac{1}{Z_{,\aa}},\mathbb H}\mathbb P_A(Z_t\bar Z_{t,\aa})-U_l \bracket{\frac{1}{\Zf_{,\aa}},\mathbb H}\mathbb P_A(\Zf _t\bar{ \Zf}_{t,\aa})     }_{\dot H^{1/2}}
\lec \mathcal F(t)^{1/2};\label{2387}\\
\nm{\frac{\mathbb P_H(\bar Z_t Z_{t,\aa})}{Z_{,\aa}}-U_l\frac{\mathbb P_H(\bar {\Zf}_t \Zf_{t,\aa})}{\Zf_{,\aa}}   }_{\dot H^{1/2}}\lec \mathcal F(t)^{1/2};\label{2388}
\end{align}
consequently by \eqref{2385} and \eqref{2388}, \eqref{2359},
 \begin{equation}\label{2389}
\nm{\frac{\mathbb P_A(\bar Z_t Z_{t,\aa})}{Z_{,\aa}}-U_l\frac{\mathbb P_A(\bar {\Zf}_t \Zf_{t,\aa})}{\Zf_{,\aa}}   }_{\dot H^{1/2}}\lec \mathcal F(t)^{1/2}.
\end{equation}
 Similarly we have
 \begin{align}
\nm{\bracket{\frac{1}{Z_{,\aa}},\mathbb H}\mathbb P_A\paren{Z_t \partial_\aa\frac1{Z_{,\aa}}}-U_l \bracket{\frac{1}{\Zf_{,\aa}},\mathbb H}\mathbb P_A\paren{\Zf _t  \partial_\aa\frac1{\Zf_{,\aa}}   }    }_{\dot H^{1/2}}
\lec \mathcal F(t)^{1/2};\label{2390}\\
\nm{\bracket{\frac{1}{Z_{,\aa}}, \mathbb H}\paren{ A_1\paren{\bar{D_\aa Z_t}+\frac{\frak a_t}{\frak a}\circ h^{-1}}  }-U_l\bracket{\frac{1}{\Zf_{,\aa}}, \mathbb H}\paren{ \tAone\paren{\bar{\tD_\aa \Zf_t}+\frac{\taf_t}{\taf }\circ \th^{-1}}  }}_{\dot H^{1/2}}\lec \mathcal F(t)^{1/2};\label{2390-1}
\end{align}
provided we can show that \begin{equation}\label{2390-2}
\nm{ \mathbb H\paren{  A_1\paren{\bar{D_\aa Z_t}+\frac{\frak a_t}{\frak a}\circ h^{-1}}} }_{L^\infty}\lec C(\frak E).\end{equation}

We now prove \eqref{2390-2}. It suffices to show $\nm{ \mathbb P_A\paren{ A_1\paren{\bar{D_\aa Z_t}+\frac{\frak a_t}{\frak a}\circ h^{-1}}} }_{L^\infty}\lec C(\frak E)$, since we have  $ \nm{  A_1\paren{\bar{D_\aa Z_t}+\frac{\frak a_t}{\frak a}\circ h^{-1}} }_{L^\infty}\lec C(\frak E)$.
We know
$$2\mathbb P_A (A_1 \bar{D_\aa Z_t})=\bracket{\frac{A_1}{\bar Z_{,\aa}},\mathbb H}\bar Z_{t,\aa}=-i[Z_{tt},\mathbb H]\bar Z_{t,\aa},$$
hence by Cauchy-Schwarz inequality and Hardy's inequality \eqref{eq:77},
\begin{equation}\label{2390-3}
\nm{2\mathbb P_A (A_1 \bar{D_\aa Z_t})}_{L^\infty}\lec \|Z_{tt,\aa}\|_{L^2}\|Z_{t,\aa}\|_{L^2}\lec C(\frak E).
\end{equation}
For the second term we use the formula (2.23) of \cite{wu6}, \footnote{This formula can be checked directly from \eqref{at}-\eqref{ba}-\eqref{dta1} via similar manipulations as in \eqref{2446-1}-\eqref{2448}.}
\begin{equation}\label{2390-4}
{A_1}\frac{\frak a_t}{\frak a}\circ h^{-1}=-\Im( 2[Z_t,\mathbb H]{\bar Z}_{tt,\alpha'}+2[Z_{tt},\mathbb H]\partial_{\alpha'} \bar Z_t-
[Z_t, Z_t; D_{\alpha'} \bar Z_t]),
\end{equation}
observe that the quantities $[Z_t,\mathbb H]{\bar Z}_{tt,\alpha'}$, $[Z_{tt},\mathbb H]\partial_{\alpha'} \bar Z_t$ are anti-holomorphic by \eqref{comm-hilbe}, and $[Z_t, Z_t; D_{\alpha'} \bar Z_t]$ is anti-holomorphic by integration by parts and \eqref{comm-hilbe}, so 
\begin{equation}\label{2390-5}
\mathbb P_A\paren{{A_1}\frac{\frak a_t}{\frak a}\circ h^{-1}}=i\paren{ [Z_t,\mathbb H]{\bar Z}_{tt,\alpha'}+[Z_{tt},\mathbb H]\partial_{\alpha'} \bar Z_t-\frac12
[Z_t, Z_t; D_{\alpha'} \bar Z_t]};
\end{equation}
therefore
\begin{equation}\label{2390-6}
\nm{\mathbb P_A\paren{{A_1}\frac{\frak a_t}{\frak a}\circ h^{-1}}}_{L^\infty}\lec C(\frak E)
\end{equation}
by Cauchy-Schwarz inequality and Hardy's inequality \eqref{eq:77}.  This proves \eqref{2390-2}. 

In what follows we  will need the bound for $\nm{Z_{ttt,\aa}}_{L^2}$. We begin with \eqref{eq:dztt} and calculate $\bar Z_{ttt,\aa}$. We have
\begin{equation}\label{2410}
\bar Z_{ttt,\aa}=\bar Z_{tt,\aa} (\bar{D_\aa Z_t}+\frac{\frak a_t}{\frak a}\circ h^{-1})-iA_1D_\aa(\bar{D_\aa Z_t}+\frac{\frak a_t}{\frak a}\circ h^{-1})
\end{equation}
where we substituted the factor $\bar Z_{tt}-i$ in the second term by $-\frac{iA_1}{Z_{,\aa}}$, see \eqref{eq:dzt}. We know from \S\ref{proof} that all the quantities in \eqref{2410} are controlled and we have
\begin{equation}\label{2411}
\nm{Z_{ttt,\aa}}_{L^2}\le C(\frak E(t)).
\end{equation}

 \subsubsection{Controlling the $\dot H^{1/2}$ norms of $Z_{,\aa}((\partial_t+b\partial_\aa)\Theta+\frak c)$ for $\Theta=\bar Z_t, \frac1{Z_{,\aa}}-1, \bar Z_{tt}$, with $\frak c=-i, 0, 0$ respectively}\label{hhalf-norm}
 We will use Proposition~\ref{half-product} to control the item 3 above. To do so we need to check that
 the assumptions of the proposition hold. One of them is $Z_{,\aa}((\partial_t+b\partial_\aa)\Theta+\frak c)\in \dot H^{1/2}\cap L^\infty$, for $\Theta=\bar Z_t, \frac1{Z_{,\aa}}-1, \bar Z_{tt}$;  $\frak c=-i, 0, 0$ respectively; with the norms bounded by $C(\frak E(t))$. By \eqref{eq:dzt}, \eqref{eq:dztt} and \eqref{eq:dza},
 \begin{equation}\label{2391}
 Z_{,\aa}(\bar Z_{tt}-i)=-iA_1,\quad Z_{,\aa}(\partial_t+b\partial_\aa)\frac1{Z_{,\aa}}=b_\aa-D_\aa Z_t,\quad Z_{,\aa} \bar Z_{ttt}=-iA_1\paren{\bar {D_\aa Z_t}+\frac{\frak a_t}{\frak a}\circ h^{-1}}.
 \end{equation}
In \S\ref{proof}, we have shown that these quantities are in $L^\infty$, with their $L^\infty$ norms controlled by $C(\frak E(t))$. So we only need to estimate their $\dot H^{1/2}$ norms. 
  
Applying Proposition~\ref{hhalf4}, \eqref{hhalf42} to \eqref{a1} and \eqref{ba},  we get $A_1, b_\aa-2\Re D_\aa Z_t\in \dot H^{1/2}$, with
  \begin{align}
 & \|A_1\|_{\dot H^{1/2}}\lec \nm{\partial_\aa Z_t}_{L^2}^2\lec C(\frak E(t));\label{2392}
   \\
&   \|b_\aa-2\Re D_\aa Z_t\|_{\dot H^{1/2}}\lec \nm{\partial_\aa Z_t}_{L^2}\nm{\partial_\aa \frac1{Z_{,\aa}}}_{L^2}\lec C(\frak E(t)).\label{2393}
  \end{align}
  
  We next compute $\|D_\aa Z_t(t)\|_{\dot H^{1/2}}$. By definition, 
  \begin{equation}\label{2394}
  \begin{aligned}
 \|D_\aa Z_t(t)\|_{\dot H^{1/2}}^2&= \int (i\partial_\aa \mathbb H D_\aa Z_t)\, \bar{D_\aa Z_t }\,d\aa\\&= \int i\partial_\aa \bracket{\mathbb H, \frac1{Z_{,\aa}}}  Z_{t,\aa} \, \bar{D_\aa Z_t }\,d\aa+ \int i\partial_\aa \paren{\frac1{Z_{,\aa}} \mathbb H  Z_{t,\aa}} \bar{D_\aa Z_t }\,d\aa\\&=
 \int i\bar{D_\aa}\bracket{\mathbb H, \frac1{Z_{,\aa}}}  Z_{t,\aa}\, \partial_\aa\bar{ Z_t }\,d\aa+ \int i   Z_{t,\aa} (D_\aa \bar{D_\aa Z_t })\,d\aa
 \end{aligned}
  \end{equation}
where in the last step we used integration by parts and the fact $\mathbb H Z_{t,\aa}=-Z_{t,\aa}$.
Recall in \eqref{2042}, we have shown $\nm{D_\aa\bracket{\mathbb H, \frac1{Z_{,\aa}}}  Z_{t,\aa} }_{L^2}\le C(\frak E(t))$. So by Cauchy-Schwarz inequality, we have
\begin{equation}\label{2395}
 \|D_\aa Z_t(t)\|_{\dot H^{1/2}}^2\lec \nm{D_\aa\bracket{\mathbb H, \frac1{Z_{,\aa}}}  Z_{t,\aa} }_{L^2}\nm{Z_{t,\aa} }_{L^2}+\nm{Z_{t,\aa} }_{L^2}\nm{D_\aa^2Z_{t} }_{L^2}\le C(\frak E(t)).
\end{equation}

Now we consider $\nm{\frac{\frak a_t}{\frak a}\circ h^{-1}}_{\dot H^{1/2}}$. By \eqref{at}-\eqref{ba}-\eqref{dta1}, we know Proposition~\ref{hhalf4}, \eqref{hhalf42} can be used to handle all terms, except for  $[Z_t, b; \bar Z_{t,\aa}]$. 

Let $p\in C_0^\infty(\mathbb R)$, we have, by duality,
\begin{equation}\label{2396}
\abs{\int \partial_\aa p [Z_t, b; \bar Z_{t,\aa}]\,d\aa} =\abs{\int [Z_t, b;  \partial_\aa p]\bar Z_{t,\aa}\,d\aa}\lec \|Z_{t,\aa}\|_{L^2}^2\|b_\aa\|_{L^\infty}\|p\|_{\dot H^{1/2}},
\end{equation}
where in the last step we used Cauchy-Schwarz inequality and \eqref{eq:b111}. Therefore
$\nm{[Z_t, b; \bar Z_{t,\aa}]}_{\dot H^{1/2}}\le C(\frak E(t))$. Applying Proposition~\ref{hhalf4}, \eqref{hhalf42} to the remaining terms and using \eqref{hhalf-1} yields 
\begin{equation}\label{2397}
\nm{ \frac{\frak a_t}{\frak a}\circ h^{-1} }_{\dot H^{1/2}}\le C(\frak E(t)).
\end{equation}

We can now conclude that for $\Theta=\bar Z_t, \ \frac1{Z_{,\aa}}-1,\ \bar Z_{tt}$, with $\frak c=i, 0, 0$ respectively, 
\begin{equation}\label{2398}
\nm{Z_{,\aa}((\partial_t+b\partial_\aa)\Theta+\frak c)}_{L^\infty}+\nm{Z_{,\aa}((\partial_t+b\partial_\aa)\Theta+\frak c)}_{\dot H^{1/2}}\le C(\frak E(t)).
\end{equation}

\subsubsection{Controlling $\int \bar{\paren{\frac1{Z_{,\aa}}-\frac1{\Zf_{,\aa}}\circ l }} \paren{\bar{\Zf_{,\aa}\circ l((\partial_t+\tb\partial_\aa)\tTheta\circ l+\frak c)}\Theta_\aa-\bar{Z_{,\aa}((\partial_t+b\partial_\aa)\Theta+\frak c)}(\tTheta\circ l)_\aa}\,d\aa$}
We begin with studying $\frac1{Z_{,\aa}}-\frac1{\Zf_{,\aa}}\circ l$. By \eqref{2100}, 
$\frac1{Z_{,\aa}}(h(\a,t),t)=\frac1{Z_{,\aa}}(\a,0)e^{\int_0^t (b_\aa\circ h(\a,\tau)-D_\a z_t(\a,\tau))\,d\tau}$, so
\begin{equation}\label{2399}
\begin{aligned}
\paren{\frac1{Z_{,\aa}}-\frac1{\Zf_{,\aa}}\circ l}\circ h&=\paren{\frac1{Z_{,\aa}}(0)-\frac1{\Zf_{,\aa}}(0)} e^{\int_0^t (\tb_\aa-\tD_\aa \Zf_t)\circ \th (\tau)\,d\tau}\\&+\frac1{Z_{,\aa}}\circ h\paren{1-  e^{\int_0^t ((\tb_\aa-\tD_\aa \Zf_t)\circ \th-(b_\aa-D_\aa Z_t)\circ h) (\tau)\,d\tau}}.
\end{aligned}
\end{equation}
We know for $t\in [0, T]$, 
\begin{equation}\label{2470}
\nm{\paren{\frac1{Z_{,\aa}}(0)-\frac1{\Zf_{,\aa}}(0)} e^{\int_0^t (\tb_\aa-\tD_\aa \Zf_t)\circ \th (\tau)\,d\tau}}_{L^\infty}\le C(\sup_{[0, T]}\tEf (t))\nm{\frac1{Z_{,\aa}}(0)-\frac1{\Zf_{,\aa}}(0)}_{L^\infty};
\end{equation}
and
\begin{equation}\label{2471}
\nm{  1-  e^{\int_0^t ((\tb_\aa-\tD_\aa \Zf_t)\circ \th-(b_\aa-D_\aa Z_t)\circ h) (\tau)\,d\tau}  }_{L^2}\lec \int_0^t  \mathcal F(\tau)^{1/2}\,d\tau.
\end{equation}

Now we rewrite 
\begin{equation}\label{2472}
\begin{aligned}
&\int \bar{\paren{\frac1{Z_{,\aa}}-\frac1{\Zf_{,\aa}}\circ l }} \paren{\bar{\Zf_{,\aa}\circ l((\partial_t+\tb\partial_\aa)\tTheta\circ l+\frak c)}\Theta_\aa-\bar{Z_{,\aa}((\partial_t+b\partial_\aa)\Theta+\frak c)}(\tTheta\circ l)_\aa}\,d\aa\\&
=\int \bar{\paren{\frac1{Z_{,\aa}}-\frac1{\Zf_{,\aa}}\circ l }}\Theta_\aa \paren{\bar{\Zf_{,\aa}\circ l((\partial_t+\tb\partial_\aa)\tTheta\circ l+\frak c)}-\bar{Z_{,\aa}((\partial_t+b\partial_\aa)\Theta+\frak c)}}\,d\aa\\&+
\int \bar{\paren{\frac1{Z_{,\aa}}-\frac1{\Zf_{,\aa}}\circ l }} \bar{Z_{,\aa}((\partial_t+b\partial_\aa)\Theta+\frak c)}\,(\Theta-\tTheta\circ l)_\aa\,d\aa=I+II.
\end{aligned}
\end{equation}
We apply Proposition~\ref{half-product} to $II$, with $g=\frac1{Z_{,\aa}}-\frac1{\Zf_{,\aa}}\circ l$, and $f=Z_{,\aa}((\partial_t+b\partial_\aa)\Theta+\frak c)$, where $\Theta=\bar Z_t, \ \frac1{Z_{,\aa}}-1,\ \bar Z_{tt}$, with $\frak c=i, 0, 0$ respectively. We know
$$\partial_\aa\paren{\frac1{Z_{,\aa}} f}= \bar Z_{tt,\aa},\quad \partial_\aa(\partial_t+b\partial_\aa)\frac1{Z_{,\aa}},\quad \bar Z_{ttt,\aa},\qquad \text{for  }\Theta=\bar Z_t, \ \frac1{Z_{,\aa}}-1,\ \bar Z_{tt},$$
so $\nm{\partial_\aa\paren{\frac1{Z_{,\aa}} f}}_{L^2}\le C(\frak E(t))$, by \S\ref{proof} and \eqref{2411}. Applying Proposition~\ref{half-product} to the $g$ and $f$ given  above yields
\begin{equation}\label{2473}
\begin{aligned}
&\nm{\paren{\frac1{Z_{,\aa}}-\frac1{\Zf_{,\aa}}\circ l}\, Z_{,\aa}((\partial_t+b\partial_\aa)\Theta+\frak c)}_{\dot H^{1/2}}\\&\qquad \lec \nm{\frac1{Z_{,\aa}}-\frac1{\Zf_{,\aa}}\circ l}_{\dot H^{1/2}}+ \nm{\frac1{Z_{,\aa}}(0)-\frac1{\Zf_{,\aa}}(0)}_{L^\infty}+\int_0^t \mathcal F(\tau)^{1/2}\,d\tau;
\end{aligned}
\end{equation}
consequently
\begin{equation}\label{2474}
\abs{II}\lec \mathcal F(t)+T\int_0^t\mathcal F(\tau)\,d\tau+\nm{\frac1{Z_{,\aa}}(0)-\frac1{\Zf_{,\aa}}(0)}_{L^\infty}\mathcal F(t)^{1/2}.
\end{equation}

We apply the decomposition \eqref{2399} and Cauchy-Schwarz inequality to $I$, notice that $\nm{\Theta_\aa}_{L^2}\le C(\frak E(t))$, and $\nm{D_\aa \Theta}_{L^\infty}\le C(\frak E(t))$, for $\Theta=\bar Z_t, \ \frac1{Z_{,\aa}}-1,\ \bar Z_{tt}$. We have
\begin{equation}\label{2475}
\abs{I}\lec \nm{\frac1{Z_{,\aa}}(0)-\frac1{\Zf_{,\aa}}(0)}_{L^\infty}\mathcal F(t)^{1/2}+\mathcal F(t)^{1/2}\int_0^t\mathcal F(\tau)^{1/2}\,d\tau.
\end{equation}
This shows that for $\Theta=\bar Z_t, \ \frac1{Z_{,\aa}}-1,\ \bar Z_{tt}$, with $\frak c=i, 0, 0$ respectively, 
\begin{equation}\label{2476}
\begin{aligned}
&\abs{\int \bar{\paren{\frac1{Z_{,\aa}}-\frac1{\Zf_{,\aa}}\circ l }} \paren{\bar{\Zf_{,\aa}\circ l((\partial_t+\tb\partial_\aa)\tTheta\circ l+\frak c)}\Theta_\aa-\bar{Z_{,\aa}((\partial_t+b\partial_\aa)\Theta+\frak c)}(\tTheta\circ l)_\aa}\,d\aa}\\&\qquad\qquad\lec  \mathcal F(t)+T\int_0^t\mathcal F(\tau)\,d\tau+\nm{\frac1{Z_{,\aa}}(0)-\frac1{\Zf_{,\aa}}(0)}_{L^\infty}\mathcal F(t)^{1/2}.
\end{aligned}
\end{equation}

\subsubsection{Controlling   $\nm{\frac{\partial_\aa A_1}{|Z_{,\aa}|^2}-\frac{\partial_\aa \tAone}{|\Zf_{,\aa}|^2}\circ l}_{L^2}$}\label{da1z2}

 We will take advantage of the fact that $\frac{\partial_\aa A_1}{|Z_{,\aa}|^2}$ is purely real to use $(I+\mathbb H)$ to convert it to some commutator forms to which the Propositions in \S\ref{prepare} can be applied. 

Observe that
\begin{equation}\label{2369}
i\,\frac{\partial_\aa A_1}{|Z_{,\aa}|^2}= \frac{1}{Z_{,\aa}}\partial_\aa\frac{ i A_1}{\bar Z_{,\aa}}-\frac{ i A_1}{Z_{,\aa}}\partial_\aa\frac{1}{\bar Z_{,\aa}}=\frac{1}{Z_{,\aa}}\partial_\aa Z_{tt}+(\bar Z_{tt}-i)\partial_\aa\frac{1}{\bar Z_{,\aa}};
\end{equation}
  we apply $(I+\mathbb H)$ to \eqref{2369}, and use the fact $\partial_\aa\frac{1}{\bar Z_{,\aa}}=-\mathbb H\paren{\partial_\aa\frac{1}{\bar Z_{,\aa}}}$  to write the second term in a commutator form.  We have 
  \begin{equation}\label{2370}
i\,(I+\mathbb H)\frac{\partial_\aa A_1}{|Z_{,\aa}|^2}=(I+\mathbb H)\paren{\frac{1}{Z_{,\aa}}\partial_\aa Z_{tt}}-\bracket{\bar Z_{tt}, \mathbb H}\partial_\aa\frac{1}{\bar Z_{,\aa}}.
\end{equation}
For the first term on the right hand side, we commute out $\frac1{Z_{,\aa}}$, then use the fact $Z_t=-\mathbb H Z_t$ to write $(I+\mathbb H)Z_{tt}$ as a commutator (see \eqref{2354}),
 \begin{equation}\label{2371}
(I+\mathbb H)\paren{\frac{1}{Z_{,\aa}}\partial_\aa Z_{tt}}=\bracket{\mathbb H, \frac{1}{Z_{,\aa}}}\partial_\aa Z_{tt}-\frac{1}{Z_{,\aa}}\partial_\aa [b,\mathbb H]Z_{t,\aa};
\end{equation}
we compute
\begin{equation}\label{2372}
\begin{aligned}
\frac{1}{Z_{,\aa}}\partial_\aa [b,\mathbb H]Z_{t,\aa}&=\frac{1}{Z_{,\aa}} b_\aa \mathbb H Z_{t,\aa}-\frac1{\pi i Z_{,\aa}} \int \frac{b(\aa,t)-b(\bb,t)}{(\aa-\bb)^2}Z_{t,\bb}\,d\bb\\&
=-b_\aa D_\aa Z_t-\bracket{\frac1{Z_{,\aa}}, b; Z_{t,\aa}}-\frac1{\pi i } \int \frac{b(\aa,t)-b(\bb,t)}{(\aa-\bb)^2}D_\bb Z_{t}\,d\bb\\&=
-b_\aa D_\aa Z_t-\bracket{\frac1{Z_{,\aa}}, b; Z_{t,\aa}}+[b,\mathbb H]\partial_\aa D_\aa Z_t-\mathbb H(b_\aa D_\aa Z_t),
\end{aligned}
\end{equation}
in the last step we performed integration by parts. 
We have converted the right hand side of \eqref{2370} in the desired forms. Applying \eqref{dl21-inq}, \eqref{dl221}, \eqref{d32inq}, \eqref{q3} and \eqref{2359}, then take the imaginary parts gives
\begin{equation}\label{2373}
\nm{\frac{\partial_\aa A_1}{|Z_{,\aa}|^2}-\frac{\partial_\aa \tAone}{|\Zf_{,\aa}|^2}\circ l}_{L^2}\lec \mathcal F(t)^{\frac12}.
\end{equation}

In what follows we will use the following identities in the calculations: for $f,\ g,\ p$, 
satisfying $g=\mathbb H g$ and $p=\mathbb H p$,
 \begin{align}
 [ f, \mathbb H] (gp)&=[ f g,\mathbb H]p=[\mathbb P_A( f g),\mathbb H] p;\label{H1}\\
 [ f, \mathbb H]\partial_\aa(gp)&=[ f\partial_\aa g, \mathbb H] p+[ f g,\mathbb H]\partial_\aa p=[\mathbb P_A( f\partial_\aa g), \mathbb H] p+[\mathbb P_A( f g),\mathbb H]\partial_\aa p
 \label{H2}
 \end{align}
\eqref{H1} is obtained by using the fact that the product of holomorphic functions is holomorphic, and \eqref{comm-hilbe};  \eqref{H2} is a consequence of \eqref{H1} and the product rules.

\subsubsection{Controlling   $\nm{(\partial_t+b\partial_\aa)(b_\aa-2\Re D_\aa Z_t)-(\partial_t+\tb\partial_\aa)(\tb_\aa- 2\Re \tD_\aa\Zf_t )\circ l}_{L^2}$  }\label{ddtba}

We begin with \eqref{dba-1}, 
\begin{equation}\label{2374}
\begin{aligned}
&(\partial_t+b\partial_\aa)\paren{b_\aa-2\Re D_\aa Z_t}=\Re \paren{\bracket{ (\partial_t+b\partial_\aa)\frac1{Z_{,\aa}}, \mathbb H}  Z_{t,\alpha'} + \bracket{Z_{t}, \mathbb H}\partial_\aa (\partial_t+b\partial_\aa)\frac1{Z_{,\aa}}}
\\&\qquad+\Re\paren{ \bracket{ \frac1{Z_{,\aa}}, \mathbb H}  Z_{tt,\alpha'}+ \bracket{Z_{tt}, \mathbb H}\partial_\aa \frac1{Z_{,\aa}}  -\bracket{  \frac1{Z_{,\aa}}, b;  Z_{t,\alpha'}   }    -\bracket{ Z_{t}, b; \partial_\aa \frac1{Z_{,\aa}}   } };
\end{aligned}
\end{equation}
observe that using Propositions~\ref{dl21}, ~\ref{d32} and \ref{denergy} we are able to get the desired estimates for the last four terms on the right hand side of \eqref{2374}. We need to rewrite the first two terms in order to apply the results in \S\ref{prepare}. 
First, by \eqref{eq:dza} we have 
\begin{equation}\label{2375}
(\partial_t+b\partial_{\alpha'})\paren{\frac1{Z_{,\aa}}}=\frac1{Z_{,\aa}} \paren{b_\aa-D_\aa Z_t};
\end{equation}
and by $\mathbb HZ_{t,\aa}=-Z_{t,\aa}$, 
\begin{equation}\label{2376}
\bracket{ (\partial_t+b\partial_\aa)\frac1{Z_{,\aa}}, \mathbb H}  Z_{t,\alpha'}=-(I+\mathbb H) \paren{(D_\aa Z_t )(b_\aa-D_\aa Z_t)};
\end{equation}
so we can conclude from \eqref{q3} and \eqref{2359}  that
\begin{equation}\label{2377}
\nm{\bracket{ (\partial_t+b\partial_\aa)\frac1{Z_{,\aa}}, \mathbb H}  Z_{t,\alpha'}-U_l \bracket{ (\partial_t+\tb\partial_\aa)\frac1{\Zf_{,\aa}}, \mathbb H}  \Zf_{t,\alpha'}}_{L^2}\lec \mathcal F(t)^{1/2}.
\end{equation}

For the second term on the right hand side of \eqref{2374}, we use \eqref{b} to further rewrite \eqref{2375},
 \begin{equation}\label{2378}
 \begin{aligned}
 &(\partial_t+b\partial_{\alpha'})\paren{\frac1{Z_{,\aa}}}=\frac1{Z_{,\aa}} \paren{\partial_\aa{\Re (I-\mathbb H)\frac {Z_t}{Z_{,\aa}} } -D_\aa Z_t}\\&=
 \frac1{Z_{,\aa}} \paren{\mathbb P_A\frac {Z_{t,\aa}}{Z_{,\aa}}+\mathbb P_H\frac {\bar Z_{t,\aa}}{\bar Z_{,\aa}} +\Re (I-\mathbb H)\paren{Z_t\partial_\aa\frac1{Z_{,\aa}}}  -D_\aa Z_t}\\&=
 \frac1{Z_{,\aa}} \paren{\mathbb P_H\paren{\bar{D_\aa Z_t}-D_\aa Z_t}+\mathbb P_A\paren{Z_{t} \partial_\aa \frac1{Z_{,\aa} }}+\bar{\mathbb P_A\paren{Z_{t} \partial_\aa \frac1{Z_{,\aa} }}}}.
 \end{aligned}
 \end{equation}
 We substitute  the right hand side of \eqref{2378} in the second term, $\bracket{Z_{t}, \mathbb H}\partial_\aa (\partial_t+b\partial_\aa)\frac1{Z_{,\aa}}$ of \eqref{2374}, term by term. For the first term we have, by \eqref{H2},
 \begin{equation}\label{2379}
 \begin{aligned}
& \bracket{Z_{t}, \mathbb H}\partial_\aa \paren{  \frac1{Z_{,\aa}}\mathbb P_H\paren{\Im\bar{ D_\aa Z_t}} }=\bracket{\mathbb P_A\paren{Z_{t} \partial_\aa \frac1{Z_{,\aa} }}, \mathbb H}\mathbb P_H\paren{\Im \bar{D_\aa Z_t}} \\&\qquad\qquad\qquad+\bracket{\mathbb P_A\paren{\frac{Z_{t} }{Z_{,\aa} }}, \mathbb H} \partial_\aa\mathbb P_H\paren{\Im\bar{D_\aa Z_t}}\\&
= (I-\mathbb H)\paren{ \mathbb P_A\paren{Z_{t} \partial_\aa \frac1{Z_{,\aa} }}\mathbb P_H\paren{\Im\bar{D_\aa Z_t}}} +\bracket{b, \mathbb H} \partial_\aa\mathbb P_H\paren{\Im\bar{D_\aa Z_t}};
 \end{aligned}
 \end{equation}
in the last step we used \eqref{bb} and \eqref{comm-hilbe}.    Therefore by \eqref{2382}-\eqref{2383}, \eqref{2359}, \eqref{q3} and \eqref{dl221}, 
 \begin{equation}\label{2400}
 \nm{\bracket{Z_{t}, \mathbb H}\partial_\aa \paren{  \frac1{Z_{,\aa}}\mathbb P_H\paren{\Im\bar{D_\aa Z_t}} }- U_l\bracket{\Zf_{t}, \mathbb H}\partial_\aa \paren{  \frac1{\Zf_{,\aa}}\mathbb P_H\paren{\Im \bar{\tD_\aa \Zf_t}} }}_{L^2}\lec \mathcal F(t)^{\frac12}.
 \end{equation}
We substitute in the second term and rewrite further by \eqref{comm-hilbe},
\begin{equation}\label{2401}
\begin{aligned}
\bracket{Z_{t}, \mathbb H}\partial_\aa \paren{ \frac1{Z_{,\aa}}\mathbb P_A\paren{Z_{t} \partial_\aa \frac1{Z_{,\aa} }}}&=\bracket{Z_{t}, \mathbb H}\partial_\aa \mathbb P_H\paren{ \frac1{Z_{,\aa}}\mathbb P_A\paren{Z_{t} \partial_\aa \frac1{Z_{,\aa} }}}\\&=-\frac12\bracket{Z_{t}, \mathbb H}\partial_\aa \paren{ \bracket{\frac1{Z_{,\aa}},\mathbb H}\mathbb P_A\paren{Z_{t} \partial_\aa \frac1{Z_{,\aa} }}}
\end{aligned}
\end{equation} 
 This allows us to conclude, by \eqref{dl21-inq}, and \eqref{2390}, \eqref{2359},\footnote{For the estimate $\nm{\partial_\aa \paren{ \frac1{Z_{,\aa}}\mathbb P_A\paren{Z_{t} \partial_\aa \frac1{Z_{,\aa} }}} }_{L^2}\le C(\frak E(t))$, see \eqref{2043}-\eqref{2044}.}
 \begin{equation}\label{2402}
 \nm{\bracket{Z_{t}, \mathbb H}\partial_\aa \paren{ \frac1{Z_{,\aa}}\mathbb P_A\paren{Z_{t} \partial_\aa \frac1{Z_{,\aa} }}}-U_l\bracket{\Zf_{t}, \mathbb H}\partial_\aa \paren{ \frac1{\Zf_{,\aa}}\mathbb P_A\paren{\Zf_{t} \partial_\aa \frac1{\Zf_{,\aa} }}}
 }_{L^2}\lec \mathcal F(t)^{\frac12}.
 \end{equation}
 Now we substitute in the last term and rewrite further by \eqref{H2}, 
 \begin{equation}\label{2403}
\begin{aligned}
&\bracket{Z_{t}, \mathbb H}\partial_\aa \paren{ \frac1{Z_{,\aa}}\bar{\mathbb P_A\paren{Z_{t} \partial_\aa \frac1{Z_{,\aa} }}}}=\bracket{\mathbb P_A\paren{Z_{t}\partial_\aa\frac1{Z_{,\aa}}}, \mathbb H}\bar{\mathbb P_A\paren{Z_{t} \partial_\aa \frac1{Z_{,\aa} }}}\\&\qquad\qquad\qquad+ \bracket{\mathbb P_A\paren{\frac{Z_{t}}{Z_{,\aa}}}, \mathbb H}\partial_\aa\bar{\mathbb P_A\paren{Z_{t} \partial_\aa \frac1{Z_{,\aa} }}}\\&=(I-\mathbb H)\paren{ \mathbb P_A\paren{Z_{t}\partial_\aa\frac1{Z_{,\aa}}}\bar{\mathbb P_A\paren{Z_{t} \partial_\aa \frac1{Z_{,\aa} }}}   }  + \bracket{b, \mathbb H}\partial_\aa\bar{\mathbb P_A\paren{Z_{t} \partial_\aa \frac1{Z_{,\aa} }}}.
\end{aligned}
\end{equation} 
 Again, this puts it in the right form to allow us to conclude, from \eqref{2382}-\eqref{2383}, \eqref{q3}, and \eqref{dl221}, that
 \begin{equation}\label{2404}
 \nm{\bracket{Z_{t}, \mathbb H}\partial_\aa \paren{ \frac1{Z_{,\aa}}\bar{\mathbb P_A\paren{Z_{t} \partial_\aa \frac1{Z_{,\aa} }}}}-U_l\bracket{\Zf_{t}, \mathbb H}\partial_\aa \paren{ \frac1{\Zf_{,\aa}}\bar{\mathbb P_A\paren{\Zf_{t} \partial_\aa \frac1{\Zf_{,\aa} }}}}
 }_{L^2}\lec \mathcal F(t)^{\frac12}.
 \end{equation}
This finishes the proof of 
\begin{equation}\label{2405}
\nm{(\partial_t+b\partial_\aa)(b_\aa-2\Re D_\aa Z_t)-(\partial_t+\tb\partial_\aa)(\tb_\aa- 2\Re \tD_\aa\Zf_t )\circ l}_{L^2}\lec \mathcal F(t)^{1/2}.
\end{equation}
 
 \subsubsection{Controlling $ \nm{(\partial_t+b\partial_\aa)\paren{\frac{\frak a_t}{\frak a}\circ h^{-1}}}_{L^\infty}$} \label{dtati}
 We begin with \eqref{at} and take a $\partial_t+b\partial_\aa$ derivative. We get
 \begin{equation}\label{2406}
 (\partial_t+b\partial_\aa)\paren{\frac{\frak a_t}{\frak a}\circ h^{-1}}=\frac{(\partial_t+b\partial_\aa)^2A_1}{A_1}-\paren{\frac{(\partial_t+b\partial_\aa)A_1}{A_1}}^2+(\partial_t+b\partial_\aa)(b_\aa-2\Re D_\aa Z_t).
 \end{equation}
 We have controlled all the quantities on the right hand side of \eqref{2406} in \S\ref{proof}, except for $\|(\partial_t+b\partial_\aa)^2A_1\|_{L^\infty}$. 
 We proceed  from \eqref{dta1} and use \eqref{eq:c14} to compute,  
 \begin{equation}\label{2407}
 \begin{aligned}
(\partial_t +b\partial_\aa)^2 A_1&= -\Im \paren{\bracket{2\bracket{Z_{tt},\mathbb H}\bar Z_{tt,\alpha'}-[Z_{tt}, b; \bar Z_{t,\alpha'}]-[Z_{t}, b; \bar Z_{tt,\alpha'}]}}\\&
-\Im \paren{\bracket{Z_{ttt},\mathbb H}\bar Z_{t,\alpha'}+\bracket{Z_t,\mathbb H}\partial_\aa \bar Z_{ttt}-  (\partial_t +b\partial_\aa)[Z_t, b; \bar Z_{t,\aa}]},
\end{aligned}
\end{equation}
and we expand similarly
 \begin{equation}\label{2408}
 \begin{aligned}
(\partial_t +b\partial_\aa)[Z_t, b; \bar Z_{t,\aa}]&=[Z_{tt}, b; \bar Z_{t,\aa}]+[Z_t, (\partial_t +b\partial_\aa)b; \bar Z_{t,\aa}]+[Z_t, b; \bar Z_{tt,\aa}]\\&-\frac2{\pi i}\int \frac{(b(\aa,t)-b(\bb,t))^2(Z_t(\aa,t)-Z_t(\bb,t))   }{(\aa-\bb)^3} \bar Z_{t,\bb}\,d\bb
\end{aligned}
\end{equation}
Applying Cauchy-Schwarz inequality and Hardy's inequality, we get
 \begin{equation}\label{2409}
 \begin{aligned}
\nm{(\partial_t +b\partial_\aa)^2 A_1}_{L^\infty}&\lec \nm{Z_{tt,\aa}}_{L^2}^2+ \nm{Z_{tt,\aa}}_{L^2}\nm{b_\aa}_{L^\infty}\nm{Z_{t,\aa}}_{L^2}+\nm{Z_{ttt,\aa}}_{L^2}\nm{Z_{t,\aa}}_{L^2}\\&+\nm{\partial_\aa(\partial_t +b\partial_\aa)b}_{L^\infty}\nm{Z_{t,\aa}}_{L^2}^2+\nm{b_\aa}_{L^\infty}^2\nm{Z_{t,\aa}}_{L^2}^2
\end{aligned}
\end{equation}
Observe that all quantities on the right hand side of \eqref{2409} are controlled in \S\ref{proof} and in 
\eqref{2411}. 
This shows that
\begin{equation}\label{2412}
\nm{(\partial_t +b\partial_\aa)^2 A_1}_{L^\infty}\le C(\frak E(t)),\qquad  \nm{(\partial_t+b\partial_\aa)\paren{\frac{\frak a_t}{\frak a}\circ h^{-1}}}_{L^\infty}\le C(\frak E(t)).
\end{equation}

\subsubsection{Controlling $ \nm{(\partial_t+b\partial_\aa)\paren{\frac{\frak a_t}{\frak a}\circ h^{-1}}-(\partial_t+\tb\partial_\aa)\paren{\frac{\taf_t}{\taf} \circ \th^{-1} }\circ l}_{L^2}$ } \label{ddtat}
 By the expansions \eqref{dta1}, \eqref{2406}, and \eqref{2407}, \eqref{2408}, we see that by the results in \S\ref{prepare} and by \eqref{2405}, we can directly conclude the desired estimates for all but the following three 
 \begin{itemize}
 \item
 $ \nm{\bracket{Z_{ttt},\mathbb H}\bar Z_{t,\alpha'}-(\bracket{\Zf_{ttt},\mathbb H}\bar {\Zf}_{t,\alpha'})   \circ l}_{L^2}$;
 \item
 $ \nm{\bracket{Z_t,\mathbb H}\partial_\aa \bar Z_{ttt}-(\bracket{\Zf_t,\mathbb H}\partial_\aa \bar {\Zf}_{ttt} )  \circ l}_{L^2}$;
 \item 
 $ \nm{[Z_t, (\partial_t +b\partial_\aa)b; \bar Z_{t,\aa}]-U_l [\Zf_t, (\partial_t +b\partial_\aa)b; \bar {\Zf}_{t,\aa}]}_{L^2}$. 
 \end{itemize}
 The first two items can be analyzed similarly as in \S\ref{ddtba}. We begin with $\bracket{Z_{ttt},\mathbb H}\bar Z_{t,\alpha'}$ and rewrite it using $\mathbb H\bar Z_{t,\alpha'}=\bar Z_{t,\alpha'}$, and substitute in by \eqref{eq:dztt}, \eqref{eq:dzt},
 \begin{equation}\label{2413}
 \bracket{Z_{ttt},\mathbb H}\bar Z_{t,\alpha'}=(I-\mathbb H)(Z_{ttt} \bar Z_{t,\aa})=(I-\mathbb H)\paren{iA_1 \bar {D_\aa Z_t}\paren{D_\aa Z_t+\frac{\frak a_t}{\frak a}\circ h^{-1}}}.
 \end{equation}
From here we are ready to conclude from \eqref{q3}, \eqref{2359} that
  \begin{equation}\label{2414} \nm{\bracket{Z_{ttt},\mathbb H}\bar Z_{t,\alpha'}-(\bracket{\Zf_{ttt},\mathbb H}\bar {\Zf}_{t,\alpha'})   \circ l}_{L^2}\lec \mathcal F(t)^{1/2}.
  \end{equation}
 
 Now substitute in by \eqref{eq:dztt}, \eqref{eq:dzt}, and use the identity $\mathbb P_H+\mathbb P_A=I$, then use  \eqref{H2} and \eqref{comm-hilbe},
  \begin{equation}\label{2415}
  \begin{aligned}
  \bracket{Z_t,\mathbb H}\partial_\aa \bar Z_{ttt}&  = -i\bracket{Z_t,\mathbb H}\partial_\aa \paren{\frac{1}{Z_{,\aa}}(\mathbb P_H+\mathbb P_A)\paren{ A_1\paren{\bar{D_\aa Z_t}+\frac{\frak a_t}{\frak a}\circ h^{-1}}  } }\\&=
  -i\bracket{\mathbb P_A\paren{Z_t\partial_\aa\frac{1}{Z_{,\aa}}},\mathbb H}\mathbb P_H\paren{ A_1\paren{\bar{D_\aa Z_t}+\frac{\frak a_t}{\frak a}\circ h^{-1}}  }\\&-i\bracket{\mathbb P_A\paren{\frac{Z_t}{Z_{,\aa}}},\mathbb H}\partial_\aa\mathbb P_H\paren{ A_1\paren{\bar{D_\aa Z_t}+\frac{\frak a_t}{\frak a}\circ h^{-1}}  }\\&
  -i\bracket{Z_t,\mathbb H}\partial_\aa \mathbb P_H\paren{\frac{1}{Z_{,\aa}}\mathbb P_A\paren{ A_1\paren{\bar{D_\aa Z_t}+\frac{\frak a_t}{\frak a}\circ h^{-1}}  } }\\&
  =-i(I-\mathbb H)\paren{\mathbb P_A\paren{Z_t\partial_\aa\frac{1}{Z_{,\aa}}}\mathbb P_H\paren{ A_1\paren{\bar{D_\aa Z_t}+\frac{\frak a_t}{\frak a}\circ h^{-1}}  }}\\&-i\bracket{b,\mathbb H}\partial_\aa\mathbb P_H\paren{ A_1\paren{\bar{D_\aa Z_t}+\frac{\frak a_t}{\frak a}\circ h^{-1}}  }\\&
 -i\bracket{Z_t,\mathbb H}\partial_\aa \bracket{\mathbb P_H, \frac{1}{Z_{,\aa}}}\mathbb P_A\paren{ A_1\paren{\bar{D_\aa Z_t}+\frac{\frak a_t}{\frak a}\circ h^{-1}}  }. 
  \end{aligned}
  \end{equation}
 From here we can apply the Propositions in \S\ref{prepare} and \eqref{2382}-\eqref{2383}, \eqref{2390-1} to conclude
   \begin{equation}\label{2416} 
   \nm{\bracket{Z_{t},\mathbb H}\partial_\aa\bar Z_{ttt}-(\bracket{\Zf_{t},\mathbb H}\partial_\aa\bar {\Zf}_{ttt})   \circ l}_{L^2}\lec \mathcal F(t)^{1/2}.
  \end{equation}
  
  Now we consider the last term,  $[Z_t, (\partial_t +b\partial_\aa)b; \bar Z_{t,\aa}]$. The problem with this term is that we don't yet have the  estimate, $\nm{\partial_\aa(\partial_t +b\partial_\aa)b-(\partial_\aa(\partial_t +\tb\partial_\aa)\tb)\circ l}_{L^2}\lec \mathcal F(t)^{1/2}$, to apply Proposition~\ref{d32}. We will not prove this estimate. Instead, we will identify the trouble term in $\partial_\aa(\partial_t +b\partial_\aa)b$, and handle it differently. We compute, by \eqref{eq:c7}, \eqref{eq:c1-1}, 
  \begin{equation}\label{2417}
  \begin{aligned}
  \partial_\aa&(\partial_t +b\partial_\aa)b=b_\aa^2+(\partial_t +b\partial_\aa)b_\aa\\&=b_\aa^2+(\partial_t +b\partial_\aa)(b_\aa-2\Re D_\aa Z_t)-2\Re \{(D_\aa Z_t)^2\}+2\Re D_\aa Z_{tt};
  \end{aligned}
  \end{equation}
  observe that we have the  estimate for the first three terms. We expand the last term by substituting in \eqref{aa1}, 
  \begin{equation}\label{2418}
  2\Re D_\aa Z_{tt}= 2\Re\frac1{Z_{,\aa}}\partial_\aa{\frac{iA_1}{\bar Z_{,\aa}}}
    =\partial_\aa\paren{\frac{iA_1}{|Z_{,\aa}|^2}}-\frac{\partial_\aa\paren{iA_1}}{|Z_{,\aa}|^2}-2\frac{iA_1}{\bar Z_{,\aa}}\partial_\aa\frac1{Z_{,\aa}}.
  \end{equation}
 Substitute \eqref{2418} in \eqref{2417}, and then apply $\mathbb P_A$, writing the last term as a commutator; we get
 \begin{equation}\label{2419}
 \begin{aligned}
& \mathbb P_A \partial_\aa \paren{(\partial_t +b\partial_\aa)b - \frac{iA_1}{|Z_{,\aa}|^2}  }\\&=\mathbb P_A\paren{ b_\aa^2+(\partial_t +b\partial_\aa)(b_\aa-2\Re D_\aa Z_t)-2\Re \{(D_\aa Z_t)^2\} -i\frac{\partial_\aa A_1}{|Z_{,\aa}|^2}   } -\bracket{\frac{iA_1}{\bar Z_{,\aa}},\mathbb H}\partial_\aa\frac1{Z_{,\aa}};
 \end{aligned}
 \end{equation}
a direct application of the results in \S\ref{prepare}, \S\ref{ddtba} and \S\ref{da1z2} to the right hand side of \eqref{2419} yields  
 \begin{equation}\label{2420}
 \nm{\mathbb P_A \partial_\aa \paren{(\partial_t +b\partial_\aa)b - \frac{iA_1}{|Z_{,\aa}|^2}  }-U_l \mathbb P_A \partial_\aa \paren{(\partial_t +\tb\partial_\aa)\tb - \frac{i\tAone}{|\Zf_{,\aa}|^2}  }}_{L^2}\lec \mathcal F(t)^{1/2},
 \end{equation}
 which of course holds also for its real part. We know the real part
$$\Re\mathbb P_A \partial_\aa \paren{(\partial_t +b\partial_\aa)b- \frac{iA_1}{|Z_{,\aa}|^2}  }=\frac12\partial_\aa \paren{(\partial_t +b\partial_\aa)b+\mathbb H\paren{\frac{iA_1}{|Z_{,\aa}|^2}} }.
 $$
We split $[Z_t, (\partial_t +b\partial_\aa)b; \bar Z_{t,\aa}]$ in two:
 \begin{equation}\label{2421}
[Z_t, (\partial_t +b\partial_\aa)b; \bar Z_{t,\aa}]=[Z_t, (\partial_t +b\partial_\aa)b+\mathbb H\paren{\frac{iA_1}{|Z_{,\aa}|^2}}; \bar Z_{t,\aa}]-[Z_t,   \mathbb H\paren{\frac{iA_1}{|Z_{,\aa}|^2}}  ; \bar Z_{t,\aa}] 
 \end{equation}
 and  we can conclude from Proposition~\ref{d32} for the first term that,\footnote{The fact that $\nm{\partial_\aa\paren{(\partial_t +b\partial_\aa)b+\mathbb H\paren{\frac{iA_1}{|Z_{,\aa}|^2}}}}_{L^\infty}\le C(\frak E(t))$  follows from \eqref{2039-1} and \eqref{2052}.}
 \begin{equation}\label{2422}
  \nm{ [Z_t, (\partial_t +b\partial_\aa)b+\mathbb H\paren{\frac{iA_1}{|Z_{,\aa}|^2}}; \bar Z_{t,\aa}]-U_l[\Zf_t, (\partial_t +\tb\partial_\aa)\tb+\mathbb H\paren{\frac{i\tAone}{|\Zf_{,\aa}|^2}}; \bar {\Zf}_{t,\aa}] }_{L^2}\lec \mathcal F(t)^{1/2}.
 \end{equation}
 We are left with the term $[Z_t,   \mathbb H\paren{\frac{iA_1}{|Z_{,\aa}|^2}}  ; \bar Z_{t,\aa}] $. We will convert it to a form so that on which we can directly apply known results to conclude the desired estimate, $$\nm{ [Z_t, \mathbb H(\frac{iA_1}{|Z_{,\aa}|^2}); \bar Z_{t,\aa}]-U_l[\Zf_t, \mathbb H(\frac{i\tAone}{|\Zf_{,\aa}|^2}); \bar {\Zf}_{t,\aa}] }_{L^2}\lec \mathcal F(t)^{1/2}.$$    
 
 We need the following basic identities: 1. for $f$, $g$ satisfying $f=\mathbb H f$, $g=\mathbb H g$, 
  \begin{equation}\label{2446-1}
 [ f, g; 1]=0;
\end{equation}

 2. for $f, p, g$, satisfying  
 $g=\mathbb Hg$ and $p=\mathbb Hp$,
 \begin{equation}\label{2424}
 [\bar p, \mathbb P_H f; g]=[\mathbb P_H f, \bar p g; 1]= [f, \mathbb P_A(\bar p g); 1]
 \end{equation}
 \eqref{2446-1} can be verified by \eqref{comm-hilbe} and integration by parts. \eqref{2424} can be verified by  \eqref{2446-1}.

 We split
 \begin{equation}\label{2423}
 \bracket{Z_t,   \mathbb H\paren{\frac{iA_1}{|Z_{,\aa}|^2}}  ; \bar Z_{t,\aa}}=\bracket{Z_t,   2\mathbb P_H\paren{\frac{iA_1}{|Z_{,\aa}|^2}}  ; \bar Z_{t,\aa}}-\bracket{Z_t,   \frac{iA_1}{|Z_{,\aa}|^2}; \bar Z_{t,\aa}}=2I-II.
 \end{equation}
 Applying  \eqref{2424} to $I$ yields
 \begin{equation}\label{2425}
 I:=\bracket{Z_t,   \mathbb P_H\paren{\frac{iA_1}{|Z_{,\aa}|^2}}  ; \bar Z_{t,\aa}} 
 =\bracket{\frac{iA_1}{|Z_{,\aa}|^2}, \mathbb P_A(Z_t\bar Z_{t,\aa});1};
 \end{equation}
 substituting in \eqref{2425} the identity 
 \begin{equation}\label{2440}
 \frac{iA_1(\aa)}{|Z_{,\aa}|^2}-\frac{iA_1(\bb)}{|Z_{,\bb}|^2}=\paren{\frac{iA_1(\aa)}{Z_{,\aa}}-\frac{iA_1(\bb)}{Z_{,\bb}}}\frac1{\bar Z_{,\bb}}+\frac{iA_1(\aa)}{Z_{,\aa}}\paren{\frac1{\bar Z_{,\aa}}-\frac1{\bar Z_{,\bb}}};
 \end{equation}
gives
\begin{equation}\label{2441}
I=\frac1{\pi i}\int\frac{\paren{\mathbb P_A(Z_t\bar Z_{t,\aa})(\aa)-\mathbb P_A(Z_t\bar Z_{t,\bb})(\bb)}\paren{\frac{iA_1(\aa)}{Z_{,\aa}}-\frac{iA_1(\bb)}{Z_{,\bb}}    }\frac1{\bar Z_{,\bb}}}{(\aa-\bb)^2}\,d\bb;
\end{equation}
 here the second term disappears  because of  
 the fact \eqref{2446-1}. 
 Using the identity
 \begin{equation}\label{2442}
\frac{\mathbb P_A(Z_t\bar Z_{t,\aa})-\mathbb P_A(Z_t\bar Z_{t,\bb})}{\bar Z_{,\bb}}= \frac{\mathbb P_A(Z_t\bar Z_{t,\aa})}{\bar Z_{,\aa}}-\frac{\mathbb P_A(Z_t\bar Z_{t,\bb})}{\bar Z_{,\bb}}-\mathbb P_A(Z_t\bar Z_{t,\aa})\paren{\frac1 {\bar Z_{,\aa}}-\frac1{\bar Z_{,\bb}}}  
 \end{equation}
 we get
  \begin{equation}\label{2443}
I=\bracket{\frac{\mathbb P_A(Z_t\bar Z_{t,\aa})}{\bar Z_{,\aa}}, \frac{iA_1}{Z_{,\aa}}; 1}-\mathbb P_A(Z_t\bar Z_{t,\aa})\bracket{\frac1 {\bar Z_{,\aa}}, \frac{iA_1}{Z_{,\aa}}; 1   }
\end{equation}
 from here we are readily to conclude from Proposition~\ref{d33}. We now work on $II$. By \eqref{2440},
 \begin{equation}\label{2444}
 II:= \bracket{Z_t,   \frac{iA_1}{|Z_{,\aa}|^2}; \bar Z_{t,\aa}}=\bracket{Z_t,   \frac{iA_1}{Z_{,\aa}}; \bar {D_\aa Z_{t}}}+\frac{iA_1}{Z_{,\aa}} \bracket{Z_t,   \frac{1}{\bar Z_{,\aa}}; \bar Z_{t,\aa}};
 \end{equation}
 the first term can be handled by Proposition~\ref{d33}. We focus on the second term. By a \eqref{2442} type identity, we have
 \begin{equation}\label{2445}
 \begin{aligned}
 &\frac{1}{Z_{,\aa}} \bracket{Z_t,   \frac{1}{\bar Z_{,\aa}}; \bar Z_{t,\aa}}=\bracket{\frac{Z_t}{Z_{,\aa}},   \frac{1}{\bar Z_{,\aa}}; \bar Z_{t,\aa}} -\bracket{\frac{1}{Z_{,\aa}},   \frac{1}{\bar Z_{,\aa}}; Z_t\bar Z_{t,\aa}}\\&\qquad=\bracket{\mathbb P_A\paren{\frac{Z_t}{Z_{,\aa}}},   \frac{1}{\bar Z_{,\aa}}; \bar Z_{t,\aa}} -\bracket{\frac{1}{Z_{,\aa}},   \frac{1}{\bar Z_{,\aa}}; \mathbb P_A\paren{Z_t\bar Z_{t,\aa}}}\\&\qquad\qquad+\bracket{\mathbb P_H\paren{\frac{Z_t}{Z_{,\aa}}},   \frac{1}{\bar Z_{,\aa}}; \bar Z_{t,\aa}} -\bracket{\frac{1}{Z_{,\aa}},   \frac{1}{\bar Z_{,\aa}}; \mathbb P_H\paren{Z_t\bar Z_{t,\aa}}}=I_1-I_2+I_3-I_4
 \end{aligned}
 \end{equation}
 The first two terms $I_1$, $I_2$ in \eqref{2445} can be handed by Propositions~\ref{d32} and \ref{d33}, because $\mathbb P_A \paren{\frac{Z_t}{Z_{,\aa}}}=\mathbb P_A b$. We need to manipulate further the last two terms. We begin with $I_4$, and use the first equality in \eqref{2424},  then use the identity $\mathbb P_H=-\mathbb P_A+I$,
 \begin{equation}\label{2446}
 \begin{aligned}
 &I_4:= \bracket{   \frac{1}{\bar Z_{,\aa}}, \frac{1}{Z_{,\aa}}; \mathbb P_H\paren{Z_t\bar Z_{t,\aa}}}
 =\bracket{\frac{1}{Z_{,\aa}}, \frac{ \mathbb P_H\paren{Z_t\bar Z_{t,\aa}}}{ \bar Z_{,\aa} }       ;1 }\\&=
- \bracket{\frac{1}{Z_{,\aa}}, \frac{ \mathbb P_A\paren{Z_t\bar Z_{t,\aa}}}{ \bar Z_{,\aa} }       ;1 }+\bracket{\frac{1}{Z_{,\aa}}, \mathbb P_A\paren{Z_t\mathbb P_H\bar {D_\aa Z_t}}    ;1 }+\bracket{\frac{1}{Z_{,\aa}}, Z_t\mathbb P_A\bar {D_\aa Z_t}    ;1 }
 \end{aligned}
 \end{equation}
 because of the fact \eqref{2446-1},  the $\mathbb P_A$ can be inserted in the second term.
 Now the first two terms on the right hand side of \eqref{2446} can be handled by Propositions~\ref{d33} and \ref{dhhalf2}, we need to work further on the last term, $$I_{43}:=\bracket{\frac{1}{Z_{,\aa}}, Z_t\mathbb P_A\bar {D_\aa Z_t}    ;1 }.$$ 
 We consider it together with $I_3$. 
By \eqref{2424}, 
 \begin{equation}\label{2447}
 I_3: =\bracket{ \frac{1}{\bar Z_{,\aa}}, \mathbb P_H\paren{\frac{Z_t}{Z_{,\aa}}}; \bar Z_{t,\aa}}=\bracket{\frac{Z_t}{Z_{,\aa}}, \mathbb P_A(\bar {D_\aa Z_t}); 1   }.
 \end{equation}
Sum up $I_3$ and $-I_{43}$ gives 
\begin{equation}\label{2448}
I_3-I_{43}=\frac1{\pi i} \int\frac{ \paren{\frac{\mathbb P_A(\bar {D_\aa Z_t})}{Z_{,\bb}}-   \frac{\mathbb P_A(\bar {D_\bb Z_t})}{Z_{,\aa}}   }(Z_t(\aa,t)-Z_t(\bb,t))     } {(\aa-\bb)^2}\,d\bb=
-\mathbb P_A(\bar{D_\aa Z_t})\bracket{Z_t, \frac1{Z_{,\aa}};1}
\end{equation}
 here  we used \eqref{2446-1} in the second step.
 
Through the steps in \eqref{2423}--\eqref{2448}, we have converted  $\bracket{Z_t,   \mathbb H\paren{\frac{iA_1}{|Z_{,\aa}|^2}}  ; \bar Z_{t,\aa}}$  into a sum of terms that can be handled with known results in \S\ref{prepare}-\S\ref{additional}. We can conclude now that
 \begin{equation}\label{2449}
 \nm{\bracket{Z_t,   \mathbb H\paren{\frac{iA_1}{|Z_{,\aa}|^2}}  ; \bar Z_{t,\aa}}-U_l\bracket{\Zf_t,   \mathbb H\paren{\frac{i\tAone}{|\Zf_{,\aa}|^2}}  ; \bar {\Zf}_{t,\aa}}}_{L^2}\lec \mathcal F(t)^{1/2}.
 \end{equation}
 Combine with \eqref{2422}, we obtain
  \begin{equation}\label{2450}
 \nm{\bracket{Z_t,   (\partial_t+b\partial_\aa)b ; \bar Z_{t,\aa}}-U_l\bracket{\Zf_t,    (\partial_t+\tb\partial_\aa)\tb   ; \bar {\Zf}_{t,\aa}}}_{L^2}\lec \mathcal F(t)^{1/2}.
 \end{equation}
 Now combine all the steps in \S\ref{ddtat}, we get
 \begin{equation}\label{2451}
  \nm{(\partial_t+b\partial_\aa)\paren{\frac{\frak a_t}{\frak a}\circ h^{-1}}-(\partial_t+\tb\partial_\aa)\paren{\frac{\taf_t}{\taf} \circ \th^{-1} }\circ l}_{L^2}\lec \mathcal F(t)^{1/2}.
 \end{equation}
 
 Combine all the steps above we have \eqref{denergy-inq}. This finishes the proof for Proposition~\ref{denergy-est},  and Theorem~\ref{unique}.

 \end{proof}
\section{The proof of Theorem~\ref{th:local}}\label{proof2}

For the data given in \S\ref{id}, we construct the solution of the Cauchy problem in the class where $\mathcal E<\infty$ via a sequence of approximating solutions  obtained by mollifying the initial data by the Poisson kernel, where we use Theorem~\ref{unique} and a compactness argument to prove the convergence of the sequence.
To prove the uniqueness of the solutions we use Theorem~\ref{unique}. 

 In what follows, we denote $z'=x'+iy'$, where $x', y'\in\mathbb R$.  $K$ is the Poisson kernel as defined by \eqref{poisson},  $f\ast g$ is the convolution in the spatial variable. For any function $\phi$, $\phi_\epsilon(x)=\frac1{\epsilon}\phi(\frac x\epsilon)$ for $x\in \mathbb R$.

\subsection{Some basic preparations}\label{analysis-2}    
Observe that in inequality \eqref{stability}, the stability is proved for the difference of the solutions in Lagrangian coordinates. We begin with some inequalities that will allow us to control the difference in Riemann mapping coordinates.

 We have

 \begin{lemma}\label{lemma4}
 Let $l:\mathbb R\to \mathbb R$ be a diffeomorphism with $l-\aa\in H^1(\mathbb R)$. Then  
 
 1. for any $f\in \dot H^1(\mathbb R)$, 
 \begin{equation}\label{lemma4-inq}
 \nm{f\circ l-f}_{\dot H^{1/2}}\lec \|\partial_\aa f\|_{L^2}\|l-\aa\|_{L^2}^{1/4}\|l_\aa-1\|_{L^2}^{1/4}
 +
 C(\nm{(l^{-1})_\aa}_{L^\infty}, \nm{l_\aa}_{L^\infty})\|l_\aa-1\|_{L^2}\|\partial_\aa f\|_{L^2}.
 \end{equation}
 
 2. for any  function $b:\mathbb R\to\mathbb R$, with $b_\aa\in H^{1/2}(\mathbb R)\cap L^\infty(\mathbb R)$,
 \begin{equation}\label{lemma4-inq2}
 \|b_\aa\circ l-b_\aa\|_{L^2}^2\lec \|b_\aa\|_{L^2}\|b_\aa\|_{L^\infty}\|l_\aa\|_{L^\infty}^{1/2}\|l_\aa-1\|_{L^2}+\|b_\aa\|_{\dot H^{1/2}}\|b\circ l-b\|_{\dot H^{1/2}}+\|b_\aa\|_{L^\infty}^2\|l_\aa-1\|_{L^2}^2.
 \end{equation}
 \end{lemma}
 \begin{proof}
 We know
 \begin{equation}\label{3000}
i \int\partial_\aa(f\circ l-f)\overline {(f\circ l-f)}\,d\aa=2\Re i\int\partial_\aa f\overline{(f-f\circ l)}\,d\aa
 \end{equation}
 so
 \begin{equation}\label{3001}
 \abs{i \int\partial_\aa(f\circ l-f)\overline {(f\circ l-f)}\,d\aa}\le 2\|\partial_\aa f\|_{L^2}\|f-f\circ l\|_{L^2}.
 \end{equation}
 Now
 \begin{equation}\label{3002}
 \int |f(\aa)-f(l(\aa))|^2\,d\aa\le \|l-\aa\|^2_{L^\infty}\int |\mathcal M(\partial_\aa f))(\aa)|^2\,d\aa\lec \|l-\aa\|^2_{L^\infty}\|\partial_\aa f\|^2_{L^2},
 \end{equation}
where $\mathcal M$ is the Hardy-Littlewood maximal operator. Therefore by  Sobolev embedding \eqref{eq:sobolev} and Lemma~ \ref{hhalf1},  
 \begin{align}
 \nm{f\circ l-f}_{\dot H^{1/2}}\lec \|\partial_\aa f\|_{L^2}\|l-\aa\|_{L^2}^{1/4}\|l_\aa-1\|_{L^2}^{1/4}
 +\nm{\mathbb P_A(f\circ l-f)}_{\dot H^{1/2}},\label{3006}\\
 \nm{f\circ l-f}_{\dot H^{1/2}}\lec \|\partial_\aa f\|_{L^2}\|l-\aa\|_{L^2}^{1/4}\|l_\aa-1\|_{L^2}^{1/4}
 +\nm{\mathbb P_H(f\circ l-f)}_{\dot H^{1/2}}.\label{3007}
 \end{align}
 Now 
 $$2\mathbb P_A (f\circ l-f)=(2\mathbb P_A f)\circ l-2\mathbb P_A f+\mathcal Q_l(f\circ l).$$
Applying \eqref{3007} to $(\mathbb P_A f)\circ l-\mathbb P_A f$   and using \eqref{q1} gives 
$$\|\mathbb P_A (f\circ l-f)\|_{\dot H^{1/2}}\lec  \|\partial_\aa f\|_{L^2}\|l-\aa\|_{L^2}^{1/4}\|l_\aa-1\|_{L^2}^{1/4}
+
 C(\nm{(l^{-1})_\aa}_{L^\infty}, \nm{l_\aa}_{L^\infty})\|l_\aa-1\|_{L^2}\|\partial_\aa f\|_{L^2}.
 $$
 This proves \eqref{lemma4-inq}. 
 
 To prove \eqref{lemma4-inq2}, we begin with
 \begin{equation}\label{3019}
 b_\aa\circ l-b_\aa=\partial_\aa(b\circ l-b)+b_\aa\circ l(1-l_\aa);
 \end{equation}
 and by expanding the integral, we have
 \begin{equation}\label{3018}
\|\partial_\aa(b\circ l-b)\|_{L^2}^2=\int (b_\aa\circ l)^2(l_\aa^2-l_\aa)\,d\aa+2\int b_\aa\partial_\aa(b-b\circ l)\,d\aa.
 \end{equation}
 \eqref{lemma4-inq2} follows directly from the Triangular, Cauchy-Schwarz and H\"older's inequalities. 
  \end{proof}
  
  \begin{lemma}\label{lemma3}
For any $\phi\in C^\infty(\mathbb R)$, with $\int\phi(x)\,dx=1$ and $\int |x\phi(x)|^2\,dx<\infty$, and  for any $f\in \dot H^1(\mathbb R)$, 
\begin{equation}\label{lemma3-inq}
\|\phi_\epsilon\ast f-f\|_{L^\infty}\lec \epsilon^{1/2}\|\partial_x f\|_{L^2}\|x\phi\|_{L^2}.
\end{equation}
\end{lemma}
The proof is straightforward by Cauchy-Schwarz inequality and Hardy's inequality \eqref{eq:77}. We omit the details. 
 
 Let $Z, \frak Z$ be solutions of the system \eqref{interface-r}-\eqref{interface-holo}-\eqref{a1}-\eqref{b}, satisfying the assumptions of Theorem~\ref{unique},  and let $l$ be given by \eqref{def-l}. 
 We know
$$(\partial_t+b\partial_\aa)(l-\aa)=U_{h^{-1}}(\th_t-h_t)=\tb\circ l-b,$$ and  $l(\aa, 0)=\aa$ for $\aa\in\mathbb R$. By Lemma~\ref{basic-e2},  \begin{equation}\label{3003}
\frac d{dt}\|l(t)-\aa\|^2_{L^2}\le 2\|\tb\circ l(t)-b(t)\|_{L^2}\|l(t)-\aa\|_{L^2}+\|b_\aa(t)\|_{L^\infty}\|l(t)-\aa\|_{L^2}^2,
\end{equation}
and from \eqref{b} and Sobolev embedding,
\begin{equation}\label{3010}
\|b(t)\|_{H^1(\mathbb R)}\lec \|Z_{t}(t)\|_{H^1(\mathbb R)}\paren{\nm{\frac1{Z_{,\aa}}(t)-1}_{H^1(\mathbb R)}+1}.
\end{equation}
Therefore by Gronwall's inequality, we have
 \begin{equation}\label{3004} 
\sup_{[0, T]}\|l(t)-\aa\|_{L^2(\mathbb R)}\le C,
\end{equation}
 where $C$ is a constant depending on $\sup_{[0, T]}\paren{\|Z_t(t)\|_{L^2}+\|\frak Z_t(t)\|_{L^2}+\nm{\frac1{Z_{,\aa}}(t)-1}_{L^2} +\nm{\frac1{\frak Z_{,\aa}}(t)-1}_{L^2}}$ and $\sup_{[0, T]} (\mathcal E(t)+\tilde{\mathcal E}(t) )$. 
 
 Let
 \begin{equation}\label{3008}
 \begin{aligned}
 \nm{(Z-\Zf)(0)}:=& \|\paren{\bar Z_t-\bar \Zf_t}(0)\|_{\dot{H}^{1/2}}+\|\paren{\bar Z_{tt}-\bar \Zf_{tt}}(0)\|_{\dot{H}^{1/2}}+\nm{\paren{\frac1{ Z_{,\aa}}-\frac 1{ \Zf_{,\aa}}}(0)}_{\dot{H}^{1/2}}\\&+\|\paren{D_\aa Z_t-(\tD_\aa \Zf_t)}(0)\|_{L^2}
 +\nm{\paren{\frac1{ Z_{,\aa}}-\frac 1{ \Zf_{,\aa}}}(0)}_{L^\infty}.
 \end{aligned}
 \end{equation}
Applying \eqref{lemma4-inq} to $f=\bar {\frak Z}_t$, $\frac1{\frak Z_{,\aa}}-1$ and $\bar {\frak Z}_{tt}$ and use \eqref{stability} gives
\begin{equation}\label{3005}
\begin{aligned}
\sup_{[0, T]}&\paren{\|\paren{\bar Z_t-\bar \Zf_t}(t)\|_{\dot{H}^{1/2}(\mathbb R)}+\nm{\paren{\frac1{ Z_{,\aa}}-\frac 1{ \Zf_{,\aa}}}(t)}_{\dot{H}^{1/2}(\mathbb R)}+\|\paren{\bar Z_{tt}-\bar \Zf_{tt}}(t)\|_{\dot{H}^{1/2}(\mathbb R)}}\\&\le C(\nm{(Z-\Zf) (0)}+  \nm{(Z-\Zf) (0)}  ^{1/4});
 \end{aligned}
\end{equation}
and applying \eqref{lemma4-inq2}, \eqref{lemma4-inq} to $\tb$ and use \eqref{3010}, \eqref{2393}, \eqref{2395}, Appendix~\ref{quantities} and \eqref{stability} yields
\begin{equation}\label{3020}
\sup_{[0, T]}\|(b_\aa-\tb_\aa)(t)\|_{L^2(\mathbb R)} \le C(\nm{(Z-\Zf) (0)}+  \nm{(Z-\Zf) (0)}  ^{1/8}),
\end{equation}
where $C$ is a constant depending on $\sup_{[0, T]}\paren{\|Z_t(t)\|_{L^2}+\|\frak Z_t(t)\|_{L^2}+\nm{\frac1{Z_{,\aa}}(t)-1}_{L^2} +\nm{\frac1{\frak Z_{,\aa}}(t)-1}_{L^2}}$ and $\sup_{[0, T]} (\mathcal E(t)+\tilde{\mathcal E}(t) )$. By Sobolev embedding \eqref{eq:sobolev}, 
\begin{equation}\label{3012}
\|l(\cdot,t)-\aa\|_{L^\infty(\mathbb R)}^2\lec \|l(\cdot,t)-\aa\|_{L^2(\mathbb R)}\|(l_\aa-1)(t)\|_{L^2(\mathbb R)},
\end{equation}
therefore by \eqref{3004}, \eqref{2345}-\eqref{2346}, and \eqref{stability},
\begin{equation}\label{3011}
\sup_{[0, T]}(\|h(t)-\th(t)\|^2_{L^\infty(\mathbb R)}+\|h^{-1}(t)-\th^{-1}(t)\|^2_{L^\infty(\mathbb R)})\le C \|(Z-\Zf)(0)\|,
\end{equation}
where C is a constant depending on $\sup_{[0, T]}\paren{\|Z_t(t)\|_{L^2}+\|\frak Z_t(t)\|_{L^2}+\nm{\frac1{Z_{,\aa}}(t)-1}_{L^2} +\nm{\frac1{\frak Z_{,\aa}}(t)-1}_{L^2}}$ and $\sup_{[0, T]} (\mathcal E(t)+\tilde{\mathcal E}(t) )$.

We also have, from Sobolev embedding \eqref{eq:sobolev}, \eqref{3010}, \eqref{3004}, \eqref{3012} and  \eqref{stability} that for $t\in [0, T]$, 
\begin{equation}\label{3013}
\begin{aligned}
&\nm{(b-\tb)(t)}_{L^\infty(\mathbb R)}^2\lec \nm{(b-\tb\circ l)(t)}_{L^\infty(\mathbb R)}^2+\nm{(\tb\circ l-\tb)(t)}_{L^\infty(\mathbb R)}^2\\&\lec \nm{(b-\tb\circ l)(t)}_{L^2(\mathbb R)}\nm{\partial_\aa(b-\tb\circ l)(t)}_{L^2(\mathbb R)}+\nm{l(t)-\aa}_{L^\infty(\mathbb R)}^2\nm{\tb_\aa(t)}_{L^\infty(\mathbb R)}^2\\&\le C(\|(Z-\Zf)(0)\|+\|(Z-\Zf)(0)\|^2),
\end{aligned}
\end{equation}
where C is a constant depending on $\sup_{[0, T]}\paren{\|Z_t(t)\|_{L^2}+\|\frak Z_t(t)\|_{L^2}+\nm{\frac1{Z_{,\aa}}(t)-1}_{L^2} +\nm{\frac1{\frak Z_{,\aa}}(t)-1}_{L^2}}$ and $\sup_{[0, T]} (\mathcal E(t)+\tilde{\mathcal E}(t) )$. 

 We have 
\begin{lemma}\label{lemma5}
1. Assume that $f\in H^{1/2}(\mathbb R)$. Then
\begin{equation}\label{lemma5-inq1}
\|f\|_{L^4(\mathbb R)}^2\lec \|f\|_{L^2(\mathbb R)}\|f\|_{\dot H^{1/2}(\mathbb R)}.
\end{equation}
2. Let $\phi\in C^\infty(\mathbb R)\cap L^q(\mathbb R)$, and $f\in L^p(\mathbb R)$, where $1\le p\le \infty$, $\frac1p+\frac1q=1$.  For any $y'<0$, $x'\in\mathbb R$,
\begin{equation}\label{lemma5-inq2}
|\phi_{y'}\ast f(x')| \le (-y')^{-1/p}\|\phi\|_{L^q(\mathbb R)}\|f\|_{L^p(\mathbb R)}.
\end{equation}
\end{lemma}

\begin{proof}
By the Theorem 1 on page 119 of \cite{s}, Plancherel's Theorem and Cauchy-Schwarz inequality, we have, for any $f\in H^{1/2}(\mathbb R)$, 
$$\|f\|_{L^4(\mathbb R)}\lec \|\partial_x^{1/4}f\|_{L^2(\mathbb R)}\lec  \|f\|_{L^2(\mathbb R)}\|f\|_{\dot H^{1/2}(\mathbb R)}.$$
\eqref{lemma5-inq2} is a direct consequence of  H\"older's inequality.

\end{proof}
We need in addition the following compactness results in the proof of the existence of solutions.

\begin{lemma}\label{lemma1} Let $\{f_n\}$ be a sequence of smooth functions on $\mathbb R\times [0, T]$. Let $1<p\le\infty$. Assume that there is a constant $C$, independent of $n$, such that 
\begin{equation}
\sup_{[0, T]}\|f_n(t)\|_{L^\infty}+ \sup_{[0, T]}\|{\partial_x f_n}(t)\|_{L^p}+ \sup_{[0, T]}\|\partial_t f_n(t)\|_{L^\infty}\le C.
\end{equation}
Then there is a function $f$, continuous and bounded on $\mathbb R\times [0, T]$,  and a subsequence $\{f_{n_j}\}$, such that $f_{n_j}\to f$ uniformly on compact subsets of $\mathbb R\times [0, T]$.
\end{lemma}
Lemma~\ref{lemma1} is an easy consequence of Arzela-Ascoli Theorem, we omit the proof.

\begin{lemma}\label{lemma2} 
Assume that $f_n\to f$ uniformly on compact subsets of $\mathbb R\times [0, T]$, and assume there is a constant $C$, such that $\sup_{n}\|f_n\|_{L^\infty(\mathbb R\times [0, T])}\le C$. Then $K_{y'}\ast f_n$ converges uniformly  to $K_{y'}\ast f$ on compact subsets of $\bar {\mathscr P}_-\times [0, T]$.
\end{lemma}
The proof follows easily by considering the convolution on the sets $|x'|<N$, and $|x'|\ge N$ separately. We omit the proof.

\begin{definition} We write
\begin{equation}\label{unif-notation}
f_n\Rightarrow f\qquad \text{on }E
\end{equation}
if $f_n$ converge uniformly to $f$ on compact subsets of $E$.
\end{definition}

 \subsection{The proof of Theorem~\ref{th:local}} The uniqueness of the solution to the Cauchy problem is a direct consequence of \eqref{3005} and Definition~ \ref{de}. In what follows we  prove the existence of solutions to the Cauchy problem.
\subsubsection{The initial data}\label{ID}
 Let $U(z', 0)$ be the initial fluid velocity in the Riemann mapping coordinate, $\Psi(z',0):{\mathscr P}_-\to\Omega(0)$ be the Riemann mapping as given in \S\ref{id} with $Z(\alpha', 0)=\Psi(\alpha', 0)$  the initial interface. We note that by the assumption
$$
\begin{aligned}&\sup_{y'<0}\nm{\partial_{z'}\paren{\frac1{\Psi_{z'}(z',0)}}}_{L^2(\mathbb R, dx')}\le \mathcal E_1(0)<\infty,\quad \sup_{y'<0}\nm{\frac1{\Psi_{z'}(z',0)}-1}_{L^2(\mathbb R, dx')}\le c_0<\infty,\\&
\sup_{y'<0}\|U_{z'}(z',0)\|_{L^2(\mathbb R, dx')}\le \mathcal E_1(0)<\infty\quad \text{and } \ \sup_{y'<0}\|U(z',0)\|_{L^2(\mathbb R, dx')}\le c_0<\infty,
\end{aligned}
$$
$ \frac1{\Psi_{z'}}(\cdot, 0)$, $U(\cdot, 0)$ can be extended continuously  onto  $\bar {\mathscr P}_-$.   
So $Z(\cdot,0):= \Psi(\cdot+i0, 0)$ is continuous differentiable on the open set where $\frac1{\Psi_{z'}}(\alpha', 0)\ne 0$, and $\frac1{\Psi_{z'}}(\alpha', 0)=\frac1{Z_{,\alpha'}(\alpha', 0)}$ where
$\frac1{\Psi_{z'}}(\alpha', 0)\ne 0$.
By $\frac1{\Psi_{z'}}(\cdot, 0)-1\in H^1(\mathbb R)$ and Sobolev embedding, there is $N>0$ sufficiently large, such that for $|\alpha'|\ge N$, $|\frac1{\Psi_{z'}}(\alpha', 0)-1| \le 1/2$, so $Z=Z(\cdot, 0)$ is continuous differentiable on $(-\infty, -N)\cup (N, \infty)$, with $|Z_{,\alpha'}(\alpha', 0)|\le 2$, for all $ |\alpha'|\ge N$. Moreover, $Z_{,\alpha'}(\cdot, 0)-1\in H^1\{(-\infty, -N)\cup (N, \infty)\}$.

\subsubsection{The mollified data and the approximate solutions}\label{mo-ap}
Let $\epsilon>0$. We take
\begin{equation}\label{m-id}
\begin{aligned}
Z^\epsilon(\alpha', 0)&=\Psi(\alpha'-\epsilon i, 0),\quad \bar Z^\epsilon_t(\alpha', 0)=U(\alpha'-\epsilon i, 0),\quad
h^\epsilon(\alpha,0)=\alpha,\\& U^\epsilon(z',0)=U(z'-\epsilon i, 0),\quad \Psi^\epsilon(z',0)=\Psi(z'-\epsilon i,0).
\end{aligned}
\end{equation}
Notice that $U^\epsilon(\cdot, 0)$, $\Psi^\epsilon(\cdot, 0)$ are holomorphic on ${\mathscr P}_-$, $Z^\epsilon(0)$ satisfies \eqref{interface-holo} and $\bar Z^\epsilon_t(0)=\mathbb H \bar Z^\epsilon_t( 0)$. Let $Z_{tt}^\epsilon(0)$ be given by \eqref{interface-a1}. 
It is clear that $Z^\epsilon(0)$, $ Z_t^\epsilon(0)$ and $Z_{tt}^\epsilon(0)$ satisfy the assumption of Theorem~\ref{blow-up}. Let
$Z^\epsilon(t):=Z^\epsilon(\cdot,t)$, $t\in [0, T_\epsilon^*)$,   be the solution given by Theorem~\ref{blow-up},  with the maximal time of existence $T_\epsilon^*$, the diffeomorphism $h^\epsilon(t)=h^\epsilon(\cdot,t):\mathbb R\to\mathbb R$, the quantity $b^\epsilon:=h_t^\epsilon\circ (h^\epsilon)^{-1}$, and $z^\epsilon(\alpha,t)=Z^\epsilon(h^\epsilon(\alpha,t),t)$. We know $z_t^\epsilon(\alpha,t)=Z_t^\epsilon(h^\epsilon(\alpha,t),t)$.
Let  $$U^\epsilon(x'+iy', t)=K_{y'}\ast \bar Z^\epsilon_t(x', t),\quad \Psi_{z'}^\epsilon(x'+iy', t) =K_{y'}\ast Z^\epsilon_{,\alpha'}(x',t),\quad \Psi^\epsilon(\cdot, t)$$
be the holomorphic functions on ${\mathscr P}_-$ with boundary values $\bar Z^\epsilon_t( t)$, $Z^\epsilon_{,\alpha'}(t)$ and $Z^\epsilon(t)$; we have
$$\frac1{\Psi_{z'}^\epsilon}(x'+iy', t) =K_{y'}\ast \frac1{Z^\epsilon_{,\alpha'}}(x',t)$$
by uniqueness.\footnote{By the maximum principle, $\big(K_{y'}\ast \frac1{Z^\epsilon_{,\alpha'}}\big)\big(K_{y'}\ast {Z^\epsilon_{,\alpha'}}\big)\equiv 1$ on ${\mathscr P}_-$. }
We denote the energy functional $\mathcal E$  for $\paren{Z^\epsilon(t), \bar Z_t^\epsilon(t)}$ by $\mathcal E^\epsilon(t)$ and the energy functional $\mathcal E_1$ for $(U^\epsilon(t),\Psi^\epsilon(t))$ by $\mathcal E^\epsilon_1(t)$. It is clear that $\mathcal E^\epsilon_1(0)\le \mathcal E_1(0)$, and $\|Z^\epsilon_t(0)\|_{L^2}+\nm{\frac1{Z^\epsilon_{,\aa}(0)}-1}_{L^2}\le c_0$ for all $\epsilon>0$; and by the continuity of $\frac1{\Psi_{z'}}(\cdot, 0)$ on $\bar{\mathscr P}_-$, there is a $\epsilon_0>0$, such that for all $0<\epsilon\le\epsilon_0$, $|\frac1{Z^\epsilon_{,\aa}}(0,0)|^2\le |\frac1{Z_{,\aa}}(0,0)|^2+1$.
By Theorem~\ref{blow-up}, 
Theorem~\ref{prop:a priori} and Proposition~\ref{prop:energy-eq}, there exists $T_0>0$, $T_0$ depends only on $\mathcal E(0)=\mathcal E_1(0)+|\frac1{Z_{,\aa}}(0,0)|^2$,\footnote{By \eqref{domain-energy1}.} such that for all $0<\epsilon\le \epsilon_0$, 
$T^*_\epsilon> T_0$ 
and 
\begin{equation}\label{eq:400}
\sup_{[0, T_0]}\paren{\mathcal E_1^\epsilon(t)+\abs{\frac1{Z^\epsilon_{,\aa}}(0,t)}^2}=\sup_{[0, T_0]}\mathcal E^\epsilon(t)\le M\paren{\mathcal E(0) }<\infty;
\end{equation}
and by \eqref{interface-a1}, \eqref{2109} and \eqref{2110},  
\begin{equation}\label{eq:401}
\sup_{[0, T_0]}\paren{\|Z^\epsilon_t(t)\|_{L^2}+\|Z^\epsilon_{tt}(t)\|_{L^2}+\nm{\frac1{Z^\epsilon_{,\alpha'}(t)}-1}_{L^2}}\le c\paren{c_0, \mathcal E(0)},
\end{equation}
so there is a constant $C_0:=C(c_0, \mathcal E(0))>0$, such that 
\begin{equation}\label{eq:402}
\sup_{[0,T_0]}\{\sup_{y'<0}\|U^\epsilon(\cdot+iy', t)\|_{L^2(\mathbb R)}+\sup_{y'<0}\nm{\frac1{\Psi^\epsilon_{z'}(\cdot+iy',t)}-1}_{L^2(\mathbb R)}\}<C_0<\infty.
\end{equation}

\subsubsection{Uniformly bounded quantities}\label{ubound}
Besides \eqref{3005}, \eqref{3011} and \eqref{3013}, we would like to apply the compactness results Lemma~\ref{lemma1}, Lemma~\ref{lemma2} to pass to limits of some of the quantities. To this end we discuss  the boundedness properties of these quantities. We begin with two  inequalities. 

We have, from \eqref{eq:dza},
\begin{equation}\label{eq:411}
\nm{(\partial_t+b^\epsilon\partial_\aa) \frac1{Z^\epsilon_{,\alpha'}}(t)}_{L^\infty}\le \nm{\frac1{Z^\epsilon_{,\alpha'}}(t)}_{L^\infty}(\|b^\epsilon_{\alpha'}(t)\|_{L^\infty}+\|D_{\alpha'}Z^\epsilon_{t}(t)\|_{L^\infty})
\end{equation}
and by \eqref{eq:dztt}, 
\begin{equation}\label{eq:4400}
\|(\partial_t+b^\epsilon\partial_\aa)Z^\epsilon_{tt}(t)\|_{L^\infty}\le \|Z^\epsilon_{tt}(t)+i\|_{L^\infty}\paren{\|D_{\aa}Z^\epsilon_t(t)\|_{L^\infty}+\nm{\frac{\frak a^\epsilon_t}{\frak a^\epsilon}\circ (h^{\epsilon})^{-1}(t)}_{L^\infty}}.
\end{equation}

Let $0<\epsilon\le\epsilon_0$, and $M(\mathcal E(0))$, $c(c_0,\mathcal E(0))$, $C_0$ be the bounds in \eqref{eq:400}, \eqref{eq:401} and \eqref{eq:402}. 
By Proposition~\ref{prop:energy-eq}, Sobolev embedding, Appendix~\ref{quantities} and \eqref{eq:411}, \eqref{3010},  the following quantities are uniformly bounded with bounds depending only on $M(\mathcal E(0))$, $c(c_0,\mathcal E(0))$, $C_0$:
\begin{equation}\label{eq:404}
\begin{aligned}
&\sup_{[0, T_0]}\|Z^\epsilon_t(t)\|_{L^\infty}, \quad\sup_{[0, T_0]}\|Z^\epsilon_{t,\alpha'}(t)\|_{L^2}, \quad\sup_{[0, T_0]}\|Z^\epsilon_{tt}(t)\|_{L^\infty}, \quad\sup_{[0, T_0]}\|Z^\epsilon_{tt,\alpha'}(t)\|_{L^2},\\&
\sup_{[0, T_0]}\nm{\frac1{Z^\epsilon_{,\alpha'}}(t)}_{L^\infty}, \quad\sup_{[0, T_0]}\nm{\partial_{\alpha'}\frac1{Z^\epsilon_{,\alpha'}}(t)}_{L^2}, \quad\sup_{[0, T_0]}\nm{(\partial_t+b^\epsilon\partial_\aa)\frac1{Z^\epsilon_{,\alpha'}}(t)}_{L^\infty},\quad \sup_{[0, T_0]}\nm{b^\epsilon}_{L^\infty};
\end{aligned}
\end{equation}
and with a change of the variables and \eqref{2346}, \eqref{eq:4400} and Appendix~\ref{quantities},
\begin{equation}\label{eq:414}
\begin{aligned}
&\sup_{[0, T_0]}\|z^\epsilon_t(t)\|_{L^\infty}+ \sup_{[0, T_0]}\|z^\epsilon_{t\alpha}(t)\|_{L^2}+ \sup_{[0, T_0]}\|z^\epsilon_{tt}(t)\|_{L^\infty}\le C(c_0, \mathcal E(0)), \\&
\sup_{[0, T_0]}\nm{\frac{h^\epsilon_{\alpha}}{z^\epsilon_{\alpha}}(t)}_{L^\infty}+ \sup_{[0, T_0]}\nm{\partial_{\alpha}(\frac {h^\epsilon_{\alpha}}{z^\epsilon_{\alpha}})(t)}_{L^2}+ \sup_{[0, T_0]}\nm{\partial_t \frac{h^\epsilon_\alpha}{z^\epsilon_{\alpha}}(t)}_{L^\infty}\le C(c_0, \mathcal E(0)),\\&
\sup_{[0, T_0]}\|z^\epsilon_{tt}(t)\|_{L^\infty}+ \sup_{[0, T_0]}\|z^\epsilon_{tt\alpha}(t)\|_{L^2}+ \sup_{[0, T_0]}\|z^\epsilon_{ttt}(t)\|_{L^\infty}\le C(c_0, \mathcal E(0)).
\end{aligned}
\end{equation}
Observe that 
$h^\epsilon(\alpha, t)-\alpha=\int_0^t h^\epsilon_t(\alpha, s)\,ds$, so 
\begin{equation}\label{eq:425}
\sup_{\mathbb R\times [0, T_0]} |h^\epsilon(\alpha, t)-\alpha|\le T_0\sup_{[0, T_0]}\|h^\epsilon_t(t)\|_{L^\infty}\le T_0C(c_0,\mathcal E(0))<\infty.
\end{equation}
Furthermore by \eqref{2346} and Appendix~\ref{quantities}, there are $c_1, c_2>0$, depending only on $\mathcal E(0)$, such that
\begin{equation}\label{eq:416}
0<c_1\le\frac{h^\epsilon(\alpha,t)-h^\epsilon(\beta,t)}{\alpha-\beta}\le c_2<\infty,\qquad \forall \alpha,\beta\in \mathbb R, \ t\in [0, T_0].
\end{equation}

\subsubsection{Passing to the limit} 
It is easy to check by Lemma~\ref{lemma3} and \eqref{hhalf-1}, \eqref{hhalf42} that the sequence $(Z^\epsilon(0), \bar Z_t^\epsilon(0))$ converges in the norm $\|\cdot\|$ defined by \eqref{3008}, so 
by \eqref{3020}, \eqref{3011} and \eqref{3013}, there are functions $b$ and $h-\alpha$, continuous and bounded on $\mathbb R\times [0, T_0]$, with $h(\cdot, t):\mathbb R\to \mathbb R$ a homeomorphism for $t\in [0, T_0]$, $b_\aa\in L^\infty([0, T_0], L^2(\mathbb R))$, such that
\begin{equation}\label{3015}
\lim_{\epsilon\to 0} \paren{b^\epsilon,\, h^\epsilon,\, (h^\epsilon)^{-1}}=\paren{b,\, h, \, h^{-1}},\qquad \text{uniformly on } \mathbb R\times [0, T_0];
\end{equation}
\begin{equation}\label{3015-1}
\lim_{\epsilon\to 0} b^\epsilon_\aa=b_\aa \qquad \text{in } \ \ L^\infty([0, T_0], L^2(\mathbb R));
\end{equation}
and \eqref{eq:416} yields 
\begin{equation}\label{eq:420}
0<c_1\le\frac{h(\alpha,t)-h(\beta,t)}{\alpha-\beta}\le c_2<\infty,\qquad \forall \alpha,\beta\in \mathbb R, \ t\in [0, T_0].
\end{equation}

By Lemma~\ref{lemma1}, \eqref{eq:414} and \eqref{3005}, there are functions $w$, $u$, $q:=w_t$,  continuous and bounded on $\mathbb R\times [0, T_0]$, such that
\begin{equation}\label{eq:417}
  z^\epsilon_t\Rightarrow w,\quad \frac{h^\epsilon_{\alpha}}{z^\epsilon_\alpha}\Rightarrow u,\quad z^\epsilon_{tt}\Rightarrow q,\qquad \text{on } \mathbb R\times [0, T_0],
\end{equation}
as $ \epsilon\to 0$; 
this gives
\begin{equation}\label{eq:419}
\bar Z_t^\epsilon\Rightarrow w\circ h^{-1},\qquad \frac1{Z^\epsilon_{,\alpha'}}\Rightarrow u\circ h^{-1}, \quad\bar Z_{tt}^\epsilon\Rightarrow w_t\circ h^{-1},\qquad\text{on }\mathbb R\times [0, T_0]
\end{equation}
as  $ \epsilon\to 0$. \eqref{3005} also gives that
\begin{equation}\label{3016}
\lim_{\epsilon\to 0} \paren{\bar Z_t^\epsilon,\,  \frac1{Z^\epsilon_{,\alpha'}},\, \bar Z_{tt}^\epsilon}=  \paren{w\circ h^{-1}, \, u\circ h^{-1},\,  w_t\circ h^{-1}},\qquad\text{in } L^\infty([0, T_0], \dot H^{1/2}(\mathbb R)).
\end{equation}

Now 
 \begin{equation}\label{eq:421}
U^\epsilon(z',t)=K_{y'}\ast \bar Z_t^\epsilon,\qquad \frac1{\Psi^\epsilon_{z'}}(z',t)=K_{y'}\ast \frac1{Z^\epsilon_{,\alpha'}}.
\end{equation}
Let $U(z',t)=   K_{y'}\ast (w\circ h^{-1})(x', t)$, $\Lambda(z',t)=K_{y'}\ast (u\circ h^{-1})(x',t)$.
By Lemma~\ref{lemma2}, 
\begin{equation}\label{eq:422}
U^\epsilon(z',t)\Rightarrow U(z',t),\qquad \frac1{\Psi^\epsilon_{z'}}(z',t)\Rightarrow \Lambda(z',t)\qquad\text{on }\bar {\mathscr P}_-\times [0, T_0];
\end{equation}
as  $ \epsilon\to 0$. Moreover $U(\cdot,t)$, $\Lambda(\cdot,t)$ are holomorphic on ${\mathscr P}_-$ for each $t\in [0, T_0]$, and continuous on $\bar {\mathscr P}_-\times [0, T]$. Applying  Cauchy integral formula to the first limit in \eqref{eq:422} yields,  as  $ \epsilon\to 0$, 
\begin{equation}\label{eq:430}
U^\epsilon_{z'}(z',t)\Rightarrow U_{z'}(z',t) \qquad\text{on } {\mathscr P}_-\times [0, T_0].
\end{equation}

\subsubsection*{Step 1. The limit of $\Psi^\epsilon$}\label{step4.1}
We consider the limit of $\Psi^{\epsilon}$, as $\epsilon\to 0$.  Let $0<\epsilon\le\epsilon_0$. We know
\begin{equation}\label{eq:423}
\begin{aligned}
z^\epsilon(\alpha,t)&=z^\epsilon(\alpha,0)+\int_0^t z_t^\epsilon(\alpha, s)\,ds\\&
=\Psi(\alpha-\epsilon i, 0)+\int_0^t z_t^\epsilon(\alpha, s)\,ds,
\end{aligned}
\end{equation}
therefore
\begin{equation}\label{eq:424}
\begin{aligned}
Z^\epsilon(\alpha',t)-Z^\epsilon(\alpha',0)&
=\Psi((h^\epsilon)^{-1}(\alpha',t)-\epsilon i, 0)-\Psi(\alpha'-\epsilon i, 0)\\&+\int_0^t z_t^\epsilon((h^\epsilon)^{-1}(\alpha',t)   , s)\,ds.
\end{aligned}
\end{equation}
Let
\begin{equation}\label{eq:431}
W^\epsilon(\alpha',t):=\Psi((h^\epsilon)^{-1}(\alpha',t)-\epsilon i, 0)-\Psi(\alpha'-\epsilon i, 0)+\int_0^t z_t^\epsilon((h^\epsilon)^{-1}(\alpha',t)   , s)\,ds.
\end{equation}
Observe $Z^\epsilon(\alpha',t)-Z^\epsilon(\alpha',0)$ is the boundary value of the holomorphic function
$\Psi^\epsilon(z', t)-\Psi^\epsilon(z', 0)$. By \eqref{eq:417} and \eqref{3015}, $\int_0^t z_t^\epsilon((h^\epsilon)^{-1}(\alpha',t), s)\,ds\to \int_0^t w(h^{-1}(\alpha',t), s)\,ds$ uniformly on compact subsets of $\mathbb R\times [0, T_0]$, and by \eqref{eq:414}, $\int_0^t z_t^\epsilon((h^\epsilon)^{-1}(\alpha',t), s)\,ds$ is continuous and uniformly bounded in $L^\infty(\mathbb R\times [0, T_0])$. 
By the assumptions $\lim_{z'\to 0}\Psi_{z'}(z',0)=1$ and $\Psi(\cdot, 0)$ is continuous on $\bar {\mathscr P}_-$, and by \eqref{eq:425}, \eqref{3015}, $$\Psi((h^\epsilon)^{-1}(\alpha',t)-\epsilon i, 0)-\Psi(\alpha'-\epsilon i, 0)$$ is continuous and uniformly bounded in $L^\infty(\mathbb R\times [0, T_0])$ for $0<\epsilon<1$, and converges uniformly on compact subsets of $\mathbb R\times [0, T_0]$ as $\epsilon\to 0$. This gives\footnote{Because $W^\epsilon(\cdot,t)$ and $\partial_{\alpha'}W^\epsilon(\cdot,t):=Z^\epsilon_{,\alpha'}(\alpha',t)-Z^\epsilon_{,\alpha'}(\alpha',0)$ are continuous and bounded on $\mathbb R$, $\Psi^\epsilon_{z'}(z',t)-\Psi^\epsilon_{z'}(z',0)=K_{y'}\ast (\partial_{\alpha'}W^\epsilon)(x',t)=\partial_{z'}K_{y'}\ast W^\epsilon(x',t)$. \eqref{eq:426} holds because both sides of \eqref{eq:426} have the same value on $\partial{\mathscr P}_-$.}
\begin{equation}\label{eq:426}
\Psi^\epsilon(z', t)-\Psi^\epsilon(z', 0)=K_{y'}\ast W^\epsilon(x',t)
\end{equation}
and by Lemma~\ref{lemma2}, $\Psi^\epsilon(z', t)-\Psi^\epsilon(z', 0)$ converges uniformly on compact subsets of $\bar {\mathscr P}_-\times [0, T_0]$ to a function that is holomorphic on ${\mathscr P}_-$ for every $t\in [0, T_0]$ and continuous on $\bar {\mathscr P}_-\times [0, T_0]$. Therefore there is a function $\Psi(\cdot,t)$, holomorphic on ${\mathscr P}_-$ for every $t\in [0, T_0]$  and continuous on $\bar {\mathscr P}_-\times [0, T_0]$, such that
\begin{equation}\label{eq:427}
\Psi^\epsilon(z',t)\Rightarrow \Psi(z',t)\qquad\text{on }\bar {\mathscr P}_-\times [0, T_0]
\end{equation}
as $\epsilon\to 0$; as a consequence of the Cauchy integral formula, 
\begin{equation}\label{eq:428}
\Psi^\epsilon_{z'}(z',t)\Rightarrow \Psi_{z'}(z',t)\qquad\text{on } {\mathscr P}_-\times [0, T_0]
\end{equation} 
as $\epsilon\to 0$. Combining with \eqref{eq:422}, we have $\Lambda(z',t)=\frac1{\Psi_{z'}(z',t)}$, so $\Psi_{z'}(z',t)\ne 0$ for all $(z',t)\in \bar{\mathscr P}_-\times [0, T_0]$ and
\begin{equation}\label{eq:429}
\frac1{\Psi^\epsilon_{z'}(z',t)}\Rightarrow \frac1{\Psi_{z'}(z',t)}\qquad\text{on }\bar {\mathscr P}_-\times [0, T_0]
\end{equation}
as $\epsilon\to 0$.

Denote $Z(\alpha', t):=\Psi(\alpha', t)$, $\alpha'\in \mathbb R$, and $z(\alpha,t)=Z(h(\alpha,t),t)$. \eqref{eq:427} yields $Z^\epsilon(\alpha',t)\Rightarrow Z(\alpha',t)$, together with \eqref{3015} it implies
$z^\epsilon(\alpha,t)\Rightarrow z(\alpha,t)$ 
on $\mathbb R\times [0, T_0]$, as $\epsilon\to 0$.
Furthermore by \eqref{eq:423}, 
$$z(\alpha', t)=z(\alpha',0)+\int_0^t w(\alpha,s)\,ds,$$ so $w=z_t$. We denote $Z_t=z_t\circ h^{-1}$. 

\subsubsection*{Step 2. The limits of $\Psi^\epsilon_t$ and $U_t^\epsilon$} Observe that by \eqref{eq:431}, for fixed $\epsilon>0$,  $\partial_t W^\epsilon(\cdot, t)$ is a bounded function on $\mathbb R\times [0, T_0]$, so by \eqref{eq:426} and the dominated convergence Theorem, $\Psi^\epsilon_t=K_{y'}\ast \partial_t W^\epsilon$, hence $\Psi_t^\epsilon$ is bounded on ${\mathscr P}_-\times [0, T_0]$. 

Since for given $t\in [0, T_0]$ and $\epsilon>0$,  $\frac{\Psi^\epsilon_t}{\Psi^\epsilon_{z'}}$ is bounded and holomorphic on ${\mathscr P}_-$,  by \eqref{eq:271}, 
\begin{equation}\label{eq:432}
\frac{\Psi^\epsilon_t}{\Psi^\epsilon_{z'}}=K_{y'}\ast(\frac{Z^\epsilon_t}{Z^\epsilon_{,\alpha'}}-b^\epsilon).
\end{equation}
Therefore by \eqref{3015}, \eqref{eq:419} and Lemma~\ref{lemma2}, as $\epsilon\to 0$, $\frac{\Psi^\epsilon_t}{\Psi^\epsilon_{z'}}$ converges uniformly on compact subsets of $\bar {\mathscr P}_-\times [0, T_0]$ to a function that is holomorphic on ${\mathscr P}_-$ for each $t\in [0, T_0]$ and continuous on $\bar {\mathscr P}_-\times [0, T_0]$.  Hence we can conclude from \eqref{eq:427} and \eqref{eq:428} that $\Psi$ is continuously differentiable and 
\begin{equation}\label{eq:433}
\Psi^\epsilon_t\Rightarrow \Psi_t\qquad \text{on }{\mathscr P}_-\times [0, T_0]
\end{equation}
 as $\epsilon\to 0$.
 
Now we consider the limit of $U^\epsilon_t$ as $\epsilon\to 0$. Since for fixed $\epsilon>0$, 
$\partial_t Z_t^\epsilon=Z_{tt}^\epsilon-b^\epsilon Z_{t,\alpha'}^\epsilon$ is in $L^\infty(\mathbb R\times [0, T_0])$, by \eqref{eq:421} and the dominated convergence Theorem,
\begin{equation}\label{eq:434}
U^\epsilon_t(z',t)=K_{y'}\ast  \partial_t \bar Z_t^\epsilon=K_{y'}\ast (\bar Z_{tt}^\epsilon-b^\epsilon \bar Z_{t,\alpha'}^\epsilon).
\end{equation}
We rewrite
\begin{equation}\label{3017}
K_{y'}\ast (\bar Z_{tt}^\epsilon-b^\epsilon \bar Z_{t,\alpha'}^\epsilon)=K_{y'}\ast \bar Z_{tt}^\epsilon-(\partial_{x'} K_{y'})\ast (b^\epsilon \bar Z_{t}^\epsilon)+ K_{y'}\ast (b^\epsilon_\aa \bar Z_{t}^\epsilon).
\end{equation}
Now we apply \eqref{3015}, \eqref{3015-1}, \eqref{3016} and Lemma~\ref{lemma5} to each term on the right hand side of \eqref{3017}. We can conclude that 
 $U$ is continuously differentiable with respect to $t$,
  and 
\begin{equation}\label{eq:435}
U^\epsilon_t \Rightarrow U_t\qquad \text{on }{\mathscr P}_-\times [0, T_0]  
\end{equation}
as $ \epsilon\to 0$.

\subsubsection*{Step 3. The limit of $\mathfrak P^\epsilon$} By the calculation in \S\ref{general-soln}, we know 
there is a real valued function $\frak P^\epsilon$, such that
\begin{equation}\label{eq:437}
\Psi^\epsilon_{z'} 
U^\epsilon_t- {\Psi^\epsilon_t}U^\epsilon_{z'}+{\bar U^\epsilon}U^\epsilon_{z'} -i\Psi^\epsilon_{z'}=-(\partial_{x'}-i\partial_{y'})\frak P^\epsilon,\qquad\text{in }{\mathscr P}_-;
\end{equation}
and 
\begin{equation}\label{eq:438}
\frak P^\epsilon=constant,\qquad \text{on }\partial {\mathscr P}_-.
\end{equation}
Without loss of generality we take the $constant=0$. We now explore a few other properties of $\frak P^\epsilon$. Moving ${\bar U^\epsilon}U^\epsilon_{z'}=\partial_{z'}({\bar U^\epsilon}U^\epsilon)$ to the right of  \eqref{eq:437} gives
\begin{equation}\label{eq:440}
\Psi^\epsilon_{z'} 
U^\epsilon_t- {\Psi^\epsilon_t}U^\epsilon_{z'} -i\Psi^\epsilon_{z'}=-(\partial_{x'}-i\partial_{y'})(\frak P^\epsilon+ \frac12 |U^\epsilon|^2),\qquad\text{in }{\mathscr P}_-;
\end{equation}
Applying $(\partial_{x'}+i\partial_{y'})=2\bar \partial_{z'}$ to \eqref{eq:440} yields
\begin{equation}\label{eq:439}
-\Delta  (\frak P^\epsilon+ \frac12 |U^\epsilon|^2) =0,\qquad\text{in } {\mathscr P}_-.
\end{equation}
So $\frak P^\epsilon+ \frac12 |U^\epsilon|^2$ is a harmonic function on ${\mathscr P}_-$ with boundary value $\frac12 |\bar Z_t^\epsilon|^2$. On the other hand, it is easy to check that $\lim_{y'\to-\infty} (\Psi^\epsilon_{z'} 
U^\epsilon_t- {\Psi^\epsilon_t}U^\epsilon_{z'} -i\Psi^\epsilon_{z'})=-i$. Therefore 
\begin{equation}\label{eq:441}
\frak P^\epsilon(z',t)=- \frac12 |U^\epsilon(z',t)|^2-y + \frac12 K_{y'}\ast (|\bar Z_t^\epsilon|^2)(x',t).
\end{equation}
By \eqref{eq:422}, \eqref{eq:419} and Lemma~\ref{lemma2}, 
\begin{equation}\label{eq:442}
\frak P^\epsilon(z',t)\Rightarrow - \frac12 |U(z',t)|^2-y + \frac12 K_{y'}\ast (|\bar Z_t|^2)(x',t),\qquad \text{on }\bar {\mathscr P}_-\times [0, T_0] 
\end{equation}
as $\epsilon\to 0$. We write $$\frak P:=- \frac12 |U(z',t)|^2-y + \frac12 K_{y'}\ast (|\bar Z_t|^2)(x',t).$$ We have $\frak P$ is continuous on $\bar {\mathscr P}_-\times [0, T_0]$ with $\frak P \in C([0, T_0], C^\infty({\mathscr P}_-))$, and 
\begin{equation}\label{eq:443}
\frak P=0,\qquad \text{on }\partial {\mathscr P}_-.
\end{equation}
Moreover, since $K_{y'}\ast (|\bar Z_t^\epsilon|^2)(x',t)$ is harmonic on ${\mathscr P}_-$, by interior derivative estimate for harmonic functions and by \eqref{eq:422}, 
\begin{equation}\label{eq:4450}
(\partial_{x'}-i\partial_{y'})\frak P^\epsilon\Rightarrow (\partial_{x'}-i\partial_{y'})\frak P\qquad\text{on }{\mathscr P}_-\times [0, T_0]
\end{equation}
as $\epsilon\to 0$. And by \eqref{eq:441} and a similar argument as that in \eqref{eq:434}-\eqref{eq:435}, we have that $\frak P$ is continuously differentiable with respect to $t$ and 
\begin{equation}\label{eq:4451}
\partial_t \frak P^\epsilon\Rightarrow \partial_t\frak P\qquad\text{on }{\mathscr P}_-\times [0, T_0]
\end{equation}
as $\epsilon\to 0$.

\subsubsection*{Step 4. Conclusion} We now sum up Steps 1-3. We have shown that there are
functions $\Psi(\cdot, t)$ and $U(\cdot, t)$,  holomorphic on ${\mathscr P}_-$ for each fixed $t\in [0, T_0]$,  continuous on $\bar {\mathscr P}_-\times [0, T_0]$, and continuous differentiable on ${\mathscr P}_-\times [0, T_0]$, with $ \frac1{\Psi_{z'}}$ continuous on $\bar {\mathscr P}_-\times [0, T_0]$, 
 such that 
$\Psi^\epsilon \to \Psi$, $\frac1{\Psi^\epsilon_{z'}} \to \frac1{\Psi_{z'}}$, $ U^\epsilon\to U$ uniform on compact subsets of $\bar {\mathscr P}_-\times [0, T_0]$, $\Psi^\epsilon_t \to \Psi_t$, $\Psi^\epsilon_{z'}\to \Psi_{z'}$, $U^\epsilon_{z'}\to U_{z'}$ and $U^\epsilon_t\to U_t$ uniform on compact subsets of ${\mathscr P}_-\times [0, T_0]$,  as  $\epsilon\to 0$. We have also shown that there is $\frak P$,
continuous on $\bar {\mathscr P}_-\times [0, T_0]$ with $\frak P=0$ on $\partial {\mathscr P}_-$, and  continuous differentiable on ${\mathscr P}_-\times [0, T_0]$, such that $\frak P^\epsilon\to \frak P$ uniform on compact subsets of $\bar{\mathscr P}_-\times [0, T_0]$ and,  $(\partial_{x'}-i\partial_{y'})\frak P^\epsilon\to (\partial_{x'}-i\partial_{y'})\frak P$ and $\partial_t\frak P^\epsilon\to \partial_t \frak P$ uniformly on compact subsets of $  {\mathscr P}_-\times [0, T_0]$, as $\epsilon\to 0$. Let $\epsilon\to 0$ in equation \eqref{eq:437}, we have 
\begin{equation}\label{eq:444}
\Psi_{z'} 
U_t- {\Psi_t}U_{z'}+{\bar U}U_{z'} -i\Psi_{z'}=-(\partial_{x'}-i\partial_{y'})\frak P,\qquad\text{on }{\mathscr P}_-\times [0, T_0].
\end{equation}
This shows that $(U, \Psi, \frak P)$ is a solution of the Cauchy problem for the system \eqref{eq:273}-\eqref{eq:274}-\eqref{eq:275} in the sense of Definition~\ref{de}. Furthermore because of \eqref{eq:400}, \eqref{eq:402}, letting $\epsilon\to 0$ gives 
\begin{equation}
\sup_{[0, T_0]}\mathcal E(t)\le M(\mathcal E(0))<\infty.
\end{equation}
and
\begin{equation}
\sup_{[0,T_0]}\{\sup_{y'<0}\|U(x'+iy', t)\|_{L^2(\mathbb R, dx')}+\sup_{y'<0}\|\frac1{\Psi_{z'}(x'+iy',t)}-1\|_{L^2(\mathbb R,dx')}\}<C_0<\infty.
\end{equation}

By the argument at the end of \S\ref{general-soln}, if $\Sigma(t):=\{Z=\Psi(\aa,t)\,|\, \aa\in\mathbb R\}$ is a Jordan curve, then $\Psi(\cdot, t):{\mathscr P}_-\to \Omega(t)$, where $\Omega(t)$ is the domain bounded from the above by $\Sigma(t)$, is invertible; and the solution $(U,\Psi,\frak P)$ gives rise to a solution $(\bar{\bf v}, P):= (U\circ \Psi^{-1}, \frak P\circ \Psi^{-1})$ of the water wave equation \eqref{euler}. This finishes the proof for part 1 of Theorem~\ref{th:local}.

\subsection{The chord-arc interfaces} Now assume at time $t=0$, the interface $Z=\Psi(\alpha',0):=Z(\alpha',0)$, $\alpha'\in\mathbb R$ is chord-arc, that is,  there is $0<\delta<1$, such that
$$\delta \int_{\alpha'}^{\beta'} |Z_{,\alpha'}(\gamma,0)|\,d\gamma\le |Z(\alpha', 0)-Z(\beta', 0)|\le \int_{\alpha'}^{\beta'} |Z_{,\alpha'}(\gamma,0)|\,d\gamma,\quad \forall -\infty<\alpha'<  \beta'<\infty.$$
 We want to show there is $T_1>0$, depending only on $\mathcal E(0)$, such that for $t\in [0, \min\{T_0, \frac\delta{T_1}\}]$, the interface $Z=Z(\alpha',t):=\Psi(\alpha',t)$ remains chord-arc. We begin with
 \begin{equation}\label{eq:446}
- z^\epsilon(\alpha,t)+z^\epsilon(\beta,t)+z^\epsilon(\alpha,0)-z^\epsilon(\beta,0)=\int_0^t\int_{\alpha}^\beta z^\epsilon_{t\alpha}(\gamma,s)\,d\gamma\,ds
 \end{equation}
for $\alpha<\beta$. Because 
 \begin{equation}\label{eq:447}
 \frac d{dt} |z^\epsilon_{\alpha}|^2=2|z^\epsilon_{\alpha}|^2 \Re D_{\alpha}z^\epsilon_t,
 \end{equation}
 by  Gronwall's inequality, for $t\in [0, T_0]$, 
 \begin{equation}\label{eq:448}
 |z^\epsilon_{\alpha}(\alpha,t)|^2\le |z^\epsilon_{\alpha}(\alpha,0)|^2 e^{2\int_0^t |D_{\alpha}z^\epsilon_t(\alpha,\tau)|\,d\tau };
 \end{equation}
 so 
 \begin{equation}\label{eq:449}
 |z^\epsilon_{t\alpha}(\alpha,t)|\le |z^\epsilon_{\alpha}(\alpha,0)| |D_{\alpha}z^\epsilon_t(\alpha,t)|e^{\int_0^t |D_{\alpha}z^\epsilon_t(\alpha,\tau)|\,d\tau };
 \end{equation}
 by Appendix~\ref{quantities}, \eqref{eq:400} and Proposition~\ref{prop:energy-eq},
  \begin{equation}\label{eq:450}
\sup_{[0, T_0]} |z^\epsilon_{t\alpha}(\alpha,t)|\le |z^\epsilon_{\alpha}(\alpha,0)| C(\mathcal E(0));
 \end{equation}
 therefore for $t\in [0, T_0]$, 
 \begin{equation}\label{eq:451}
 \int_0^{t}\int_{\alpha}^\beta |z^\epsilon_{t\alpha}(\gamma,s)|\,d\gamma\,ds\le t C(\mathcal E(0))\int_{\alpha}^\beta |z^\epsilon_{\alpha}(\gamma,0)| \,d\gamma.
 \end{equation}
 Now $z^\epsilon(\alpha,0)=Z^\epsilon(\alpha,0)=\Psi(\alpha-\epsilon i, 0)$. Because $Z_{,\alpha'}(\cdot,0)\in L^1_{loc}(\mathbb R)$, and $Z_{,\alpha'}(\cdot,0)-1\in H^1(\mathbb R\setminus [-N, N])$ for some large $N$, 
 \begin{equation}\label{eq:452}
 \overline{\lim_{\epsilon\to 0}}\int_{\alpha}^\beta |\Psi_{z'}(\gamma-\epsilon i,0)|\,d\gamma\le \int_{\alpha}^\beta |Z_{,\alpha'}(\gamma, 0)|\,d\gamma.
 \end{equation}
 Let $\epsilon\to 0$ in \eqref{eq:446}. We get, for $t\in [0, T_0]$, 
 \begin{equation}\label{eq:453}
| |z(\alpha,t)-z(\beta,t)| -|Z(\alpha,0)-Z(\beta,0)||\le 
tC(\mathcal E_1(0))\int_{\alpha}^\beta |Z_{,\alpha'}(\gamma, 0)|\,d\gamma,
 \end{equation}
 hence for all $\alpha<\beta$ and $0\le t\le \min\{T_0, \frac{\delta}{2C(\mathcal E(0))}\}$, 
 \begin{equation}\label{eq:454}
\frac12\delta \int_\alpha^\beta |Z_{,\alpha'}(\gamma,0)|\,d\gamma\le  |z(\alpha,t)-z(\beta,t)|\le 2 \int_\alpha^\beta |Z_{,\alpha'}(\gamma,0)|\,d\gamma.
 \end{equation}
 This show that for $\le t\le \min\{T_0, \frac{\delta}{2C(\mathcal E(0))}\}$,
 $z=z(\cdot,t)$ is absolute continuous on compact intervals of $\mathbb R$, with $z_{\alpha}(\cdot,t)\in L_{loc}^1(\mathbb R)$, and is chord-arc.  So $\Sigma(t)=\{z(\alpha,t) \ | \ \alpha\in \mathbb R\}$ is Jordan. This finishes the proof of Theorem~\ref{th:local}.
 
\begin{appendix}

\section{Basic analysis preparations}\label{ineq}
We present in this section some basic analysis results that will be used in this paper. 
First we have, as a consequence of the fact that product of holomorphic functions is holomorphic, the following identity.

\begin{proposition}\label{prop:comm-hilbe}
Assume that $f,\ g \in L^2(\mathbb R)$. Assume either both $f$, $g$ are holomorphic: $f=\mathbb H f$, $g=\mathbb H g$, or both are anti-holomorphic: $f=-\mathbb H f$, $g=-\mathbb H g$. Then
\begin{equation}\label{comm-hilbe}
[f, \mathbb H]g=0.
\end{equation}
\end{proposition}

Let $f :\mathbb R\to\mathbb C$ be a function in $\dot H^{1/2}(\mathbb R)$, we define
\begin{equation}\label{def-hhalf}
\|f\|_{\dot H^{1/2}}^2=\|f\|_{\dot H^{1/2}(\mathbb R)}^2:= \int i\mathbb H \partial_x f(x) \bar f(x)\,dx=\frac1{2\pi}\iint\frac{|f(x)-f(y)|^2}{(x-y)^2}\,dx\,dy.
\end{equation}
We have the following results on $\dot H^{1/2}$ norms and $\dot H^{1/2}$ functions.

\begin{lemma}\label{hhalf1}
For any function $f\in \dot H^{1/2}(\mathbb R)$,
\begin{align}
\nm{f}_{\dot H^{1/2}}^2&=\nm{\mathbb P_H f}_{\dot H^{1/2}}^2+\nm{\mathbb P_A f}_{\dot H^{1/2}}^2;\label{hhalfp}\\
\int i\partial_\aa f \, \bar f\,d\aa&=\nm{\mathbb P_H f}_{\dot H^{1/2}}^2-\nm{\mathbb P_A f}_{\dot H^{1/2}}^2.\label{hhalfn}
\end{align}

\end{lemma}
\begin{proof}
Lemma~\ref{hhalf1} is an easy consequence of the decomposition $f=\mathbb P_H f+\mathbb P_A f$, the 
definition \eqref{def-hhalf} and the Cauchy integral Theorem. We omit the details.
\end{proof}

\begin{proposition}\label{prop:Hhalf}
Let $f,\ g\in C^1(\mathbb R)$. Then
\begin{align}
\nm{fg}_{\dot H^{1/2}}\lec \|f\|_{L^\infty}\|g\|_{\dot H^{1/2}}+\|g\|_{L^\infty}\|f\|_{\dot H^{1/2}};\label{hhalf-1}\\
\|g\|_{\dot H^{1/2}}\lesssim \|f^{-1}\|_{L^\infty}(\|fg\|_{\dot H^{1/2}}+\|f'\|_{L^2}\|g\|_{L^2}).\label{Hhalf}
\end{align}

\end{proposition}
The proof is straightforward from the definition of $\dot H^{1/2}$ and the Hardy's inequality. We omit the
details.

We next present the basic estimates we will rely on for this paper.  We start with the Sobolev inequality.

\begin{proposition}[Sobolev inequality]\label{sobolev}
Let $f\in C^1_0(\mathbb R)$. Then
\begin{equation}\label{eq:sobolev}
\|f\|_{L^\infty}^2\le 2\|f\|_{L^2}\|f'\|_{L^2}.
\end{equation}
\end{proposition}

\begin{proposition}[Hardy's inequalities]
\label{hardy-inequality}  
Let $f \in C^1(\mathbb R)$, 
with $f' \in L^2(\mathbb R)$. Then there exists $C > 0$ independent of $f$ such that for any $x \in \mathbb R$,
\begin{equation}
  \label{eq:77}
\abs{\int \f{(f(x) - f(y))^2}{(x-y)^2} dy} \le C \nm{f'}_{L^2}^2;
\end{equation}
and
\begin{equation}
  \label{eq:771}
\iint \f{|f(x) - f(y)|^4}{|x-y|^4} \,dx dy \le C \nm{f'}_{L^2}^4.
\end{equation}
\end{proposition}

Let $ H\in C^1(\mathbb R; \mathbb R^d)$, $A_i\in C^1(\mathbb R)$, $i=1,\dots m$, and $F\in C^\infty(\mathbb R)$. Define
\begin{equation}\label{3.15}
C_1(A_1,\dots, A_m, f)(x)=\text{pv.}\int F\paren{\frac{H(x)-H(y)}{x-y}} \frac{\Pi_{i=1}^m(A_i(x)-A_i(y))}{(x-y)^{m+1}}f(y)\,dy.
\end{equation}

\begin{proposition}\label{B1} There exist  constants $c_1=c_1(F, \|H'\|_{L^\infty})$, $c_2=c_2(F, \|H'\|_{L^\infty})$, such that 

1. For any $f\in L^2,\ A_i'\in L^\infty, \ 1\le i\le m, $
\begin{equation}\label{3.16}
\|C_1(A_1,\dots, A_m, f)\|_{L^2}\le c_1\|A_1'\|_{L^\infty}\dots\|A_m'\|_{L^\infty}\|f\|_{L^2}. 
\end{equation}
2. For any $ f\in L^\infty, \ A_i'\in L^\infty, \ 2\le i\le m,\ A_1'\in L^2$, 
\begin{equation}\label{3.17}
\|C_1(A_1,\dots, A_m, f)\|_{L^2}\le c_2\|A_1'\|_{L^2}\|A'_2\|_{L^\infty}\dots\|A_m'\|_{L^\infty}\|f\|_{L^\infty}.
\end{equation}
\end{proposition}
\eqref{3.16} is a result of Coifman, McIntosh and Meyer \cite{cmm}. \eqref{3.17} is a consequence of the Tb Theorem, a proof  is given in \cite{wu3}.

Let  $H$, $A_i$ $F$ satisfy the same assumptions as in \eqref{3.15}. Define
\begin{equation}\label{3.19}
C_2(A, f)(x)=\int F\paren{\frac{H(x)-H(y)}{x-y}}\frac{\Pi_{i=1}^m(A_i(x)-A_i(y))}{(x-y)^m}\partial_y f(y)\,dy.
\end{equation}
We have the following inequalities.
\begin{proposition}\label{B2} There exist constants $c_3$, $c_4$ and $c_5$, depending on $F$ and $\|H'\|_{L^\infty}$, such that 

1. For any $f\in L^2,\ A_i'\in L^\infty, \ 1\le i\le m, $
\begin{equation}\label{3.20}
\|C_2(A, f)\|_{L^2}\le c_3\|A_1'\|_{L^\infty}\dots\|A_m'\|_{L^\infty}\|f\|_{L^2}.
\end{equation}

2. For any $ f\in L^\infty, \ A_i'\in L^\infty, \ 2\le i\le m,\ A_1'\in L^2$,
\begin{equation}\label{3.21}
\|C_2(A, f)\|_{L^2}\le c_4\|A_1'\|_{L^2}\|A'_2\|_{L^\infty}\dots\|A_m'\|_{L^\infty}\|f\|_{L^\infty}.\end{equation}

3. For any $f'\in L^2, \ A_1\in L^\infty,\ \ A_i'\in L^\infty, \ 2\le i\le m, $
\begin{equation}\label{3.22}
\|C_2(A, f)\|_{L^2}\le c_5\|A_1\|_{L^\infty}\|A'_2\|_{L^\infty}\dots\|A_m'\|_{L^\infty}\|f'\|_{L^2}.\end{equation}

\end{proposition}

Using integration by parts, the operator $C_2(A, f)$ can be easily converted into a sum of operators of the form $C_1(A,f)$. \eqref{3.20} and \eqref{3.21} follow from \eqref{3.16} and \eqref{3.17}.  To get \eqref{3.22}, we rewrite $C_2(A,f)$ as the difference of the two terms $A_1C_1(A_2,\dots, A_m, f')$ and $C_1(A_2,\dots, A_m, A_1f')$ and apply \eqref{3.16} to each term.

\begin{proposition}\label{prop:half-dir}
 There exists a constant $C > 0$ such that for any $f, g, m$ smooth and decays fast at infinity,
\begin{align} 
 &\nm{[f,\HH]  g}_{L^2} \le C \nm{f}_{\dot{H}^{1/2}}\nm{g}_{L^2};\label{eq:b10}\\&
    \nm{[f,\HH] g}_{L^\infty} \le C \nm{f'}_{L^2} \nm{g}_{L^2};  \label{eq:b13}\\&
  \nm{[f,\HH] \partial_\aa g}_{L^2} \le C \nm{f'}_{L^2} \nm{g}_{\dot{H}^{1/2}};\label{eq:b11}\\&
  \nm{[f, m; \partial_\aa g]}_{L^2}\le C \nm{f'}_{L^2} \nm{m'}_{L^\infty}\nm{g}_{\dot{H}^{1/2}}.\label{eq:b111}
  \end{align}
\end{proposition}
Here $[f,g;h]$ is as given in \eqref{eq:comm}. \eqref{eq:b10} is straightforward  by Cauchy-Schwarz inequality and the definition of $\dot H^{1/2}$. \eqref{eq:b13} is straightforward from Cauchy-Schwarz inequality and Hardy's inequality \eqref{eq:77}. \eqref{eq:b11} and \eqref{eq:b111} follow from integration by parts,
then Cauchy-Schwarz inequality, Hardy's inequality \eqref{eq:77}, and the definition of $\dot H^{1/2}$.

\begin{proposition}
 There exists a constant $C > 0$ such that for any $f, \ g, \ h$, smooth and decay fast at spatial infinity, 
  \begin{align}
    \label{eq:b12}
    \nm{[f,g;h]}_{L^2} &
    \le C \nm{f'}_{L^2} \nm{g'}_{L^2} \nm{h}_{L^2};\\ \label{eq:b15}
     \nm{[f,g;h]}_{L^\infty}&\le C \nm{f'}_{L^2} \nm{g'}_{L^\infty} \nm{h}_{L^2};\\
     \label{eq:b16}
     \nm{[f,g;h]}_{L^\infty}&\le C \nm{f'}_{L^2} \nm{g'}_{L^2} \nm{h}_{L^\infty}.
  \end{align}
\end{proposition}
\eqref{eq:b12} follows directly from Cauchy-Schwarz inequality, Hardy's inequality \eqref{eq:77} and Fubini Theorem; \eqref{eq:b15} follows from Cauchy-Schwarz inequality, Hardy's inequality \eqref{eq:77} and the mean value Theorem; \eqref{eq:b16} follows from Cauchy-Schwarz inequality and Hardy's inequality \eqref{eq:77}.

\section{Identities}\label{iden}

\subsection{Commutator identities}\label{comm-iden}

We include here various commutator identities that are necessary for the proofs. The first set: \eqref{eq:c1}-\eqref{eq:c4} has already appeared in \cite{kw}.
\begin{align}
  \label{eq:c1}
  [\partial_t,D_\a] &= - (D_\a z_t) D_\a;\\ 
  \label{eq:c2}
    \bracket{\partial_t,D_\a^2} &   
        = -2(D_\a z_t) D_\a^2 - (D_\a^2 z_t) D_\a;\\
  \label{eq:c3}
     \bracket{\partial_t^2,D_\a} &=(-D_\a z_{tt}) D_\a + 2(D_\a z_t)^2 D_\a - 2(D_\a z_t) D_\a \partial_t;\\
  \label{eq:c5}
   \bracket{\partial_t^2 + i \mathfrak{a} \partial_\a, D_\a} &=  (-2D_\a z_{tt}) D_\a -2(D_\a z_t)\partial_t D_\a;
   \end{align}
   and
   \begin{equation} \label{eq:c4}
   \begin{aligned}
 \bracket{(\partial_t^2 + i \af \partial_\a),D_\a^2} & =(-4D_\a z_{tt}) D_\a^2 + 6(D_\a z_t)^2 D_\a^2  - (2D_\a^2 z_{tt}) D_\a \\&+ 6(D_\a z_t) (D_\a^2 z_t) D_\a - 2(D_\a^2 z_t) D_\a \partial_t - 4(D_\a z_t) D_\a^2 \partial_t.
   \end{aligned}
   \end{equation}

Let $$\mathcal P:=(\partial_t+b\partial_\aa)^2+i\mathcal A\partial_\aa.$$   Notice that $U_h^{-1}\partial_t U_h=\partial_t+b\partial_\aa$, $U_h^{-1}D_\a U_h=D_\aa$ and 
$\mathcal P=U_h^{-1}(\partial_t^2+i\frak a\partial_\a)U_h$, we precompose with $h^{-1}$ to equations  \eqref{eq:c1}-\eqref{eq:c4}, and get
\begin{align}
  \label{eq:c1-1}
  [\partial_t+b\partial_\aa,D_\aa] &= - (D_\aa Z_t) D_\aa;\\ 
  \label{eq:c2-1}
    \bracket{\partial_t+b\partial_\aa,D_\aa^2} &   
        = -2(D_\aa Z_t) D_\aa^2 - (D_\aa^2 Z_t) D_\aa;\\
  \label{eq:c3-1}
     \bracket{(\partial_t+b\partial_\aa)^2,D_\aa} &=(-D_\aa Z_{tt}) D_\aa + 2(D_\aa Z_t)^2 D_\aa - 2(D_\aa Z_t) D_\aa (\partial_t+b\partial_\aa);\\
  \label{eq:c5-1}
   \bracket{\mathcal P, D_\aa} &=  (-2D_\aa Z_{tt}) D_\aa -2(D_\aa Z_t)(\partial_t+b\partial_\aa) D_\aa;
   \end{align}
   and
   \begin{equation} \label{eq:c4-1}
   \begin{aligned}
 \bracket{\mathcal P,D_\aa^2} & =(-4D_\aa Z_{tt}) D_\aa^2 + 6(D_\aa Z_t)^2 D_\aa^2  - (2D_\aa^2 Z_{tt}) D_\aa \\&+ 6(D_\aa Z_t) (D_\aa^2 Z_t) D_\aa - 2(D_\aa^2 Z_t) D_\aa (\partial_t+b\partial_\aa) - 4(D_\aa Z_t) D_\aa^2 (\partial_t+b\partial_\aa).
   \end{aligned}
   \end{equation}

We need some additional commutator identities. In general,  for operators $A, B$ and $C$,
\begin{equation}\label{eq:c12}
[A, BC^k]=[A, B]C^k+ B[A, C^k]=[A, B]C^k+ \sum_{i=1}^k BC^{i-1}[A, C]C^{k-i}.
\end{equation}
We have 
\begin{align}
\label{eq:c7}
[\partial_t+b\partial_\aa, \partial_\aa]f&=-b_\aa\partial_\aa f;\\
\label{eq:c8}[(\partial_t+b\partial_\aa)^2, \partial_\aa]f&=-(\partial_t+b\partial_\aa)(b_\aa\partial_\aa f)-b_\aa\partial_\aa (\partial_t+b\partial_\aa)f;\\
\label{eq:c9}[i\mathcal A\partial_\aa,\partial_\aa]f&=-i\mathcal A_\aa \partial_\aa f;\\
\label{eq:c10}
[\mathcal P, \partial_\aa]f&=-(\partial_t+b\partial_\aa)(b_\aa\partial_\aa f)-b_\aa\partial_\aa (\partial_t+b\partial_\aa)f-i\mathcal A_\aa \partial_\aa f;\\
\label{eq:c11}[\partial_t+b\partial_\aa, \partial_\aa^2]f&=-\partial_\aa(b_\aa\partial_\aa f)-b_\aa\partial_\aa^2 f.
\end{align}
Here \eqref{eq:c8}, \eqref{eq:c11} are obtained by \eqref{eq:c12} and \eqref{eq:c7}. 
We also have
\begin{equation}\label{eq:c21}
[\partial_t+b\partial_\aa, \mathbb H]=[b,\mathbb H]\partial_{\alpha'}
\end{equation}
We compute
$$
\begin{aligned}
(\partial_t+b\partial_\aa) [f,\mathbb H]g&=[(\partial_t+b\partial_\aa) f,\mathbb H]g+\bracket{f, \bracket{\partial_t+b\partial_\aa, \mathbb H}}g+ [f,\mathbb H](\partial_t+b\partial_\aa) g\\&
=[(\partial_t+b\partial_\aa) f,\mathbb H]g+\bracket{f, \bracket{b, \mathbb H}\partial_\aa}g+ [f,\mathbb H](\partial_t+b\partial_\aa) g
\\&
=[(\partial_t+b\partial_\aa) f,\mathbb H]g+ [f,\mathbb H]\paren{(\partial_t+b\partial_\aa) g+b_\aa g}\\&
\qquad+\bracket{f, \bracket{b, \mathbb H}}\partial_\aa g -[b,\mathbb H](f_\aa g)-[f,\mathbb H](b_\aa g).
\end{aligned}
$$
It can be checked easily,  by  integration by parts,  that
 $$\bracket{f, \bracket{b, \mathbb H}}\partial_\aa g -[b,\mathbb H](f_\aa g)-[f,\mathbb H](b_\aa g)=-[f,b;g].$$
So
\begin{equation}\label{eq:c14'}
\begin{aligned}
(\partial_t+b\partial_\aa)& [f,\mathbb H]g=[(\partial_t+b\partial_\aa) f,\mathbb H]g\\&+ [f,\mathbb H]((\partial_t+b\partial_\aa) g+b_{\alpha'} g)-[f, b; g];
\end{aligned}
\end{equation}
with an application of \eqref{eq:c7} yields
\begin{equation}\label{eq:c14}
\begin{aligned}
(\partial_t+b\partial_\aa)& [f,\mathbb H]\partial_{\alpha'}g=
[(\partial_t+b\partial_\aa) f,\mathbb H]\partial_{\alpha'}g\\&+ [f,\mathbb H]\partial_{\alpha'}(\partial_t+b\partial_\aa) g-[f, b; \partial_{\alpha'}g].
\end{aligned}
\end{equation}
We compute, by \eqref{eq:c21}, \eqref{eq:c12} and \eqref{eq:c14} that
\begin{equation}\label{eq:c23}
\begin{aligned}
\bracket{(\partial_t+b\partial_\aa)^2, \mathbb H}f&=(\partial_t+b\partial_\aa)\bracket{b,\mathbb H}\partial_\aa f+\bracket{b,\mathbb H}\partial_\aa (\partial_t+b\partial_\aa)f
\\&=\bracket{(\partial_t+b\partial_\aa)b,\mathbb H}\partial_\aa f+2\bracket{b,\mathbb H}\partial_\aa (\partial_t+b\partial_\aa)f-[b,b; \partial_\aa f].
\end{aligned}
\end{equation}
 We also have
\begin{equation}\label{eq:c24}
\bracket{i\mathcal A\partial_\aa, \mathbb H}f=\bracket{i\mathcal A,\mathbb H}\partial_\aa f.
\end{equation}
Sum up \eqref{eq:c23} and \eqref{eq:c24} yields
\begin{equation}\label{eq:c25}
\bracket{\mathcal P, \mathbb H}f=\bracket{(\partial_t+b\partial_\aa)b,\mathbb H}\partial_\aa f+2\bracket{b,\mathbb H}\partial_\aa (\partial_t+b\partial_\aa)f-[b,b; \partial_\aa f]+\bracket{i\mathcal A,\mathbb H}\partial_\aa f.
\end{equation}

 We have, by product rules, that
\begin{equation}\label{eq:c15}
[(\partial_t+b\partial_\aa)^2, \frac{1}{Z_{,\alpha'}}]f=(\partial_t+b\partial_\aa)^2\paren{\frac{1}{Z_{,\alpha'}}}f+2(\partial_t+b\partial_\aa)\paren{\frac{1}{Z_{,\alpha'}}}(\partial_t+b\partial_\aa)f; 
\end{equation}
and 
\begin{equation}\label{eq:c17}
[i\mathcal A\partial_\aa, \frac{1}{Z_{,\alpha'}}]f=i\mathcal A\partial_\aa \paren{\frac{1}{Z_{,\alpha'}}} f;
\end{equation}
so
\begin{equation}\label{eq:c16}
[\mathcal P, \frac{1}{Z_{,\alpha'}}]f=(\partial_t+b\partial_\aa)^2\paren{\frac{1}{Z_{,\alpha'}}}f+2(\partial_t+b\partial_\aa)\paren{\frac{1}{Z_{,\alpha'}}}(\partial_t+b\partial_\aa)f+i\mathcal A\partial_\aa \paren{\frac{1}{Z_{,\alpha'}}} f.
\end{equation}
And we compute, by \eqref{eq:dza},
\begin{align}\label{eq:c26}
(\partial_t+b\partial_\aa)\paren{\frac{1}{Z_{,\alpha'}}}&=\frac{1}{Z_{,\alpha'}}(b_\aa-D_\aa Z_t);\\
\label{eq:c27}
(\partial_t+b\partial_\aa)^2\paren{\frac{1}{Z_{,\alpha'}}}&=\frac{1}{Z_{,\alpha'}}(b_\aa-D_\aa Z_t)^2+\frac{1}{Z_{,\alpha'}}(\partial_t+b\partial_\aa)(b_\aa-D_\aa Z_t). 
\end{align}

\section{Main quantities controlled by $\frak E$} \label{quantities}
We have shown in \S\ref{basic-quantities} that the following quantities are controlled by a polynomial of $\frak E$ (or equivalently by $\mathcal E$):
\begin{equation}\label{2020-1}
\begin{aligned}
&\nm{ D_\aa\bar Z_t}_{\dot H^{1/2}}, \quad\nm{\frac{1}{ Z_{,\alpha'}} D_\aa^2\bar Z_t}_{\dot H^{1/2}}, \quad \|\bar Z_{t,\aa}\|_{L^2},\quad \|D_\aa^2\bar Z_t\|_{L^2},\quad \nm{\partial_\aa \frac{1}{ Z_{,\alpha'}}}_{L^2},\quad \abs{\frac{1}{ Z_{,\alpha'}}(0,t)},\\&
\|A_1\|_{L^\infty}, \quad \|b_\aa\|_{L^\infty}, \quad \|D_\aa Z_t\|_{L^\infty},\quad \|D_\aa Z_{tt}\|_{L^\infty},
\quad \|(\partial_t+b\partial_\aa)\bar Z_{t,\aa}\|_{L^2} \\&
\|Z_{tt,\aa}\|_{L^2}, \quad \|D_\aa^2 \bar Z_{tt}\|_{L^2},\quad \|D_\aa^2  Z_{tt}\|_{L^2},\quad \|(\partial_t+b\partial_\aa)D_\aa^2\bar Z_t\|_{L^2},\quad \|D_\aa^2 Z_t\|_{L^2}.
\end{aligned}
\end{equation}
  In the remainder of \S\ref{proof0} we have controlled the following quantities by a polynomila of $\frak E$ (or equivalently by $\mathcal E$):
\begin{equation}\label{2020-2}
\begin{aligned}
& \nm{\frac{\frak a_t}{\frak a}\circ h^{-1}}_{L^\infty},\quad \nm{(\partial_t+b\partial_\aa)A_1}_{L^\infty},\quad \nm{\mathcal A_{\aa}}_{L^\infty},\quad \nm{\frac1{Z_{,\aa}}\partial_\aa \frac1{Z_{,\aa}}}_{L^\infty},\\& \nm{\partial_\aa(\partial_t+b\partial_\aa)\frac1{Z_{,\aa}}}_{L^2},\quad \nm{(\partial_t+b\partial_\aa)\partial_\aa\frac1{Z_{,\aa}}}_{L^2},\quad \nm{\partial_\aa(\partial_t+b\partial_\aa)b}_{L^\infty},\quad \nm{(\partial_t+b\partial_\aa)b_\aa}_{L^\infty},\\& \nm{(\partial_t+b\partial_\aa)D_\aa Z_t}_{L^\infty}, \quad
\nm{D_\aa\paren{\frac{\frak a_t}{\frak a}\circ h^{-1}}}_{L^2}.
\end{aligned}
\end{equation}

As a consequence of \eqref{2045} and \eqref{2020-1} we have 
$$\nm{D_{\alpha'}b_{\alpha'}}_{L^2}\lesssim C(\frak E).$$

\end{appendix}

\end{document}